\documentclass[msc,10pt, oneside]{ubcthesis}
\usepackage{afterpage}
\usepackage{longtable}
\usepackage{graphicx}
\usepackage{pdflscape}
\usepackage[numbers,sort&compress]{natbib}
\usepackage{psfrag}
\usepackage[percent]{overpic}
\usepackage[unicode=true,
  linktocpage,
  linkbordercolor={0.5 0.5 1},
  citebordercolor={0.5 1 0.5},
  linkcolor=blue]{hyperref}
  \usepackage{graphicx}
\usepackage{pdflscape}
\usepackage{psfrag}
\usepackage{afterpage}
\usepackage{graphicx}				
\usepackage{lipsum}								
\usepackage{amssymb}
\usepackage{amssymb}
\usepackage{color}
\usepackage{calligra}
\usepackage{amsmath}
\usepackage{amssymb}
\usepackage{xfrac}
\usepackage{stmaryrd}
\usepackage{booktabs}
\usepackage{bookman}
\usepackage{siunitx}
\usepackage{bm}
\usepackage{empheq}
\usepackage{amsthm}
\usepackage[margin=1.3in]{geometry}
\usepackage{float}
\usepackage{tikz}
\usepackage{tikz-cd}
\usepackage{tgtermes}
\usepackage{mathrsfs}
\usepackage{fancybox}
\usepackage{mathrsfs}
\usepackage{mathtools}
\newtheorem{defn}{Definition}
\numberwithin{defn}{section}

\numberwithin{naidefn}{section}
\newtheorem{proppy}{Proposition}
\numberwithin{proppy}{section}
\newtheorem{thm}{Theorem}
\numberwithin{thm}{section}
\newtheorem{lemmy}{Lemma}
\numberwithin{lemmy}{section}

\numberwithin{defnconj}{section}
\newtheorem{rmk}{Remark}
\numberwithin{rmk}{section}
\newtheorem{conj}{Conjecture}
\numberwithin{conj}{section}
\newtheorem{cory}{Corollary}
\numberwithin{cory}{section}
\newtheorem{Ex}{Example}
\numberwithin{Ex}{section}

\usepackage{mathtools}
\makeatletter
\newcommand{\dotr}[1]{%
  \mathpalette\@dotr{#1}%
}
\newcommand*{\@dotr}[2]{%
  \sbox0{$\m@th#1#2$}%
  \usebox{0}%
  \raisebox{\dimexpr\ht0-\height}{$\m@th#1\@smallbullet#1\bullet$}%
  \kern\scriptspace
}
\newcommand*{\@smallbullet}[2]{%
  \scalebox{.5}{$\m@th#1#2$}%
}
\makeatother
\newcommand{\Mod}[1]{\ (\mathrm{mod}\ #1)}

\definecolor{shadowcolor}{RGB}{0, 0, 102}

\makeatletter
\g@addto@macro\bfseries{\boldmath}
\makeatother

\newcommand*{\sheafhom}{\mathscr{H}\kern -3pt om}
\DeclareMathOperator{\shHom}{\mathscr{H}\text{\kern -3pt {\calligra\large om}}\,}

\newcommand\smvee{\raise0.4ex\hbox{$\scriptscriptstyle\vee$}}

\makeatletter
\renewcommand{\@chapapp}{}

\makeatother

\DeclarePairedDelimiter\floor{\lfloor}{\rfloor}

\institution{The University Of British Columbia}
\faculty{The Faculty of Graduate and Postdoctoral Studies}
\institutionaddress{Vancouver}

\previousdegree{B.Sc., The University of New Mexico, 2014}
\previousdegree{M.Sc. in Physics, The University of British Columbia, 2017}

\title{An Introduction to Modern Enumerative Geometry with Applications to the Banana Manifold}
\author{Stephen Pietromonaco}
\copyrightyear{}
\submitdate{\monthname\ \number\year} 
\program{Mathematics}

\renewcommand\thepart         {\Roman{part}}
\renewcommand\thechapter      {\arabic{chapter}}
\renewcommand\thesection      {\thechapter.\arabic{section}}
\renewcommand\thesubsection   {\thesection.\arabic{subsection}}
\renewcommand\thesubsubsection{\thesubsection.\arabic{subsubsection}}
\renewcommand\theparagraph    {\thesubsubsection.\arabic{paragraph}}

\floatstyle{ruled}
\newfloat{Program}{htbp}{lop}[chapter]

\begin{document}

\frontmatter


\maketitle                      

The following individuals certify that they have read, and recommend to the Faculty of Graduate and Postdoctoral Studies for acceptance, a thesis entitled:
$$\textbf{An Introduction to Modern Enumerative Geometry with Applications to the Banana Manifold}$$

submitted by \textbf{Stephen Pietromonaco} in partial fulfillment of the requirements of the degree of \textbf{Master's of Science} in \textbf{Mathematics}.\\
\\

\textbf{Examining Committee:}\\
\\
\textbf{Jim Bryan}\\
Supervisor\\
\\
\textbf{Georg Oberdieck}\\
Supervisory Committee Member

\begin{abstract}                

The banana manifold $X_{\text{ban}}$ is a smooth projective Calabi-Yau threefold fibered over $\mathbb{P}^{1}$ by abelian surfaces.  Each singular fiber contains a ``banana configuration of curves" which generates the three-dimensional lattice $\Gamma$ of curve classes supported in the fibers of $X_{\text{ban}} \to \mathbb{P}^{1}$.  The Donaldson-Thomas partition function of $X_{\text{ban}}$ in fiber classes was computed by J. Bryan \cite{bryan_donaldson-thomas_2018} to be the infinite product
\[Z_{\text{DT}}(X_{\text{ban}})_{\Gamma}= \prod_{d_{1}, d_{2}, d_{3} \geq 0} \prod_{k \in \mathbb{Z}} \big(1-Q_{1}^{d_{1}} Q_{2}^{d_{2}} Q_{3}^{d_{3}}t^{k} \big)^{-12 c(||\underline{\bf{d}}||, k)}\]  
where $||\underline{\bf{d}}|| = 2d_{1} d_{2} + 2d_{1} d_{3} + 2d_{2}d_{3} -d_{1}^{2}-d_{2}^{2}-d_{3}^{2}$, and $c(||\underline{\bf{d}}||, k)$ are coefficients of the equivariant elliptic genus of $\mathbb{C}^{2}$.  We observe that under a change of variables, $Z_{\text{DT}}(X_{\text{ban}})_{\Gamma}$ behaves formally like a Borcherds lift of (12 times) the equivariant elliptic genus.  

The main result of this thesis is that the associated Gromov-Witten potentials $F_{g}$ in genus $g \geq 2$ are meromorphic genus two Siegel modular forms of weight $2g-2$.  They arise as Maass lifts 
\[ F_{g} = \text{ML} \bigg(\frac{6 |B_{2g}|}{g (2g-2)!} E_{2g}(\tau) \Theta^{2}(\tau, z)\bigg)\]
of weak Jacobi forms of weight $2g-2$ and index 1 arising in an expansion of the elliptic genus in the equivariant parameter.  Here, $\Theta^{2}$ is the unique weak Jacobi form of weight -2 and index 1.  We show the equivariant elliptic genus of $\mathbb{C}^{2}$ encodes the Gopakumar-Vafa invariants of $X_{\text{ban}}$.  Therefore, one can regard $X_{\text{ban}}$ as an example where the generating functions of Gromov-Witten and Donaldson-Thomas invariants in fiber classes are produced by standard lifts of a modular object encoding the Gopakumar-Vafa invariants.  We note that because this is a Masters thesis, the first six chapters offer an extended introduction to the relevant background material, while the original results are presented in the final chapter.

\end{abstract}

\tableofcontents                

\chapter{Acknowledgements}      

First and foremost, I am indebted to my advisor, Professor Jim Bryan.  I am thankful not only for his tremendous expertise and guidance, but also his patience, enthusiastic support, and encouragement as I learned this material.  I could not have imagined a better supervisor.  Special thanks is due to Georg Oberdieck for reading a draft, providing helpful comments, and also offering a suggestion early on which ultimately became one of the major components of my results.  I am also especially thankful to Hiroki Aoki and Eric Sharpe for their assistance.  I would like to thank the following individuals for their helpful comments and enlightening discussions: Jake Bian, Elliot Cheung, Javier Gonzalez-Anaya, Sheldon Katz, Giorgos Korpas, Oliver Leigh, Yu-Hsiang Liu, and Nina Morishige.

\mainmatter

\chapter{Introduction}

The \emph{banana manifold} $X_{\text{ban}}$ is a smooth projective Calabi-Yau threefold fibered over $\mathbb{P}^{1}$ with generic fiber a smooth abelian surface.  To construct $X_{\text{ban}}$ explicitly, let $r:S \to \mathbb{P}^{1}$ be a generic rational elliptic surface.  There are 12 singular fibers of $r$, each of which is a nodal elliptic curve.  One can form the fibered product $S \times_{\mathbb{P}^{1}} S$ and consider the diagonal $\Delta$ as a Weil divisor.  There are 12 conifold singularities of $S \times_{\mathbb{P}^{1}} S$, all of which lie on $\Delta$.  We define the banana manifold to be
\begin{equation}
X_{\text{ban}} \coloneqq \text{Bl}_{\Delta}(S \times_{\mathbb{P}^{1}} S)
\end{equation}
which is a full conifold resolution of singularities.  There is a natural map $\pi : X_{\text{ban}} \to \mathbb{P}^{1}$ whose generic fibers are $E \times E$, where $E$ is a smooth elliptic curve.  There are 12 singular fibers of $\pi$, each containing a \emph{banana configuration} of curves -- this consists of three rational curves $C_{1}, C_{2}, C_{3}$ all meeting in two distinct points $p,q \in X_{\text{ban}}$ (see Figure \ref{fig: banana configuration}).  The classes in homology of $C_{1}, C_{2}, C_{3}$ generate the lattice of fiber curve classes
\[\Gamma = \text{ker}(\pi_{*}) \subset H_{2}(X_{\text{ban}}, \mathbb{Z}).\]

\begin{figure}[h] 
\centering
\begin{tikzpicture}[xshift=5cm,
		    scale = 1.0
		    ]

\begin{scope}  
\draw (0,0) ellipse (2.4 and 2);
\draw (0,0) ellipse (0.6cm and 2cm);
\draw (0,0) ellipse (1.2cm and 2cm);
\draw (-0.6,0) arc(180:360:0.6 and 0.3);
\draw[dashed](-0.6,0) arc(180:0:0.6cm and 0.3cm);
\draw (1.2,0) arc(180:360:0.6cm and 0.3cm);
\draw[dashed](1.2,0) arc(180:0:0.6cm and 0.3cm);
\draw (-2.4,0) arc(180:360:0.6cm and 0.3cm);
\draw[dashed](-2.4,0) arc(180:0:0.6cm and 0.3cm);

\draw (0,-2) node[below] {$p$};
\draw(0,2) node[above] {$q$};
\draw(0,0.6) node {$C_{2}$};
\draw(-1.8,0.6) node {$C_{1}$};
\draw(1.8,0.6)node {$C_{3}$};
\end{scope}


\begin{scope}[xshift=4.5cm,yshift=-2cm]
\draw (0,1.5)--(1.5,1.5)--(2.5,2.5)--(4,2.5);
\draw (1.5,0)--(1.5,1.5)--(2.5,2.5)--(2.5,4);
\draw (0.5,1.5)node{$||$};
\draw (3.5,2.5)node{$||$};
\draw (1.5,0.5)node{$-$};
\draw (2.5,3.5)node{$-$};
\draw (2.5,3.1)node[left]{$C_{3}$};
\draw (2,2)node[below right]{$C_{2}$};
\draw (0.9,1.5)node[above]{$C_{1}$};
\draw (1.5,1.5)node[below left]{$p$};
\draw (2.5,2.5)node[above right]{$q$};
\end{scope}
\end{tikzpicture}
\caption{A banana configuration of curves}
\label{fig: banana configuration}
\end{figure}
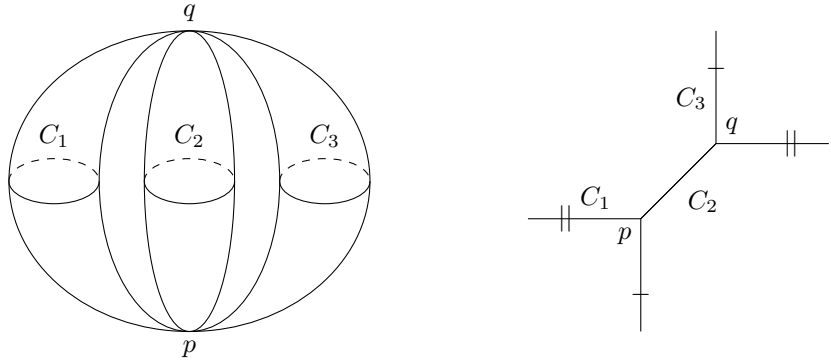
The Donaldson-Thomas partition function of $X_{\text{ban}}$ for the three fiber classes has been computed in a theorem of J. Bryan \cite{bryan_donaldson-thomas_2018}.  \emph{Using this result, the ultimate goal of this thesis is to understand the automorphic and arithmetic properties of the Donaldson-Thomas partition function, and to a much greater extent, the associated Gromov-Witten potentials.}

Since this is a Masters thesis, we take the opportunity to survey some of the necessary background material that a beginner would hopefully find useful.  Our work on the banana manifold not only requires understanding  Gromov-Witten, Donaldson-Thomas, and Gopakumar-Vafa invariants, but also equivariant instanton counting on $\mathbb{C}^{2}$, elliptic genera, automorphic forms, and arithmetic lifts.  As such, Chapters 2-6 offer partial introductions to these subjects, along with some of the interactions with physics, and a guide to the literature.  Our original results then come in Chapter 7.  For the sake of clarity, let us briefly outline the contents of each chapter, and how each topic fits into the thesis as a whole.  

\subsection{Outline of Topics}

We begin in Chapter 2 by studying Yang-Mills theory.  This is an exciting theory in its own right, though for those interested in algebraic or complex geometry, Yang-Mills theory really comes to life when studying holomorphic bundles on K\"{a}hler manifolds (Section \ref{sec:YMThhKahler}).  One of the main goals of this chapter is to understand in this setting how certain spaces of connections can be interpreted as spaces of holomorphic structures on bundles.  By the Donaldson-Uhlenbeck-Yau theorem, the stability of the bundle (a familiar notion to algebraic geometers) translates into the bundle carrying a unique connection solving the Hermitian Yang-Mills equation (something familiar to differential geometers and physicists).    

We also want to understand \emph{instantons}, which are anti-self-dual connections on bundles.  We will eventually show (\ref{eqn:rerrrnBLandSQEEG}, \ref{eqn:DTpartfuncBorLiftTWOO}) that the Donaldson-Thomas partition function of $X_{\text{ban}}$ is very closely related to (framed) instanton counting on $\mathbb{C}^{2}$.  

One reason for introducing Yang-Mills theory, is to consistently extend in Chapter 3 the notion of stability to coherent sheaves on smooth projective varieties.  We will introduce Gieseker stability and slope stability on torsion-free sheaves, as well as Simpson stability on pure sheaves (Section \ref{sec:StabConddCOHsh}).  These stability conditions are imposed to produce moduli schemes of stable or semistable sheaves with fixed topological features, so we also include a brief exposition on moduli problems in general (Section \ref{subsec:RevModProbb}), which will be useful in later chapters as well.  Along the way we hope to acquaint the reader with some basics on coherent sheaves and the Grothendieck group (or K-theory) of coherent sheaves (Section \ref{subsecc:GgrpHRR}).  

In Section \ref{sec:Dbranesstabstrth}, applying much of the previous material, we give a discussion of D-branes which we hope is approachable to mathematicians.  D-branes are objects in string theory which intrinsically carry a Yang-Mills theory on their support.  In algebraic geometry, a moduli space of sheaves with fixed topological type is what a physicist would call a moduli space of D-branes with fixed charges in K-theory.  The algebraic geometers typically impose stability or semistability which the physicist interprets as requiring the D-brane to produce a supersymmetric or BPS state.  We must mention that this story is very much incomplete without passing to the derived category, which we will not do in this thesis.  

Chapter 4 is devoted to introducing Gromov-Witten theory, Donaldson-Thomas theory, as well as the Gopakumar-Vafa invariants, all of which lie at the heart of modern enumerative geometry.  One common feature shared by these theories is that invariants are extracted from moduli spaces which are different compactifications of the space of smooth curves in a projective variety.  In Gromov-Witten theory we study curves via stable maps, while in Donaldson-Thomas theory we study curves as subschemes.  For the Gopakumar-Vafa invariants, the central object is not only a moduli space of pure stable torsion sheaves supported on curves, but also the Hilbert-Chow morphism to the Chow variety.  Clearly we will make use of our discussion in the previous chapter of moduli problems in general, and moduli spaces of sheaves in particular.  The way invariants are extracted from the moduli spaces is by using deformation and obstruction theories along with the existence of a virtual fundamental class.  These are very deep ideas, and we will only scratch the surface.  

We will also explain physical manifestations of these invariants (see Sections \ref{subsec:GWA-modTSSP}, \ref{subsecc:PHYSSDTtherry}, \ref{subsecc:PHYSGVinvv}).  Gromov-Witten theory is equivalent to the A-model topological string theory, and the Donaldson-Thomas invariants are quantities in the B-model topological string.  But one can think of them as a supersymmetric index computing a virtual number of BPS states of particles engineered by bound states of D2-D0 branes inside a single D6-brane in Type IIA string theory.  The Gopakumar-Vafa invariants are a virtual count of M2-branes in M-theory.  One should consult the D-brane section of Section \ref{sec:Dbranesstabstrth} to compliment this material.

One of the insights from physics is that these invariants should be packaged into generating functions.  For a smooth projective Calabi-Yau threefold $X$, the genus $g$ Gromov-Witten potential is (see Section \ref{sec:GWThhhh} for full details)
\begin{equation}
F_{g}(v) = \sum_{\beta \in H_{2}(X, \mathbb{Z})} \text{GW}_{g, \beta}(X) v^{\beta}.
\end{equation}
This is a generating function whose coefficients are virtual counts of genus $g$ curves in homology class $\beta$.  In physics, $F_{g}$ is a genus $g$ topological string amplitude.  Similarly, the Donaldson-Thomas partition function is defined to be (see Section \ref{eqn:DTThhhh})
\begin{equation}
Z_{\text{DT}}(X) = \sum_{\beta \in H_{2}(X, \mathbb{Z})} \sum_{n \in \mathbb{Z}} \text{DT}_{\beta, n}(X) v^{\beta} p^{n}.
\end{equation}
This is a generating function whose coefficients are virtual counts of ideal sheaves of curves and points in $X$.  Physically, $Z_{\text{DT}}(X)$ is the partition function of certain BPS black holes engineered from a single D6-brane, no D4-branes, and bound states of D2-D0 branes.  For the banana manifold, we will be interested in both $F_{g}$ and $Z_{\text{DT}}(X_{\text{ban}})$ restricted to the fiber classes described above.   

A fundamental quantity in our work on the banana manifold is the \emph{elliptic genus}.  As we introduce in Chapter 5, the elliptic genus $\text{Ell}_{q,y}(X)$ of a compact, complex manifold $X$ is a topological index generalizing the Euler characteristic, the Poincar\'{e} polynomial, and the $\chi_{y}$-genus.  We introduce equivariant cohomology and Atiyah-Bott localization as a method for computing these topological indices in certain cases (Section \ref{sec:EquivCohABLoc}).  In fact, we will be primarily interested in the elliptic genus of non-compact toric varieties, which one must \emph{define} via Atiyah-Bott localization.  This is an example of an equivariant index.  Of particular interest to the banana manifold, we will show that in the case of $\mathbb{C}^{2}$ this equivariant elliptic genus is
\begin{equation}\label{eqn:copiedEQELLGENC666}
\text{Ell}_{q, y}(\mathbb{C}^{2}; t) = y^{-1} \prod_{n=1}^{\infty} \frac{(1-yq^{n-1}t)(1-y^{-1}q^{n}t^{-1})(1-yq^{n-1}t^{-1})(1-y^{-1}q^{n}t)}{(1-q^{n-1}t)(1-q^{n}t^{-1})(1-q^{n-1}t^{-1})(1-q^{n}t)}
\end{equation}
where $t$ is a single equivariant parameter.  One can also define and compute the equivariant elliptic genus of Hilbert schemes of points on $\mathbb{C}^{2}$ via localization.  A generating function of the form 
\begin{equation}\label{eqn:NEKKY}
\sum_{m=0}^{\infty} Q^{m} \text{Ell}_{q,y}\big( \text{Hilb}^{m}(\mathbb{C}^{2}); t \big)
\end{equation} 
is an example of a Nekrasov partition function which is (the instanton part of) a partition function in $\mathcal{N}=2$ gauge theory.  In Section \ref{sec:FULLNEkkkSec} we study Nekrasov partition functions on $\mathbb{C}^{2}$ in some generality, replacing the Hilbert scheme by a higher rank instanton moduli space, and replacing the elliptic genus by a more general equivariant index. 

The theory of automorphic forms is becoming an increasingly powerful tool for modern enumerative geometers.  Certain generating functions (like those described above) may exhibit automorphy which could be due to hidden geometrical symmetries, and might motivate conjectures one can make.  To this end, in Chapter 6 we survey just a few kinds of automorphic forms relevant to the thesis: ordinary modular forms, Jacobi forms, and Siegel modular forms.  Of particular interest will be Jacobi forms (Section \ref{sec:IntroJACform}), which are two-variable holomorphic functions $\varphi_{k,m}(\tau, z)$ transforming under $SL_{2}(\mathbb{Z}) \ltimes \mathbb{Z}^{2}$ with weight $k$ and index $m$.  Jacobi forms admit a Fourier expansion
\begin{equation}
\varphi_{k,m}(\tau, z) = \sum_{n, l \in \mathbb{Z}} c(n, l) q^{n} y^{l}
\end{equation}
where we consistently use the change of variables $q = e^{2 \pi i \tau}$ and $y = e^{2 \pi i z}$.  We say the Jacobi form is \emph{weak} if $c(n, l)=0$ unless $n \geq 0$.  To touch base with the previous section, the elliptic genus of a compact Calabi-Yau manifold is a weak Jacobi form of weight zero.  

A genus two Siegel modular form is a three-variable holomorphic function $F(\tau, z, \sigma)$ transforming under the symplectic group $Sp_{4}(\mathbb{Z})$.  The ring of such objects is finitely-generated and we study the generators in Section \ref{subsec:RingofSMF}.  Jacobi forms and genus two Siegel modular forms are closely linked.  In Section \ref{sec:HeckeJacForm} we will be interested in certain arithmetic lifts of Jacobi forms.  Given a weak Jacobi form $\varphi_{k,1}$ of index one, we can define the Hecke operators (\ref{eqn:Heckedefn}, \ref{T0WEAKHOL}) denoted $V_{m}$ for $m \geq 0$, and form the \emph{Maass lift} of $\varphi_{k,1}$
\begin{equation}\label{eqn:IntroMLL}
\text{ML}(\varphi_{k,1}) = \sum_{m=0}^{\infty} Q^{m}\big( \varphi_{k,1} \big| V_{m} \big).  
\end{equation}
It turns out \cite{aoki_formal_2014,aoki_notitle_2018} that $\text{ML}(\varphi_{k,1})$ is a genus two meromorphic Siegel modular form of weight $k$, where $Q=e^{2 \pi i \sigma}$.  One can more generally define Jacobi forms $\Phi_{k}(\tau, \bm{w})$ of matrix index and many elliptic variables.  Generalizing the Maass lift, we can define the \emph{formal Borcherds lift} in the case of weight zero
\begin{equation}
\text{BL}(\Phi_{0}) = \text{exp} \big( \text{ML}(\Phi_{0}) \big).  
\end{equation}
This lift behaves formally like a Borcherds lift \cite{borcherds_automorphic_1995}, and takes the form of an infinite product, but we make no claims about the automorphy of the resulting quantity. 

Finally, in Chapter 7 we use much of the expositional material above to present our original results on the partition functions of the banana manifold $X_{\text{ban}}$.  The central quantity in our proposal is the equivariant elliptic genus (\ref{eqn:copiedEQELLGENC666})  
\begin{equation}\label{eqn:introPhi}
\Phi_{0}(\tau, z, x) = \text{Ell}_{q,y}(\mathbb{C}^{2}; t) = \sum_{g=0}^{\infty} \lambda^{2g-2} \psi_{2g-2}(\tau, z)
\end{equation}
where $t = e^{i \lambda}$ and $\lambda = 2 \pi x$.  We will show that $\Phi_{0}$ is a weak Jacobi form of weight zero and matrix index.  Expanding in $\lambda$ as shown above, the coefficients $\psi_{2g-2}$ are weak Jacobi forms of weight $2g-2$ and index one \cite{zhou_regularized_2015}, which for genus $g \geq 2$ are given explicitly as
\begin{equation}\label{eqn:psi2gm2intro}
\psi_{2g-2}(\tau, z) = \frac{|B_{2g}|}{2g (2g-2)!} \, E_{2g}(\tau) \Theta(\tau, z)^{2}
\end{equation}
where $B_{2g}$ is a Bernoulli number, $E_{2g}(\tau)$ is the Eisenstein series of weight $2g$, and $\Theta(\tau, z)^{2}$ is the unique (up to scale) weak Jacobi form of weight -2 and index 1 (\ref{eqn:varThetaa}).  Using results of J. Bryan \cite{bryan_donaldson-thomas_2018}, we show that the Donaldson-Thomas partition function of $X_{\text{ban}}$ restricted to the lattice of fiber classes $\Gamma$, is the formal Borcherds lift of $12 \Phi_{0}$ 
\begin{equation}
Z_{\text{DT}}(X_{\text{ban}})_{\Gamma} = \text{BL}(12 \Phi_{0}) = \prod_{(m,n,l,k)>0} \big( 1-Q^{m}q^{n}y^{l}t^{k}\big)^{-12 c(4nm-l^{2}, k)}
\end{equation}
where $c(4nm-l^{2}, k)$ are the Fourier coefficients of $\Phi_{0}$, and there is a simple change of variables (\ref{eqn:changeofvarss}) from the K\"{a}hler classes of the banana curves to $Q, q, y$.  We note that there are other geometries where a weight zero automorphic object lifts to produce the Donaldson-Thomas partition function \cite{kawai_string_2000, oberdieck_holomorphic_2018}. 

Our main result is the following: assuming the GW/DT correspondence for $X_{\text{ban}}$, the genus $g$ Gromov-Witten potentials $F_{g}$ for $g \geq 2$ are the Maass lifts of (12 times) the weak Jacobi forms $\psi_{2g-2}$
\begin{equation}
F_{g}(\tau, z, \sigma) = \text{ML}(12 \psi_{2g-2}).
\end{equation} 
As discussed above, it follows that $F_{g}$ is a meromorphic genus two Siegel modular form of weight $2g-2$.  This result can be partially explained through mirror symmetry (see Remark \ref{rmk:MIRRORSYYMM}).  We describe completely the denominators of the $F_{g}$ and can therefore in principal, compute $F_{g}$ explicitly for arbitrary $g$.  

We observe that the Gopakumar-Vafa invariants of $X_{\text{ban}}$ are encoded non-trivially into the equivariant elliptic genus $12 \Phi_{0}$.  We can summarize these phenomena on the banana manifold as follows: \emph{there exists a weight zero modular object encoding the Gopakumar-Vafa invariants which is lifted in standard ways to product the Donaldson-Thomas and Gromov-Witten theories in fiber classes.  The formal Borcherds lift produces the Donaldson-Thomas partition function.  Expanding in the equivariant parameter and taking the Maass lift of the coefficient Jacobi forms, we get the Gromov-Witten potentials, which are Siegel modular forms.  These results are compatible via the asymptotic statement of the GW/DT correspondence.}  

\begin{equation}
\begin{tikzcd}
Z_{\text{DT}}(X_{\text{ban}})_{\Gamma} \arrow[dashed, swap]{rrr}{\text{(Asymptotic) GW/DT Corresopndence}}     &     &      &      \sum_{g=0}^{\infty} \lambda^{2g-2} \text{ML}(12 \psi_{2g-2}) \\
                                                            &      &      &        \\
                                                            &      &      &        \\
                                                            &      &      &        \\
                                                            &      &  12 \Phi_{0}(\tau, z, x) = \sum_{g=0}^{\infty} \lambda^{2g-2} 12\psi_{2g-2}(\tau, z) \arrow{uuuull}{\text{Formal Borcherds Lift of}\, 12\Phi_{0}} \arrow[swap]{uuuur}{\text{Maass Lift of the}\, 12 \psi_{2g-2}}    &
\end{tikzcd}
\end{equation}


\chapter{Introduction to Yang-Mills Theory and Instantons}

Yang-Mills theory (also known as gauge theory) refers to the study of connections on principal bundles or associated vector bundles which solve the Yang-Mills equations.  Equivalently, these are special connections which locally minimize a natural action functional in physics.  Identifying connections up to gauge equivalence, one can construct moduli spaces of such connections, and study their topology and geometry.  As we will explain, the holomorphic structures on a Hermitian vector bundle over a complex manifold are in one-to-one correspondence with integrable Hermitian connections.  Therefore, Yang-Mills theory on complex manifolds can produce moduli spaces of bundles, and this fact is the origin of the contact made with algebraic geometry.  

Aside from being a beautiful theory itself, there are at least two related reasons an algebraic geometer should care about Yang-Mills theory.  First, Yang-Mills theory often provides a more intuitive and physical realization of stable bundles and sheaves, which we introduce in the next chapter.  This culminates in the Donaldson-Uhlenbeck-Yau theorem where the notion of a stable bundle is reinterpreted as the existence of an irreducible Hermitian Yang-Mills connection on the bundle.  In addition, D-branes are objects carrying a Yang-Mills theory which play a large role in modern mathematics and physics.  

The second major application of Yang-Mills theory to algebraic geometry comes in the form of instantons on smooth algebraic surfaces.  Fixing discrete invariants, we get finite dimensional moduli spaces of instantons which may be compactified (or partially compactified) by adding torsion-free sheaves.  A phenomenon known as geometric engineering is a highly non-trivial relationship between instantons on a surface and curves on a threefold.  One aspect of this thesis in the final chapter, will be studying the relationship between an instanton partition function and a Donaldson-Thomas partition function of a Calabi-Yau threefold.

\section{The Differential Geometry of Yang-Mills Theory}
In this section, we assume the reader is familiar with some foundational ideas in differential topology and geometry, specifically with regards to Lie groups and Lie algebras.  For more details on the topics to follow, one can consult \cite{huybrechts_complex_2004,friedman_gauge_1997,donaldson_geometry_1997,freed_instantons_1991}.

\subsection{Principal $\boldmath{G}$-Bundles}

Let $X$ be a smooth manifold, and let $G$ be a Lie group with Lie algebra $\mathfrak{g}$.  The most fundamental object in a Yang-Mills theory is a principal $G$-bundle, which we define now to be a special type of fiber bundle over $X$ with fiber $G$.  

\begin{defn}
A principal $G$-bundle (often shortened to principal bundle) is a fiber bundle $\pi : \mathcal{P} \to X$ with fiber diffeomorphic to $G$ such that the total space $\mathcal{P}$ is smooth, $\pi$ is a smooth surjection, and there is a smooth free and transitive right $G$-action on $\mathcal{P}$ preserving the fibers of $\pi$.  We will refer to $X$ as the base space and to $G$ as the structure group.  
\end{defn}

Until emphasized otherwise, we will take $X$ to be an arbitrary smooth manifold of any dimension.  In various contexts to follow, we may impose Riemannian or complex structure, as well as a specialization to two or four dimensions.  We will also eventually ask $G$ to be compact, but that is not necessary for now.  It is convenient to sometimes denote the right action on a principal bundle by $R_{g}$, for example when considering the pushforward of this action.  At other times, it will be denoted simply as multiplication by $g$ on the right.  

From the definition of a principal bundle, there are two main components to specify: the projection map $\pi$ and the free action.  Therefore, when defining morphisms of principal bundles, we expect them to be bundle morphisms commuting with the projection maps, but also satisfying an \emph{equivariance} property with respect to the group action.  

\begin{defn}
Let $\mathcal{P}$ and $\mathcal{P}'$ be two principal $G$-bundles with the same base space $X$ and projection maps $\pi$ and $\pi'$, respectively.  A principal bundle morphism from $\mathcal{P}$ to $\mathcal{P}'$ consists of a smooth map $\varphi: \mathcal{P} \to \mathcal{P}'$ compatible with the projections, and equivariant with respect to the group actions.  This is summarized in the commutative diagram (\ref{eqn:commdiag1}) where the equivariance condition is $R'_{g} \circ \varphi = \varphi \circ R_{g}$.  

\begin{equation}\label{eqn:commdiag1}
\begin{tikzcd}
\mathcal{P} \arrow{rr}{\varphi} & & \mathcal{P}' \\
\mathcal{P} \arrow{rr}{\varphi} \arrow{dr}{\pi} \arrow{u}{R_{g}} & & \mathcal{P}'\arrow[swap]{dl}{\pi'}\arrow[swap]{u}{R'_{g}} \\
& X &
\end{tikzcd}
\end{equation}
A principal bundle isomorphism is an equivariant diffeomorphism $\varphi: \mathcal{P} \to \mathcal{P}'$ compatible with projections.  We call a principal bundle $\mathcal{P}$ trivial if it is isomorphic as a principal bundle to $X \times G$.  
\end{defn}

\noindent This definition can be generalized to allow for different base spaces and even different structure groups.  For our purposes, the above definition will suffice.  

\begin{rmk}
The set of principal $G$-bundle automorphisms $\text{Aut}\mathcal{P}$ forms a group under composition.  The precise structure of $\text{Aut}\mathcal{P}$ as well as its crucial role in Yang-Mills theory will be established in subsequent sections.  
\end{rmk}

Vector bundles always have global sections, but the following proposition shows this to be false for principal bundles.    

\begin{proppy}
A principal bundle $\mathcal{P}$ admits a global section if and only if it is trivial.  
\end{proppy}
 
\begin{proof}
If $\mathcal{P}$ is trivial, there exists an equivariant diffeomorphism $\varphi: X \times G \to \mathcal{P}$, and $\varphi( -, e): X \to \mathcal{P}$ defines a global section.  Conversely, assume $s: X \to \mathcal{P}$ is a global section.  We define equivariant maps $\varphi: X \times G \to \mathcal{P}$ by $\varphi(x,g) = R_{g}s(x)$ and $\widetilde{\varphi}: \mathcal{P} \to X \times G$ by $\widetilde{\varphi}(p) = \big(\pi(p), g\big)$ where $g \in G$ is the unique group element such that $R_{g}s\big(\pi(p)\big) =p$.  A simple verification shows that $\varphi$ and $\widetilde{\varphi}$ are equivariant diffeomorphisms and mutual inverses.  
\end{proof}

Because a principal bundle is in particular a fiber bundle, it must trivialize on some open cover of $X$.  

\begin{defn}
A local trivialization of $\mathcal{P}$ consists of an open cover $\{U_{\alpha} \}$ of $X$, and equivariant maps $g_{\alpha}: \pi^{-1}U_{\alpha} \to G$ making the following diagram commute:
\begin{equation}
\begin{tikzcd}
\pi^{-1}U_{\alpha} \arrow{rr}{\Psi_{\alpha}=(\pi, g_{\alpha})} \arrow[swap]{dr}{\pi} & & U_{\alpha} \times G \arrow{dl}{\text{pr}_{1}} \\
& U_{\alpha} &
\end{tikzcd}
\end{equation}
\end{defn}

\noindent On overlapping trivializing open sets $U_{\alpha \beta} = U_{\alpha} \cap U_{\beta}$, we have two different ways of locally identifying the bundle with $U_{\alpha \beta} \times G$, using $\Psi_{\alpha} =(\pi, g_{\alpha})$ or $\Psi_{\beta} = (\pi, g_{\beta})$.  Over $\pi^{-1}U_{\alpha \beta}$, the two trivializations $\Psi_{\alpha}(\, p)$ and $\Psi_{\beta}(\, p)$ do not have to agree, but there must exist functions $\widetilde{g}_{\alpha \beta}: \pi^{-1}U_{\alpha \beta} \to G$ defined by $g_{\alpha}(\, p) = \widetilde{g}_{\alpha \beta}(\, p) g_{\beta}(\, p)$ for all $p \in \pi^{-1}U_{\alpha \beta}$.  This is summarized in the following diagram.

\begin{equation}
\begin{tikzcd}
U_{\alpha \beta} \times G \arrow[swap]{dr}{\text{pr}_{1}}  & \pi^{-1}U_{\alpha \beta} \arrow[swap]{l}{\Psi_{\alpha}} \arrow{r}{\Psi_{\beta}}\arrow{d}{\pi} & U_{\alpha \beta} \times G \arrow{dl}{\text{pr}_{1}} \\
& U_{\alpha \beta} &
\end{tikzcd}
\end{equation}
By the equivariance of $g_{\alpha}$ and $g_{\beta}$ it is easy to see that $\widetilde{g}_{\alpha \beta}$ is constant on each fiber, $\widetilde{g}_{\alpha \beta}(p \cdot g) = \widetilde{g}_{\alpha \beta}(\, p)$.  Therefore, $\widetilde{g}_{\alpha \beta}$ descends to $G$-valued functions
\begin{equation} \label{eqn:PBundTrFunc}
g_{\alpha \beta} : U_{\alpha \beta} \to G
\end{equation}
which we call the \emph{transition functions}.  One should read $g_{\alpha \beta}$ as transforming a group element in the $\beta$ trivialization to a group element in the $\alpha$ trivialization.  The transition functions define cocycles
\begin{equation}
\{g_{\alpha \beta}\} \in \check{H}^{1}(X, \underbar{G})
\end{equation}
in the \v{C}ech cohomology of the sheaf $\underbar{G}$ of $G$-valued functions.  Explicitly, this means that $g_{\alpha \beta}(x)g_{\beta \alpha}(x) =e$ for all $x \in U_{\alpha \beta}$, and $g_{\alpha \beta}(x) g_{\beta \gamma}(x) g_{\gamma \alpha}(x) =e$ for all $x \in U_{\alpha \beta \gamma}$.

\subsection{Associated Vector Bundles}

The goal of this section is to introduce vector bundles which are canonically constructed from a principal $G$-bundle along with a particular representation of $G$.  These are called associated vector bundles.  It is really the transition functions of the principal bundle and the representation which determine the properties of the resulting vector bundle.  We will also introduce the frame bundle as a way of recovering a principal bundle from a vector bundle, and argue that at least for all structure groups of interest, this gives a bijection between the two types of objects.  

Let $V$ be a complex vector space called the \emph{fiber} and let $\rho: G \to GL(V)$ be a representation of the Lie group $G$ by invertible linear transformations of $V$.

\begin{defn}
Given a principal $G$-bundle $\pi: \mathcal{P} \to X$ and a representation $\rho : G \to GL(V)$, the associated vector bundle $\mathcal{P} \times_{\rho} V$ is defined by
\[\mathcal{P} \times_{\rho} V = \mathcal{P} \times V / \sim,\]
where $(p_{1},v_{1}) \sim (p_{2}, v_{2})$ when $p_{2} = p_{1}g$ and $v_{2} = \rho(g^{-1})v_{1}$ for some $g \in G$.  Denote the equivalence classes by $[p,v] \in \mathcal{P} \times_{\rho} V$.  Finally, the projection map $\pi_{V}: \mathcal{P} \times_{\rho} V \to X$ is defined by $\pi_{V}\big([p,v]\big) = \pi(p)$.  
\end{defn}
\noindent One can show that $\mathcal{P} \times_{\rho}V$ is indeed a vector bundle with fiber $V$ and trivializes over the same open cover of $M$ as $\mathcal{P}$ with transition functions,  
\[\varphi_{\alpha \beta} = \rho(g_{\alpha \beta}): U_{\alpha \beta} \longrightarrow G \overset{\rho}{\longrightarrow}  GL(V)\]

\noindent The \emph{structure group} of the bundle $\mathcal{P} \times_{\rho} V$ is the group $\rho(G) \subseteq GL(V)$.

Conversely, given a complex vector bundle $E \to X$ of rank $n$ with fibers $E_{x}$, a frame at $x \in X$ is a choice of ordered basis of $E_{x}$, and the set of all frames $\text{Fr}_{x}$ at $x$ has a natural right action by $GL_{n}(\mathbb{C})$.  The \emph{frame bundle} $\text{Fr}(E)$ of $E$ is the principal $GL_{n}(\mathbb{C})$-bundle over $X$ where the fiber at $x \in X$ is $\text{Fr}_{x}$.  If $E$ is a real vector bundle or has reduced structure group, the same construction is valid.

\subsubsection{The Associated Adjoint Bundles}

Given a Lie group $G$, we have a group homomorphism $\text{Ad}: G \to \text{Aut}(G)$ such that $\text{Ad}(g)$ acts on $G$ by conjugation  
\begin{equation}
\text{Ad}(g)(h) = ghg^{-1}.
\end{equation}
We call $\text{Ad}$ the adjoint action.  Note that for all $g \in G$, $\text{Ad}(g)$ preserves the identity, so we get an induced map $\text{ad}: G \to GL(\mathfrak{g})$ defined by $\text{ad}(g) = \big(\text{Ad}(g)\big)_{*} : \mathfrak{g} \to \mathfrak{g}$.  In general, for all $x \in \mathfrak{g}$ we have
\begin{equation}
\text{ad}(g)(x) = \frac{d}{dt}\big( g e^{tx} g^{-1}\big) \big|_{t=0}.
\end{equation}
In the case of a matrix group, one can make sense of group elements acting on Lie algebra elements.  The adjoint action can be written in this case as
\begin{equation}
\text{ad}(g)(x) = g x g^{-1},
\end{equation}
and called the adjoint representation.  Associated to a principal $G$-bundle $\mathcal{P}$ along with the adjoint representation, we get the following adjoint bundle via the associated bundle construction,
\begin{equation}
\text{ad} \mathcal{P} = \mathcal{P} \times_{\text{ad}} \mathfrak{g}.
\end{equation}

The adjoint bundle is a real or complex vector bundle depending on whether $G$ is a real or complex Lie group.  A fiber of $\text{ad}\mathcal{P}$ is of course $\mathfrak{g}$ itself and therefore, the rank of the adjoint bundle is the dimension of $G$.  In the case of a matrix group $G$, over the intersection $U_{\alpha \beta}$ of two trivializing open sets, the transition functions of $\text{ad}\mathcal{P}$ are given by
\begin{equation}
\text{ad}\big(g_{\alpha \beta}(u)\big)(x) = g_{\alpha \beta}(u) \,x \, g_{\alpha \beta}^{-1}(u)
\end{equation}
for all $u \in U_{\alpha \beta}$ and $x \in \mathfrak{g}$, where $g_{\alpha \beta}$ is as in (\ref{eqn:PBundTrFunc}).  This example illustrates how the transition functions of the principal bundle along with the representation determine the associated vector bundle.  
  
One can construct more general \emph{associated fiber bundles} with a manifold $F$ as a fiber, and homomorphism $\rho : G \to \text{Aut}(F)$.  Given a principal $G$-bundle $\mathcal{P}$, choosing $F = G$ and homomorphism $\text{Ad}: G \to \text{Aut}(G)$, we get a fiber bundle
\begin{equation}
\text{Ad}\mathcal{P} = \mathcal{P} \times_{\text{Ad}} G
\end{equation}
with fiber $G$.  Note that $\text{Ad}\mathcal{P}$ is \emph{not} a principal $G$-bundle, however the space $\Omega^{0}(\text{Ad}\mathcal{P})$ of smooth sections of $\text{Ad}\mathcal{P}$ naturally inherits a group structure given by fiberwise multiplication.  This in fact turns out to be isomorphic to a familiar group.  A proof of the following lemma can be found in \cite{friedman_gauge_1997}.

\begin{lemmy} \label{lemmy:autgauge}
The group $\Omega^{0}(\text{Ad}\mathcal{P})$ is naturally isomorphic to the group $\text{Aut}\mathcal{P}$ of principal bundle automorphisms.  
\end{lemmy} 

\noindent The automorphism group $\text{Aut}\mathcal{P} \cong \Omega^{0}(\text{Ad}\mathcal{P})$ is an infinite-dimensional Lie group under fiberwise multiplication, with Lie algebra $\Omega^{0}(\text{ad}\mathcal{P})$.

\subsection{Relationship Between Principal Bundles and Vector Bundles}

Let $\mathcal{P}$ be a principal $GL_{n}(\mathbb{C})$-bundle over base space $X$.  Choosing $\rho$ to be the fundamental representation of $GL_{n}(\mathbb{C})$ on $\mathbb{C}^{n}$, we get an associated complex vector bundle $E$ of rank $n$.  Conversely, given a rank $n$ complex vector bundle $E$, we recover a principal $GL_{n}(\mathbb{C})$-bundle as the frame bundle $\text{Fr}(E)$.  These two constructions are mutually inverse and therefore, for structure group $GL_{n}(\mathbb{C})$, we have the following equivalence

\[ \begin{Bmatrix*}[l]
 \,\,\,\,\,\,\,\,\,\,\, \text{Principal}\\ GL_{n}(\mathbb{C})\text{-bundles}
\end{Bmatrix*}  \Longleftrightarrow  \begin{Bmatrix*}\text{Complex vector bundles}\\ \text{of rank }n \end{Bmatrix*}.\]

\noindent In practice, this allows one to work with complex vector bundles instead of principal $GL_{n}(\mathbb{C})$-bundles.

A principal $SL_{n}(\mathbb{C})$-bundle gives rise to a complex vector bundle $E$ of rank $n$ with trivial determinant, $\Lambda^{n} E \cong X \times \mathbb{C}$.  Conversely, given a rank $n$ complex vector bundle $E$ with trivial determinant, the frame bundle $\text{Fr}(E)$ is a principal $SL_{n}(\mathbb{C})$-bundle.  As above, these two constructions are mutually inverse and give rise to the correspondence

\[ \begin{Bmatrix*}[l]
 \,\,\,\,\,\,\,\,\, \text{Principal}\\ SL_{n}(\mathbb{C})\text{-bundles}
\end{Bmatrix*}  \Longleftrightarrow  \begin{Bmatrix*}\text{Complex vector bundles} \,\, E \\ \text{of rank }n \,\, \text{with} \,\, \Lambda^{n}E \cong X \times \mathbb{C} \end{Bmatrix*}.\]

The two additional structure groups we will be interested in are $U(n)$ and $SU(n)$, which arise as the compact real forms of $GL_{n}(\mathbb{C})$ and $SL_{n}(\mathbb{C})$, respectively.  Both groups $G=U(n)$ or $G=SU(n)$ are defined as the group of symmetries preserving a Hermitian form on a complex vector space $V$ of dimension $n$, where for $SU(n)$ it must also preserve a volume form.  We therefore have a Hermitian form $\mathfrak{q} : V \times V \to \mathbb{C}$ such that $\mathfrak{q}(g v, g w) = \mathfrak{q}(v,w)$ for all $g \in G$ and all $v,w \in V$.  One can show that for all Lie algebra elements $a \in \mathfrak{g} \subseteq \text{End}(V)= \mathfrak{gl}(V)$ and all $v,w \in V$ the following important constraint must be satisfied
\begin{equation}\label{eqn:Liealgconstr}
\mathfrak{q}(a v, w) + \mathfrak{q}(v, a w) =0.
\end{equation}
In the above equation, it is important to interpret $a \in \mathfrak{g}$ as an endomorphism of $V$.  An important notion in what follows will be that of a Hermitian metric on a complex vector bundle.

\begin{defn}
Let $E$ be a complex vector bundle.  A Hermitian metric $h$ on $E$ is a collection of Hermitian forms $h_{x} : E_{x} \times E_{x} \to \mathbb{C}$ for all $x \in X$ which varies smoothly with $x$.  A Hermitian vector bundle $(E,h)$ is a complex vector bundle with a Hermitian metric $h$.  
\end{defn}

One can show that a Hermitian vector bundle has its structure group reduced from $GL_{n}(\mathbb{C})$ to $U(n)$ by the metric, and further to $SU(n)$ if $\Lambda^{n}E \cong X \times \mathbb{C}$.  Moreover, we have the following correspondence via the associated bundle and frame bundle constructions
 
\[ \begin{Bmatrix*}[l]
 \,\,\,\,\,\,\, \text{Principal}\\ U(n)\text{-bundles}
\end{Bmatrix*}  \Longleftrightarrow  \begin{Bmatrix*}\text{Hermitian vector bundles}\\ \text{of rank }n \end{Bmatrix*}.\]

\noindent In the obvious manner, Hermitian vector bundles with trivial determinant give rise to the similar equivalence

\[ \begin{Bmatrix*}[l]
 \,\,\,\,\,\,\,\, \text{Principal}\\ SU(n)\text{-bundles}
\end{Bmatrix*}  \Longleftrightarrow  \begin{Bmatrix*}\text{Hermitian vector bundles} \,\, (E,h) \\ \text{of rank }n \,\, \text{with} \,\, \Lambda^{n}E \cong X \times \mathbb{C} \end{Bmatrix*}.\]

For readers interested in complex and algebraic geometry, the takeaway from this section should be that for these structure groups, instead of considering principal bundles, one can work entirely with complex vector bundles with a metric, and possibly trivial determinant.  This leads us to an important remark.

\begin{rmk}
In this chapter we will be primarily interested in structure groups $U(n)$ or $SU(n)$, so we will consistently let $G$ denote one of these two groups.  At times we will say ``let $(E, h)$ be a Hermitian vector bundle with structure group $G$" which simply means that $(E, h)$ is a Hermitian vector bundle, and it has trivial determinant if $G = SU(n)$.  We will also always take trivializing charts such that the transition functions take values in $G$.  
\end{rmk}

\subsubsection{Bundle Automorphisms and Endomorphisms}

Given a complex vector bundle $E$ of rank $n$, the bundle of endomorphisms $\text{End}E \cong E \otimes E^{\vee}$ is a complex vector bundle of rank $n^{2}$ with fiber $\text{End}E_{x} \cong E_{x} \otimes E_{x}^{\vee} \cong \mathfrak{gl}_{n}(\mathbb{C})$ for all $x \in X$.  In addition, the set of automorphisms $\text{Aut}E$ of $E$ is a group under composition.  In fact, $\text{Aut}E$ is a fiber bundle over the base space $X$ with fiber $GL_{n}(\mathbb{C})$ -- it is \emph{not} however a principal $GL_{n}(\mathbb{C})$-bundle.  

If $\mathcal{P}_{E}$ is the principal $GL_{n}(\mathbb{C})$-bundle corresponding uniquely to $E$, then $\text{Aut}E$ and $\text{End}E$ are both bundles associated to $\mathcal{P}_{E}$ coinciding with the adjoint bundles introduced earlier
\begin{equation} \label{eqn:adjautend}
\setlength{\jot}{10pt}
\begin{split}
& \text{Aut}E = \text{Ad}\mathcal{P}_{E} = \mathcal{P}_{E} \times_{\text{Ad}} GL_{n}(\mathbb{C}) \\
& \text{End}E = \text{ad}\mathcal{P}_{E} = \mathcal{P}_{E} \times_{\text{ad}} \mathfrak{gl}_{n}(\mathbb{C})
\end{split}
\end{equation}
where both adjoint representations above are taken with respect to the group $GL_{n}(\mathbb{C})$.  Indeed, if the transition functions of $E$ are $\{\varphi_{\alpha \beta} \}$, the transition functions of $\text{End}E$ are $\{ \text{ad}(\varphi_{\alpha \beta})\}$, consistent with the transition functions of the associated bundle, and similarly for $\text{Aut}E$.  Comparing (\ref{eqn:adjautend}) and Lemma \ref{lemmy:autgauge} we see that automorphisms of $\mathcal{P}_{E}$ correspond to global sections of $\text{Aut}E$.    

Let $(E,h)$ be a Hermitian vector bundle with structure group $G$, assumed to be either $U(n)$ or $SU(n)$, and let $\mathcal{P}_{E}$ be the uniquely corresponding principal $G$-bundle.  We denote by
\begin{equation}
\setlength{\jot}{8pt}
\begin{split}
& G_{E} \coloneqq \text{Ad}\mathcal{P}_{E} \subset \text{Aut}E \\
& \mathfrak{g}_{E} \coloneqq \text{ad}\mathcal{P}_{E} \subset \text{End}E
\end{split}
\end{equation}
respectively, the vector bundle automorphisms and endomorphisms, compatible with the structure group $G$.  Note that $G_{E}$ is a subbundle of $\text{Aut}E$ with fiber $G \subset GL_{n}(\mathbb{C})$ and $\mathfrak{g}_{E}$ is a subbundle of $\text{End}E$ with fiber $\mathfrak{g} \subset \mathfrak{gl}_{n}(\mathbb{C})$.  Just as in the case of $GL_{n}(\mathbb{C})$, the automorphisms of $\mathcal{P}_{E}$ correspond to global sections of $G_{E}$.

\subsection{Connections on Complex and Hermitian Vector Bundles}

A smooth function on a manifold $X$ valued in $\mathbb{C}^{n}$ can be thought of as a smooth section of the trivial bundle $X \times \mathbb{C}^{n}$.  Such sections can be differentiated in a standard way using the exterior derivative.  One introduces the notion of a \emph{connection} to generalize the differentiation of smooth functions to sections of an arbitrary vector bundle.  A choice of a connection provides a way of identifying nearby fibers in a vector bundle, which allows one to then make sense of a directional derivative of a section.  

Let $E$ be a complex vector bundle of rank $n$ on a smooth manifold $X$.  Throughout, we denote by $\Omega^{p}(E)$ the $C^{\infty}(X)$-module of $p$-forms on $X$ valued in $E$.  More specifically, this is the space of global sections of the bundle $\Lambda^{p}T^{\smvee}_{X} \otimes E$.  For an open set $U \subset X$, we will denote the local $p$-forms valued in $E$ by $\Omega^{p}_{U}(E)$.  

\begin{defn}
A connection on a complex vector bundle $E$ is a $\mathbb{C}$-linear map 
\begin{equation}
d_{A} : \Omega^{0}(E) \to \Omega^{1}(E)
\end{equation}
satisfying the Leibniz rule $d_{A}(f \cdot s) = f \cdot d_{A}(s) + s \cdot df$, for all functions $f \in C^{\infty}(X)$ and sections $s \in \Omega^{0}(E)$.  
\end{defn}
One can use a connection to make sense of differentiating sections of $E$ in the direction of a particular tangent vector field to $X$.  If $\langle \cdot , \cdot \rangle$ is the natural contraction of vector fields and one-forms on $X$, then given any section $s \in \Omega^{0}(E)$ and vector field $v \in \Omega^{0}(T_{X})$, one can define the derivative of $s$ along $v$ to be $\langle d_{A}(s) , v \rangle \in \Omega^{0}(E)$.  Because of this property, a connection as we have defined it, is sometimes called a covariant derivative.  

If $\mathcal{P}_{E}$ is the principal $GL_{n}(\mathbb{C})$-bundle corresponding uniquely to $E$, then a connection on $E$ is equivalent to a connection on $\mathcal{P}_{E}$.  For details, we refer the reader to \cite[Section 2.10]{friedman_gauge_1997}.

The origin of the notation $d_{A}$ can be understood as follows.  Suppose $E$ trivializes over an open set $U \subset X$.  We can identify local sections of $E$ with $C^{\infty}(U, \mathbb{C}^{n})$.  A local frame $(e_{1}, \ldots, e_{n})$ of $E$ over $U$ is a collection of functions $e_{i} \in C^{\infty}(U, \mathbb{C}^{n})$ which at each point of $U$, form a basis of the fiber $\mathbb{C}^{n}$.  For each $e_{i}$ we can write
\begin{equation}\label{eqn:Conn1formframe}
d_{A}(e_{i}) = \sum_{j=1}^{n} A_{i j} e_{j}
\end{equation}
for a matrix $A = (A_{i j})$ of one-forms on $U$.  We refer to $A$ as the \emph{connection one-form} associated to $d_{A}$ on $U$.  In physics, the $A_{i j}$ are called \emph{gauge fields}.

\begin{proppy}
If $d_{A}$ and $d'_{A'}$ are connections on a complex vector bundle $E$, then $d_{A} - d'_{A'} \in \Omega^{1}(\text{End}E)$.  Conversely, given $a \in \Omega^{1}(\text{End}E)$ and any connection $d_{A}$, then $d_{A} + a$ is again a connection.  
\end{proppy}

\begin{proof}
To show that $d_{A} - d'_{A'}$ is a $\text{End}E$-valued one-form, we simply need to show it to be $C^{\infty}(X)$-linear.  But by the Leibniz rule, it is indeed clear that
\begin{equation}
(d_{A} - d'_{A'})(f \cdot s) = f \cdot (d_{A} - d'_{A'})(s)
\end{equation}
for all $f \in C^{\infty}(X)$ and $s \in \Omega^{0}(E)$.  For the second claim, note that $a \in \Omega^{1}(\text{End}E)$ can act on $\Omega^{0}(E)$ by multiplication in the form part, and evaluation in the endomorphism part.  Therefore
\begin{equation}
\setlength{\jot}{12pt}
\begin{split}
(d_{A} + a)(f \cdot s) & = f \cdot d_{A} (s) + s \cdot df + a(f \cdot s)\\
& = f \cdot (d_{A} + a)(s) + s \cdot df
\end{split}
\end{equation}
which verifies that $d_{A} + a$ satisfies the Leibniz rule.  
\end{proof}

Because a connection is not itself $C^{\infty}(X)$-linear, it is not a tensor.  But by the proposition, \emph{differences} between connections are indeed one-forms on $X$ valued in the endomorphism bundle.  We typically denote by $\mathscr{A}(E)$ the space of all connections on a fixed complex vector bundle $E$.  The above proposition implies the following corollary establishing the structure of $\mathscr{A}(E)$.

\begin{cory}
The space of connections $\mathscr{A}(E)$ on a complex vector bundle $E$ is an infinite-dimensional affine space modeled on $\Omega^{1}(\text{End}E)$.  In particular, there is no canonically distinguished connection.  
\end{cory}

Given a connection $d_{A}$ on $E$, one can induce a connection on the standard bundles constructed from $E$.  We choose to also denote these connections by $d_{A}$.  For example, the connection induced on the determinant $\Lambda^{n}E$ is defined by
\begin{equation}\label{eqn:defnconndett}
d_{A}(s_{1} \wedge \ldots, \wedge s_{n}) = \sum_{i=1}^{n} s_{1} \wedge \ldots \wedge d_{A}(s_{i}) \wedge \ldots \wedge s_{n}
\end{equation}
for local or global sections $s_{i}$ of $E$.    

On a Hermitian vector bundle with possibly trivial determinant, we ask that a connection be compatible with these extra structures or properties.  We take an orthonormal frame of $(E, h)$ to mean a local frame $(e_{1}, \ldots, e_{n})$ of $E$, orthonormal with respect to $h$.  That is, for all $i, j$ we have $h(e_{i}, e_{j}) = \delta_{i j}$.   

\begin{defn}\label{defn:Herrrrmconnnn}
A Hermitian connection on a Hermitian vector bundle $(E, h)$ of rank $n$ over $X$ is a connection $d_{A}$ on the complex vector bundle $E$ which additionally satisfies
\begin{equation} \label{eqn:Gconnconstr}
dh(s,t) = h\big(d_{A} (s), t \big) + h\big(s, d_{A} (t)\big)
\end{equation}
for all sections $s,t \in \Omega^{0}(E)$.  If the determinant of $E$ is trivial, there must also exist around each point an orthonormal frame $(e_{1}, \ldots, e_{n})$ of $E$ such that $d_{A}(e_{1} \wedge \ldots \wedge e_{n})=0$.    
\end{defn}

It is helpful to understand what this definition implies locally.  If $(e_{1}, \ldots, e_{n})$ is an orthonormal frame of $E$, then applying (\ref{eqn:Conn1formframe}) and (\ref{eqn:Gconnconstr}) we see
\begin{equation}
\begin{split}
0 = dh(e_{i}, e_{j}) & = \sum_{k=1}^{n} A_{i k} h(e_{k}, e_{j}) + \sum_{k=1}^{n} \overline{A}_{j k} h(e_{i}, e_{k})\\
& =A_{i j} + \overline{A}_{j i}
\end{split}
\end{equation}
which is the statement that the connection one-form $A$ is skew-Hermitian.  In other words, $A$ is a one-form valued in the Lie algebra $\mathfrak{u}_{n}$ of $U(n)$.  Similarly, using also (\ref{eqn:defnconndett}) we have 
\begin{equation}\label{eqn:connonwedgeer}
d_{A}(e_{1} \wedge \ldots \wedge e_{n}) = \text{Tr}(A) (e_{1} \wedge \ldots \wedge e_{n})
\end{equation}
which means that if $d_{A}$ is a connection on a bundle with trivial determinant, $A$ must also be traceless.  This makes $A$ a one-form valued in the Lie algebra $\mathfrak{su}_{n}$ of $SU(n)$.    

\begin{defn}
Assuming $G$ to be either $U(n)$ or $SU(n)$, a $G$-connection on a Hermitian vector bundle $(E, h)$ is a Hermitian connection satisfying the additional condition in Definition \ref{defn:Herrrrmconnnn} if $G = SU(n)$.  We denote the space of $G$-connections on $(E, h)$ by $\mathscr{A}_{G}(E, h)$.  
\end{defn}

\noindent Though we will not prove it, $G$-connections always exist.  Let $\mathfrak{g}_{E} \subset \text{End}E$ be the bundle of Lie algebras associated to the adjoint representation of $G$.  Depending on $G$, the fiber of $\mathfrak{g}_{E}$ is either $\mathfrak{u}_{n}$ or $\mathfrak{su}_{n}$.

\begin{cory}\label{cory:spofGconnnns}
The space of $G$-connections $\mathscr{A}_{G}(E, h)$ is an infinite-dimensional affine space modeled on $\Omega^{1}(\mathfrak{g}_{E})$.  In particular, there is no canonically distinguished $G$-connection.  
\end{cory}

\begin{proof}
Given any $G$-connection $d_{A}$ and any $a \in \Omega^{1}(\mathfrak{g}_{E})$, we want to show that $d_{A} + a$ is again a $G$-connection.  We have previously shown that $d_{A} + a$ is at least a connection on $E$, so it remains to verify that (\ref{eqn:Gconnconstr}) is satisfied.  We have
\begin{equation}
h\big( (d_{A} + a)(s), t\big) + h\big(s, (d_{A} + a)(t) \big) = dh(s,t) + h(as, t) + h(s, at)
\end{equation}
using that $d_{A}$ is a $G$-connection as well as the bilinearity of $h$.  However, we have $h(as, t) + h(s, at)=0$, as can be seen locally, where this becomes the statement that $a_{i j} + \overline{a}_{j i} =0$.  If the determinant of $E$ is trivial, the connection one-form $A$ as well as $a$ are one-forms valued in traceless skew-Hermitian matrices.  We must then verify that 
\begin{equation}
(d_{A} + a)(e_{1} \wedge \ldots \wedge e_{n}) =0.
\end{equation} 
Because we have already shown that $d_{A} + a$ is again a connection, we know by (\ref{eqn:connonwedgeer})
\begin{equation}
(d_{A} + a)(e_{1} \wedge \ldots \wedge e_{n}) = \text{Tr}(A + a) (e_{1} \wedge \ldots \wedge e_{n})
\end{equation}
and this indeed vanishes, completing the proof.  
\end{proof}

\subsubsection{The Local Description of Connections}

Let $(E,h)$ be a Hermitian vector bundle with structure group $G$ trivializing on an open cover $\{U_{\alpha}\}$ of $X$ through isomorphisms $E|_{U_{\alpha}} \cong U_{\alpha} \times \mathbb{C}^{n}$.  Then on intersecting open sets $U_{\alpha \beta} = U_{\alpha} \cap U_{\beta}$, we have transition functions $\varphi_{\alpha \beta} : U_{\alpha \beta} \to G$.  Let $(e_{1}, \ldots, e_{n})$ be an orthonormal frame of $(E,h)$ over $U_{\alpha}$, and let $s_{\alpha}$ be a local section, with components $s_{\alpha}^{(i)} : U_{\alpha} \to \mathbb{C}$ relative to the frame for all $i=1, \ldots, n$.  Using the Leibniz rule, we can compute the action of $d_{A}$ on $s_{\alpha}$ to be
\begin{equation}
\begin{split}
d_{A}(s_{\alpha}) = d_{A} \big( \sum_{i=1}^{n} s_{\alpha}^{(i)}e_{i}\big) & = \sum_{i=1}^{n}\bigg( d(s_{\alpha}^{(i)})e_{i} + s_{\alpha}^{(i)}d_{A}(e_{i})\bigg) \\
& = \sum_{i=1}^{n} \bigg( d(s_{\alpha}^{(i)})e_{i} + s_{\alpha}^{(i)} \sum_{j=1}^{n}(A_{\alpha})_{i j} e_{j} \bigg)
\end{split}
\end{equation}
where we have also used the definition (\ref{eqn:Conn1formframe}) of the connection one-form $A_{\alpha} \in \Omega^{1}(U_{\alpha}) \otimes \mathfrak{g}$.  We therefore see that the local description of the connection on $U_{\alpha}$ is 
\begin{equation}
d_{A}|_{U_{\alpha}} = d + A_{\alpha}.
\end{equation}  

\begin{proppy}
The connection one-forms $A_{\alpha}$ and $A_{\beta}$ are related on $U_{\alpha \beta}$ by  
\begin{equation} \label{eqn:localcon}
A_{\beta} = \varphi_{\alpha \beta}^{-1} A_{\alpha} \varphi_{\alpha \beta} + \varphi_{\alpha \beta}^{-1} \, d \varphi_{\alpha \beta}.
\end{equation}
\end{proppy}

\begin{proof}
Let $s \in \Omega^{0}(E)$ be a global section of $E$ which restricts to local sections $s_{\alpha}: U_{\alpha} \to \mathbb{C}^{n}$.  For all $x \in U_{\alpha \beta}$, we have an invertible linear map $\varphi_{\alpha \beta}(x) \in G$ on fiber $E_{x}$ and $s_{\alpha}(x) = \varphi_{\alpha \beta}(x) s_{\beta}(x)$.  The condition we require to hold is,
\begin{equation}
\varphi_{\alpha \beta} \big( d + A_{\beta}\big) s_{\beta} = \big(d + A_{\alpha}\big) s_{\alpha}
\end{equation}
from which the result follows after a straightforward computation.  
\end{proof}

The term $\varphi_{\alpha \beta}^{-1} d\varphi_{\alpha \beta}$ in (\ref{eqn:localcon}) is the local manifestation of a $G$-connection \emph{not} being a tensor.  However, because this term depends only on $\varphi_{\alpha \beta}$, taking the difference of two $G$-connections $d_{A} - d'_{A'}$ gives the transformation
\begin{equation}
A_{\beta} - A'_{\beta} = \varphi^{-1}_{\alpha \beta}(A_{\alpha} - A'_{\alpha}) \varphi_{\alpha \beta}
\end{equation}
which is precisely how a one-form valued in $\mathfrak{g}_{E}$ should transform (see Corollary \ref{cory:spofGconnnns}).

\subsection{The Curvature of a Connection}

Given a connection $d_{A}$ on a complex vector bundle $E$ over a smooth manifold $X$ of dimension $m$, we have seen that $d_{A}$ defines a covariant differentiation of sections of $E$.  Using $d_{A}$, we may generalize the de Rham complex
\begin{equation} \label{eqn:deRhamcomp}
\Omega^{0}(X) \overset{d}{\longrightarrow} \Omega^{1}(X) \overset{d}{\longrightarrow} \cdots \overset{d}{\longrightarrow} \Omega^{m}(X) \longrightarrow 0
\end{equation}
to differential forms on $X$ valued in the bundle $E$
\begin{equation} \label{eqn:bundvalseq}
\Omega^{0}(E) \overset{d_{A}}{\longrightarrow} \Omega^{1}(E) \overset{d_{A}}{\longrightarrow} \cdots \overset{d_{A}}{\longrightarrow} \Omega^{m}(E)\longrightarrow 0.
\end{equation}
The covariant derivative $d_{A}$ is uniquely defined by requiring that it coincides with the connection on $\Omega^{0}(E)$ and for $d_{A}: \Omega^{p}(E) \to \Omega^{p+1}(E)$, we have the Leibniz rule
\begin{equation} \label{eqn:curvLeibrulle}
d_{A}( \omega \wedge s ) = d \omega \otimes s + (-1)^{p} \omega \wedge d_{A}(s)
\end{equation}
for all $\omega \in \Omega^{p}(X)$ and $s \in \Omega^{0}(E)$. The exterior derivative of course satisfies $d^{2} =0$, but we do \emph{not} in general have $d_{A} \circ d_{A} =0$.

\begin{defn}
The curvature $F_{A}$ of the connection $d_{A}$ on a complex vector bundle $E$ is defined by
\begin{equation}\label{eqn:curvmap}
F_{A} \coloneqq d_{A} \circ d_{A} : \Omega^{0}(E) \to \Omega^{2}(E).
\end{equation}
The connection $d_{A}$ is said to be flat if the curvature is identically zero, $F_{A} = 0$.  
\end{defn}  

\noindent Slightly generalizing the Leibniz rule (\ref{eqn:curvLeibrulle}), one can show that 
\begin{equation}
d_{A}( \beta \wedge s) = d \beta \wedge s + (-1)^{k} \beta \wedge d_{A}(s)
\end{equation}
for all $\beta \in \Omega^{k}(X)$ and $s \in \Omega^{l}(E)$.  It follows from this that $F_{A}$ is $C^{\infty}(X)$-linear, meaning we have $F_{A} \in \Omega^{2}(\text{End}E)$.  If $(E, h)$ is a Hermitian vector bundle with structure group $G$, assumed to be either $U(n)$ or $SU(n)$, we simply have $F_{A} \in \Omega^{2}(\mathfrak{g}_{E})$.

We therefore interpret the map (\ref{eqn:curvmap}) to be given as action by the endomorphism part of $F_{A}$, and multiplication by the form part.  The curvature should be thought to measure the obstruction to the sequence (\ref{eqn:bundvalseq}) being a complex.  It is then evidently the flat connections on which the covariant derivative behaves analogously to the exterior derivative.

Given a covariant derivative $d_{A}$ on a complex vector bundle $E$, there is a canonical way to induce a covariant derivative on $\text{End}E$, which we will also denote by $d_{A}$.  Given sections $\sigma \in \Omega^{k}(\text{End}E)$ and $s \in \Omega^{0}(E)$, the condition uniquely determining $d_{A}: \Omega^{k}(\text{End}E) \to \Omega^{k+1}(\text{End}E)$ is the Leibniz-like rule 
\begin{equation}
d_{A}(\sigma \cdot s) = d_{A} (\sigma)(s) + \sigma \cdot  d_{A}(s) 
\end{equation}
which can be rearranged and taken as the definition of $d_{A}(\sigma)$.  We evidently have
\begin{equation}
d_{A} (\sigma) \coloneqq d_{A} \circ \sigma - \sigma \circ d_{A} = [ d_{A}, \sigma ]
\end{equation}
which one may see defined as $d_{A} \coloneqq [d_{A}, - \, ]$.  It should be clear from the context whether $d_{A}$ is the covariant derivative on $E$ or that on $\text{End}E$.  If $(E, h)$ is a Hermitian vector bundle with structure group $G$, one may similarly induce a covariant derivative on the adjoint bundle $\mathfrak{g}_{E}$.  

Understanding now how to differentiate endomorphism-valued forms, the following result, known as the Bianchi identity, is an important global constraint on the curvature.

\begin{proppy}[\bfseries Bianchi Identity]
The curvature $F_{A} \in \Omega^{2}(\mathfrak{g}_{E})$ satisfies
\begin{equation} \label{eqn:BianchiId}
d_{A} (F_{A}) = 0,
\end{equation}
where $d_{A}$ is interpreted as the induced covariant derivative on $\mathfrak{g}_{E}$.  
\end{proppy}

\begin{proof}
By the above discussion, for all sections $s \in \Omega^{0}(E)$ we have
\begin{equation}
\begin{split}
d_{A}(F_{A})(s) & = d_{A}\big(F_{A}(s)\big) - F_{A}\big(d_{A}(s)\big) \\
& = (d_{A} \circ d_{A} \circ d_{A}) (s) - (d_{A} \circ d_{A} \circ d_{A}) (s) =0.
\end{split}
\end{equation}
\end{proof}

\begin{rmk}
Despite the fact that $F_{A} = d_{A} \circ d_{A}$, one must not interpret the Bianchi identity to say $d_{A} \circ d_{A} \circ d_{A}=0$.  The $d_{A}$ appearing in (\ref{eqn:BianchiId}) is the covariant derivative on $\mathfrak{g}_{E}$.  
\end{rmk}

Recall that because the set of $G$-connections $\mathscr{A}_{G}(E,h)$ on a Hermitian vector bundle $(E,h)$ is an affine space, given a connection $d_{A}$, any other is of the form $d_{A} + a$ for some $a \in \Omega^{1}(\mathfrak{g}_{E})$, which we can abbreviate as $A + a$.  One can ask how the curvature interacts with this affine structure of $\mathscr{A}_{G}(E,h)$.  A direct calculation shows that
\begin{equation} \label{eqn:connaff}
F_{A + a} = F_{A} + d_{A}(a) + a \wedge a
\end{equation}
where in this context, $d_{A}(a) \in \Omega^{2}(\mathfrak{g}_{E})$ and $a \wedge a$ is interpreted as the exterior product of the form parts and the composition of the endomorphism parts.

\subsubsection{The Local Description of the Curvature}

Locally, the curvature of a $G$-connection is a matrix-valued two-form.  More specifically, on a trivializing chart, in terms of the connection one-form $A$ the curvature takes the form
\begin{equation}
F_{A} = dA + A \wedge A.  
\end{equation}
On intersecting trivializing charts $U_{\alpha \beta} = U_{\alpha} \cap U_{\beta}$, let $\varphi_{\alpha \beta}$ be the transition functions.  The transformation of the curvature on $U_{\alpha \beta}$ is
\begin{equation}\label{eqn:curvchangtriv}
F_{A_{\beta}} = \varphi_{\alpha \beta}^{-1} F_{A_{\alpha}} \varphi_{\alpha \beta}
\end{equation}
where $A_{\alpha}$ is the connection one-form on $U_{\alpha}$, and similarly for $A_{\beta}$.  This is precisely how one would expect a two-form valued in $\mathfrak{g}_{E}$ to transform.

\subsection{Gauge Transformations as Bundle Automorphisms}  

Let $(E,h)$ be a Hermitian vector bundle with structure group $G$, assumed to be either $U(n)$ or $SU(n)$, and corresponding principal $G$-bundle $\mathcal{P}_{E}$.  We recall that the automorphism groups of $(E,h)$ and $\mathcal{P}_{E}$ are \emph{not} the same.  The automorphism group of $(E,h)$ is the fiber bundle $Ad \mathcal{P}_{E}$ with fiber $G$, and the automorphisms of $\mathcal{P}_{E}$ are global sections of $Ad \mathcal{P}_{E}$ (see Lemma \ref{lemmy:autgauge}).  

\begin{defn}
The group of gauge transformations (or gauge group\footnote{One must beware that in some contexts, particularly in physics, the term \emph{gauge group} is used to mean what we are calling the structure group $G$.}) of a Hermitian vector bundle $(E,h)$ as above is defined by 
\begin{equation}
\mathscr{G} \coloneqq \text{Aut}\mathcal{P}_{E} \cong \Omega^{0}(Ad \mathcal{P}_{E}).
\end{equation}
We refer to an element $\sigma \in \mathscr{G}$ as a (global) gauge transformation.  Over a trivializing open set $U_{\alpha}$, a local gauge transformation is a map $\sigma_{\alpha} : U_{\alpha} \to G$ and $C^{\infty}(U_{\alpha}, G)$ is the group of local gauge transformations.  
\end{defn}

The group of gauge transformations is an infinite-dimensional Lie group with associated Lie algebra $\Omega^{0}(\mathfrak{g}_{E}) \cong \Omega^{0}(\text{ad} \mathcal{P}_{E})$.  One should think of $\mathscr{G}$ as the group of \emph{symmetries} of the bundle $(E,h)$ which forms, along with the base space $X$, the geometrical background of a Yang-Mills theory.  

\begin{defn}\label{defn:comprexgrpgTRANY}
The complex group of gauge transformations $\mathscr{G}^{\mathbb{C}}$ is the group of automorphisms of the principal bundle associated to the complexification of $(E,h)$.  
\end{defn}  

The general philosophy is that objects related by gauge transformations should be identified, in a suitable sense.  In future sections, we will consider moduli spaces of objects where the natural notion of equivalence is that two objects be related by gauge transformations.  We therefore need to know how the group of gauge transformations acts on connections and curvatures.  Given a covariant derivative or connection $d_{A}$, a gauge transformation $\sigma \in \mathscr{G}$ acts on $d_{A}$ by pullback
\begin{equation}\label{eqn:gaugetransCONN}
\sigma^{*} d_{A} = \sigma^{-1} \circ d_{A} \circ \sigma
\end{equation}
and if $F_{A}$ is the curvature of $A$, $\sigma \in \mathscr{G}$ acts again by pullback
\begin{equation}\label{eqn:gautranyFAA}
\sigma^{*} F_{A} = \sigma^{-1} F_{A} \sigma.
\end{equation}
These actions on connections and curvatures hold also for the complex group of gauge transformations $\mathscr{G}^{\mathbb{C}}$.  

Let $U_{\alpha}$ be a trivializing open chart, and let $\sigma_{\alpha} : U_{\alpha} \to G$ be a local gauge transformation.  If $d + A_{\alpha}$ is the local description of the connection on $U_{\alpha}$, one can show that the action of $\sigma_{\alpha}$ on the connection one-form is
\begin{equation} \label{eqn:conngaugetrans}
\sigma_{\alpha}^{*} A_{\alpha} = \sigma_{\alpha}^{-1} A_{\alpha} \sigma_{\alpha} + \sigma_{\alpha}^{-1} d \sigma_{\alpha}
\end{equation}
while the action on the local connection two-form 
\begin{equation}\label{eqn:curvgaugetranss}
\sigma_{\alpha}^{*} F_{A_{\alpha}} = \sigma_{\alpha}^{-1} F_{A_{\alpha}} \sigma_{\alpha} 
\end{equation}
can be written in terms of the adjoint representation $\text{ad} : G \to GL(\mathfrak{g})$.  Similar statements of course hold for $\mathscr{G}^{\mathbb{C}}$.  Noting the likeness of (\ref{eqn:conngaugetrans}) to (\ref{eqn:localcon}) and (\ref{eqn:curvgaugetranss}) to (\ref{eqn:curvchangtriv}), we conclude that a local connection one-form and curvature two-form transform the same way under a change of trivialization and under a gauge transformation, i.e. a bundle automorphism.  This is analogous to the case of \emph{finite} dimensional Lie groups acting on Euclidean space where one can equivalently transform a vector or transform a coordinate system.

\subsection{Characteristic Classes of Bundles} \label{sec:chercrassbundre}

Given a complex vector bundle $E \to X$, characteristic classes are cohomology classes of $X$ associated to $E$ which describe some (but not all) of the global topological features of the vector bundle.  The content of this section will later be generalized to apply to coherent sheaves.  In practice, we will often want to fix as many topological features of a bundle or sheaf as possible, which we will do by prescribing fixed characteristic classes.  

One standard construction of characteristic classes on complex vector bundles is via Chern-Weil theory.  To sketch the idea, let $E$ be a complex vector bundle of rank $n$ on a smooth manifold $X$ with connection $d_{A}$ and corresponding curvature $F_{A}$.  The \emph{total Chern class} is a de Rham cohomology class on $X$ defined by
\begin{equation}\label{eqn:totCHCRASS}
c(E) = \bigg[ \text{det} \big(1 + \frac{i}{2 \pi} F_{A}\big) \bigg]
\end{equation}
where the brackets denote the cohomology class of the form.  It follows from the Bianchi identity that the form is closed.  The polynomial $\text{det}(1+x)$ is invariant under conjugation of the matrix $x$ which implies that the total Chern class is invariant under gauge transformations of $F_{A}$.  In fact, something much stronger is true.  One can show that $c(E)$ is independent of the choice of connection $d_{A}$, and therefore is a topological invariant of the bundle.  When we say a characteristic class is a topological invariant, we mean that if two vector bundles have different characteristic classes, then they are not isomorphic.  The converse is in general, not true.  For detailed proofs of these claims, see \cite[Section 4.4]{huybrechts_complex_2004}.  

Thanks to the independence of $c(E)$ on the connection, we may choose it to be a Hermitian connection, which necessarily exists.  It follows that the total Chern class defines a \emph{real} cohomology class $c(E) \in H^{*}(X, \mathbb{R})$ on the base space $X$.  Actually, the normalization $i/2 \pi$ is chosen in the Chern-Weil theory such that the cohomology classes are in fact integral.  We will use the notation
\begin{equation}
c(E) = \big( 1,  c_{1}(E), c_{2}(E), \cdots, c_{n}(E) \big)
\end{equation}
where $c_{k}(E) \in H^{2k}(X, \mathbb{Z})$ is the $k$-th Chern class.  The class $c_{k}(E)$ vanishes if $k$ is larger than the rank of $E$.  We can explicitly give the first few Chern classes in terms of the curvature $F_{A}$,
\begin{equation}\label{curvfirstchcrasss}
c_{1}(E) = \frac{i}{2 \pi} \big[ \text{Tr}(F_{A}) \big], \,\,\,\,\,\,\,\,\,\, c_{2}(E) = \frac{1}{8 \pi^{2}} \big[ \text{Tr}( F_{A} \wedge F_{A}) - \text{Tr}(F_{A})^{2}\big].
\end{equation}

\noindent The Chern classes and total Chern class satisfy some nice properties, which we summarize below:

\begin{enumerate}
\item Because the trivial bundle $X \times \mathbb{C}^{n}$ of rank $n$ admits the trivial connection which is flat, it follows from (\ref{eqn:totCHCRASS}) that all Chern classes $c_{k}$ of $X \times \mathbb{C}^{n}$ vanish for $k>0$.    

\item For complex vector bundles $E$ and $E'$, the \emph{Whitney product formula} says that
\begin{equation}
c(E \oplus E') = c(E) \cdot c(E').
\end{equation}
This can be seen from (\ref{eqn:totCHCRASS}) noting that the determinant is multiplicative on direct sums.  

\item It follows directly from the Whitney product formula that the first Chern class is additive on direct sums
\begin{equation}
c_{1}(E \oplus E') = c_{1}(E) + c_{1}(E').  
\end{equation}

\item If $F_{A}$ and $F'_{A'}$ are the curvatures of connections on $E$ and $E'$ respectively, then the curvature of the induced connection on $E \otimes E'$ is $F_{A} \otimes 1 + 1 \otimes F'_{A'}$.  It therefore follows from (\ref{curvfirstchcrasss}) that
\begin{equation}
c_{1}(E \otimes E') = \text{rk}(E')c_{1}(E) + \text{rk}(E) c_{1}(E')
\end{equation}
where $\text{rk}(E)$ and $\text{rk}(E')$ are the ranks of $E$ and $E'$, respectively.  
\end{enumerate}

One drawback to the total Chern class is that it doesn't behave nicely on tensor products.  In many ways, a preferable topological invariant of a complex vector bundle $E$ is the Chern character, defined as follows.  By the splitting principle \cite[Section 4.4]{huybrechts_complex_2004}, in a specific sense we can \emph{formally} decompose a bundle $E$ as a direct sum of line bundles 
\begin{equation}
E = L_{1} \oplus \cdots \oplus L_{n}
\end{equation}
where $n$ is the rank of $E$.  The formal Chern roots of $E$ are given by $x_{1}, \ldots, x_{n}$ where $x_{i} = c_{1}(L_{i})$.  The Chern class $c_{k}(E)$ can be expressed as the $k$-th symmetric function of the formal Chern roots.  

We define the \emph{Chern character} of the bundle $E$ in terms of the formal Chern roots as
\begin{equation}
\text{ch}(E) = \sum_{i=1}^{n} e^{x_{i}} = \big(n, \text{ch}_{1}(E), \text{ch}_{2}(E), \cdots \big)
\end{equation}
where the $k$-th Chern character $\text{ch}_{k}(E) \in H^{2k}(X, \mathbb{Q})$ is a rational cohomology class.  The exponential is defined via Taylor expansion, and the $k$-th Chern character can therefore be written as a polynomial in the lower Chern classes.  For example, we have $\text{ch}_{1}(E) = c_{1}(E)$ and
\begin{equation}
\text{ch}_{2}(E) = \frac{1}{2}\big( c_{1}(E)^{2} - 2c_{2}(E)\big) = - \frac{1}{8 \pi^{2}}\big[ \text{Tr}(F_{A} \wedge F_{A})\big].
\end{equation}
Unlike the Chern classes, $\text{ch}_{k}(E)$ will typically not vanish for $k$ larger than the rank of $E$.  For example, if $L$ is a line bundle, then for all $k \geq 0$, we have $\text{ch}_{k}(L) = \frac{1}{k!}c_{1}(L)^{k}$.  

The following two properties of the Chern character are the main reasons one might prefer it to the total Chern class.  For all vector bundles $E$ and $E'$, we have
\begin{equation}
\text{ch}\big( E \oplus E' \big) = \text{ch}(E) + \text{ch}(E'), \,\,\,\,\,\,\,\,\,\,\,\,\,\,\,\,\,\,\, \text{ch}\big( E \otimes E' \big) = \text{ch}(E) \, \text{ch}(E').
\end{equation}
These in fact show that the Chern character is a ring homomorphism from what is called \emph{topological K-theory} into $H^{*}(X, \mathbb{Q})$.  In the following chapter on coherent sheaves, we will encounter algebraic K-theory, where the Chern character will play a similar role (see Section \ref{subsecc:GgrpHRR}).  

Because the characteristic classes discussed above are topological invariants of the bundle, they are insensitive to the introduction of a bundle metric.  Therefore Chern classes and Chern characters on complex vector bundles are equivalent to those on $U(n)$-bundles.  However, an $SU(n)$-bundle $E$ has not only a Hermitian metric, but also a trivial determinant.  Because $c_{1} (\Lambda^{n}E) = c_{1}(E)$, the first Chern class of $SU(n)$-bundles must vanish.  

We mentioned above that part of the importance of characteristic classes is that by fixing them, one may classify certain topological features of a bundle.  We now give a few results in this direction.  For example, if $L$ is a complex line bundle (a $U(1)$-bundle) on an arbitrary smooth manifold $X$, then $c_{1}(L) \in H^{2}(X, \mathbb{Z})$ is the only non-vanishing Chern class, and one can show the following \cite[Theorem E.5]{freed_instantons_1991}.  

\begin{proppy}
The first Chern class $c_{1}(L)$ completely topologically classifies complex line bundles or $\text{U}(1)$-bundles on any base space $X$.  In other words, two line bundles are isomorphic if and only if their first Chern classes agree.  In addition, each element of $H^{2}(X, \mathbb{Z})$ is realized as the first Chern class of some line bundle on $X$.  
\end{proppy}

If $X$ is now a compact, orientable four-dimensional manifold and $E$ is a complex vector bundle on $X$, then the only non-vanishing topological invariants are $c_{1}(E)$ and $\text{ch}_{2}(E)$.  The \emph{topological charge} is typically defined to be the half-integer
\begin{equation}\label{eqn:topcharge4man}
k = \int_{X} \text{ch}_{2}(E) = - \frac{1}{8 \pi^{2}} \int_{X} \big[ \text{Tr} (F_{A} \wedge F_{A}) \big] \in \tfrac{1}{2} \mathbb{Z}.
\end{equation}
If $E$ is an $SU(n)$-bundle, then $c_{1}(E)=0$ and the topological charge is the only invariant.  Moreover, it is in fact an integer.  The following is an important result in the case of $n=2$.

\begin{proppy}{\bfseries \cite[Theorem E.5]{freed_instantons_1991}}
The topology of an $SU(2)$-bundle on a compact, oriented four-dimensional manifold is completely classified by the topological charge $k$.  
\end{proppy}

\noindent One should not get the impression that the Chern classes completely classify vector bundles topologically; this fails in general.

\subsection{The Yang-Mills Action Functional}\label{sec:YMACfunnn}
Up to this point, the base space $X$ was not assumed to have any additional structure beyond that of a smooth manifold.  In order to introduce the Yang-Mills functional and the Yang-Mills equations, we need further structure on $X$.  Let $(X,g)$ be an $m$-dimensional, oriented, Riemannian manifold with Riemannian metric $g$.  An action functional will require integrating over spacetime, so we will further assume $X$ is compact.  To provide some perspective, let us briefly describe an analogy in basic Hodge theory.

\subsubsection{Hodge Theory in Riemannian Geometry}

It is well-known that the de Rham cohomology groups of a smooth $m$-dimensional manifold $X$ are isomorphic to the singular cohomology groups $H^{p}(X, \mathbb{R})$ with real coefficients, and are finite-dimensional real vector spaces if $X$ is compact.  If $X$ additionally has a Riemannian metric $g$, then we can introduce the Hodge star operator on differential forms $\star : \Omega^{k}(X) \to \Omega^{m-k}(X)$ defined for all $k=1, \ldots, m$ by 
\begin{equation}
\alpha \wedge \star \beta = \langle \alpha, \beta \rangle \text{dvol}_{g}
\end{equation}
for all $\alpha, \beta \in \Omega^{k}(X)$ where $\langle \alpha, \beta \rangle$ is the natural pairing on $k$-forms producing a smooth function on $X$, and $\text{dvol}_{g}$ is the Riemannian volume element.  The Hodge star is an involution up to a sign: for $\alpha \in \Omega^{k}(X)$, we have $\star^{2} \alpha = (-1)^{k(m-k)}\alpha$.  By integrating over the compact manifold $X$, we get a symmetric positive-definite non-degenerate $L^{2}$-inner product on $k$-forms
\begin{equation}\label{eqn:innprod}
(\alpha, \beta) = \int_{X} \alpha \wedge \star \beta, \,\,\,\,\,\,\,\,\,\, \alpha, \beta \in \Omega^{k}(X).
\end{equation}

Making use of the Hodge star, one typically defines the differential operator lowering the degree of $k$-forms $d^{*} : \Omega^{k}(X) \to \Omega^{k-1}(X)$ by
\begin{equation}
d^{*} = (-1)^{m(k-1) + 1} \star d \star.
\end{equation}

\begin{lemmy}\label{lemmy:adjjjoplema}
The operator $d^{*}$ is adjoint to the exterior derivative $d$ with respect to (\ref{eqn:innprod}) in the sense that $(\alpha, d^{*} \beta) = (d \alpha, \beta)$ for all $\alpha \in \Omega^{k-1}(X)$ and $\beta \in \Omega^{k}(X)$.  
\end{lemmy}
\begin{proof}
The key is to note that because $d(\alpha \wedge \star \beta)$ is an exact top-form and $X$ is compact, by Stokes theorem we know $\int_{X} d(\alpha \wedge \star \beta) =0$.  The conclusion follows by a simple computation making use of the Leibniz rule as well as the sign in $\star^{2}$.  
\end{proof}

Because the above inner product is positive-definite, for all $k$ we may introduce the following $L^{2}$-norm via the functional $\mathcal{E}: \Omega^{k}(X) \to \mathbb{R}_{\geq 0}$ defined by
\begin{equation} \label{eqn:funcHodgeDeRhTh}
\mathcal{E}(\alpha) = (\alpha, \alpha) = \int_{X} \alpha \wedge \star \alpha.
\end{equation}
We want to study the differential of $\mathcal{E}$, which we denote $d\mathcal{E}$, and interpret as a global section of the cotangent bundle of the infinite-dimensional vector space $\Omega^{k}(X)$.  For all $\alpha \in \Omega^{k}(X)$ and $\zeta_{\alpha} \in T_{\alpha}\Omega^{k}(X) \cong \Omega^{k}(X)$
\begin{equation}
d \mathcal{E}(\alpha, \zeta_{\alpha}) = \frac{d}{d \epsilon} \mathcal{E}(\alpha + \epsilon \zeta_{\alpha})\big|_{\epsilon=0}
\end{equation}
represents the derivative of $\mathcal{E}$ at $\alpha$ in the $\zeta_{\alpha}$ direction.  This is also called a functional derivative or a variational derivative.  A direct computation verifies that
\begin{equation} \label{eqn:harmycond}
d \mathcal{E}(\alpha, \zeta_{\alpha}) = 2(\alpha, \zeta_{\alpha}).
\end{equation}
Given a functional like $\mathcal{E}$, one can study its \emph{critical locus} which is defined by
\begin{equation}
\text{Crit}(\mathcal{E}) \coloneqq \bigg\{ \alpha \in \Omega^{k}(X) \, \bigg| \, d \mathcal{E}(\alpha, \zeta_{\alpha}) =0, \,\, \text{for all} \,\, \zeta_{\alpha} \in T_{\alpha}\Omega^{k}(X)\bigg\}.
\end{equation}
It follows from (\ref{eqn:harmycond}) and the non-degeneracy of the inner product $( \cdot, \cdot)$ that $\text{Crit}(\mathcal{E})= \{ 0 \}$.  In other words, there are no genuine critical points of $\mathcal{E}$.  

However, we can look for critical points corresponding to variations in only certain directions.  Let $\alpha \in \Omega^{k}(X)$ be a closed form representing the cohomology class $[\alpha]$.  We want the condition that $\alpha$ minimizes $\mathcal{E}$ within the class $[\alpha]$.  Because any other representative of $[ \alpha]$ is of the form $\alpha + \epsilon d \beta$ for $\beta \in \Omega^{k-1}(X)$, the condition we want is that $d \mathcal{E}( \alpha, d \beta)=0$ for all $\beta$.  By (\ref{eqn:harmycond}) and Lemma \ref{lemmy:adjjjoplema} we have
\begin{equation}
d \mathcal{E}(\alpha, d \beta) = 2( \alpha, d \beta) = 2(d^{*} \alpha, \beta)
\end{equation}
and in order for this to vanish for all $\beta$, the following \emph{Euler-Lagrange equation} must be satisfied
\begin{equation}\label{eqn:ELeqqn}
d^{*}\alpha =0.
\end{equation}
However, because $\alpha$ defines a cohomology class and is hence closed, the Euler-Lagrange equation (\ref{eqn:ELeqqn}) is equivalent to $\Delta \alpha=0$, where $\Delta = d d^{*} + d^{*} d$ is the Laplacian.  Therefore, given a closed form $\alpha \in \Omega^{k}(X)$, the condition that $\alpha$ minimizes $\mathcal{E}$ within $[\alpha]$ is that it is a harmonic form.  Conversely, the Hodge-de Rham theorem says that solutions to $\Delta \alpha=0$, which are a priori just $L^{2}$-forms, are in fact smooth.  If $\mathcal{H}^{k}(X, g)$ denotes the space of harmonic $k$-forms on $X$, this shows we have the isomorphism of vector spaces
\begin{equation}
H^{k}(X, \mathbb{R}) \cong \mathcal{H}^{k}(X, g)
\end{equation}
which says that every real cohomology class has a unique harmonic representative.  One may hear harmonic forms described as \emph{zero modes} of the Laplacian which by the discussion above, is equivalent to minimizing the functional $\mathcal{E}$ (within a cohomology class).

\subsubsection{Generalization to Yang-Mills Theory}

The reason for reviewing Hodge theory above is that in some sense, Yang-Mills theory provides a beautiful \emph{non-linear} generalization to differential forms on a compact Riemannian manifold valued in an adjoint or endomorphism bundle.  As above, let $(X, g)$ be a smooth compact orientable Riemannian manifold of dimension $m$.  

To begin with some amount of generality, let $G$ be an arbitrary compact Lie group with Lie algebra $\mathfrak{g}$ and $\mathcal{P} \to X$ a principal $G$-bundle.  Of course the adjoint bundle $\text{ad}\mathcal{P}$ is constructed as an associated vector bundle, as we have seen.  We want an analogous inner product to (\ref{eqn:innprod}) on bundle-valued forms $\Omega^{k}(\text{ad}\mathcal{P})$.  Given two sections $\omega, \eta \in \Omega^{k}(\text{ad}\mathcal{P})$ we get an $\text{ad}\mathcal{P}$-valued top-form $\omega \wedge \star \eta \in \Omega^{m}(\text{ad}\mathcal{P})$, noting that the Hodge star applies in the obvious way on the form part, and we interpret the wedge product to be the exterior product on the forms and the composition of the bundle-valued part.  

In order to get a well-defined inner product by integrating $\omega \wedge \star \eta$ over $X$, we clearly must introduce a fiberwise inner product between vectors in the fibers of the adjoint bundle.  Recalling that the fibers of $\text{ad}\mathcal{P}$ are isomorphic to the Lie algebra $\mathfrak{g}$, what we need is an inner product $\langle \cdot , \cdot \rangle_{\mathfrak{g}}$ on $\mathfrak{g}$.  However, given an inner product on the fiber, to get a well-defined associated vector bundle, one must have the inner product invariant with respect to the chosen representation.  For the case at hand, we must have
\begin{equation}
\langle \text{ad}(g) \, a , \text{ad}(g) \, b \rangle_{\mathfrak{g}} = \langle a, b \rangle_{\mathfrak{g}}
\end{equation}
for all $g \in G$ and $a,b \in \mathfrak{g}$.  We refer to such an inner product as \emph{ad-invariant}.  If we ultimately want a positive-definite inner product on sections $\Omega^{k}(\text{ad}\mathcal{P})$, then $\langle \cdot , \cdot \rangle_{\mathfrak{g}}$ must be positive-definite.  It is this condition which requires that $G$ be a compact Lie group.

From here on, we will assume $G$ is one of the compact groups $U(n)$ or $SU(n)$ with positive definite non-degenerate ad-invariant inner product $\langle \cdot , \cdot \rangle_{\mathfrak{g}} = -\text{Tr}_{\mathfrak{g}}(\cdot \,\, \cdot)$, where $\text{Tr}_{\mathfrak{g}}$ denotes the trace in the Lie algebra $\mathfrak{g}$, which we are assuming is either $\mathfrak{u}_{n}$ or $\mathfrak{su}_{n}$.  Recall that $\omega \wedge \star \eta \in \Omega^{m}(\text{ad}\mathcal{P})$ for sections $\omega, \eta \in \Omega^{k}(\text{ad}\mathcal{P})$.  We now want to apply the inner product $\langle \cdot , \cdot \rangle_{\mathfrak{g}}$ fiberwise to get an honest volume form on $X$.  Mildly abusing notation, we will write 
\begin{equation}\label{eqn:traceadeqn}
\langle \omega , \eta \rangle_{\mathfrak{g}} = -\text{Tr}_{\mathfrak{g}}(\omega \wedge \star \eta) \in \Omega^{m}(X)
\end{equation}
to mean that $\langle \cdot , \cdot \rangle_{\mathfrak{g}}$ is applied fiberwise on the bundle-valued parts while the exterior product and Hodge star combine to give a volume form, as above.  We can finally define the symmetric positive-definite inner product
\begin{equation} \label{eqn:forminnprod}
(\omega, \eta) = - \int_{X} \text{Tr}_{\mathfrak{g}}(\omega \wedge \star \eta)
\end{equation}
for $\omega, \eta \in \Omega^{k}(\text{ad}\mathcal{P})$.  Moreover, defining $|\eta|^{2} \coloneqq \langle \eta, \eta \rangle_{\mathfrak{g}}$, we have the following $L^{2}$-norm for all $\eta \in \Omega^{k}(\text{ad}\mathcal{P})$
\begin{equation}
( \eta, \eta) = \int_{X} | \eta |^{2} = - \int_{X} \text{Tr}_{\mathfrak{g}}(\eta \wedge \star \eta).
\end{equation}

The bundle-valued forms which are of the most interest to us are curvature two-forms of $G$-connections on complex vector bundles with metrics.  Let $(E,h)$ be a Hermitian vector bundle with structure group $G$ and $\mathcal{P}_{E}$ the uniquely corresponding principal $U(n)$ or $SU(n)$-bundle.  We have seen that $\mathfrak{g}_{E} \cong \text{ad}\mathcal{P}_{E}$ is a real subbundle of $\text{End}E$ with fiber $\mathfrak{u}_{n}$ or $\mathfrak{su}_{n}$.  Given a $G$-connection $d_{A}$, the curvature is a section $F_{A} \in \Omega^{2}(\mathfrak{g}_{E})$.  In terms of the curvature, we define the Yang-Mills action functional $S_{\text{YM}} : \mathscr{A}_{G}(E,h) \to \mathbb{R}_{\geq 0}$ by
\begin{equation}\label{eqn:TANGMILLSFUNC}
S_{\text{YM}}(d_{A}) = \int_{X} |F_{A}|^{2} = - \int_{X} \text{Tr}_{\mathfrak{g}}(F_{A} \wedge \star F_{A}).
\end{equation}

We will now show that the Yang-Mills functional is invariant under gauge transformations of the connection $d_{A}$.  In physics parlance, we refer to this as the \emph{gauge invariance} of an action functional.  

\begin{proppy}
Given a global gauge transformation $\sigma \in \mathscr{G}$, we have $S_{\text{YM}}(\sigma^{*}d_{A}) = S_{\text{YM}}(d_{A})$ for all $G$-connections $d_{A}$.    
\end{proppy}
\begin{proof}
Recall from (\ref{eqn:gaugetransCONN}) and (\ref{eqn:gautranyFAA}) that $\sigma$ acts on both $d_{A}$ and $F_{A}$ by pullback.  More specifically, $\sigma^{*}d_{A} = \sigma^{-1} \circ d_{A} \circ \sigma$ and $\sigma^{*} F_{A} = \sigma^{-1} F_{A} \sigma$.  By a direct computation we have
\begin{equation}
F_{\sigma^{*}A} = (\sigma^{*} d_{A}) \circ (\sigma^{*} d_{A}) = \sigma^{-1} \circ (d_{A} \circ d_{A}) \circ \sigma = \sigma^{*} F_{A}.
\end{equation}
Applying this within the Yang-Mills functional, we get
\begin{equation}
S_{\text{YM}}(\sigma^{*} d_{A}) = - \int_{X} \text{Tr}_{\mathfrak{g}}\big( \sigma^{*}F_{A} \wedge \star \sigma^{*}F_{A}\big).
\end{equation}
The gauge transformation has no interaction with the form part or the Hodge star.  Recall that the notation $\text{Tr}_{\mathfrak{g}}\big( \sigma^{*}F_{A} \wedge \star \sigma^{*}F_{A}\big)$ is understood to involve a fiberwise application of the inner product $\langle \cdot , \cdot \rangle_{\mathfrak{g}}$.  But by the local description of gauge transformations as local $G$-valued functions $\sigma_{\alpha} : U_{\alpha} \to G$, we have at all $u \in U_{\alpha}$
\begin{equation}
\big\langle \text{ad}\big( \sigma_{\alpha}(u)\big) \cdot F_{A}(u) \, ,  \, \text{ad}\big( \sigma_{\alpha}(u)\big) \cdot F_{A}(u) \, \big\rangle_{\mathfrak{g}} = \big\langle F_{A}(u) \, , \, F_{A}(u) \big\rangle_{\mathfrak{g}}
\end{equation}
by the ad-invariance of the inner product, where $F_{A}(u) \in \mathfrak{g}$ is the value in the bundle of $F_{A} \in \Omega^{2}(\mathfrak{g}_{E})$ at $u \in U_{\alpha}$.  Therefore, $S_{\text{YM}}(\sigma^{*}d_{A}) = S_{\text{YM}}(d_{A})$.  
\end{proof}

\begin{cory}
The Yang-Mills action functional descends to a functional on the quotient of $\mathscr{A}_{G}(E,h)$ by gauge transformations $S_{\text{YM}} : \mathscr{A}_{G}(E,h)/\mathscr{G} \to \mathbb{R}_{\geq 0}$.  
\end{cory}

\subsubsection{The Yang-Mills Equations}

Given a $G$-connection $d_{A}$, we get a covariant derivative $d_{A} : \Omega^{1}(\mathfrak{g}_{E}) \to \Omega^{2}(\mathfrak{g}_{E})$ on the endomorphism bundle.  With respect to the inner product (\ref{eqn:forminnprod}) we define the formal adjoint operator to $d_{A}$ as $d_{A}^{*} : \Omega^{2}(\mathfrak{g}_{E}) \to \Omega^{1}(\mathfrak{g}_{E})$, with explicit form
\begin{equation}\label{eqn:covariantADJ}
d_{A}^{*} = (-1)^{m+1} \star d_{A} \star
\end{equation}
in terms of the Hodge star operator.  One can verify that for all $\eta \in \Omega^{1}(\mathfrak{g}_{E})$ and $\omega \in \Omega^{2}(\mathfrak{g}_{E})$ we have
\begin{equation} \label{eqn:covadj}
(\omega, d_{A} \eta) = (d_{A}^{*} \omega, \eta).
\end{equation}
Just as we did with the functional $\mathcal{E}$ in (\ref{eqn:funcHodgeDeRhTh}), we consider the derivative $dS_{\text{YM}}$ which we interpret as a global section of the cotangent bundle of $\mathscr{A}_{G}(E,h)$.  Given a $G$-connection $d_{A}$ on $E$ and $a \in \Omega^{1}(\mathfrak{g}_{E})$ we have
\begin{equation}
d S_{\text{YM}}(d_{A}, a) \coloneqq \frac{d}{d \epsilon} S_{\text{YM}}(d_{A} + \epsilon a) \big|_{\epsilon =0}
\end{equation}
which we interpret as the derivative of $S_{\text{YM}}$ at $d_{A}$ in the direction of $a$.  We may also call this the functional derivative or variational derivative of $S_{\text{YM}}$.  The \emph{critical locus} of $S_{\text{YM}}$ is defined as
\begin{equation}
\text{Crit}(S_{\text{YM}}) \coloneqq \bigg\{ d_{A} \in \mathscr{A}_{G}(E,h) \, \bigg| \, dS_{\text{YM}}(d_{A}, a) =0 \,\, \text{for all} \,\, a \in \Omega^{1}(\mathfrak{g}_{E})\bigg\}.
\end{equation}  
We interpret a connection lying in $\text{Crit}(S_{\text{YM}})$ to be one which locally minimizes the Yang-Mills functional.  We now want to ask what constraint must a $G$-connection $d_{A}$, or its corresponding curvature $F_{A}$, satisfy if it is to locally minimize the value of the functional?  Recall that equation (\ref{eqn:connaff}) provides the expression for $F_{A + \epsilon a}$, which allows us to compute
\begin{equation} \label{eqn:YMvariation}
\setlength{\jot}{12pt}
\begin{split}
dS_{\text{YM}}(d_{A}, a) = \frac{d}{d \epsilon } S_{\text{YM}}(d_{A} + \epsilon a) \big|_{\epsilon=0} & = - 2 \int_{X} \text{Tr}_{\mathfrak{g}}\big( d_{A}(a) \wedge \star F_{A}\big)\\
& = 2\big(d_{A}(a), F_{A}\big) = 2\big(F_{A}, d_{A}(a)\big) = 2\big(d_{A}^{*}F_{A}, a\big)
\end{split}
\end{equation}
where in the final equality we have used (\ref{eqn:covadj}).  In order for (\ref{eqn:YMvariation}) to vanish for all $a \in \Omega^{1}(\mathfrak{g}_{E})$ we must have $d^{*}_{A}F_{A} =0$, which is called the Euler-Lagrange equation for the Yang-Mills functional.  By (\ref{eqn:covariantADJ}) it is clear that $d_{A}^{*}F_{A}=0$ if and only if $d_{A} \star F_{A}=0$, so the Euler-Lagrange equation can equivalently be written either way.  Together with the Bianchi identity (\ref{eqn:BianchiId}) which always holds, we define the \emph{Yang-Mills equations} to be
\begin{equation}\label{eqn:YMequations}
\boxed{ d_{A} \star F_{A} =0 \,\,\,\,\,\,\,\,\,\,\,\,\,\, d_{A} F_{A}=0.}
\end{equation}  
A $G$-connection $d_{A}$ satisfying (\ref{eqn:YMequations}) is referred to as a \emph{Yang-Mills connection} while $F_{A}$ is referred to as a \emph{Yang-Mills field}.  The Yang-Mills connections correspond exactly to the critical points of $S_{\text{YM}}$,
\begin{equation}
\text{Crit}(S_{\text{YM}}) = \bigg\{ \text{Yang-Mills Connections}\bigg\} \subset \mathscr{A}_{G}(E,h).
\end{equation}

Mathematically, the Yang-Mills equations are non-linear analogs of the conditions for a two-form to be harmonic.  Replacing the covariant derivative $d_{A}$ by the exterior derivative $d$, and replacing the curvature $F_{A}$ by an ordinary two-form $\alpha$, we recover the conditions for $\alpha$ to be a harmonic form.  Moreover, Yang-Mills fields minimize the Yang-Mills functional $S_{\text{YM}}$ while harmonic forms minimize the functional $\mathcal{E}$ (within a fixed cohomology class).

One consequence of the Yang-Mills functional being gauge invariant is that given any Yang-Mills connection $d_{A}$, for all $\sigma \in \mathscr{G}$ the gauge transformed connection $\sigma^{*} d_{A}$ is also Yang-Mills.  Therefore, $\text{Crit}(S_{\text{YM}})$ inherits a well-defined action by the gauge group.  We define the \emph{Yang-Mills moduli space} to be 
\begin{equation}
\mathscr{M}^{(\text{YM})}_{G}(E,h) \coloneqq \text{Crit}(S_{\text{YM}}) / \mathscr{G}.  
\end{equation}
In general, it is infinite dimensional but it has finite dimensional subspaces which are of interest.  This leads us to a discussion of instantons.

\subsection{Instantons on Four-Manifolds}  

In the previous section we saw that for $(X,g)$ a compact, oriented Riemannian manifold and $G$ a compact Lie group, the Yang-Mills functional is the natural functional on the space of $G$-connections on vector bundles over $X$.  However, we did not place any restrictions on the dimension of $X$.  It turns out, that four-dimensional oriented Riemannian manifolds (called four-manifolds) hold a special place in Yang-Mills theory.  It is in such a case where one can study \emph{instantons}, which we will define to be a certain class of Yang-Mills connections.  To an algebraic geometer, the most important examples of four-manifolds are complex K\"{a}hler surfaces and smooth algebraic surfaces.  In fact, one theme of this thesis is that in some of these examples, instantons bridge remarkable connections between differential geometry, algebraic geometry, enumerative geometry, and physics.

Given a four-manifold $(X, g)$, we can consider the Hodge star operator acting on the vector space of two-forms $\Omega^{2}(X)$.  In such a case, applying the operator twice yields the identity $\star^{2} = 1$, and the eigenvalues of $\star$ are therefore $\pm1$.  This induces a decomposition of $\Omega^{2}(X)$ into eigenspaces with respect to $\star$
\begin{equation}\label{eqn:instdecomp}
\Omega^{2}(X) = \Omega^{2}_{+}(X) \oplus \Omega_{-}^{2}(X)
\end{equation}
where $\Omega^{2}_{+}(X)$ is defined as the vector space of \emph{self-dual} (SD) two-forms $\eta$ satisfying $\star \eta=\eta$ and $\Omega_{-}^{2}(X)$ is the space of \emph{anti-self-dual} (ASD) two-forms satisfying $\star \eta = - \eta$.  Given any two-form $\omega$ on $X$, we can uniquely write it as a sum $\omega = \omega_{+} + \omega_{-}$ of a SD and ASD form.  

\begin{lemmy}
The decomposition (\ref{eqn:instdecomp}) of $\Omega^{2}(X)$ is orthogonal with respect to the natural inner product $\langle\, , \,\rangle$ on forms.
\end{lemmy}  
\begin{proof}
By definition, $\langle \omega_{+}, \omega_{-} \rangle \text{dvol}_{g} = \omega_{+} \wedge \star \omega_{-}$ and the inner product is symmetric,
\begin{equation}
\langle \omega_{+}, \omega_{-} \rangle \text{dvol}_{g}=\langle \omega_{-}, \omega_{+} \rangle \text{dvol}_{g}.
\end{equation}
Using anti-self-duality, the lefthand side simplifies to $-\omega_{+} \wedge \omega_{-}$ and using self-duality the righthand side simplifies to $\omega_{-} \wedge \omega_{+}$.  But since $\omega_{\pm}$ are two-forms, we have $\omega_{-} \wedge \omega_{+} = \omega_{+} \wedge \omega_{-}$, which implies that $\langle \omega_{+}, \omega_{-} \rangle =0$.  
\end{proof}

The above discussion, including the definition of self-duality and anti-self-duality, holds as well for bundle-valued forms.  In particular, let $(E,h)$ be a Hermitian vector bundle with structure group $G$ assumed to be either $U(n)$ or $SU(n)$.  

\begin{defn}
On a four-manifold, we say $F \in \Omega^{2}(\mathfrak{g}_{E})$ is self-dual (SD) if $\star F = F$ and anti-self-dual (ASD) if $\star F = - F$.  We then have the deomposition
\begin{equation}
\Omega^{2}(\mathfrak{g}_{E}) = \Omega^{2}_{+}(\mathfrak{g}_{E}) \oplus \Omega_{-}^{2}(\mathfrak{g}_{E})
\end{equation}
orthogonal with respect to $\langle \cdot , \cdot \rangle_{\mathfrak{g}} = -\text{Tr}_{\mathfrak{g}}( \cdot \wedge \star \, \cdot)$ in (\ref{eqn:traceadeqn}).  Here, $\Omega^{2}_{+}(\mathfrak{g}_{E})$ is the space of SD two-forms valued in $\mathfrak{g}_{E}$ and $\Omega^{2}_{-}(\mathfrak{g}_{E})$ is the space of ASD two-forms valued in $\mathfrak{g}_{E}$.  
\end{defn}

On any Riemannian manifold of dimension $2n$ with $n$ even, the Hodge star operator induces a similar decomposition on the middle-dimensional forms.  Nonetheless, one reason dimension four is distinguished in Yang-Mills theory is because the curvature two-form admits a unique decomposition $F_{A} = F_{A}^{+} + F_{A}^{-}$ into SD and ASD components.  In the case of a four-manifold, we will refer to a connection as SD or ASD if its curvature two-form has that property, as defined above.

Suppose further that $(X,g)$ is a compact four-manifold.  Because $\Omega_{+}^{2}(\mathfrak{g}_{E})$ and $\Omega_{-}^{2}(\mathfrak{g}_{E})$ are orthogonal with respect to $\langle \cdot , \cdot \rangle_{\mathfrak{g}} = - \text{Tr}_{\mathfrak{g}}(\cdot \wedge \star \, \cdot)$ for $\mathfrak{g}$ equal to $\mathfrak{u}_{n}$ or $\mathfrak{su}_{n}$, by (\ref{eqn:TANGMILLSFUNC}) we have the following decomposition of the Yang-Mills functional,
\begin{equation} \label{eqn:YMplusminus}
S_{\text{YM}}(d_{A}) = \int_{X} |F_{A}^{+}|^{2} + |F_{A}^{-}|^{2}
\end{equation}
which is of course non-negative.  Recall from (\ref{eqn:topcharge4man}) that $\text{Tr}_{\mathfrak{g}}(F_{A} \wedge F_{A})$ gives rise to a topological invariant of a compact four-manifold called the topological charge $k$.  We can compute 
\begin{equation}
\setlength{\jot}{12pt}
\begin{split}
\text{Tr}_{\mathfrak{g}}(F_{A} \wedge F_{A}) & = \text{Tr}_{\mathfrak{g}}(F_{A}^{+} \wedge F_{A}^{+}) + \text{Tr}_{\mathfrak{g}}(F_{A}^{-} \wedge F_{A}^{-}) \\
& = \text{Tr}_{\mathfrak{g}}(F_{A}^{+} \wedge \star F_{A}^{+}) - \text{Tr}_{\mathfrak{g}}(F_{A}^{-} \wedge \star F_{A}^{-})
\end{split}
\end{equation}
where the first equality follows from the orthogonality of $\Omega_{+}^{2}(\mathfrak{g}_{E})$ and $\Omega_{-}^{2}(\mathfrak{g}_{E})$ while the second equality is by self-duality and anti-self-duality.  Integrating over $X$, and paying careful attention to minus signs, we therefore get
\begin{equation} \label{eqn:topboundd}
- 8 \pi^{2} k = \int_{X} \text{Tr}_{\mathfrak{g}}(F_{A} \wedge F_{A}) = \int_{X} |F_{A}^{-}|^{2} - |F_{A}^{+}|^{2}.
\end{equation}

Comparing (\ref{eqn:YMplusminus}) and (\ref{eqn:topboundd}) we conclude that for all $G$-connections, $S_{\text{YM}}(d_{A}) \geq |8 \pi^{2} k|$, which is a topological bound on the Yang-Mills functional.  There are two cases to consider:

\begin{enumerate}
\item For $k>0$, we have $S_{\text{YM}}(d_{A}) \geq 8 \pi^{2} k$ with equality if and only if $F_{A}$ is SD, i.e. $F_{A}^{-} =0$.  

\item For $k<0$, we have $S_{\text{YM}}(d_{A}) \geq - 8 \pi^{2} k$ with equality if and only if $F_{A}$ is ASD, i.e $F_{A}^{+} =0$.  
\end{enumerate}

\noindent Recall that the Yang-Mills equations (\ref{eqn:YMequations}) are made up of the Bianchi identity $d_{A} F_{A} =0$, which holds for all connections, as well as $d_{A} \star F_{A} =0$.  In the case of SD or ASD connections, $d_{A} \star F_{A}=0$ is satisfied and clearly follows from the Bianchi identity.  Therefore, on a compact four-manifold the Yang-Mills equations specialize to the \emph{SD or ASD Yang-Mills equations}
\begin{equation}
\star F_{A} = \pm F_{A}.
\end{equation}

\begin{defn}
Let $(X,g)$ be a compact four-manifold and let $(E,h)$ be a Hermitian vector bundle with structure group $G$, assumed to be either $U(n)$ or $SU(n)$.  An instanton\footnote{The name \emph{instanton} comes from physics, and refers to a field configuration which is localized in both time and space.  In other words, the configuration exists only for an ``instant".} is simply a SD or ASD $G$-connection on $E$.  By the discussion above, instantons are examples of Yang-Mills connections which globally minimize the Yang-Mills functional $S_{\text{YM}}$.  
\end{defn}

\noindent By (\ref{eqn:topboundd}), the proof of the following proposition is trivial.  

\begin{proppy}
If $(E,h)$ is a Hermitian vector bundle on a compact four-manifold and if the topological charge $k$ vanishes, then an instanton on $E$ is a flat connection.  
\end{proppy}

\section{Yang-Mills Theory on K\"{a}hler Manifolds}\label{sec:YMThhKahler}

For a complex or algebraic geometer, the Yang-Mills theory introduced in the previous section primarily becomes of interest as a theory of holomorphic vector bundles over K\"{a}hler manifolds.  Some features of the general theory (for example, the infinite-dimensional affine space of connections) are nicer in the more restrictive setting.  Indeed, we will see that there is a unique connection compatible with the extra structure in such a way that we can translate a space of such connections to a space of holomorphic structures on a Hermitian vector bundle.  We will introduce the Hermitian Yang-Mills equation, which arises in the study of instantons as well as D-branes in string theory.  By the Donaldson-Uhlenbeck-Yau theorem, irreducible connections solving the Hermitian Yang-Mills equation correspond to stable holomorphic vector bundles.  This section is a short prequel in some sense of our study in the following chapter on stability of coherent sheaves.  We will not provide full details, and we will often refer to the literature for proofs.  Some great resources for various parts of this material are \cite{huybrechts_complex_2004,friedman_gauge_1997,donaldson_geometry_1997,lubke_kobayashi-hitchin_1995}.  In fact, much of what follows in this brief survey is modelled on the book \cite{lubke_kobayashi-hitchin_1995}.

\subsection{From Connections to Holomorphic Structures on Bundles}

Let $X$ be a complex manifold, and let $E \to X$ be a complex vector bundle.  It is a standard fact that the complex structure of $X$ induces the decomposition of the exterior derivative
\begin{equation}
d = \partial + \overline{\partial}
\end{equation}
and allows us to define $\Omega^{p,q}(E)$ as the vector space of smooth $(p,q)$-forms valued in $E$.  A holomorphic vector bundle of rank $r$ is typically thought of as a complex manifold $\mathcal{E}$ called the total space, with a surjective holomorphic map $\mathcal{E} \to X$ such that the fiber at each point is a complex vector space of dimension $r$.  Equivalently, it is a complex vector bundle whose transition functions are holomorphic.  As shown in \cite[Theorem 5.1]{atiyah_self-duality_1978}, there is yet another equivalent characterization which will be useful in what follows.  
\begin{defn}\label{defn:defnholstrbundre}
A holomorphic structure on the complex vector bundle $E$ over $X$ is a $\mathbb{C}$-linear map 
\begin{equation}
\overline{\delta} : \Omega^{0}(E) \to \Omega^{0,1}(E)
\end{equation}
such that $\overline{\delta} \circ \overline{\delta} =0$, and the Leibniz rule $\overline{\delta}(f s) = f \overline{\delta}(s) + s \overline{\partial}(f)$ holds for all $f \in C^{\infty}(X)$ and $s \in \Omega^{0}(E)$.  We will typically denote by $\mathcal{E}_{\bar{\delta}}$ a complex vector bundle $E$ with a choice of a holomorphic structure $\overline{\delta}$, and refer to it as a holomorphic bundle.
\end{defn}

We say a holomorphic vector bundle $\mathcal{E}$ is \emph{simple} if $H^{0}(X, \text{End}\mathcal{E})$ is one-dimensional and generated over $\mathbb{C}$ by $\text{id}_{\mathcal{E}}$.  We call $\overline{\delta}$ a simple holomorphic structure if $\mathcal{E}_{\bar{\delta}}$ is simple.  One should think that simple holomorphic vector bundles are those with the minimal number of global holomorphic endomorphisms.  
  
\begin{defn}
We define $\mathscr{H}(E)$ to be the space of holomorphic structures on the underlying complex vector bundle $E$, and $\mathscr{H}_{\text{simp}}(E)$ to be the space of simple holomorphic structures on $E$.  
\end{defn}  

Recall from Definition \ref{defn:comprexgrpgTRANY} the complex group of gauge transformations $\mathscr{G}^{\mathbb{C}}$, which acts on $\mathscr{H}(E)$ from the right as follows.  For all $\overline{\delta} \in \mathscr{H}(E)$ and $\sigma \in \mathscr{G}^{\mathbb{C}}$ we define
\begin{equation}
\overline{\delta} \cdot \sigma \coloneqq \sigma^{-1} \circ \overline{\delta} \circ \sigma. 
\end{equation}
Indeed, it is clear that if $\overline{\delta} \circ \overline{\delta} =0$, then $(\overline{\delta} \cdot \sigma) \circ (\overline{\delta} \cdot \sigma)=0$ as well.  The complex gauge transformations also act in a well-defined way on $\mathscr{H}_{\text{simp}}(E)$.  We want to identify holomorphic structures related by gauge transformations, as in the following definition.  

\begin{defn}
Two holomorphic structures $\overline{\delta}_{1}, \overline{\delta}_{2} \in \mathscr{H}(E)$ are said to be isomorphic if there exists $\sigma \in \mathscr{G}^{\mathbb{C}}$ such that $\overline{\delta}_{1} = \overline{\delta}_{2} \cdot \sigma$.  Equivalently, we will say that $\mathcal{E}_{\bar{\delta}_{1}}$ and $\mathcal{E}_{\bar{\delta}_{2}}$ are isomorphic holomorphic vector bundles.  We make the same definition for simple holomorphic structures.  
\end{defn}
  
Taking the quotient by the right action of $\mathscr{G}^{\mathbb{C}}$ on $\mathscr{H}_{\text{simp}}(E)$, we get the moduli space of isomorphism classes of simple holomorphic structures on $E$, which we denote
\begin{equation}\label{eqn:SImpHolMOD}
\mathcal{M}_{\text{simp}}(E) \coloneqq \mathscr{H}_{\text{simp}}(E) \big/ \mathscr{G}^{\mathbb{C}}.
\end{equation}
This has the structure of a complex analytic space, though it is generally non-reduced and non-Hausforff \cite[Lemma 4.3.5]{lubke_kobayashi-hitchin_1995}.  We do not want non-separated points in a moduli space, which is part of the reason why in the next section we will restrict attention to the stable locus in $\mathscr{H}_{\text{simp}}(E)$.  

Having laid some foundations on holomorphic structures on bundles, we will now see that a Hermitian connection $d_{A}$ on a Hermitian vector bundle $(E,h)$ determines a holomorphic structure if an integrability condition is satisfied.  The complex structure on $X$ induces a decomposition on the connection
\begin{equation}
d_{A} = \partial_{A} + \overline{\partial}_{A},
\end{equation}
such that $\partial_{A} : \Omega^{0}(E) \to \Omega^{1,0}(E)$ and $\overline{\partial}_{A} : \Omega^{0}(E) \to \Omega^{0,1}(E)$ are the holomorphic and anti-holomorphic components.  There is the corresponding splitting of the curvature two-form
\begin{equation}
F_{A} = F_{A}^{2,0} + F_{A}^{1,1} + F_{A}^{0,2} \in \Omega^{2,0}(\mathfrak{g}_{E}) \oplus \Omega^{1,1}(\mathfrak{g}_{E}) \oplus \Omega^{0,2}(\mathfrak{g}_{E})
\end{equation}
where $F_{A}^{2,0} = \partial_{A} \circ \partial_{A}$ and $F_{A}^{0,2} = \overline{\partial}_{A} \circ \overline{\partial}_{A}$.  

\begin{defn}
We say the connection $d_{A}$ is integrable if $F_{A} \in \Omega^{1,1}(\mathfrak{g}_{E})$.  In addition, we say the connection is irreducible if the kernel of the induced connection on $\text{End}E$ is one-dimensional and generated over $\mathbb{R}$ by $i \cdot \text{id}_{E}$.    
\end{defn}

\begin{defn}
On a Hermitian vector bundle $(E, h)$ we denote the infinite-dimensional affine space of Hermitian connections by $\mathscr{A}(E,h)$.  The spaces of integrable and irreducible Hermitian connections on $(E,h)$ are denoted by $\mathscr{A}_{\text{int}}(E,h)$ and $\mathscr{A}^{*}(E, h)$, respectively.  
\end{defn}

One should think of $\mathscr{A}_{\text{int}}(E,h)$ as an infinite-dimensional \emph{non-affine} space cut out of $\mathscr{A}(E,h)$ by the conditions $\partial_{A} \circ \partial_{A}=0$ and $\overline{\partial}_{A} \circ \overline{\partial}_{A}=0$.  

If $d_{A}$ is an integrable connection then by Definition \ref{defn:defnholstrbundre}, $\overline{\partial}_{A}$ is a holomorphic structure and $\mathcal{E}_{\overline{\partial}_{A}}$ is a holomorphic vector bundle.  The following theorem establishes the converse to this observation.  A proof can be found in \cite[Proposition 4.2.14]{huybrechts_complex_2004}.  

\begin{thm}
Let $(E, h)$ be a Hermitian vector bundle with holomorphic structure $\overline{\delta}$.  There is a unique integrable Hermitian connection $d_{A}$ on $(E,h)$, called the Chern connection, such that $\overline{\partial}_{A} = \overline{\delta}$.  Equivalently, we have a bijection 
\begin{equation}\label{eqn:corrintconnholstr}
\begin{split}
& \Psi : \mathscr{A}_{\text{int}}(E,h) \longrightarrow \mathscr{H}(E) \\
& \,\,\,\,\,\,\,\,\,\,\,\,\,\,\,\,\,\,\,\,\,\,\,\,\,\,\,\, d_{A} \longmapsto \overline{\partial}_{A}
\end{split}
\end{equation}
between integrable Hermitian connections and holomorphic structures.  
\end{thm}

\noindent This theorem represents the preliminary link between spaces of connections (of interest to differential geometers and physicists) and spaces of holomorphic bundles (of interest to complex and algebraic geometers).  

As we defined in (\ref{eqn:gaugetransCONN}), there is an action of the (uncomplexified) group of gauge transformations $\mathscr{G}$ on $\mathscr{A}(E, h)$, which extends to $\mathscr{A}_{\text{int}}(E,h)$ and $\mathscr{A}^{*}(E,h)$.  We want to quotient $\mathscr{A}_{\text{int}}(E,h)$ by the action of $\mathscr{G}$, and quotient $\mathscr{H}(E)$ by $\mathscr{G}^{\mathbb{C}}$.  But because $\mathscr{G} \subset \mathscr{G}^{\mathbb{C}}$, to have any hope of (\ref{eqn:corrintconnholstr}) inducing an isomorphism of moduli spaces, there must be additional conditions on the connections.  It turns out that we must impose irreducibility, as well as the Hermitian Yang-Mills equation.  Moreoever, we will be primarily interested in simple holomorphic bundles, as these are more amenable to parameterization in a moduli problem.  The inverse image of $\mathscr{H}_{\text{simp}}(E)$ under $\Psi$ is contained in $\mathscr{A}_{\text{int}}(E,h) \cap \mathscr{A}^{*}(E,h)$.  Equivalently, the Chern connection associated to a simple holomorphic structure is irreducible.  The converse is not true, but it will be true for irreducible connections solving the Hermitian Yang-Mills equation, to which we now turn.

\subsection{The Hermitian Yang-Mills Equation and Stability of Bundles}

We have now seen that integrable Hermitian connections on Hermitian vector bundles over complex manifolds are equivalent to holomorphic structures on the bundle.  Moreover, connections related by gauge transformations precisely define isomorphic holomorphic structures.  However, just as with Yang-Mills connections and instantons, we do not want to consider all integrable Hermitian connections, but only those satisfying a condition known as the \emph{Hermitian Yang-Mills equation}.  

Let us specialize to the case where $(X,J)$ is a compact K\"{a}hler manifold of dimension $n$ with K\"{a}hler form $J$.  Let $(E,h)$ be a Hermitian vector bundle on $X$, and let $\mathcal{E}$ be the holomorphic bundle corresponding uniquely to Chern connection $d_{A}$ on $(E,h)$ via (\ref{eqn:corrintconnholstr}).  Noting that $F_{A} \in \Omega^{1,1}(\mathfrak{g}_{E})$ and $J \in \Omega^{1,1}(X)$, the Hermitian Yang-Mills equation is a constraint relating these two $(1,1)$-forms on $X$.

\begin{defn}
The connection $d_{A}$ described above is called a Hermitian Yang-Mills connection if
\begin{equation}\label{eqn:HermEin}
i \, F_{A} \wedge J^{n-1} = \frac{\lambda}{n} \, J^{n} \cdot \text{id}_{E}
\end{equation}
for some constant $\lambda \in \mathbb{R}$.  This condition is called the Hermitian Yang-Mills equation.  
\end{defn}

\noindent This definition is often given instead for the metric $h$, and is called a Hermitian-Einstein metric if the Hermitian Yang-Mills equation is satisfied.  
  
\begin{defn}
Given a holomorphic vector bundle $\mathcal{E}$ on $X$ with rank $\text{rk}(\mathcal{E})$, we define the slope of $\mathcal{E}$ with respect to the K\"{a}hler form $J$ as
\begin{equation} \label{eqn:slopevect}
\mu(\mathcal{E}) = \frac{1}{\text{rk}(\mathcal{E})} \int_{X} c_{1}(\mathcal{E}) \wedge J^{n-1}.
\end{equation}
In addition, we say that $\mathcal{E}$ is a $\mu$-stable (or slope stable) vector bundle with respect to $J$ if for all subbundles $\mathcal{F} \hookrightarrow \mathcal{E}$ with $\text{rk}(\mathcal{F}) < \text{rk}(\mathcal{E})$, we have $\mu(\mathcal{F}) < \mu(\mathcal{E})$.  
\end{defn}

We can use the slope to give an explicit expression for the constant $\lambda$.  Both sides of (\ref{eqn:HermEin}) are $(n,n)$-forms valued in $\mathfrak{g}_{E}$, so let us take the trace of both sides and integrate over the compact manifold $X$.  Noting that $i \text{Tr}(F_{A}) = 2 \pi c_{1}(\mathcal{E})$ as well as $\int_{X} J^{n} = n! \cdot \text{Vol}_{J}(X)$, after standard simplification we find 
\begin{equation}
\lambda =  \frac{2 \pi \, \mu(\mathcal{E})}{(n-1)! \, \text{Vol}_{J}(X)}.
\end{equation} 

Let us define $\mathscr{A}^{\text{HYM}}_{J}(E,h)$ to be the space of \emph{irreducible} Hermitian Yang-Mills connections on a Hermitian vector bundle $(E,h)$ over a compact K\"{a}hler manifold $(X, J)$.  Recall that there is not a well-defined map from $\mathscr{A}_{\text{int}}(E,h) \cap \mathscr{A}^{*}(E,h)$ to $\mathscr{H}_{\text{simp}}(E)$ because an irreducible integrable connection might not give rise to a simple holomorphic bundle.  However, if the connection is additionally Hermitian Yang-Mills, then we do get a map \cite[Corollary 2.3.4]{lubke_kobayashi-hitchin_1995}
\begin{equation}\label{eqn:HYMsimpmap}
\mathscr{A}^{\text{HYM}}_{J}(E,h) \longrightarrow \mathscr{H}_{\text{simp}}(E)  
\end{equation}
which is equivariant \cite[Remark 2.1.9]{lubke_kobayashi-hitchin_1995} with respect to the action of $\mathscr{G}$ on $\mathscr{A}^{\text{HYM}}_{J}(E,h)$ and $\mathscr{G}^{\mathbb{C}}$ on $\mathscr{H}_{\text{simp}}(E)$.  Recall from (\ref{eqn:SImpHolMOD}) that we denote by $\mathcal{M}_{\text{simp}}(E)$ the moduli space of simple holomorphic structures on $E$ up to isomorphism.  We define the moduli space of Hermitian Yang-Mills connections 
\begin{equation}
\mathcal{M}^{\text{HYM}}_{J}(E, h) \coloneqq \mathscr{A}^{\text{HYM}}_{J}(E, h) \big/\mathscr{G}
\end{equation}
which carries the structure of a complex analytic space.  The map (\ref{eqn:HYMsimpmap}) then induces an open embedding 
\begin{equation}
\mathcal{M}^{\text{HYM}}_{J}(E, h) \longrightarrow \mathcal{M}_{\text{simp}}(E).  
\end{equation}

The Donaldson-Uhlenbeck-Yau theorem says that the image of the above injective map is precisely the locus of $\mu$-stable holomorphic bundles.  

\begin{thm}[\bfseries Donaldson-Uhlenbeck-Yau]
Let $(X, J)$ be a compact K\"{a}hler manifold, and let $(E,h)$ be a Hermitian vector bundle on $X$.  If $\mathcal{M}_{J}^{s}(E)$ denotes the coarse moduli space\footnote{In the following Chapter (Section \ref{subsec:RevModProbb}) we will give a proper treatment of moduli problems, and in particular discuss fine and coarse moduli spaces.} of holomorphic structures on $E$ up to isomorphism, stable with respect to $J$, then we have the complex analytic isomorphism 
\begin{equation}
\mathcal{M}^{\text{HYM}}_{J}(E, h) \cong \mathcal{M}_{J}^{s}(E).
\end{equation}
In particular, a holomorphic structure on $(E,h)$ is $\mu$-stable with respect to $J$ if and only if the unique Chern connection is irreducible and Hermitian Yang-Mills.    
\end{thm}

This theorem was proven initially by Narasimhan-Seshadri \cite{narasimhan_stable_1965} in the case of a curve, and then by Donaldson \cite{donaldson_anti_1985} for compact K\"{a}hler surfaces.  Finally, Uhlenbeck and Yau \cite{uhlenbeck_existence_1986} proved the theorem for K\"{a}hler manifolds of arbitrary dimension.  

In the next chapter we will see that stability of vector bundles is really a concept belonging to algebraic geometry -- the stability conditions we will define will specialize consistently to $\mu$-stability defined above.  Therefore, the Donaldson-Uhlenbeck-Yau theorem is a relationship between differential geometry and physics (special connections on Hermitian bundles) and algebraic geometry (stable holomorphic bundles).  We will also see in Section \ref{sec:Dbranesstabstrth} that this relationship arises when studying D-branes in string theory.  

Let us close this chapter by briefly discussing a few important examples.

\begin{Ex}[\bfseries Narasimhan-Seshadri]
Let $C$ be a smooth projective curve with K\"{a}hler form $J$.  In this case, the Hermitian Yang-Mills equation reads
\begin{equation}
\frac{i}{2 \pi} F_{A} = \frac{\mu(\mathcal{E})}{\text{Vol}_{J}(C)} J \cdot \text{id}_{E}.  
\end{equation}
We refer to this as the condition that the connection $d_{A}$ be projectively flat -- the curvature does not vanish, but up to scale it is $J \cdot \text{id}_{E}$.  Therefore, in this case the Donaldson-Uhlenbeck-Yau theorem says that for all Hermitian metrics $h$, a holomorphic structure on $(E,h)$ is stable if and only if the resulting Chern connection is irreducible and projectively flat.  This is the Narasimhan-Seshadri theorem \cite{narasimhan_stable_1965}.  In particular, the moduli space of stable holomorphic bundles with vanishing first Chern class on a curve is isomorphic to the moduli space of irreducible flat Hermitian connections.   
\end{Ex}

\begin{Ex}[\bfseries Instantons on K\"{a}hler Surfaces]
We saw in (\ref{eqn:instdecomp}) that on a Riemannian four-manifold $X$, we have a decomposition of two-forms
\begin{equation}\label{eqn:instdecomp222}
\Omega^{2}(X) = \Omega^{2}_{+}(X) \oplus \Omega_{-}^{2}(X)
\end{equation}
into self-dual and anti-self-dual (ASD) forms via the Hodge star operator.  Letting $(X,J)$ be a smooth compact K\"{a}hler surface, we furthermore have that $J \in \Omega^{2}_{+}(X)$, while the orthogonal compliment to $J$ lies in $\Omega^{2}_{-}(X)$.  If we now let $(E,h)$ be a Hermitian vector bundle on $X$ with $c_{1}(E)=0$, and let $d_{A}$ be a Hermitian Yang-Mills connection, we have
\begin{equation}
F_{A} \wedge J =0.
\end{equation}
But the curvature $F_{A}$ is a $(1,1)$-form (valued in $\mathfrak{g}_{E}$) and it is orthogonal to $J$ by the Hermitian Yang-Mills equation above.  Therefore, $F_{A} \in \Omega^{2}_{-}(X)$ which is the statement that a Hermitian Yang-Mills connection on a bundle with vanishing first Chern class is an instanton; more specifically, an ASD connection.  Conversely, an ASD connection on a bundle over a K\"{a}hler surface is Hermitian Yang-Mills.  Therefore, by the Donaldson-Uhlenbeck-Yau theorem, the moduli space of irreducible $SU(n)$ instantons on a K\"{a}hler surface is isomorphic to the moduli space of rank $n$ holomorphic bundles with vanishing first Chern class.  
\end{Ex}

\chapter{Stability Conditions on Coherent Sheaves and D-branes}

Partially motivated by the appearance of stable vector bundles in Yang-Mills theory, the main goal of this chapter is to work entirely in the world of algebraic geometry, and to extend in a consistent way our understanding of stability to more general coherent sheaves.  We will introduce Gieseker and slope stability of torsion-free sheaves (which one can use to compactify moduli spaces in Yang-Mills theory, though we will not do so) as well as Simpson stability of pure torsion sheaves.  Along the way, we will review some basic material on coherent sheaves in general, and the Grothendieck group.  We close the chapter with an application to D-branes in string theory, which we hope is approachable to mathematicians.

\section{Generalities on Coherent Sheaves}

In this chapter, all schemes $(X, \mathcal{O}_{X})$ will be Noetherian schemes over $\mathbb{C}$, unless otherwise mentioned.  We assume the reader is familiar with some of the basics of sheaves of $\mathcal{O}_{X}$-modules and coherent sheaves.  We denote by $\text{Coh}(X)$ the abelian category of coherent sheaves whose objects are coherent sheaves on $X$, and whose morphisms are morphisms of $\mathcal{O}_{X}$-modules.  Given a sheaf $\mathscr{E}$ on $X$, the stalk at $x \in X$ is denoted by $\mathscr{E}_{x}$ and the support of $\mathscr{E}$ is defined by
\begin{equation}
\text{Supp}(\mathscr{E}) \coloneqq \big\{ \, x \in X \, \big| \, \mathscr{E}_{x} \neq 0 \, \big\} \subseteq X.  
\end{equation}
If $X$ is a Noetherian scheme and $\mathscr{E}$ is coherent, then the support of $\mathscr{E}$ is a closed subscheme of $X$.

\subsection{Ideal Sheaves and their Subschemes}\label{subsec:idsheee}

Let us begin by introducing an important class of coherent sheaves known as ideal sheaves.  As we will see, the name stems from the fact that they locally correspond to ideals in the ring of local functions.  For the purposes of this thesis, ideal sheaves will primarily arise in the study of Donaldson-Thomas theory in the next chapter.     

\begin{defn}
An ideal sheaf on an arbitrary scheme $X$ is an $\mathcal{O}_{X}$-submodule $\mathscr{I}$ of $\mathcal{O}_{X}$.  That is to say, we have an injection of coherent sheaves $\mathscr{I} \hookrightarrow \mathcal{O}_{X}$.  
\end{defn}

\noindent If $\mathscr{I}$ is an ideal sheaf, then for all open sets $U \subset X$, the sections $\mathscr{I}(U)$ are an $\mathcal{O}_{X}(U)$-submodule of $\mathcal{O}_{X}(U)$.  This is equivalent to $\mathscr{I}(U)$ being an ideal of $\mathcal{O}_{X}(U)$, and is clearly the origin of the name.

Let $R$ be a Noetherian ring.  Imagining $R$ as a module over itself, any $R$-submodule (equivalently, an ideal of $R$) is finitely generated.  The sheaf-theoretic analog of this statement is that on a Noetherian scheme $X$, any $\mathcal{O}_{X}$-submodule of $\mathcal{O}_{X}$ is coherent.  In other words, ideal sheaves on Noetherian schemes are necessarily coherent.  
 
\begin{proppy}
Let $X$ be a Noetherian scheme.  Given an ideal sheaf which we denote $\mathscr{I}_{Z}$, there is a unique closed subscheme $Z \subset X$ such that we have the following short exact sequence of $\mathcal{O}_{X}$-modules
\begin{equation} \label{eqn:idshexseq}
0 \longrightarrow \mathscr{I}_{Z} \longrightarrow \mathcal{O}_{X} \longrightarrow \mathcal{O}_{Z} \longrightarrow 0
\end{equation}
known as the ideal sheaf exact sequence.  We will sometimes call $\mathscr{I}_{Z}$ the ideal sheaf of the subscheme $Z$.  
\end{proppy}

\begin{proof}
Because $\mathscr{I}_{Z}$ is an ideal sheaf, we have an injection $\mathscr{I}_{Z} \hookrightarrow \mathcal{O}_{X}$ of coherent sheaves, which implies the cokernel $\mathcal{O}_{X}/\mathscr{I}_{Z}$ is coherent, and fits into the short exact sequence
\begin{equation}
0 \longrightarrow \mathscr{I} \longrightarrow \mathcal{O}_{X} \longrightarrow \mathcal{O}_{X} / \mathscr{I} \longrightarrow 0.
\end{equation}
We define $Z$ uniquely as the support of $\mathcal{O}_{X}/\mathscr{I}_{Z}$ (which is a closed subscheme since $X$ is Noetherian) with structure sheaf $\mathcal{O}_{Z} \coloneqq \mathcal{O}_{X}/\mathscr{I}_{Z}$.  
\end{proof}

One should think of $\mathcal{O}_{X}$ as the sheaf of functions on $X$, and of $\mathscr{I}_{Z}$ as the subsheaf consisting locally of functions vanishing on the subscheme $Z$.  In other words, $\mathscr{I}_{Z}$ consists locally of equations cutting out $Z \subset X$ and accordingly, one should identify $\mathcal{O}_{Z}$ as the sheaf of functions on $Z$.  

A converse to the above proposition can be understood as follows.  One standard fact in algebraic geometry is that every closed subscheme $Z$ of a Noetherian scheme is locally cut out by finitely many equations.  These local equations correspond to local sections of a coherent $\mathcal{O}_{X}$-module.  This is precisely the ideal sheaf $\mathscr{I}_{Z}$ of $Z$.  However, there is a $\mathbb{C}^{*}$ ambiguity in recovering $\mathscr{I}_{Z}$ from $Z$ essentially because a vanishing locus is unchanged upon multiplying each equation by a non-zero constant.  Modulo this subtlety, there is a bijection between ideal sheaves on a Noetherian scheme and closed subschemes.

\subsection{Dimension and Purity of Coherent Sheaves}\label{subsec:dimpurCohhShh}

Unless otherwise noted, in this section $X$ will be an integral Noetherian scheme -- the notions we will discuss are well-behaved primarily in this setting.  Given a coherent sheaf $\mathscr{E}$ on $X$, there is a canonical ideal sheaf we can associate to it.  The \emph{annihilator ideal sheaf} of $\mathscr{E}$ is defined to be the kernel of the morphism
\begin{equation}\label{eqn:annnidshh}
\mathcal{O}_{X} \longrightarrow \shHom_{\mathcal{O}_{X}}(\mathscr{E}, \mathscr{E})
\end{equation}
defined on an open set $U \subset X$ by mapping a local function $f \in \mathcal{O}_{X}(U)$ to the $\mathcal{O}_{X}|_{U}$-module morphism $\mathscr{E}|_{U} \to \mathscr{E}|_{U}$ given by multiplication by $f$.  Because it injects into $\mathcal{O}_{X}$, the annihilator ideal sheaf is indeed an ideal sheaf and its local sections are functions vanishing on the support of $\mathscr{E}$.  It is therefore the ideal sheaf of the closed subscheme $\text{Supp}(\mathscr{E})$.  In fact, $\text{Supp}(\mathscr{E})$ is a priori just a closed \emph{subset} of $X$, and its scheme structure is induced from the annihilator ideal sheaf.  

\begin{defn}
The dimension $\text{dim}(\mathscr{E})$ of a coherent sheaf $\mathscr{E}$ is defined to be the dimension of $\text{Supp}(\mathscr{E})$.  
\end{defn}

\noindent For example, if $\mathscr{I}_{Z}$ is the ideal sheaf of a closed subscheme $Z \subset X$, then $\text{dim}(\mathscr{I}_{Z}) = \text{dim} \, X$ and $\text{dim}(\mathcal{O}_{Z}) = \text{dim} \, Z$.  Locally-free sheaves are also supported on all of $X$, and hence have dimension equal to the dimension on $X$.  

Ideal sheaves and locally-free sheaves are examples of the following important class of coherent sheaves, characterized by having support on all of $X$.   

\begin{defn}
A coherent sheaf $\mathscr{E}$ on $X$ is called torsion-free if the canonical morphism (\ref{eqn:annnidshh}) restricted to stalks
\begin{equation}\label{eqn:torfreeedefn}
\mathcal{O}_{X, x} \longrightarrow \text{Hom}_{\mathcal{O}_{X, x}}(\mathscr{E}_{x}, \mathscr{E}_{x})
\end{equation}
is injective for all $x \in X$.  Or equivalently, if the morphism (\ref{eqn:annnidshh}) is itself injective. 
\end{defn}

\noindent Throughout this thesis, if using torsion-free sheaves we will assume the underlying scheme is integral.  Some authors might not require this assumption, but our philosophy is that torsion-free modules should only be defined over rings which are integral domains.  

A coherent sheaf is said to be \emph{torsion} if it is not torsion-free.  Let us now prove two basic results giving an intuitive characterization of torsion-free sheaves in terms of their dimension, and a lack of torsion subsheaves.  

\begin{proppy}
A coherent sheaf $\mathscr{E}$ on $X$ is torsion-free if and only if $\text{Supp}(\mathscr{E}) = X$.  
\end{proppy}
\begin{proof}
For the forward direction, if $\mathscr{E}$ is a torsion-free sheaf with $\mathscr{E}_{x} =0$ for some $x \in X$, then by (\ref{eqn:torfreeedefn}) we have an injection $\mathcal{O}_{X, x} \to 0$, which is a contradiction.  Conversely, if $\text{Supp}(\mathscr{E}) = X$, then the annihilator ideal sheaf is the ideal sheaf of $X$ itself, and must therefore vanish.  We conclude that $\mathcal{O}_{X} \to \shHom_{\mathcal{O}_{X}}(\mathscr{E}, \mathscr{E})$ is an injection.  
\end{proof}

\begin{proppy}
A coherent sheaf $\mathscr{E}$ is torsion-free if and only if it has no torsion subsheaves.  
\end{proppy}

\begin{proof}
Beginning with the forward direction, let $\mathscr{F} \hookrightarrow \mathscr{E}$ be a torsion subsheaf.  By definition, there must exist $x \in X$ such that $\mathcal{O}_{X,x} \to \text{Hom}_{\mathcal{O}_{X,x}}(\mathscr{F}_{x}, \mathscr{F}_{x})$ is not injective. In other words, there must exist non-zero germs $f \in \mathcal{O}_{X,x}$ and $s \in \mathscr{F}_{x}$ such that $f \cdot s=0$.  But because $\mathscr{F}_{x} \hookrightarrow \mathscr{E}_{x}$ is an injective morphism of $\mathcal{O}_{X, x}$-modules, and $\mathscr{E}$ is torsion-free we must have $s=0$, a contradiction.  For the converse, if $\mathscr{E}$ has no torsion subsheaves, then $\text{Supp}(\mathscr{E})=X$.  Applying the above proposition, this implies $\mathscr{E}$ is torsion-free.  
\end{proof}

There are important examples of torsion sheaves which will be of interest in this thesis.  Given a subscheme $Z \subset X$, one example of a torsion sheaf is the structure sheaf $\mathcal{O}_{Z}$.  One can also produce a torsion sheaf on $X$ by pushing forward a locally-free sheaf on $Z$ by the inclusion.  Notice from these examples that if $Z$ has connected components of various dimensions, a torsion sheaf supported on $Z$ will accordingly have subsheaves of various dimensions.  The following notion of \emph{purity} is a generalization of torsion-free meant to distinguish sheaves that are essentially torsion-free on their support.  

\begin{defn}
Let $X$ be a Noetherian scheme, not necessarily integral.  A coherent sheaf $\mathscr{E}$ is said to be pure of dimension $d$ if $\text{dim}(\mathscr{E})=d$, and if $\mathscr{F} \hookrightarrow \mathscr{E}$ is a non-zero subsheaf, then $\text{dim}(\mathscr{F})=d$.    
\end{defn}

\noindent If we reinstate the assumption that $X$ is integral, then a pure sheaf of dimension $\text{dim} \, X$ is simply a torsion-free sheaf.  Moreover, pure sheaves supported on integral subschemes are torsion-free restricted to their support.

\subsection{Grothendieck Group and the Hirzebruch-Riemann-Roch Theorem}\label{subsecc:GgrpHRR}

In this section we will state, mostly without details, some important constructions and results pertaining to coherent sheaves on smooth varieties.  We will introduce the Chern character of a coherent sheaf and discuss some topological features it encodes.  We will also see that the Chern character relates an object called the Grothendieck group with either the Chow groups or cohomology, and we will state the powerful Hirzebruch-Riemann-Roch theorem.  Many of these results rely on the following \cite[III, Ex. 6.9]{hartshorne_algebraic_1997}.  

\begin{proppy}
Let $X$ be a smooth variety of dimension $n$ and $\mathscr{E}$ a coherent sheaf on $X$.  There exists a locally-free resolution of $\mathscr{E}$ of length $n$.  That is to say, there exists an exact sequence of $\mathcal{O}_{X}$-modules
\begin{equation}\label{eqn:cohres}
E_{n} \longrightarrow  \cdots \longrightarrow E_{1} \longrightarrow \mathscr{E} \longrightarrow 0
\end{equation}
such that each $E_{i}$ is a locally-free sheaf on $X$.
\end{proppy}

For the purposes of this thesis, we will use the above locally-free resolution to define the determinant and Chern character of coherent sheaves on smooth varieties.  It is worth noting that when one studies moduli spaces of sheaves (as we will shortly) these two quantities are frequently part of the data fixed in the moduli problem.  Recall that the determinant $\text{det}(E)$ of a locally-free sheaf $E$ of rank $r$ on a smooth variety is an invertible sheaf given by the top exterior power $\Lambda^{r}E$.  We can extend this definition to arbitrary coherent sheaves in the following way.  

\begin{defn}\label{defn:dettofcohshref}
Let $\mathscr{E}$ be a coherent sheaf on a smooth variety $X$ of dimension $n$.  The determinant of $\mathscr{E}$ is defined by
\begin{equation}
\text{det}(\mathscr{E}) = \bigotimes_{i=1}^{n} \text{det}(E_{i})^{(-1)^{i}} \in \text{Pic}(X)
\end{equation}
where $E_{i}$ are the entries of the resolution (\ref{eqn:cohres}) and $\text{Pic}(X)$ is the Picard group -- the group of isomorphism classes of invertible sheaves on $X$.  
\end{defn}  

In algebraic geometry, the Chow groups $A_{k}(X)$ are abelian groups of algebraic $k$-cycles modulo rational equivalence \cite{fulton_intersection_1998}.  If $X$ is smooth of dimension $n$, we define $A^{k}(X) \coloneqq A_{n-k}(X)$.  If $X$ is additionally projective, by associating an $(n-k)$-cycle to its class in homology, we get the \emph{cycle map}
\begin{equation}\label{eqn:cycreemap}
\text{cl} : A^{k}(X) \longrightarrow H^{2k}(X, \mathbb{Z}) \cong H_{2n-2k}(X, \mathbb{Z})
\end{equation}  
where we have applied Poincar\'{e} duality\footnote{Strictly speaking, the target of the cycle map is the Borel-Moore homology of $X$ where Poincar\'{e} duality holds without the hypothesis of projectivity.  This hypothesis can therefore be dropped if one is content to work with Borel-Moore homology, which coincides with ordinary homology when $X$ is projective.}.  One of the goals of \emph{intersection theory} is to enhance $A^{k}(X)$ to an associative, commutative graded ring such that the cycle map becomes a graded ring homomorphism with respect to the cup product in cohomology.   

The Chern character of a coherent sheaf is constructed to define an element of $A^{*}(X)_{\mathbb{Q}} \coloneqq A^{*}(X) \otimes_{\mathbb{Z}} \mathbb{Q}$.  We define it using the locally-free resolution (\ref{eqn:cohres}) noting that the Chern character of a locally-free sheaf is essentially a Chow-valued version of what we introduced in Section \ref{sec:chercrassbundre}.  
  
\begin{defn}
Let $X$ be a smooth variety of dimension $n$ and $\mathscr{E}$ a coherent sheaf on $X$.  The Chern character $\text{ch}(\mathscr{E})$ is a cycle in the Chow group with rational coefficients defined as
\begin{equation}\label{eqn:defnnnChcarssd}
\text{ch}(\mathscr{E}) = \sum_{i=1}^{n} (-1)^{i} \text{ch}(E_{i}) \in A^{*}(X)_{\mathbb{Q}}
\end{equation}
where the $E_{i}$ are the entries of the locally-free resolution (\ref{eqn:cohres}).
\end{defn}

\noindent If $X$ is projective, by applying the cycle map (\ref{eqn:cycreemap}) we can view $\text{ch}(\mathscr{E})$ as taking values in $H^{2*}(X, \mathbb{Q})$, as is frequently done.  The Chern character as defined is additive on short exact sequences and multiplicative on tensor products.  We will typically expand $\text{ch}(\mathscr{E})$ into degrees with the following notation
\begin{equation}\label{eqn:chcarrrcompp}
\text{ch}(\mathscr{E}) = \big(\text{ch}_{0}(\mathscr{E}), \text{ch}_{1}(\mathscr{E}), \ldots, \text{ch}_{n}(\mathscr{E}) \big), \,\,\,\,\,\,\,\,\,\,\,\, \text{ch}_{k}(\mathscr{E}) \in A^{k}(X)_{\mathbb{Q}}
\end{equation}
where $\text{ch}_{k}(\mathscr{E})$ is called the $k$-th Chern character of $\mathscr{E}$.  The higher Chern characters do not generally vanish, though it is standard to neglect those in degree higher than the dimension.  

Let $X$ be a smooth variety of dimension $n$, and let $\mathscr{E}$ be a coherent sheaf of dimension $d$ on $X$.  Let $\{ Z_{i} \}_{i=1}^{s}$ be the irreducible components of the \emph{reduced} support of $\mathscr{E}$ -- each $Z_{i}$ is an integral subscheme with generic point $\eta_{i}$.  

\begin{defn}\label{defn:Suuprcrassdefn}
Let $X$ and $\mathscr{E}$ be as above.  The support cycle of $\mathscr{E}$ is an effective cycle defined by
\begin{equation}\label{eqn:supportcycledefnnn}
[\mathscr{E}] \coloneqq \sum_{i=1}^{s} \text{length}(\mathscr{E}_{\eta_{i}}) [Z_{i}] \in A^{n-d}(X)
\end{equation}
where $\text{length}(\mathscr{E}_{\eta_{i}})$ is called the multiplicity of $\mathscr{E}$ along $Z_{i}$ and is defined as the length of the module $\mathscr{E}_{\eta_{i}}$ over the local ring $\mathcal{O}_{Z_{i}, \eta_{i}}$.  Here, $[Z_{i}] \in A^{n-d}(X) = A_{d}(X)$ is the class of the subvariety $Z_{i}$ .  
\end{defn}

As one might expect, the Chern character of a coherent sheaf vanishes in degree below the codimension of the support, and in the degree of the codimension it is precisely given by the support cycle.  One can find the following in \cite{kawai_string_2000} or \cite[sec 5.9]{chriss_representation_2010}.  

\begin{proppy}\label{proppy:Cherncharrvannnishingss}
Let $X$ and $\mathscr{E}$ be as above.  For the Chern character valued in the Chow groups we have
\begin{equation}\label{eqn:Cherncarrrcopppps}
\text{ch}_{k}(\mathscr{E}) = 
\begin{cases}
\begin{aligned}
&  0 & k < n-d  \\[1ex]
& [\mathscr{E}], & k=n-d
\end{aligned}
\end{cases}
\end{equation}
\end{proppy}

\begin{Ex}\label{Ex:inclHOLVBss}
With $X$ as above, let $\iota : Z \hookrightarrow X$ be the inclusion of the $d$-dimensional integral subscheme $Z$ into $X$, and let $E$ be a locally-free sheaf of rank $r$ on $Z$.  Applying the above proposition to $\iota_{*}E$, we know
\begin{equation}
\text{ch}_{k}(\iota_{*}E ) = 
\begin{cases}
\begin{aligned}
&  0 & k < n-d  \\[1ex]
& r \, [Z], & k=n-d
\end{aligned}
\end{cases}
\end{equation}
\end{Ex}

\subsubsection{The Grothendieck Group of Coherent Sheaves}  

One groundbreaking insight of Grothendieck is that it is often useful to consider not all objects in $\text{Coh}(X)$, but rather only classes of objects in what is called the \emph{Grothendieck group} of coherent sheaves.  Basic questions pertaining to the Chern character of the sheaf, or dimensions of sheaf cohomology groups are potentially easier to answer in this simpler setting.  What we call the Grothendieck group is also known as K-theory, and we will sometimes use this terminology.  For more details, we refer the reader to \cite[Ch. 15]{fulton_intersection_1998}.  

\begin{defn}
Given an algebraic variety $X$, the Grothendieck group of coherent sheaves on $X$ is denoted $K_{0}(X)$ and is defined to be the free abelian group generated by isomorphism classes $[\mathscr{E}]$ of coherent sheaves modulo the relations
\[ [\mathscr{E}] = [\mathscr{E}'] + [\mathscr{E}''] \]
for all short exact sequences $0 \to \mathscr{E}' \to \mathscr{E} \to \mathscr{E}'' \to 0$ in $\text{Coh}(X)$.  The Grothendieck group $K^{0}(X)$ of locally-free sheaves is defined in the same way using isomorphism classes $[E]$ of locally-free sheaves.  
\end{defn}

Tensoring by a locally-free sheaf is exact, so the multiplication $[E] \cdot [F] \coloneqq [E \otimes_{\mathcal{O}_{X}} F]$ turns $K^{0}(X)$ into a commutative ring.  Tensoring by objects of $\text{Coh}(X)$ need not be exact, so it is not immediately obvious that $K_{0}(X)$ is a ring.  However, if $X$ is smooth the locally-free resolution (\ref{eqn:cohres}) of a coherent sheaf gives rise to the following result \cite{fulton_intersection_1998}.  

\begin{proppy}
If $X$ is a smooth variety, then $K_{0}(X) \cong K^{0}(X)$.  It follows that the Grothendieck group of coherent sheaves is a commutative ring with product coming from $\otimes_{\mathcal{O}_{X}}$.  
\end{proppy}

\noindent The Chern character is additive on short exact sequences and multiplicative on tensor products.  We therefore observe that for $X$ smooth, it respects both operations on $K_{0}(X)$ and descends from $\text{Coh}(X)$ to a ring homomorphism from $K_{0}(X)$ into the Chow group after tensoring by the rational numbers.  Colloquially, we say that the Chern character descends to K-theory.  Let us define $K_{0}(X)_{\mathbb{Q}} \coloneqq K_{0}(X) \otimes_{\mathbb{Z}} \mathbb{Q}$.

\begin{proppy}
If $X$ is a smooth variety, the Chern character provides the following ring isomorphism,
\begin{equation} \label{eqn:chcharisom}
\text{ch}: K_{0}(X)_{\mathbb{Q}} \overset{\sim}{\longrightarrow} A^{*}(X)_{\mathbb{Q}}
\end{equation}
between the Grothendieck group of coherent sheaves on $X$ and the Chow group, after tensoring by $\mathbb{Q}$.    
\end{proppy}

\noindent If $X$ is projective, by applying the cycle map (\ref{eqn:cycreemap}) we can think of the Chern character as inducing a ring homomorphism into $H^{2*}(X, \mathbb{Q})$ in the following sense 
\begin{equation}
\begin{tikzcd}
K_{0}(X)_{\mathbb{Q}} \arrow[]{rr}{\sim}\arrow[swap, dashed]{drr}{\text{ch}} & & A^{*}(X)_{\mathbb{Q}} \arrow{d}{\text{cl}}\\
  &   & H^{2*}(X, \mathbb{Q})
\end{tikzcd}
\end{equation}

\begin{defn}
For a projective variety $X$ of dimension $n$ and a coherent sheaf $\mathscr{E}$ on $X$, the Euler characteristic (or holomorphic Euler characteristic) is the integer defined by
\begin{equation}
\chi(X, \mathscr{E}) = \sum_{k=0}^{n} (-1)^{k} \text{dim}\, H^{k}(X, \mathscr{E}).
\end{equation}
\end{defn}

\noindent The dimensions of the sheaf cohomology groups $H^{k}(X, \mathscr{E})$ are extremely difficult to compute in general, yet they may encode useful information.  If one is lucky, knowing $\chi(X, \mathscr{E})$ and applying certain additional theorems, the relevant dimensions might be within reach.  If $X$ is also smooth, the following theorem gives a powerful method for computing $\chi(X, \mathscr{E})$ in terms of topological data associated to $X$ and $\mathscr{E}$.  

\begin{thm}[\bfseries Hirzebruch-Riemann-Roch]
If $X$ is a smooth projective variety, and $\mathscr{E}$ is a coherent sheaf on $X$, then we have
\begin{equation}\label{eqn:HirzzzBBRoch}
\chi(X, \mathscr{E}) = \int_{X} \text{ch}(\mathscr{E}) \text{td}(X)
\end{equation}
where $\text{ch}(\mathscr{E})$ is as defined in (\ref{eqn:defnnnChcarssd}), and $\text{td}(X) = \text{td}(T_{X})$ is the Todd class of the tangent bundle
\begin{equation}
\text{td}(X) = \big( 1, \frac{1}{2}c_{1}(X), \frac{1}{12}\big(c_{1}(X)^{2} + c_{2}(X)\big), \ldots \big).
\end{equation}
\end{thm}

Because the Chern character descends from $\text{Coh}(X)$ to K-theory, the Euler characteristic does as well.  We can think of $\chi(X, -)$ as a ring homomorphism from $K_{0}(X)_{\mathbb{Q}}$ to $\mathbb{Q}$ landing in $\mathbb{Z}$, and one fundamental insight of Grothendieck was that one should think of $\mathbb{Z}$ as $K_{0}(\text{pt})$.  We then have
\[ \chi(X, -) : K_{0}(X)_{\mathbb{Q}} \longrightarrow K_{0}(\text{pt})_{\mathbb{Q}}\]
as the map induced on K-theory from the unique map $X \to \text{pt}$.  Grothendieck therefore noticed that the Hirzebruch-Riemann-Roch theorem should be an example of a more general statement associated to a proper morphism $f : X \to Y$.  This is \emph{Grothendieck-Riemann-Roch}, but we will not discuss this further.

\begin{Ex}
Let $X$ be a smooth projective Calabi-Yau variety of dimension $n$, and let $\mathscr{E}$ be a one-dimensional sheaf on $X$.  By Hirzebruch-Riemann-Roch along with the Calabi-Yau condition, we have
\[ \chi(X, \mathscr{E}) = \int_{X} \text{ch}_{n}(\mathscr{E}).\] 
It follows by Proposition \ref{proppy:Cherncharrvannnishingss} that the Chern character of $\mathscr{E}$ valued in cohomology (not Chow) is given as
\begin{equation}\label{eqn:CherncharrrCY3onedish}
\text{ch}(\mathscr{E}) = \big(0,0, \ldots, 0, [\mathscr{E}], \chi(X, \mathscr{E})\big).  
\end{equation} 
\end{Ex}

\section{Stability Conditions on Coherent Sheaves}\label{sec:StabConddCOHsh}

In certain settings, one should restrict attention to only special objects in $\text{Coh}(X)$ which we call \emph{stable} or \emph{semistable}.  The goal is typically to construct a nice moduli space of sheaves.  When $X$ is a curve, there is an essentially unique notion of stability while in higher dimensions, there are various stability conditions which involve the choice of an ample class on $X$.  The definitions can seem quite unmotivated at first, but one can roughly think of stability conditions as the right ones to impose to get well-behaved moduli spaces.  One of the miracles in the world of `physical mathematics' is that these definitions of stability are brought to life in a sense by physical objects in Yang-Mills theory and string theory.  We saw in the previous chapter that stable holomorphic vector bundles are equivalently bundles admitting a unique irreducible Hermitian Yang-Mills connection.  In the same spirit, we will study D-branes as they relate to certain stable sheaves.

For a quick historical survey, stable vector bundles on a curve were originally studied by Narasimhan and Seshadri \cite{narasimhan_stable_1965, seshadri_space_1967} and generalized by Gieseker \cite{gieseker_moduli_1977} to torsion-free sheaves on surfaces.  Maruyama \cite{maruyama_moduli_1977, maruyama_moduli_1978} generalized the results to torsion-free sheaves on more general varieties.  Ultimately, Simpson \cite{simpson_moduli_1994} generalized the story further to allow for torsion sheaves.  Let us begin by understanding the case of torsion-free sheaves before pushing on to torsion sheaves.

\subsection{Stability Conditions on Torsion-free Sheaves}

Let $X$ be a smooth, irreducible projective variety of dimension $n$ with an ample line bundle $\mathcal{O}_{X}(1)$.  Under the induced embedding $X \hookrightarrow \mathbb{P}^{N}$, the hyperplane bundle $\mathcal{O}_{\mathbb{P}^{N}}(1)$ pulls back to $\mathcal{O}_{X}(1)$.  If $H \subset X$ is a divisor in the linear system $|\mathcal{O}_{X}(1)|$, we define the \emph{degree} of $X$, denoted $\text{deg}(X)$, to be the self-intersection $H^{n}$ which is equivalently understood as the degree of the image of $X$ in $\mathbb{P}^{N}$.  For the duration of this section, a polarized variety $(X,H)$ will mean a variety with an ample divisor $H$, understood to correspond to the line bundle $\mathcal{O}_{X}(1)$.  

Fix a coherent sheaf $\mathscr{E}$ on $X$ of dimension $\text{dim}(\mathscr{E}) = d$.  We define the \emph{Hilbert polynomial} to be
\begin{equation}
P(\mathscr{E}, m) \coloneqq \chi(X, \mathscr{E}(m))
\end{equation}
where $\mathscr{E}(m) = \mathscr{E} \otimes \mathcal{O}_{X}(m)$.  Combining Hirzebruch-Riemann-Roch with Proposition \ref{proppy:Cherncharrvannnishingss}, one can see that $P(\mathscr{E}, m)$ is indeed a polynomial in $m$ of degree $d$ with rational coefficients, which can be expanded as
\begin{equation} \label{eqn:HilbPolyExp}
P(\mathscr{E}, m) = \sum_{i=0}^{d} \alpha_{i}(\mathscr{E}) \frac{m^{i}}{i!}
\end{equation}
for some integers $\alpha_{i}(\mathscr{E})$.  When we want to suppress the variable, we will write $P(\mathscr{E})$ for the Hilbert polynomial.  It is an easy computation to see that the first non-vanishing coefficient is given by
\begin{equation}\label{eqn:firnonvannncoeffHilbpoly}
\alpha_{d}(\mathscr{E}) = H^{d} \cdot [\mathscr{E}] > 0
\end{equation}
where $[\mathscr{E}] = \text{ch}_{n-d}(\mathscr{E})$ is the support cycle defined in (\ref{eqn:supportcycledefnnn}), and related to the Chern character in (\ref{eqn:Cherncarrrcopppps}).  The intersection number (\ref{eqn:firnonvannncoeffHilbpoly}) is positive because $[\mathscr{E}]$ is an effective class (possibly reducible) and $H$ is an ample class.  With this result, one can define the \emph{reduced Hilbert polynomial} $p(\mathscr{E})$ by
\begin{equation}
p(\mathscr{E}, m) = \frac{P(\mathscr{E}, m)}{\alpha_{d}(\mathscr{E})}.
\end{equation}
There is a natural ordering of polynomials where we say $p \leq q$ if $p(m) \leq q(m)$ for $m \gg 0$, and similarly for strict inequality.

Let us now specialize to the case of a torsion-free coherent sheaf $\mathscr{E}$.  Recall this means $\text{dim}(\mathscr{E}) =n$, and note that we require $X$ to be irreducible so that our notion of torsion-free makes sense.  We wish to provide the definitions of the rank and degree of such a sheaf which are closely related to the Hilbert polynomial coefficients $\alpha_{n}(\mathscr{E})$ and $\alpha_{n-1}(\mathscr{E})$, respectively.  The following lemma will be useful.  

\begin{lemmy}
Given a polarized variety $(X, H)$ as above, the first two coefficients of the Hilbert polynomial of $\mathcal{O}_{X}$ encode the following intersection numbers
\begin{equation}
\alpha_{n}(\mathcal{O}_{X}) = H^{n} = \text{deg}(X), \,\,\,\,\,\,\,\,\,\, \alpha_{n-1}(\mathcal{O}_{X}) = -\frac{1}{2} H^{n-1} \cdot K_{X}
\end{equation} 
where $K_{X}$ is the canonical divisor of $X$.  
\end{lemmy}

\begin{proof}
The first claim follows from (\ref{eqn:firnonvannncoeffHilbpoly}) since $\mathcal{O}_{X}$ is a rank one sheaf of dimension $n$.  For the second claim, noting that $\text{td}_{1}(X) = c_{1}(X)/2$ we have
\begin{equation}
\alpha_{n-1}(\mathcal{O}_{X}) = \frac{1}{2} \int_{X} c_{1}\big(\mathcal{O}_{X}(1)\big)^{n-1} \, c_{1}(X) = -\frac{1}{2} H^{n-1} \cdot K_{X} 
\end{equation}
where the canonical divisor is related to the first Chern class by $K_{X} = -\text{PD}\big(c_{1}(X)\big)$.  
\end{proof}

\begin{defn}
Given a torsion-free sheaf $\mathscr{E}$ on a polarized variety $(X, H)$ as above, we define the degree $\text{deg}(\mathscr{E})$ and rank $\text{rk}(\mathscr{E})$ of $\mathscr{E}$ as follows.  First,
\begin{equation} \label{eqn:rankdef}
\text{rk}(\mathscr{E}) = \frac{\alpha_{n}(\mathscr{E})}{\alpha_{n}(\mathcal{O}_{X})} = \frac{\alpha_{n}(\mathscr{E})}{\text{deg}(X)}
\end{equation}
where the second equality follows by the above lemma.  Using this definition of rank, we define the degree   
\begin{equation} \label{eqn:degdef}
\setlength{\jot}{12pt}
\begin{split}
\text{deg}(\mathscr{E}) & = \alpha_{n-1}(\mathscr{E}) - \text{rk}(\mathscr{E}) \alpha_{n-1}(\mathcal{O}_{X}) \\
& = \alpha_{n-1}(\mathscr{E}) + \frac{\text{rk}(\mathscr{E})}{2} H^{n-1} \cdot K_{X},
\end{split}
\end{equation}
again using the above lemma.  
\end{defn}

These definitions of rank and degree seem unmotivated.  From a theoretical perspective, it is preferable to have them arise in terms of coefficients of the Hilbert polynomial, but they can in fact be shown to coincide with more familiar quantities. 

\begin{proppy}
For a torsion-free sheaf $\mathscr{E}$ of rank $r$ (meaning $\text{ch}_{0}(\mathscr{E}) =r$) on the polarized variety $(X, H)$, we have $\text{rk}(\mathscr{E}) = r$, and
\begin{equation} \label{eqn:slopdeg}
\text{deg}(\mathscr{E}) = \int_{X} c_{1}(\mathscr{E}) \, \text{PD}(H)^{n-1} = \beta \cdot H^{n-1}
\end{equation}
where $\beta$ is Poincar\'{e} dual to $c_{1}(\mathscr{E})$.  
\end{proppy}

\begin{proof}
By an application of Hirzebruch-Riemann-Roch we have,
\begin{equation}
\chi\big(X, \mathscr{E}(m)\big) = \frac{m^{n}}{n!} r \, \text{deg}(X) + \frac{m^{n-1}}{(n-1)!} H^{n-1} \cdot \big(\beta - \frac{r}{2} K_{X}\big) + \ldots
\end{equation}
Therefore, $\alpha_{n}(\mathscr{E}) = r \, \text{deg}(X)$ which implies $\text{rk}(\mathscr{E}) = r$.  With this, one can read off $\alpha_{n-1}(\mathscr{E})$ above and applying the definition of degree (\ref{eqn:degdef}), the claim is immediate.  
\end{proof}

\begin{defn}\label{eqn:Gieeekstab}
Given a torsion-free coherent sheaf $\mathscr{E}$ on a smooth, irreducible projective polarized variety $(X, H)$, with reduced Hilbert polynomial
\begin{equation}
p(\mathscr{E}, m) = \frac{P(\mathscr{E}, m)}{\text{rk}(\mathscr{E})}
\end{equation}
we say that $\mathscr{E}$ is Gieseker semistable if for all proper subsheaves $\mathscr{F} \hookrightarrow \mathscr{E}$, we have $p(\mathscr{F}) \leq p(\mathscr{E})$.  We say that $\mathscr{E}$ is Gieseker stable if it is Gieseker semistable and the inequality above is strict for all subsheaves.    
\end{defn}

\begin{defn}
With the same assumptions on $(X, H)$, we define the slope of a torsion-free coherent sheaf
\begin{equation} \label{eqn:torfreeslope}
\mu(\mathscr{E}) = \frac{\text{deg}(\mathscr{E})}{\text{rk}(\mathscr{E})} = \frac{1}{\text{rk}(\mathscr{E})} \int_{X} c_{1}(\mathscr{E}) \text{PD}(H)^{n-1}  
\end{equation}
and we say $\mathscr{E}$ is $\mu$-semistable if $\mu(\mathscr{F}) \leq \mu(\mathscr{E})$ for all proper subsheaves $\mathscr{F} \hookrightarrow \mathscr{E}$ with $\text{rk}(\mathscr{F}) < \text{rk}(\mathscr{E})$.  We define $\mathscr{E}$ to be $\mu$-stable if it is $\mu$-semistable with strict inequality.  
\end{defn}

These definitions provide the two related stability conditions we will consider on the torsion-free objects of $\text{Coh}(X)$.  Notice that we are relying on $\text{Coh}(X)$ being an abelian category, since we require a notion of subsheaf; and because subsheaves of torsion-free sheaves are torsion-free, these conditions are well-defined.  It is straightforward to show the following relationship between the stability conditions above
\[\mu-\text{stability}  \Longrightarrow \text{Gieseker stability}  \Longrightarrow \text{Gieseker semistability}  \Longrightarrow \mu-\text{semistability}.\]

\begin{Ex}
Notice that torsion-free sheaves of rank one are automatically $\mu$-stable, and therefore Gieseker stable.  This is because $\mu$-stability requires subsheaves to be of strictly smaller rank.  But torsion-free sheaves cannot have rank zero (torsion) subsheaves.  For example, invertible sheaves and ideal sheaves are stable.  
\end{Ex}

\begin{Ex}[\bfseries Torsion-free Sheaves on Smooth Curves]
Historically, the notion of stability originated with vector bundles on curves \cite{narasimhan_stable_1965, seshadri_space_1967}.  Let $X$ be a smooth, projective curve of genus $g$ polarized by an ample line bundle $\mathcal{O}_{X}(1)$.  On a curve, every torsion-free sheaf $E$ is locally-free and is determined topologically by the rank $r$ and degree $\text{deg}(E)$.  The degree of a bundle on a curve is typically defined as $\text{deg}(E) = \int_{X}c_{1}(E)$, which is consistent with (\ref{eqn:slopdeg}).  By Riemann-Roch on curves we compute
\begin{equation}
P(E, m) = \text{deg}\big( E \otimes \mathcal{O}_{X}(m)\big) + r(1-g) = r \, \text{deg}(X) \, m + \text{deg}(E)+ r(1-g)
\end{equation}
where we have used the standard fact that for bundles $E_{1}$ and $E_{2}$ on a curve, the above definition of degree implies $\text{deg}(E_{1} \otimes E_{2}) = \text{rk}(E_{1}) \text{deg}(E_{2}) + \text{rk}(E_{2}) \text{deg}(E_{1})$.  The reduced Hilbert polynomial of $E$ is
\begin{equation}
p(E,m) = m + \frac{1}{\text{deg}(X)}\big(1-g + \mu(E)\big)
\end{equation}
where the slope of $E$ is $\mu(E) = \text{deg}(E)/r$.  Notice that $\mu(E)$ is the only term of $p(E)$ depending on the bundle $E$.  This shows that for vector bundles on a curve, Gieseker stability is equivalent to slope stability.  Therefore, there exists a unique notion of stability on a curve measured by the slope $\mu$.  Moreover, this unique notion of stability is independent of a choice of polarization. 
\end{Ex}

Recall in the previous chapter we introduced the Hermitian Yang-Mills equation (\ref{eqn:HermEin}) on a compact K\"{a}hler manifold $(X, J)$, which is a relationship between the K\"{a}hler form $J$ and the curvature of the Chern connection associated to a holomorphic bundle.  By the Donaldson-Uhlenbeck-Yau theorem, the bundle is $\mu$-stable if and only if the Chern connection is an irreducible Hermitian Yang-Mills connection.  If $J = \text{PD}(H)$ for an ample divisor $H$, we can study $\mu$-stability of bundles, and notice that the slope (\ref{eqn:slopevect}) of a vector bundle with respect to $J$ coincides precisely with the slope (\ref{eqn:torfreeslope}) of a torsion-free sheaf with respect to $H$.  

Before discussing moduli spaces of torsion-free sheaves and pushing on to torsion sheaves, let us pause to give a brief overview of moduli problems in general.  This will be helpful not just in this chapter, but throughout the thesis.

\subsection{A Review of Moduli Problems in Algebraic Geometry} \label{subsec:RevModProbb}

One problem which has been of interest to mathematicians in all eras is the classification of geometric objects of a fixed type up to some notion of equivalence.  In the language of modern algebraic geometry, such a problem is known as a \emph{moduli problem}.  To give some simple examples, one can study linear subspaces of $\mathbb{C}^{n}$ up to honest equality (which leads to the Grassmanian varieties), or subvarieties of $\mathbb{P}^{n}$ up to $PGL_{n+1}(\mathbb{C})$ transformations.  

Given a moduli problem, one would like to extract a geometric space known as a \emph{moduli space}.  As a first pass at what such an object should be, one would like a moduli space to be a scheme such that the points over an algebraically closed field $k$ are in bijection with the equivalence classes of objects one is classifying.  The problem is that this is not refined enough to encode the scheme structure of the moduli space -- it might as well consist just of a disjoint union of points $\text{Spec} \, k$, one for each class.  

To get a properly refined notion of a moduli space, one should understand how objects deform in a family.  As we will see shortly, it is therefore natural to associate to a moduli problem a moduli \emph{functor} which associates a scheme to the set of families parameterized by that scheme.  In this formulation, a moduli functor is closely related to what is called a stack.  However, one frequently prefers to recover a scheme from the functor (for example, when defining enumerative invariants).  By imposing stability conditions, this may be possible as we will discuss, and leads to fine moduli spaces or coarse moduli spaces.  

The goal of this section is to provide a brief and readable overview of a very deep subject, so we refer the interested reader to \cite{hartshorne2009deformation, huybrechts_geometry_2010, neumann_algebraic_2009} for more details.  We hope this section provides some background to understanding not only moduli spaces of stable sheaves (arising in Donaldson-Thomas and Gopakumar-Vafa invariants) but also moduli spaces of stable curves and stable maps in Gromov-Witten theory.

\subsubsection{Flat Families and the Moduli Functor}  

The heuristic picture of a \emph{family} of objects $(\star)$ parameterized by a base scheme $S$ is that for all closed points of $S$, one has an object $(\star)$ and that these objects vary algebraically as one moves in the base.  Inevitably, the precise definition of a family will depend on what category the objects $(\star)$ come from -- the definition will be slightly different for families of varieties, maps, and sheaves.  These will all be of interest in this thesis at some point.  We will primarily use \emph{flat families} which one should think of as roughly families such that the objects don't jump too wildly as one moves in $S$.  We will not include here a proper treatment of flatness, and instead refer the reader to \cite{eisenbud_geometry_2000}.  

To start in the simplest setting, we define a family of algebraic varieties of a fixed type parameterized by a scheme $S$ (always of finite type over an algebraically closed field $k$) to be a morphism
\[ \pi : \mathcal{X} \longrightarrow S,\]
such that for all closed points $s \in S$, the fiber $\pi^{-1}(s)$ is an algebraic variety of the fixed type.  One should think of a flat family of varieties as one where the fibers have some fixed discrete invariants.  In particular, the Hilbert polynomial is locally constant in a flat family.  In a moduli problem, one also needs a notion of equivalence $\sim_{S}$ for each scheme $S$.  We say two such families $\pi : \mathcal{X} \to S$ and $\pi' : \mathcal{X}' \to S$ are equivalent if there exists an isomorphism $f : \mathcal{X} \to \mathcal{X}'$ such that the following diagram commutes 
\begin{equation}
\begin{tikzcd}
\mathcal{X} \arrow{rr}{f} \arrow[swap]{dr}{\pi} & & \mathcal{X}' \arrow{dl}{\pi'} \\
& S  &
\end{tikzcd}
\end{equation}
The main example of this type we will see is with flat families of stable curves in Section \ref{subsec:ModSpStCURVV}.  In Section \ref{eqn:secssonStabMapss} we will see an example of the similar notion of a flat family of maps between varieties.  

Let us briefly describe the corresponding picture for sheaves.  For a smooth projective variety $X$, we want to set up the moduli problem for coherent sheaves on $X$ satisfying some properties $(\star)$.  Here, $(\star)$ can represent a specification of purity, fixed determinant, stability with respect to a polarization, or as we will address shortly, a fixed Chern character or Hilbert polynomial.  

\begin{defn}
A family of coherent sheaves on $X$ with properties $(\star)$ parameterized by a base scheme $T$ is a coherent $\mathcal{O}_{X \times T}$-module $\mathcal{F}$ such that for all closed points $t \in T$, $\mathcal{F}_{t} \coloneqq \mathcal{F} |_{X \times \{t\}}$ is a coherent $\mathcal{O}_{X}$-module with properties $(\star)$.  The family is said to be flat if $\mathcal{F}$ is flat over $T$.  
\end{defn}

\noindent As above, we have omitted a proper discussion of flatness.  But in particular, a flat family is such that the Hilbert polynomial of $\mathcal{F}_{t}$ is locally constant as one moves in the base \cite{huybrechts_geometry_2010}.  The converse is also true when $T$ is reduced.  Let $\pi_{T} : X \times T \to T$ be the canonical projection.  We define $\mathcal{E} \sim_{T} \mathcal{F}$ for two flat families $\mathcal{E}, \mathcal{F}$ over $T$ if $\mathcal{E} \cong \mathcal{F} \otimes \pi_{T}^{*}L$ for some line bundle $L$ on $T$.  

We have remarked above that it is natural to associate a functor to a moduli problem in order to encode how objects deform in a family.  Let $\mathbf{Set}$ denote the category of sets, and $\mathbf{Sch}$ the category of schemes of finite type over an algebraically closed field.  

\begin{defn}
Given a moduli problem, a moduli functor is the functor $\mathcal{M}: \mathbf{Sch}^{op} \to \mathbf{Set}$ defined on objects by
\begin{equation}\label{eqn:defnnnmoddspppp}
\mathcal{M}(S) \coloneqq \bigg\{ \text{Flat families of objects} \, (\star)\, \text{parameterized by} \, S \, \text{up to} \, \sim_{S}\bigg\}. 
\end{equation}
Associated to a morphism $f : S \to T$ of schemes, we define $\mathcal{M}(f)$ by the pullback
\begin{equation}
\mathcal{M}(f) \coloneqq f^{*} : \mathcal{M}(T) \longrightarrow \mathcal{M}(S).
\end{equation}
\end{defn}

\noindent In whatever specific category the objects in a family live, there is clearly a well-defined notion of the pullback of a family by a morphism of base schemes.  

\begin{rmk}\label{rmk:speckpointsmodspp}
Notice that in the special case of families over the point $\text{Spec} \, k$, the set $\mathcal{M}(\text{Spec} \, k)$ is simply the set of isomorphism classes of the objects one is parameterizing.  
\end{rmk}

\noindent We will see many examples of moduli functors in coming sections and chapters so for now, we proceed with introducing some of the formal concepts.

\subsubsection{Fine and Coarse Moduli Spaces}  

In more categorical language, a moduli functor $\mathcal{M} : \mathbf{Sch}^{op} \to \mathbf{Set}$ is an example of what is called a presheaf on $\mathbf{Sch}$.  These are objects in the category $\text{Fun}(\mathbf{Sch}^{op}, \mathbf{Set})$.  Given a scheme $M$, another well-known example of a presheaf on $\mathbf{Sch}$ is the functor of points $h_{M} \coloneqq \text{Hom}( -, M)$.  By mapping $M \mapsto h_{M}$, we get a faithful functor known as the Yoneda embedding 
\begin{equation}
\begin{tikzcd}
\mathbf{Sch} \arrow[hookrightarrow]{r}{} & \text{Fun}(\mathbf{Sch}^{op}, \mathbf{Set}). 
\end{tikzcd}
\end{equation}

\begin{defn}
A scheme $M$ is a fine moduli space for a moduli functor $\mathcal{M}$ if there is an isomorphism of functors $\mathcal{M} \cong h_{M}$.  In other words, a fine moduli space is a scheme representing $\mathcal{M}$.    
\end{defn} 

\noindent We therefore see that fine moduli spaces exist precisely for those moduli functors lying in the image of the Yoneda embedding.  If such a scheme exists, a fine moduli space is unique up to unique isomorphism.   

One frequently hears a fine moduli space described in terms of a universal family.  Let us now define such an object and see how it relates to the representability of the moduli functor.  We assume $M$ to be a fine moduli space for a moduli functor $\mathcal{M}$.  For all schemes $S$ we get a bijection of sets
\begin{equation}
\begin{tikzcd}
\eta_{S} : \mathcal{M}(S) \arrow[rightarrow]{r}{\sim} & h_{M}(S) = \text{Hom}(S, M).
\end{tikzcd}
\end{equation}
In particular for $S = \text{Spec} \, k$, by Remark \ref{rmk:speckpointsmodspp} we see that the closed points of $M$ are in bijection with the isomorphism classes of the objects one is parameterizing.  This is indeed one of the properties we would ask a moduli space to satisfy.  Moreover, associated to a morphism $f : S \to T$ we get a commutative diagram
\begin{equation}
\begin{tikzcd}
\mathcal{M}(T)\arrow[rightarrow]{r}{\eta_{T}} \arrow[rightarrow]{d}   &    \text{Hom}(T, M) \arrow[rightarrow]{d} \\
\mathcal{M}(S)\arrow[rightarrow]{r}{\eta_{S}}      &  \text{Hom}(S, M)
\end{tikzcd}
\end{equation}
In particular, one can choose the identity morphism $\text{id}_{M} : M \to M$ on the fine moduli space itself which induces the bijection $\eta_{M} : \mathcal{M}(M)  \to \text{Hom}(M, M)$.  
\begin{defn}
The universal family $\mathcal{U} \to M$ over a fine moduli space is a flat family over $M$ up to equivalence defined by $\mathcal{U} \coloneqq \eta_{M}^{-1}(\text{id}_{M}) \in \mathcal{M}(M)$.  
\end{defn}

\noindent Associated to a family $\mathcal{X} \to S$ over a scheme $S$, we get an induced map $f : S \to M$ by sending a closed point of $S$ to the isomorphism class of its fiber.  Pulling back the universal family $\mathcal{U}$ by $f$, we get a family $f^{*}\mathcal{U} \to S$ with the same induced map into $M$.  We therefore must have the equivalence
\[ \mathcal{X} \sim_{S} f^{*} \mathcal{U}.\]
In other words, assuming a fine moduli space exists, every family is pulled back from the universal family.  

\begin{Ex}
Most of the moduli problems we will encounter in this thesis will not admit a fine moduli space, so let us briefly state without details some important examples of ones which do.
\begin{enumerate}
\item The moduli problem which parameterizes families of subschemes of a projective scheme $X$ with a fixed Hilbert polynomial $P$ admits a fine moduli space which is a projective scheme called the Hilbert scheme $\text{Hilb}_{P}(X)$.  For us, this will arise when we study Donaldson-Thomas theory.  

\item There is a fine moduli space of stable holomorphic vector bundles with fixed coprime rank and degree on a smooth projective curve $C$.  The moduli space is projective, and is simply the Jacobian $\text{Jac}(C)$ in the case of degree 0 line bundles.  

\item The dual abelian variety $\text{Pic}^{0}(A)$ is a fine moduli space for degree 0 line bundles on a smooth abelian variety $A$.  The moduli space is again projective, and the universal family is known as the Poincar\'{e} line bundle.  
\end{enumerate}
\end{Ex}

Given a moduli problem, the ideal solution is a fine moduli space, though one might not exist.  If a fine moduli space doesn't exist there are a few options.  A moduli functor is nearly defined to be a \emph{stack}, so depending on the goal one may be content to work with the stack (as we will do in Gromov-Witten theory).  Alternatively, one can \emph{rigidify} the moduli problem by specifying additional data and constraints in the hope that the new moduli problem admits a fine moduli space.  

The third option is to weaken the demands on a scheme solving a moduli problem.  We can ask for a scheme whose closed points are in bijection with isomorphism classes of the objects, and which satisfies a universal property.  This leads to the following definition.  

\begin{defn}
Given a moduli functor $\mathcal{M}$, a coarse moduli space for $\mathcal{M}$ is a scheme $M$ along with a morphism of functors $\eta : \mathcal{M} \to h_{M}$ such that:
\begin{enumerate}
\item $\eta_{\text{Spec}k} : \mathcal{M}(\text{Spec} \, k) \longrightarrow \text{Hom}( \text{Spec} \, k, M)$ is a bijection.  
\item For any scheme $S$ with a morphism of functors $\nu : \mathcal{M} \to h_{S}$, there exists a unique morphism of schemes $f : M \to S$ such that $\nu = h_{f} \circ \eta$, where $h_{f} : h_{M} \to h_{S}$.  
\end{enumerate}
\end{defn}

\noindent One can show that if a coarse moduli space exists, then it is unique up to unique isomorphism.  

\subsection{Moduli Spaces of Torsion-free Sheaves}

With the foundational ideas in hand, we would now like to study the moduli problem of parameterizing torsion-free coherent sheaves on $X$ with fixed discrete data.  The issue is that considering all torsion-free sheaves will not result in well-behaved moduli spaces (though this may not actually be a concern if one is content to work with stacks).  The stability conditions discussed previously were invented by Mumford, Gieseker and others precisely to construct moduli spaces (using geometric invariant theory) which are schemes.  

In order to specify some of the discrete, topological features of a sheaf one typically fixes either the Chern character or the Hilbert polynomial.  By Hirzebruch-Riemann-Roch, the choice of a Chern character along with an ample class $H$ clearly determines the Hilbert polynomial.  For generic $H$ one can recover the Chern character, though this will fail for special choices of $H$.  We prefer to have the ample class only arising in the stability condition, not in the specification of the topological features of the sheaf.  We will therefore fix the Chern character (not as a class in Chow, but in cohomology).  
  
Let $(X, H)$ be a smooth irreducible projective variety polarized by an ample divisor $H$, and let $v \in H^{2*}(X, \mathbb{Q})$ be a cohomology class with lowest graded entry $v_{0} > 0$.  Following (\ref{eqn:defnnnmoddspppp}), and assuming the notion of stability is that of Gieseker in Definition \ref{eqn:Gieeekstab}, we define the moduli functor $\mathcal{M}_{H,v}^{ss}(X)$ by
\begin{equation}\label{eqn:moddsppssss}
\mathcal{M}_{H,v}^{ss}(X)(T) \coloneqq 
\begin{Bmatrix*}[l]
\text{Flat families over}\,\, T \,\, \text{of coherent sheaves on} \,\, X \,\, \text{with Chern} \\  \,\,\,\,\,\,\,\, \text{character} \, v, \, \text{semistable with respect to} \, H, \, \text{up to} \, \sim_{T}
\end{Bmatrix*}.
\end{equation}
By our assumption $v_{0} > 0$, we are for the time being working with torsion-free sheaves.  Of course, we will drop this condition on $v$ when we generalize to torsion sheaves.  Replacing semistability with stability, we can correspondingly define the open subfunctor $\mathcal{M}^{s}_{H, v}(X)$.  Because stable sheaves are simple, their automorphism group is $\mathbb{C}^{*}$, and the moduli functor is nearly representable.  On the other hand, strictly semistable sheaves typically have large automorphism groups, and the moduli functor therefore has no hope of being representable.  

One can content themselves to work with a coarse moduli space of semistable sheaves at the expense of introducing the notion of \emph{S-equivalence} \cite{simpson_moduli_1994,huybrechts_geometry_2010}.  Every semistable sheaf $\mathscr{E}$ has a Jordan-H\"{o}lder filtration
\[ 0 = \mathscr{E}_{0} \subset \cdots \subset \mathscr{E}_{n-1} \subset \mathscr{E}_{n} = \mathscr{E}\]
such that all quotients $\mathscr{E}_{k}/\mathscr{E}_{k-1}$ are stable with slope equal to the slope of $\mathscr{E}$.  We define 
\[ \text{gr} (\mathscr{E}) \coloneqq \bigoplus_{k=1}^{n} \mathscr{E}_{k}/\mathscr{E}_{k-1}\]
and we say that two semistable sheaves $\mathscr{E}$ and $\mathscr{F}$ are S-equivalent if $\text{gr}(\mathscr{E}) \cong \text{gr}(\mathscr{F})$.  This condition is vacuous in the case of stable sheaves as two stable sheaves are S-equivalent if and only if they are isomorphic.

\begin{thm}
Let $(X, H)$ be as above, and let $v \in H^{2*}(X, \mathbb{Q})$ be a cohomology class with $v_{0} >0$.    
\begin{enumerate}
\item The moduli functor $\mathcal{M}_{H,v}^{ss}(X)$ has a coarse moduli space $M_{H,v}^{ss}(X)$ which is a projective scheme, and whose closed points are S-equivalence classes of torsion-free sheaves, semistable with respect to $H$, and with Chern character $v$.  

\item There is a quasi-projective scheme $M_{H,v}^{s}(X) \subseteq M_{H,v}^{ss}(X)$ whose closed points correspond to isomorphism classes of torsion-free sheaves, stable with respect to $H$, and with Chern character $v$.  

\item If there are no strictly semistable torsion-free sheaves with Chern character $v$, then the moduli space $M_{H,v}^{s}(X)$ is a projective scheme which we typically denote $M_{H,v}(X)$.  
\end{enumerate} 
Because the automorphism group of all stable sheaves is $\mathbb{C}^{*}$, the moduli functor $\mathcal{M}_{H, v}^{s}(X)$ is what is called a $\mathbb{C}^{*}$-gerbe over the scheme $M_{H, v}^{s}(X)$.    
\end{thm}

Moduli spaces of irreducible Hermitian Yang-Mills connections or equivalently, stable holomorphic vector bundles, give examples of moduli spaces of torsion-free sheaves.  However, these are typically non-compact, as they only account for moduli of locally-free sheaves.  It turns out that in certain settings one can compactify the moduli space by adding torsion-free sheaves -- these are known as Gieseker compactifications in algebraic geometry.  We will not provide a full discussion of Gieseker compactifications, though we will see a very similar construction in Section \ref{sec:FULLNEkkkSec} in the context of instanton counting on the non-compact four-manifold $\mathbb{C}^{2}$.  Aside from this, the central object in Donaldson-Thomas theory is a moduli space of torsion-free sheaves which we will study in detail in Section \ref{eqn:DTThhhh}.

\subsection{Simpson Stability: Generalization to Torsion Sheaves}\label{eqn:subsecSimppsStablllty}

We have now discussed stability conditions on torsion-free sheaves, and seen that it specializes to the notion of stability of vector bundles in Yang-Mills theory.  We saw in Section \ref{subsec:dimpurCohhShh} that pure sheaves are the natural generalization of torsion-free sheaves.  One can therefore study stability conditions on pure torsion sheaves.  This was done by Simpson \cite{simpson_moduli_1994} and goes under the name of \emph{Simpson stability}.  The moduli spaces of Simpson stable sheaves will play an integral role in parts of this thesis.  

In what follows $(X, H)$ will be a smooth projective variety (not necessarily irreducible) polarized by the ample divisor $H$.  Let $\mathscr{E}$ be a coherent sheaf of dimension $d$ on $X$ with corresponding reduced Hilbert polynomial $p(\mathscr{E})$, of degree $d$.  

\begin{defn}
We say the sheaf $\mathscr{E}$ is Gieseker-Simpson semistable if it is pure, and if for all proper subsheaves $\mathscr{F} \hookrightarrow \mathscr{E}$, we have $p(\mathscr{F}) \leq p(\mathscr{E})$.  We say $\mathscr{E}$ is Gieseker-Simpson stable if it is Gieseker-Simpson semistable and the inequality is strict.  
\end{defn}

\begin{defn}
Using coefficients of the Hilbert polynomial, one can define the Simpson slope to be 
\begin{equation}\label{eqn:simmmpssrope}
\mu_{S}(\mathscr{E}) = \frac{\alpha_{d-1}(\mathscr{E})}{\alpha_{d}(\mathscr{E})}.
\end{equation}
The sheaf $\mathscr{E}$ is said to be Simpson slope semistable if it is pure, and if for any proper subsheaf $\mathscr{F} \hookrightarrow \mathscr{E}$ with $\alpha_{d}(\mathscr{F}) < \alpha_{d}(\mathscr{E})$, we have $\mu_{S} (\mathscr{F}) \leq \mu_{S} (\mathscr{E})$.  We say $\mathscr{E}$ is Simpson slope stable if the inequality is strict. 
\end{defn}

Just as in the case of torsion-free sheaves, these stability conditions depend on the ample class $H$, as well as the notion of a subsheaf in the abelian category $\text{Coh}(X)$.  Assuming $X$ is irreducible, it is clear that Gieseker-Simpson stability specializes consistently to Definition \ref{eqn:Gieeekstab} when applied to torsion-free sheaves.  Comparing (\ref{eqn:torfreeslope}) and (\ref{eqn:simmmpssrope}), we also see that if $\mathscr{E}$ is torsion-free   
\begin{equation}
\mu(\mathscr{E}) = \text{deg}(X) \mu_{S}(\mathscr{E}) - \alpha_{n-1}(\mathcal{O}_{X}).
\end{equation}
So $\mu(\mathscr{E})$ and $\mu_{S}(\mathscr{E})$ do not exactly coincide, but they clearly produce the same numerical ordering and therefore, define the same stability condition.

\subsubsection{Moduli Spaces of Simpson Stable Sheaves}

Just as we did with torsion-free sheaves, we want to study the moduli problem of parameterizing pure sheaves with fixed Chern character on a polarized variety.  We will use (\ref{eqn:moddsppssss}) as the definition of $\mathcal{M}^{ss}_{H,v}(X)$ where we allow $X$ to be reducible, we drop the condition that $v_{0} >0$, and we understand stability to be that of Gieseker-Simpson.  The following theorem was proven in \cite{simpson_moduli_1994}, which we record here without proof.  

\begin{thm}[\bfseries Simpson]
Let $(X,H)$ be a projective variety polarized by the ample divisor $H$, and let $v \in H^{2*}(X, \mathbb{Q})$ be a cohomology class.
\begin{enumerate}
\item The moduli functor $\mathcal{M}_{H,v}^{ss}(X)$ has a coarse moduli space $M_{H,v}^{ss}(X)$ which is a projective scheme, and whose closed points are S-equivalence classes of pure sheaves, semistable with respect to $H$, and with Chern character $v$.  

\item There is a quasi-projective scheme $M_{H,v}^{s}(X) \subseteq M_{H,v}^{ss}(X)$ whose closed points correspond to isomorphism classes of pure sheaves, stable with respect to $H$, and with Chern character $v$.  

\item If there are no strictly semistable pure sheaves with Chern character $v$, then the moduli space $M_{H,v}^{s}(X)$ is a projective scheme which we typically denote $M_{H,v}(X)$.  
\end{enumerate} 
Because the automorphism group of all stable sheaves is $\mathbb{C}^{*}$, the open subfunctor $\mathcal{M}_{H, v}^{s}(X)$ is what is called a $\mathbb{C}^{*}$-gerbe over the scheme $M_{H, v}^{s}(X)$.   
\end{thm}

\subsubsection{One-dimensional Sheaves on a Projective Surface}

Let $X$ be a smooth, projective surface polarized by $\mathcal{O}_{X}(1)$ with $H$ an element in the linear system $|\mathcal{O}_{X}(1)|$, and let $\mathscr{E}$ be a one-dimensional sheaf on $X$.  By (\ref{eqn:Cherncarrrcopppps}) we can write
\[ \text{ch}(\mathscr{E}) = \big(0, \beta, \text{ch}_{2}(\mathscr{E})\big)\]
where $\beta$ is the support cycle of $\mathscr{E}$, and we interpret the Chern character as valued in cohomology.  By a simple application of Hirzebruch-Riemann-Roch, the Hilbert polynomial is
\begin{equation}
P(\mathscr{E}, m) = H \cdot \beta \, m  + \int_{X} \text{ch}_{2}(\mathscr{E}) - \frac{1}{2} K_{X} \cdot \beta
\end{equation}
where $K_{X}$ is the canonical divisor of $X$.  Indeed, the degree of $P(\mathscr{E})$ is 1 since $\text{dim}(\mathscr{E})=1$.  Therefore
\begin{equation}
p(\mathscr{E}, m) = m + \frac{\int_{X} \text{ch}_{2}(\mathscr{E}) -\frac{1}{2}K_{X} \cdot \beta}{H \cdot \beta} = m + \mu_{S}(\mathscr{E}).  
\end{equation}
This shows that for one-dimensional pure torsion sheaves on a smooth projective surface, Gieseker-Simpson stability is equivalent to Simpson slope stability, and there is explicit dependence on the divisor $H$.  

\begin{rmk}
An extremely important example is when $X$ is in addition Calabi-Yau, which means $K_{X} =0$ and $X$ must therefore be a K3 surface or an abelian surface.  In this case, by Hirzebruch-Riemann-Roch it is clear that $\int_{X} \text{ch}_{2}(\mathscr{E}) = \chi(X, \mathscr{E})$, and the slope $\mu_{S}(\mathscr{E})$ reduces to
\begin{equation}\label{eqn:sloopeCY2ford}
\mu_{S}(\mathscr{E}) = \frac{\chi(X, \mathscr{E})}{H \cdot \beta}.
\end{equation}
\end{rmk}

\subsubsection{One-dimensional Sheaves on a Calabi-Yau Threefold}

Now let $X$ be a smooth projective Calabi-Yau threefold polarized by an ample divisor $H$, and let $\mathscr{E}$ be a pure one-dimensional sheaf.  By (\ref{eqn:CherncharrrCY3onedish}), for the Chern character valued in cohomology, we know
\[\text{ch}(\mathscr{E}) = \big( 0, 0, \beta, \chi(X, \mathscr{E}) \big).\]
Just as above, it is straightforward to use this to compute the Hilbert polynomial 
\begin{equation}
P(\mathscr{E},m)  =  H \cdot \beta \, m  + \chi(X, \mathscr{E}).
\end{equation}
Clearly, the reduced Hilbert polynomial takes the form $p(\mathscr{E},m) = m + \mu_{S}(\mathscr{E})$ where the slope is just as shown in (\ref{eqn:sloopeCY2ford}).  In particular, Gieseker-Simpson and Simpson slope stability are equivalent.    

\begin{rmk}
Moduli spaces of stable or semistable sheaves on Calabi-Yau threefolds are remarkable in that they carry what is called a symmetric obstruction theory.  We will discuss these in some small detail in the following chapter in the context of Donaldson-Thomas and Gopakumar-Vafa invariants.  
\end{rmk}

As one can easily see from (\ref{eqn:CherncharrrCY3onedish}), the results on one-dimensional sheaves on Calabi-Yau surfaces and Calabi-Yau threefolds generalize to higher dimensions.   

\begin{proppy}\label{proppy:onedshGieSimp}
Let $X$ be a smooth, projective $n$-dimensional Calabi-Yau variety polarized by $H$.  If $\mathscr{E}$ is a pure one-dimensional coherent sheaf with support cycle $\beta$ and Euler characteristic $\chi(X, \mathscr{E}) = \int_{X} \text{ch}_{n}(\mathscr{E})$, then Gieseker-Simpson stability is equivalent to Simpson slope stability and measured by
\begin{equation} \label{eqn:slopestab}
\mu_{S}(\mathscr{E}) = \frac{\chi(X, \mathscr{E})}{H \cdot \beta}.
\end{equation}
\end{proppy}

\section{D-branes and Stability in String Theory}\label{sec:Dbranesstabstrth}

D-branes are objects emerging from string theory which have recently become of great interest in many parts of mathematics.  They will arise at various points in this thesis, so we hope to sketch a brief (and very much incomplete) overview of D-branes.  Along the way, we will very briefly describe some key features of physical and topological string theories one should keep in mind.  We must warn that the goal of this section is to give mathematicians a flavor of the subject -- precise physical details will be omitted.  For a readable account of parts of what will follow, we suggest \cite[Section 7]{ruiperez_fourier_2005}.

Let us build up to the modern interpretation of D-branes by starting with the most na\"{i}ve description.  In string theory, one can have both closed and open strings propagating in spacetime.  One must specify boundary conditions on the endpoints of the open strings.  Quite simply, D-branes (named after \emph{Dirichlet} boundary conditions) are submanifolds $Y$ of spacetime on which the endpoints of open strings live.  We call $Y$ the worldvolume of the brane.  A D-brane is called a D$p$-brane if the \emph{real} dimension of $Y$ is $p+1$.  Necessarily one of the directions in $Y$ is along the time direction, so a D$p$-brane has $p$ spatial dimensions.  If $Y$ has components of various dimensions, then the D-brane is a bound state of D$p$-branes for various $p$.  

It turns out that the endpoints of open strings appear as charged particles in a Yang-Mills theory with structure group $U(r)$, for some $r$.  This means that the submanifold $Y$ in fact supports a rank $r$ Hermitian vector bundle of $E \longrightarrow Y$.  Therefore, one of the interesting features of D-branes is their connection to Yang-Mills theory, which we studied in the previous chapter.  We will build on this connection shortly.  

D-branes are much deeper, but before continuing we must pause to record some basic ideas in string theory.  The consistency of string theory requires that spacetime be ten-dimensional.  Given that we only observe four dimensions, we model the ten-dimensional spacetime as
\[\mathbb{R}^{1,3} \times X\]
where $\mathbb{R}^{1,3}$ is four-dimensional spacetime with a Lorentzian metric, and we take $X$ to be a smooth compact Calabi-Yau threefold.  If $X$ is a generic such Calabi-Yau (not one with exceptional holonomy) then the low-energy limit of the theory will induce an $\mathcal{N}=2$ supersymmetric theory in four dimensions.  There are five consistent \emph{physical string theories}, two of which are called Type IIA and Type IIB.  Both theories contain D-branes.  In Type IIA, there are D$p$-branes where $p$ must be even or equivalently, the worldvolume must be odd-dimensional.  On the other hand, Type IIB theory admits D$p$-branes for $p$ odd.  

\begin{rmk}
Certainly some of the spatial dimensions of a D-brane are allowed to live in $\mathbb{R}^{1,3}$, but this clearly cannot be interpreted as a particle in four dimensions.  One common application is to use D-branes to engineer particles in four-dimensional spacetime so from here on, we will assume that a D$p$-brane is such that there are $p$ spatial dimensions in $X$, and only the one-dimensional worldline supported in $\mathbb{R}^{1,3}$.  
\end{rmk}

Building on the above remark, not only are we interested in using D-branes to engineer particles in four dimensions, but in order to make contact with algebraic geometry we would like to engineer \emph{BPS particles}.  Let us assume we are in the large-volume limit, meaning if $J$ is the K\"{a}hler form, the integral of $J^{3}$ over the Calabi-Yau threefold $X$ tends to infinity.  The mass of a particle engineered by a D-brane is simply the mass of the brane which is proportional to its volume.  The BPS condition is that the mass, and therefore the volume, of the brane is minimized and coincides with a central charge.  

Consider a D$p$-brane whose underlying support defines a class in $H_{p}(X, \mathbb{Z})$.  Calibrated submanifolds are those that minimize volume within their homology class.  For K\"{a}hler manifolds, the calibrated submanifolds are the holomorphic submanifolds, while in the Calabi-Yau case, special Lagrangians are also calibrated.  It follows that the BPS condition restricts attention D6, D4, D2, and D0-branes in Type IIA wrapping holomorphic submanifolds of $X$, and D3-branes in Type IIB wrapping special Lagrangians in $X$.  

However, these BPS D-branes in the large-volume limit of Type IIA are more than just the data of a holomorphic submanifold.  We have previously established a na\"{i}ve definition of a D-brane as a Hermitian vector bundle supported on a submanifold.  Let $X$ be a smooth compact Calabi-Yau threefold with K\"{a}hler form $J$.  Let $Z \subseteq X$ be a holomorphic submanifold of dimension $n \leq 3$ with $E \to Z$ a Hermitian vector bundle of rank $r$.  There is a unique connection $d_{A}$ on $E$ compatible with the Hermitian structure as well as the holomorphic structure on $Z$.  The curvature two-form satisfies 
\begin{equation}
F_{A} \in \Omega^{1,1}(\mathfrak{g}_{E})
\end{equation}
where $\mathfrak{g}_{E} \subset \text{End}E$ is the bundle of Lie algebras $\mathfrak{u}_{r}$ over $Z$.  It turns out that in order for the particle engineered by the D-brane to be BPS, not only must $Z$ be holomorphic, but the curvature two-form must satisfy the following additional condition \cite{sharpe_d-branes_1999}
\begin{equation}\label{eqn:STABconddDbrane}
i \, F_{A} \wedge J^{n-1} = \frac{\lambda}{n} \, J^{n} \cdot \text{id}_{E}
\end{equation}
for some constant $\lambda \in \mathbb{R}$.  If $n <3$, there are further conditions on the normal bundle \cite{harvey_algebras_1998} which we will not discuss.  We recognize (\ref{eqn:STABconddDbrane}) as precisely the Hermitian Yang-Mills equation (\ref{eqn:HermEin}), and by the Donaldson-Uhlenbeck-Yau theorem, we conclude that $E \to Z$ is a $\mu$-stable holomorphic vector bundle\footnote{The bundle will only be stable if the connection $d_{A}$ is irreducible -- otherwise it will be polystable.  One can safely ignore this subtlety on a first pass.}.  In other words, the physical BPS condition connects D-branes to slope stable bundles in mathematics.  

Associated to a smooth compact Calabi-Yau threefold $X$, Witten discovered two \emph{topological string theories} known as the A and B-models, respectively.  The open string sectors of both theories contain D-branes.  In the A-model, there are \emph{A-branes} wrapping a special Lagrangian three-cycle in $X$, and in the B-model there are \emph{B-branes} wrapping holomorphic cycles of $X$.  By convention, in both cases there can be any number of dimensions of the brane supported outside of $X$, and depending on this number it is unclear whether the topological configurations can be mapped to a brane configuration in either Type IIA or Type IIB.  However, if we insist that the only support outside of $X$ is along a worldline, then an A-brane comes from a BPS D3-brane in Type IIB and a B-brane comes from bound states of BPS D6, D4, D2, or D0-branes in Type IIA.  

Combining all of the above discussion, we are finally in a position to give a simple definition of the D-branes of interest to us in the large-volume limit of the Calabi-Yau.    

\begin{defn}[\bfseries D-branes of B-type at large volume]
 On a smooth compact Calabi-Yau threefold $X$ in the large-volume limit, we define a BPS D-brane of B-type to be a pair $(Z, E)$ where $Z \subseteq X$ is a holomorphic submanifold, and $E \to Z$ is a polystable holomorphic vector bundle.  The BPS D-branes of B-type are either D-branes in Type IIA engineering a BPS particle, or B-branes in the B-model topological string theory.  
\end{defn}

\noindent The following theorem states that the compactification spaces of most interest in string theory are necessarily algebraic.  

\begin{thm}{\bfseries \cite[Proposition 5.3]{joyce_lectures_2001}}
If $X$ is a smooth compact Calabi-Yau threefold with $H^{1}(X, \mathcal{O}_{X})=0$, then $X$ is a projective algebraic variety.  
\end{thm}

\noindent Since $X$ is algebraic, by Chow's theorem \cite{chow_compact_1949} holomorphic submanifolds $Z \subseteq X$ are themselves algebraic.  \emph{Therefore, taking the large-volume limit on the smooth compact Calabi-Yau threefold, we observe that the BPS sector of Type IIA string theory, as well as the B-model topological string theory, are entirely in the world of algebraic geometry!}

For an ample bundle $\mathcal{O}_{X}(1)$ on $X$, if $E \to Z$ is a $\mu$-stable holomorphic vector bundle with respect to $\mathcal{O}_{X}(1)|_{Z}$, then pushing forward by the inclusion $\iota : Z \hookrightarrow X$, we get a coherent sheaf $\iota_{*}E$ on $X$.  By a straightforward computation, one can show that $\iota_{*}E$ is a stable sheaf with respect to Simpson slope (\ref{eqn:simmmpssrope}) and ample class $\mathcal{O}_{X}(1)$.  Therefore, a BPS D-brane of B-type can be modeled in the large-volume limit as a Simpson stable coherent sheaf in algebraic geometry.


As one moves away from the large-volume limit, the (possibly non-BPS) D-branes of B-type are conjecturally \cite{aspinwall_d-branes_2004, sharpe_lectures_2003, sharpe_derived_2008} given by objects in the derived category $D^{b}(X)$.  Given an object in the derived category, one should think of the morphisms in the complex as representing open strings stretching between a brane and an anti-brane \cite{sen_tachyon_1998, sen_stable_1998}.  One non-trivial aspect of the conjecture is that quasi-isomorphisms correspond to renormalization group flow.  The homological mirror symmetry conjecture of Kontsevich \cite{kontsevich_homological_1994} is that $D^{b}(X)$ is isomorphic to the Fukaya category of A-branes in $X$.  The current understanding is that a BPS brane away from the large-volume limit should be a $\Pi$-stable object \cite{douglas_dirichlet_2002,douglas_stability_2005,aspinwall_d-brane_2002}, which is called a Bridgeland stable object \cite{bridgeland_stability_2007} in the math literature.  Indeed, if one takes the large-volume limit in the right way, $\Pi$-stability or Bridgeland stability reduces to $\mu$-stability \cite{douglas_stability_2005}, consistent with the above discussion.

\subsection{The Mukai Vector and D-brane Charges}

Being extended objects in spacetime, D-branes have certain characteristic quantities, known as \emph{charges} in physical jargon.  To name a few examples, D-branes have a mass, as well as potentially spin, and electric or magnetic charge.  Let us consider (possibly non-BPS) D-branes of B-type and possibly not in the large-volume limit.  By \cite{sharpe_d-branes_1999,witten_d-branes_1998,minasian_k-theory_1997} we understand these charges, known as \emph{D-brane charges}, to lie in the Grothendieck group $K_{0}(X)$, also called the K-theory group of $X$.  Since $X$ is smooth, by (\ref{eqn:chcharisom}) the Grothendieck group is isomorphic to the Chow group $A^{*}(X)_{\mathbb{Q}}$ via the Chern character. Indeed, passing to D-brane charges is related to taking the Chern character, but there is an important subtlety.  
\begin{defn}
On a smooth $n$-dimensional variety $X$ the Mukai vector of a coherent sheaf $\mathscr{E}$ is defined by\footnote{We define $\sqrt{\text{td}(X)}$ via a power series expansion, which is well defined as $\text{td}_{0}(X) =1$.}  
\begin{equation}\label{eqn:MuukaiVector}
\mathcal{Q}(\mathscr{E}) = \text{ch}(\mathscr{E}) \sqrt{\text{td}(X)} \in A^{*}(X)_{\mathbb{Q}}
\end{equation}
and the individual components are defined by $\mathcal{Q}(\mathscr{E}) = \big(\mathcal{Q}_{n}(\mathscr{E}), \ldots, \mathcal{Q}_{0}(\mathscr{E})\big)$ where $\mathcal{Q}_{k}(\mathscr{E}) \in A_{k}(X)_{\mathbb{Q}}$. 
\end{defn}

The Mukai vector $\mathcal{Q}(-)$ also induces a ring isomorphism from $K_{0}(X)_{\mathbb{Q}}$ to $A^{*}(X)_{\mathbb{Q}}$, and since $X$ is projective one can still apply the cycle map (\ref{eqn:cycreemap}) to recover an element in cohomology $H^{2*}(X, \mathbb{Q})$.  Though we will not discuss it in this thesis, there is a Grothendieck group $K_{0}\big( D^{b}(X) \big)$ of the derived category of $X$, and one can use the inclusion of categories $\text{Coh}(X) \hookrightarrow D^{b}(X)$ to induce the isomorphism
\begin{equation}
K_{0}\big( D^{b}(X) \big) \cong K_{0}(X).
\end{equation}
Of course, one can also extend the definition of the Chern character and the Mukai vector to complexes of coherent sheaves in the derived category.  

Let $X$ be a smooth Calabi-Yau threefold, and let us think of a D-brane as a $\Pi$-stable or Bridgeland stable object $\dotr{\mathscr{E}}$ in the derived category.  By \cite{minasian_k-theory_1997, harvey_algebras_1998} the D-brane charge is given by\footnote{Strictly speaking, the Mukai vector is only an approximation to the D-brane charge.  In \cite{halverson_perturbative_2015}, the authors introduce `Gamma classes' which provide corrections to the factor of $\sqrt{\text{td}(X)}$ in the Mukai vector.} the Mukai vector $\mathcal{Q}(\dotr{\mathscr{E}})$ such that the D$(2k)$-brane charge is $\mathcal{Q}_{k}(\dotr{\mathscr{E}}) \in A_{k}(X)_{\mathbb{Q}}$.  In particular, if a D-brane configuration can be modeled as a single coherent sheaf $\mathscr{E}$, then the D-brane charge is the Mukai vector $\mathcal{Q}(\mathscr{E})$.

\begin{Ex}
For a Calabi-Yau threefold $X$, let $\iota : Z \hookrightarrow X$ be the inclusion of an integral subscheme of dimension $d \leq 3$, and let $E$ be a $\mu$-stable holomorphic vector bundle of rank $r$ on $Z$.  We know $\iota_{*}E$ is a coherent sheaf on $X$ modeling a B-type D-brane configuration in the large-volume limit.  Using the results of Example \ref{Ex:inclHOLVBss}, we know that $\mathcal{Q}_{k}(\iota_{*}E)=0$ for all $k>d$, and additionally because $\text{td}_{0}(X)=1$, we have
\begin{equation}\label{eqn:D2dbranechh}
\mathcal{Q}_{d}(\iota_{*}E) = r[Z] \in A_{d}(X).  
\end{equation}
Equivalently, the D$(2d)$-brane charge of $\iota_{*}E$ is $r[Z]$, which is what we called (Definition \ref{defn:Suuprcrassdefn}) the support cycle of $\iota_{*}E$.  We are seeing an explanation of some common jargon in the physics literature:  the phrase ``a stack of $r$ D-branes wrapping a $d$-dimensional subvariety $Z$" is physical jargon meaning a stable holomorphic vector bundle of rank $r$ supported on $Z$.  By (\ref{eqn:D2dbranechh}), the D$(2d)$-brane charge appears as $r$ copies of the fundamental cycle of $Z$, which explains the terminology.  
\end{Ex}

Let us assume we can model a D$(2d)$-brane as a coherent sheaf $\mathscr{E}$ of dimension $d$ with smooth support.  The mass of the D-brane is proportional to the volume, which is given by the integral of suitable powers of the K\"{a}hler form over the support cycle $[\mathscr{E}]$.  In the case of $\iota_{*}E$ as in (\ref{eqn:D2dbranechh}), evidently the mass scales with the rank of the bundle.  

\begin{Ex}
Let us consider a configuration of $r$ D0-branes in a smooth projective Calabi-Yau threefold $X$.  Such a configuration can always be modeled as a zero-dimensional coherent sheaf $\mathscr{E}$ with Euler characteristic $\chi(X, \mathscr{E})=r$.  Forgetting no data at all, the system is given by skyscraper sheaves supported on specific points of $X$ such that the ranks add up to $r$.  Taking the Mukai vector, we get the D-brane charges, all of which vanish except the D0-brane charge 
\begin{equation}
\text{ch}_{3}(\mathscr{E}) = r_{1}p_{1} + \cdots + r_{s}p_{s} \in A_{0}(X), \,\,\,\,\,\,\,\,\,\,\,\,\,\,\,\,\,\,\,\,\, \sum_{k=1}^{s}r_{k} = r.  
\end{equation}
Taking values in the Chow group of zero-cycles, we see that we have lost data differentiating between two points in the support connected by a rational curve.  Notice that in this case, $\text{ch}_{3}(\mathscr{E})$ is the support cycle of $\mathscr{E}$.  In general, the Chow groups are much more complicated than the homology groups.  Applying the cycle map (\ref{eqn:cycreemap}) we get a class in $H_{0}(X, \mathbb{Z}) \cong \mathbb{Z}$, which is simply $\chi(X, \mathscr{E})=r$.  Put differently, the D0-brane charge as an element of Chow distinguishes points up to rational equivalence, while the map to homology forgets the points and only encodes the sum of multiplicities $\sum_{k=1}^{s} r_{k}=r$.  
\end{Ex}

Partially motivated by the physics, modern day algebraic geometers are often interested in developing theories which enumerate bound states of D-branes.  For example, in the next chapter we will introduce Donaldson-Thomas invariants which are a supersymmetric index virtually counting bound states of D2-D0 branes inside a single D6-brane in a Calabi-Yau threefold.  Joyce and Song \cite{joyce_theory_2012} developed a theory of generalized Donaldson-Thomas invariants which allows one to study more general bound states of D6-D4-D2-D0 branes in a Calabi-Yau threefold.  Finally, we will also discuss the Gopakumar-Vafa invariants which are counts of bound states of D2-D0 branes coming from M2-branes in M-theory.  

The unifying feature of all of these D-brane counting theories is that the invariants are extracted from a moduli space of stable or semistable sheaves with fixed Chern character on a Calabi-Yau threefold.  These moduli spaces were introduced in Section \ref{subsec:RevModProbb}.  We have seen in this section that the physical charges of a D-brane are encoded into the Mukai vector (\ref{eqn:MuukaiVector}).  However, because the characteristic class $\sqrt{\text{td}(X)}$ is invertible as a power series, fixing the Chern character is equivalent to fixing the D-brane charges.  \emph{Therefore, a specification of the physical charges of a D-brane serves to fix some of the topological features of a sheaf in the moduli problem.}

\chapter{Introduction to Modern Enumerative Geometry and String Theory}

Enumerative geometry is a subfield of algebraic geometry made up of theories attempting to count objects or families of objects in a fixed algebraic variety.  It is born out of the following distinguishing feature of algebraic geometry: many objects you may want to count come in finite, or at least finite-dimensional families.  This property is rarely achieved in the smooth or topological categories.  For one simple example, vector bundles in the algebraic category frequently have a finite-dimensional space of sections, a fact which never holds for $C^{\infty}$-bundles.  Another feature is that the moduli spaces parameterizing the objects of interest are themselves algebraic varieties (or schemes, or stacks), and one can use the moduli spaces to define invariants of the counting problem.  

An interesting enumerative problem in both mathematics and physics is to count curves of a fixed type in a smooth projective variety $X$.  One piece of data which is typically fixed in all theories is the homology class $\beta \in H_{2}(X, \mathbb{Z})$ of the curve.  For example, if $X \subset \mathbb{P}^{3}$ is a cubic surface, there are famously exactly 27 rational curves in the homology class of a line.  However, one cannot always perform such a count without specifying additional data.  Even in the simple example of $X = \mathbb{P}^{2}$, and $\beta$ the class of a line, there are infinitely many such curves.  In these cases, in order to extract a number, one must impose the correct number of incidence conditions -- these are cycles in $X$ in generic position which we require curves to intersect.  Indeed, specifying two distinct points in $\mathbb{P}^{2}$, there is a unique line through them.  

Consider the moduli stack $\mathcal{C}(X, \beta)$ of smooth curves in $X$ in the homology class $\beta$.  This is a space one might na\"{i}vely work with to construct a curve-counting theory on $X$, but it has at least two main drawbacks.  First, $\mathcal{C}(X, \beta)$ is generally not compact.  Moreover, it is not clear how to work directly with a smooth embedded curve in $X$ in a practical manner.  To remedy these issues, one should find a \emph{compactification} of $\mathcal{C}(X, \beta)$ using objects one can control in practice, and define curve-counting invariants using intersection theory on the compact moduli space.  Doing so, will require some understanding of obstruction theories and the virtual fundamental class.  

In this chapter we will describe three modern curve-counting invariants (Gromov-Witten, Donaldson-Thomas, and Gopakumar-Vafa invariants) and each arise in some form from different compactifications of $\mathcal{C}(X, \beta)$ using different objects.  All three are conjecturally equivalent, and all three are deformation invariants, meaning they are unchanged under deformations of the complex structure of $X$.  In addition, each of these theories finds a natural home in some part of string theory, as we will discuss.  

It turns out that more interesting than the invariants themselves is the generating function one can package them into.  This object not only coincides with the partition function of the associated physical theory, but often enjoys remarkable automorphic properties.  This automorphy may reflect hidden symmetries of the moduli space at hand.  This is among the common features shared by modern enumerative theories -- a few others which we will see are the use of deformation and obstruction theories on a moduli space, as well as a close interaction with physics.

\section{Gromov-Witten Theory} \label{sec:GWThhhh}

Given a smooth projective variety $X$, we consider a fixed homology class $\beta \in H_{2}(X, \mathbb{Z})$ and cycles $Z_{1}, \ldots, Z_{n}$ in $X$ of any dimension.  We are interested in counting (in some sense of the word) the curves $C$ of genus $g$ with homology class $\beta$ such that $C \cap Z_{i} \neq \varnothing$ for all $i$.  In other words, we want to count curves in $X$ with fixed discrete invariants $g$ and $\beta$ intersecting each of the $Z_{i}$.  The insight of Kontsevich \cite{kontsevich_enumeration_1995} was to replace embedded curves $C \subset X$ with abstract $n$-pointed curves $(C, p_{1}, \ldots, p_{n})$ of genus $g$ along with a holomorphic map $f : C \to X$ such that $f_{*}[C] = \beta$ and $f(p_{i}) \in Z_{i}$ for all $i$.  

It turns out that this change of perspective harmonizes the mathematical theory with the idea of a sigma model in string theory.  Mathematically, \emph{Gromov-Witten theory} is a rigorous counting of curves in $X$ as interpreted by Kontsevich, and is the underlying foundation of the \emph{A-model topological string theory} on $X$.  The relevant moduli space of stable maps will carry a perfect obstruction theory and a virtual class.  The Gromov-Witten invariants are then deformation invariants defined as integrals over this virtual class.  They will depend only on the K\"{a}hler classes, i.e. the (complexified) volumes of curves.  Finally, we mention that what has been described above is the algebro-geometric approach to the theory -- there is a purely symplectic version replacing $X$ by an arbitrary symplectic manifold, and $f$ by a J-holomorphic map \cite{mcduff_j-holomorphic_2004}.

\subsection{The Moduli Space of Stable Curves} \label{subsec:ModSpStCURVV}

As a warmup to understanding stable maps, let us briefly review the moduli of stable curves.  We will give just a rough overview of the theory, so for more details we refer the reader to \cite{deligne_irreducibility_1969, arbarello_geometry_2011}.  We will also make use of our discussion in Section \ref{subsec:RevModProbb} outlining some generalities in moduli problems.    

For the duration of the section on Gromov-Witten theory, a curve will mean a connected one-dimensional, possibly singular projective variety.  An $n$-pointed curve $(C, p_{1}, \ldots, p_{n})$ is a curve $C$ with ordered marked points $p_{i}$ lying at smooth points of $C$.  An automorphism of $(C, p_{1}, \ldots, p_{n})$ is an automorphism of $C$ preserving the marked points.  

\begin{defn}
An $n$-pointed curve $(C, p_{1}, \ldots, p_{n})$ is called prestable if it has at worst nodal singularities, and it is called stable if it is prestable with a finite automorphism group.  
\end{defn}

\noindent A stable curve is equivalently seen to be a nodal $n$-pointed curve satisfying the following two conditions:
\begin{enumerate}
\item If smooth of genus 1, then the curve has at least 1 marked point.

\item All rational components of the normalization contain at least 3 points lying over special points (either nodes or marked points).  
\end{enumerate}

\noindent If the arithmetic genus of $C$ is $g$, a necessary condition for stability is $2g-2+n>0$, a condition which clearly holds automatically for $g \geq 2$.

We can now introduce the notion of a \emph{flat family of stable curves} parameterized by a scheme, always assumed to be of finite type over $\mathbb{C}$.  

\begin{defn}
A flat family of stable $n$-pointed curves of arithmetic genus $g$ parameterized by a scheme $S$ is a flat and proper morphism of finite type
\[ \pi: \mathcal{C} \longrightarrow S\]
with $n$ disjoint sections $\sigma_{1}, \ldots, \sigma_{n} : S \to \mathcal{C}$ such that for all closed points $s \in S$,
\[ \big( C_{s}, \sigma_{1}(s), \ldots, \sigma_{n}(s) \big)\]
is a stable $n$-pointed curve of arithmetic genus $g$, where $C_{s} = \pi^{-1}(s)$.  
\end{defn}

\noindent Note that a flat family over $\text{Spec}\, \mathbb{C}$ is simply a stable curve.  We say that two such families $\mathcal{C}$ and $\mathcal{C}'$ over $S$ are equivalent, denoted $\mathcal{C} \sim_{S} \mathcal{C}'$, if there is an isomorphism of schemes $\mu: \mathcal{C} \to \mathcal{C}'$ along with the following diagram for all $i=1, \ldots, n$
\begin{equation}
\begin{tikzcd}
\mathcal{C} \arrow{rr}{\mu} \arrow{dr}{\pi} & & \mathcal{C}' \arrow[swap]{dl}{\pi'} \\
& S \arrow[bend left=40]{ul}{\sigma_{i}} \arrow[swap, bend right=37]{ur}{\sigma'_{i}} &
\end{tikzcd}
\end{equation}
By including the sections in the diagram, we require $\mu \circ \sigma_{i} = \sigma'_{i}$.  In other words, the isomorphism must preserve the marked points.  

One can study the moduli problem of parameterizing stable $n$-pointed curves of a fixed arithmetic genus.  The corresponding moduli functor (or moduli stack) is defined by $\overline{\mathcal{M}}_{g,n} : \mathbf{Sch}^{op} \to \mathbf{Set}$ such that
\begin{equation}
\overline{\mathcal{M}}_{g,n}(S) \coloneqq 
\begin{Bmatrix*}[l]
\text{Flat families over}\, S\, \text{of} \, n\text{-pointed stable}\\  \text{curves of arithmetic genus} \, g\, \text{up to} \, \sim_{S}
\end{Bmatrix*}
\end{equation}
and on morphisms of schemes, the functor acts in the obvious way by pullback.  It turns out that $\overline{\mathcal{M}}_{g,n}$ is not representable.  There is however, a coarse moduli space, typically denoted $\overline{M}_{g,n}$.  We regard $\overline{\mathcal{M}}_{g,n}$ as a compactification of the moduli stack $\mathcal{M}_{g,n}$ of smooth $n$-pointed curves of genus $g$, and likewise for the corresponding coarse moduli spaces.  

\begin{thm}[\bfseries Deligne-Mumford \cite{deligne_irreducibility_1969}]
The moduli stack $\overline{\mathcal{M}}_{g,n}$ is a smooth, proper, irreducible Deligne-Mumford stack of dimension $3g-3 + n$, assuming we have $2g-2+n>0$.  The coarse moduli space $\overline{M}_{g,n}$ is a projective variety of the same dimension with finite quotient singularities corresponding to stable curves with finite, but non-zero automorphisms.  
\end{thm}

One can also study the moduli problem for prestable curves.  As we will see in the next section, the domain of a stable map in Gromov-Witten theory is a prestable curve, so the following remark will be of some importance.  

\begin{rmk}\label{rmk:PreStabCurvvv}
The moduli stack $\mathfrak{M}_{g,n}$ of prestable $n$-pointed curves of arithmetic genus $g$ is a smooth Artin stack of dimension $3g-3+n$.  Note that because we allow infinitesimal automorphisms, this dimension may be negative.  
\end{rmk}

\subsection{The Moduli Space of Stable Maps}\label{eqn:secssonStabMapss}

Let $X$ be a smooth projective variety, and let $\mathcal{C}_{g,n}(X, \beta)$ be the (possibly empty) moduli space of smooth $n$-pointed curves in $X$ of genus $g$ and class $\beta \in H_{2}(X, \mathbb{Z})$.  As mentioned above, an insight of Kontsevich was to concretely realize a point in this moduli space as an embedding $f : C \hookrightarrow X$, where $(C, p_{1}, \ldots, p_{n})$ is a smooth $n$-pointed curve of genus $g$, and $f_{*}[C] = \beta$.  Though one can work explicitly with such an embedding, the moduli space is not compact.  One natural compactification arises by allowing for prestable curves $C$, and more general morphisms $f$.  

Given an $n$-pointed curve $(C, p_{1}, \ldots, p_{n})$ and a morphism $f: C \to X$, an automorphism of $f$ is an automorphism of the pointed curve which commutes with $f$ and preserves the marked points.  The group of automorphisms of $f$ is the subgroup $\text{Aut}(f)$ of the automorphism group of the pointed curve. 

\begin{defn}
We say the map $f$ is stable if $\text{Aut}(f)$ is finite, and denote the data by $(C, p_{1}, \ldots, p_{n}, f)$.    
\end{defn}

A \emph{collapsing component} of a stable map $(C, p_{1}, \ldots, p_{n}, f)$ is a component of $C$ mapping by $f$ to a point in $X$.  The condition that $\text{Aut}(f)$ is finite is equivalent to the requirement that any collapsing components of genus 0 must have at least 3 special points (nodes or marked points) while collapsing components of genus 1 must have at least 1 special point.  Notice that if $f$ is an embedding, there clearly are no collapsing components and $f$ is therefore stable.  
  
We must now understand how stable maps deform in families.  For a scheme $S$, always of finite type over $\mathbb{C}$, let $\pi: \mathcal{C} \to S$ be a flat family of prestable $n$-pointed curves of arithmetic genus $g$ over $S$, and let $\sigma_{1}, \ldots, \sigma_{n} : S \to \mathcal{C}$ be the disjoint sections.    
\begin{defn}
Let $X$ be a smooth projective variety.  A flat family over $S$ of stable maps into $X$ with a prestable $n$-pointed domain curve of genus $g$ is a diagram 
\begin{equation}
\begin{tikzcd}
\mathcal{C} \arrow{r}{f} \arrow{d}{\pi} & X \\
S &  &
\end{tikzcd}
\end{equation}
such that for all closed points $s \in S$, $\big( C_{s}, \sigma_{1}(s), \ldots, \sigma_{n}(s), f_{s} \big)$ is a stable map where $C_{s} = \pi^{-1}(s)$ and $f_{s} = f|_{C_{s}}$.  Moreover, we use the following terminology to describe such families:
\begin{enumerate}
\item A flat family of stable maps $f: \mathcal{C} \to X$ over $S$ is said to be $n$-pointed of genus $g$ if for all closed points $s \in S$, the curve $C_{s}$ is $n$-pointed with arithmetic genus $g$.
\item The family is said to represent the class $\beta \in H_{2}(X, \mathbb{Z})$ if $(f_{s})_{*}[C_{s}]=\beta$ for all closed points $s \in S$.  
\end{enumerate}
\end{defn}

\noindent We say that two flat families $f: \mathcal{C} \to X$ and $f': \mathcal{C}' \to X$ are isomorphic $\sim_{S}$ over $S$ if there exists an isomorphism $\mu: \mathcal{C} \to \mathcal{C}'$ with the following diagram for all $i=1, \ldots, n$  
\begin{equation}
\begin{tikzcd}
& \mathcal{C} \arrow{dl}{\pi}\arrow{dd}{\mu} \arrow{dr}{f} &   \\
S \arrow[bend left=40]{ur}{\sigma_{i}} \arrow[swap, bend right=37]{dr}{\sigma'_{i}} & & X \\
& \mathcal{C'}\arrow[swap]{ul}{\pi'} \arrow[swap]{ur}{f'}  &
\end{tikzcd}
\end{equation}  
By including the sections in the diagram, we require compatibility with marked points.  That is, $\mu \circ \sigma_{i} = \sigma'_{i}$ and $f \circ \sigma_{i} = f' \circ \sigma'_{i}$.  

We can now define the \emph{moduli stack of stable maps} by the functor $\overline{\mathcal{M}}_{g,n}(X, \beta) : \mathbf{Sch}^{op} \to \mathbf{Set}$, such that for all schemes $S$
\begin{equation}
\overline{\mathcal{M}}_{g,n}(X, \beta)(S) \coloneqq
\begin{Bmatrix*}[l]
\text{Flat families over}\, S\, \text{of} \,n\text{-pointed stable maps of} \\ \text{\, genus} \, g \, \text{representing}\, \beta \in H_{2}(X, \mathbb{Z}) \, \text{up to} \, \sim_{S}
\end{Bmatrix*}
\end{equation}
and such that the functor acts on morphisms in the obvious way by pullback.  This moduli space is the main object of interest in Gromov-Witten theory.  Let us briefly discuss some of its nice properties, as well as some not so nice.  

\begin{thm}
If $X$ is a smooth projective variety, then the moduli stack $\overline{\mathcal{M}}_{g,n}(X, \beta)$ of stable maps is a (generally singular) proper Deligne-Mumford stack.  The corresponding coarse moduli space is a projective scheme over $\mathbb{C}$.  
\end{thm}

\noindent The stack portion of the above theorem was proven by Kontsevich \cite{kontsevich_enumeration_1995}, while the result for the coarse moduli space can be found in Fulton-Pandharipande \cite{fulton_notes_1996}.  One consequence of properness is the following fact.  If two marked points come together on a domain curve, a new $\mathbb{P}^{1}$ ``bubbles" off, on which the two points are then separated.  Notice that this preserves the stability, and the new $\mathbb{P}^{1}$ component is collapsed by the map. 

We denote by $\mathcal{M}_{g,n}(X, \beta) \hookrightarrow \overline{\mathcal{M}}_{g,n}(X, \beta)$ the (possibly empty) open substack parameterizing stable maps with smooth domain curves.  Because embeddings are stable, we furthermore have the following inclusions
\begin{equation}
\begin{tikzcd}
\mathcal{C}_{g,n}(X, \beta) \arrow[hookrightarrow]{r}{} & \mathcal{M}_{g,n}(X, \beta) \arrow[hookrightarrow]{r}{} &  \overline{\mathcal{M}}_{g,n}(X, \beta).
\end{tikzcd}
\end{equation}
It is in this sense that the moduli space of stable maps is a compactification of the substacks $\mathcal{C}_{g,n}(X, \beta)$ and $\mathcal{M}_{g,n}(X, \beta)$.  However, note that it is such that the initial space may be empty or small, while the resulting space after compactifying is very large.  

The geometry of $\overline{\mathcal{M}}_{g,n}(X, \beta)$ is intimately connected with that of stable or prestable curves.  Because the domain curve of a stable map is a prestable curve, by forgetting the data of the map we get a morphism
\begin{equation}\label{eqn:morphprestabmap}
\overline{\mathcal{M}}_{g,n}(X, \beta) \longrightarrow \mathfrak{M}_{g,n}.
\end{equation}
If $2g-2+n>0$, we have a stabalization morphism $\mathfrak{M}_{g,n} \to \overline{\mathcal{M}}_{g,n}$ collapsing rational components with too few special points (nodes and marked points).  If this inequality is satisfied, we can therefore compose (\ref{eqn:morphprestabmap}) with the stabalization morphism to get
\begin{equation}
\overline{\mathcal{M}}_{g,n}(X, \beta) \longrightarrow \overline{\mathcal{M}}_{g,n}.
\end{equation}

One special case is that of $\beta=0$.  Clearly, the stability of a map collapsing the entire curve requires the domain curve to itself be stable.  We have an isomorphism 
\begin{equation} \label{eqn:DEG0GW}
\overline{\mathcal{M}}_{g,n}(X, 0) \cong \overline{\mathcal{M}}_{g,n} \times X.
\end{equation}
One should think that a moduli point in $\overline{\mathcal{M}}_{g,n}(X, 0)$ is determined by the choice of a stable curve along with a point in $X$ to which the curve collapses. 

In comparison to stable curves, the moduli space of stable maps is quite poorly behaved in the sense of classical geometry.  In addition to being usually singular, it is typically non-reduced, and contains many irreducible components of arbitrary dimensions.  The redeeming feature however, is that the moduli space of stable maps is rather nice from the perspective of modern enumerative geometry -- it carries a perfect obstruction theory and a virtual class.  We will now explain this part of the story.

\subsection{Perfect Obstruction Theory of Stable Maps and the Virtual Fundamental Class}

For many moduli spaces of interest in algebraic geometry, the local structure at a point is encoded by cohomological data associated to an object representing that point in the moduli space.  This data is interpreted as deformations of the object and obstructions to lifting deformations to higher order.  These deformation and obstruction theories are very deep (and we will hardly scratch the surface) so before jumping directly to stable maps, it is helpful to first consider the simpler case of $\overline{\mathcal{M}}_{g,n}$.

\subsubsection{The Deformation Theory of Stable Curves}

It is well-known that the infinitesimal deformations of a compact complex manifold $V$ with a finite automorphism group are given by the cohomology group $H^{1}(V, T_{V})$, while the obstructions to lifting deformations to higher order live in $H^{2}(V, T_{V})$.  The group $H^{0}(V, T_{V})$ corresponds to infinitesimal automorphisms of $V$, and vanishes by assumption.  It follows that the expected dimension of the corresponding moduli space at the point $V$ is 
\begin{equation}
h^{1}(V, T_{V}) - h^{2}(V, T_{V}),
\end{equation}
and a sufficient (but not necessary) condition for the smoothness of the moduli space at $V$ is that $h^{2}(V, T_{V})$ vanishes.  

This would suffice to understand the non-compact moduli space $\mathcal{M}_{g,0}$ (at least for $g \geq 2$), but we would like to generalize the discussion to account for arbitrary stable curves $(C, p_{1}, \ldots, p_{n})$ with marked points.  First, if $C$ is smooth one can introduce marked points by replacing $T_{C}$ by $T_{C}(-D)$ where we will use the notation $D \coloneqq \sum_{i=1}^{n} p_{i}$ for the sum of marked points as a divisor.  Of course, $T_{C}(-D)$ is the sheaf of holomorphic vector fields vanishing at the marked points.  On a smooth variety $X$, the dual of the tangent bundle is defined to agree with the sheaf $\Omega_{X}$ of K\"{a}hler differentials \cite[II, Sec. 8]{hartshorne_algebraic_1997}.  Because the K\"{a}hler differentials exist on a general scheme, if the curve $C$ is not smooth, it is natural to replace $T_{C}^{\smvee}(D)$ by $\Omega_{C}(D)$.  

The infinitesimal deformations of $\overline{\mathcal{M}}_{g,n}$ are given by $\text{Ext}^{1}_{C}(\Omega_{C}(D), \mathcal{O}_{C})$ with the obstructions living in $\text{Ext}^{2}_{C}(\Omega_{C}(D), \mathcal{O}_{C})$.  The infinitesimal automorphisms are given by $\text{Ext}^{0}_{C}(\Omega_{C}(D), \mathcal{O}_{C})$, and this group vanishes because stable curves have finite automorphisms \cite{deligne_irreducibility_1969}.  The expected dimension of $\overline{\mathcal{M}}_{g,n}$ is therefore given by 
\begin{equation}\label{eqn:verexpppdimMgn}
\text{dim} \, \text{Ext}_{C}^{1}\big( \Omega_{C}(D), \mathcal{O}_{C}\big) - \text{dim} \, \text{Ext}_{C}^{2}\big( \Omega_{C}(D), \mathcal{O}_{C}\big).
\end{equation}
If $C$ is in fact smooth, then $\Omega_{C}$ is locally-free and dual to the tangent bundle $T_{C}$.  By the two following elementary isomorphisms
\begin{equation}
\text{Ext}^{i}_{C}(\Omega_{C}(D), \mathcal{O}_{C}) \cong \text{Ext}_{C}^{i}(\mathcal{O}_{C}, T_{C}(-D)) \cong H^{i}(C, T_{C}(-D)),
\end{equation}
the automorphisms, deformations, and obstructions of a smooth $n$-pointed curve specialize consistently.  

For all stable curves $(C, p_{1}, \ldots, p_{n})$ it turns out the obstruction space $\text{Ext}_{C}^{2}\big( \Omega_{C}(D), \mathcal{O}_{C}\big)$ vanishes \cite{deligne_irreducibility_1969}.  This means that all deformations of stable curves are unobstructed and the moduli space $\overline{\mathcal{M}}_{g,n}$ is a smooth Deligne-Mumford stack, as we have stated in Section \ref{subsec:ModSpStCURVV}.  We say $\overline{\mathcal{M}}_{g,n}$ is smooth of the expected dimension, (\ref{eqn:verexpppdimMgn}) which is given numerically by \cite[Thm 3.17]{arbarello_geometry_2011}
\begin{equation}
\text{dim} \, \text{Ext}_{C}^{1}\big( \Omega_{C}(D), \mathcal{O}_{C}\big) = 3g -3 + n.
\end{equation}

We make one final comment which will be of importance in the deformation theory of stable maps.  One can also study the deformation theory of the Artin stack $\mathfrak{M}_{g,n}$ of prestable curves (see Remark \ref{rmk:PreStabCurvvv}).  The dimension is computed by
\begin{equation}\label{eqn:orbidimrangsp}
\text{dim} \, \text{Ext}_{C}^{1}\big( \Omega_{C}(D), \mathcal{O}_{C}\big) -  \text{dim} \, \text{Ext}_{C}^{0}\big( \Omega_{C}(D), \mathcal{O}_{C}\big)= 3g -3 + n,
\end{equation}
though the value can possibly be negative due to automorphisms.  In the case of a smooth curve, this result is clear by a simple application of Riemann-Roch.

\subsubsection{The Perfect Obstruction Theory of Stable Maps}

We now want to consider the deformations and obstructions of a stable map $(C, p_{1}, \ldots, p_{n}, f)$.  Unlike the case of $\overline{\mathcal{M}}_{g,n}$, the deformation theory of $\overline{\mathcal{M}}_{g,n}(X, \beta)$ is generally obstructed, resulting in either a singular moduli space or a moduli space of an unexpected dimension.  With stable maps, we can now deform both the curve $C$ and the map $f$, though we will not deform the target space $X$ which is always taken to be smooth and projective.  For a readable account of what it means to deform a map, we refer the reader to \cite{ran_deformations_1989}.  In what follows, we will draw from the exposition of \cite{cox_mirror_1999, pandharipande_13/2_2014}.  

Let us begin in the simple setting where $C$ is a smooth genus $g$ curve without marked points, and $f$ is an embedding such that $f_{*}[C] = \beta$.  The moduli in such a case are controlled by the normal bundle exact sequence
\begin{equation}\label{eqn:SESnormbundlseq}
0 \to T_{C} \to f^{*}T_{X} \to \nu_{C} \to 0
\end{equation}
where $\nu_{C}$ is the normal bundle of $C$ in $X$, and we have identified $f^{*}T_{X}$ with $T_{X}|_{C}$.  We can then pass to the long exact sequence in cohomology
\begin{equation}\label{eqn:LESINCOHOMGW}
\begin{split}
0 & \longrightarrow H^{0}(C, T_{C}) \longrightarrow H^{0}(C, f^{*}T_{X}) \longrightarrow H^{0}(C, \nu_{C}) \longrightarrow \\
& \longrightarrow H^{1}(C, T_{C}) \longrightarrow H^{1}(C, f^{*}T_{X}) \longrightarrow H^{1}(C, \nu_{C}) \longrightarrow 0.  
\end{split}
\end{equation}
Each of the cohomology groups above has a geometrical interpretation.  We have seen that $H^{0}(C, T_{C})$ and $H^{1}(C, T_{C})$ correspond to the infinitesimal automorphisms and deformations of a smooth curve, respectively.  

\begin{rmk}
Beware that the domain curve of an arbitrary stable map need only be prestable, not necessarily stable.  Therefore, $H^{0}(C, T_{C})$ might not vanish in general.  
\end{rmk}

\noindent We interpret $H^{0}(C, f^{*}T_{X})$ as the space of infinitesimal deformations of the map $f$ with fixed domain curve $C$, and $H^{1}(C, f^{*}T_{X})$ as the obstruction space to deforming of $f$.  Ultimately in this simple case, it is the sections of the normal bundle which we understand as the infinitesimal deformations of a stable map with obstruction space $H^{1}(C, \nu_{C})$.  With this in mind, the long exact sequence (\ref{eqn:LESINCOHOMGW}) is often written
\begin{equation}
\begin{split}
0 & \longrightarrow \text{Aut}(C) \longrightarrow \text{Def}(f) \longrightarrow \text{Def}(C, f) \longrightarrow \\
& \longrightarrow \text{Def}(C) \longrightarrow \text{Ob}(f) \longrightarrow \text{Ob}(C, f) \longrightarrow 0.  
\end{split}
\end{equation}
There is no entry $\text{Aut}(C, f)$ because the infinitesimal automorphisms of a stable map vanish, by definition.  

Noting that the alternating sum of dimensions in a long exact sequence of cohomology groups vanishes, we can write
\begin{equation}
h^{0}(C, \nu_{C}) - h^{1}(C, \nu_{C}) = \chi(C, f^{*}T_{X}) - \chi(C, T_{C}).
\end{equation}
A simple application of Riemann-Roch shows that $\chi(C, T_{C}) = 3-3g$ and
\begin{equation}\label{eqn:EulCharpullbcktabu}
\chi(C, f^{*}T_{X}) = \int_{\beta}c_{1}(X) + \text{dim}X(1-g)
\end{equation}
where $\beta = f_{*}[C]$.  We therefore have
\begin{equation}
h^{0}(C, \nu_{C}) - h^{1}(C, \nu_{C}) = \int_{\beta}c_{1}(X) + (\text{dim}X - 3)(1-g).
\end{equation}
Although we will see that this agrees numerically with what will soon be called the expected dimension of $\overline{\mathcal{M}}_{g,0}(X, \beta)$, it may happen that there are no moduli points corresponding to a smooth embedded curve.  In other words, $\mathcal{C}_{g,0}(X, \beta)$ may be empty.  We therefore need to mimic the above discussion in the case of general stable maps.  

Given a stable map $(C, p_{1}, \ldots, p_{n}, f)$ in $\overline{\mathcal{M}}_{g,n}(X, \beta)$, it remains true that $f^{*}T_{X}$ is locally-free but if $C$ is singular, $T_{C}$ no longer exists.  One should replace $T_{C}^{\smvee}$ by the sheaf of K\"{a}hler differentials $\Omega_{C}$, which coincides with the canonical bundle $\omega_{C}$ when $C$ is smooth.  Of course, since $X$ is smooth, $\Omega_{X}$ is simply the cotangent bundle.  

Recalling the notation $D = \sum_{i=1}^{n} p_{i}$ for the sum of marked points, the moduli are no longer controlled by the sequence (\ref{eqn:SESnormbundlseq}), but rather the complex
\begin{equation}\label{eqn:genstabmapcompll}
f^{*}\Omega_{X} \longrightarrow \Omega_{C}(D).
\end{equation}
In the case where $C$ is smooth and $f$ is an embedding, (\ref{eqn:genstabmapcompll}) is simply the dual of the canonical map $T_{C}(-D) \hookrightarrow f^{*}T_{X}$ whose kernel is the conormal bundle $\nu_{C}^{\smvee}$.  As explained in \cite{cox_mirror_1999} the automorphisms, deformations, and obstructions of a general stable map are governed by the following hyperext groups 
\[\mathbb{E}\text{xt}_{C}^{i}\big(\big\{ f^{*}\Omega_{X} \to \Omega_{C}(D)\big\}, \mathcal{O}_{C}\big).\]
The group with $i=0$ corresponds to automorphisms and vanishes by stability of the map.  Of course, $i=1$ and $i=2$ correspond to the tangent space and obstruction space, respectively.  We can now define the expected dimension of the moduli space at a point as the difference of the dimensions of the tangent space and obstruction space.  

\begin{defn}
The expected (or virtual) dimension of the moduli space of stable maps $\overline{\mathcal{M}}_{g,n}(X, \beta)$ at a stable map $(C, p_{1}, \ldots, p_{n}, f)$ is defined by
\begin{equation}\label{eqn:virdim}
\begin{split}
\text{vdim}\big(\overline{\mathcal{M}}_{g,n}(X, \beta)\big) & =  \text{dim} \, \mathbb{E}\text{xt}_{C}^{1}\big(\big\{ f^{*}\Omega_{X} \to \Omega_{C}(D)\big\}, \mathcal{O}_{C}\big)\\
& - \text{dim} \, \mathbb{E}\text{xt}_{C}^{2}\big(\big\{ f^{*}\Omega_{X} \to \Omega_{C}(D)\big\}, \mathcal{O}_{C}\big).
\end{split}
\end{equation} 
\end{defn}

Although we have not provided an honest definition, practically speaking one can take the following long exact sequence as the defining characterization of the hyperext groups introduced above
\begin{equation}\label{eqn:LESINCOHOMhyperext}
\begin{split}
0 & \to \text{Ext}_{C}^{0}(\Omega_{C}(D), \mathcal{O}_{C}) \to \text{Ext}_{C}^{0}(f^{*}\Omega_{X}, \mathcal{O}_{C}) \to \mathbb{E}\text{xt}_{C}^{1}\big(\big\{ f^{*}\Omega_{X} \to \Omega_{C}(D)\big\}, \mathcal{O}_{C}\big) \to \\
& \to \text{Ext}_{C}^{1}(\Omega_{C}(D), \mathcal{O}_{C}) \to \text{Ext}_{C}^{1}(f^{*}\Omega_{X}, \mathcal{O}_{C}) \to \mathbb{E}\text{xt}_{C}^{2}\big(\big\{ f^{*}\Omega_{X} \to \Omega_{C}(D)\big\}, \mathcal{O}_{C}\big) \to 0.  
\end{split}
\end{equation}

\begin{rmk}
The entries in the above long exact sequence have the following geometrical interpretations and specializations to known cases:
\begin{enumerate}
\item The automorphisms of the $n$-pointed curve are given by $\text{Ext}_{C}^{0}(\Omega_{C}(D), \mathcal{O}_{C})$ with the deformations given by $\text{Ext}_{C}^{1}(\Omega_{C}(D), \mathcal{O}_{C})$.  If $C$ is smooth, we have the isomorphisms
\begin{equation}
\text{Ext}_{C}^{i}(\Omega_{C}(D), \mathcal{O}_{C}) \cong H^{i}(C, T_{C}(-D)).
\end{equation}

\item Because $X$ is smooth, $f^{*}\Omega_{X}$ is a locally-free sheaf on $C$, whether the curve is smooth or not.  We therefore have the isomorphisms
\begin{equation}\label{eqn:isomEXTHi}
\text{Ext}_{C}^{i}(f^{*}\Omega_{X}, \mathcal{O}_{C}) \cong H^{i}(C, f^{*}T_{X})
\end{equation}
and we have seen that $H^{0}(C, f^{*}T_{X})$ corresponds to deformations of the map $f$ with the curve $C$ fixed, while $H^{1}(C, f^{*}T_{X})$ is the obstructions to deforming the map.  This remains true whether $C$ is smooth or not.  

\item We have already interpreted the terms $\mathbb{E}\text{xt}_{C}^{i}\big(\big\{ f^{*}\Omega_{X} \to \Omega_{C}(D)\big\}, \mathcal{O}_{C}\big)$, but we note here that if $C$ is smooth without marked points, and $f$ is an embedding, then the complex (\ref{eqn:genstabmapcompll}) is quasi-isomorphic to its kernel $\nu_{C}^{\smvee}$.  In such a case we therefore have
\begin{equation}
\mathbb{E}\text{xt}_{C}^{i}\big(\big\{ f^{*}\Omega_{X} \to \Omega_{C}(D)\big\}, \mathcal{O}_{C}\big) \cong H^{i}(C, \nu_{C})
\end{equation}
consistent with the discussion above.  
\end{enumerate}
\end{rmk}

\noindent With these geometrical interpretations in mind, one will often see the long exact sequence (\ref{eqn:LESINCOHOMhyperext}) written as
\begin{equation}
\begin{split}
0 & \longrightarrow \text{Aut}(C, p_{1}, \ldots, p_{n}) \longrightarrow \text{Def}(f) \longrightarrow \text{Def}(C, p_{1}, \ldots, p_{n}, f) \longrightarrow \\
& \longrightarrow \text{Def}(C, p_{1}, \ldots, p_{n}) \longrightarrow \text{Ob}(f) \longrightarrow \text{Ob}(C, p_{1}, \ldots, p_{n}, f) \longrightarrow 0.  
\end{split}
\end{equation}
As we warned earlier, the automorphisms of the curve might not vanish because the domain curve of a stable map might only be prestable.  The consequence of stability of the map is that a term like $\text{Aut}(C, p_{1}, \ldots, p_{n}, f)$ vanishes in the above sequence.  

Because the alternating sum of dimensions in a long exact sequence vanishes, by (\ref{eqn:virdim}) and (\ref{eqn:LESINCOHOMhyperext}) we can write the expected dimension of $\overline{\mathcal{M}}_{g,n}(X, \beta)$ as
\begin{equation}
\text{vdim}\big(\overline{\mathcal{M}}_{g,n}(X, \beta)\big) = \chi(C, f^{*}T_{X}) + \text{dim} \, \text{Ext}_{C}^{1}(\Omega_{C}(D), \mathcal{O}_{C}) - \text{dim} \, \text{Ext}_{C}^{0}(\Omega_{C}(D), \mathcal{O}_{C}).
\end{equation}
Here, we used the isomorphism (\ref{eqn:isomEXTHi}) resulting in the term $\chi(C, f^{*}T_{X})$.  This term can be computed by Riemann-Roch on a possibly singular curve and agrees numerically with (\ref{eqn:EulCharpullbcktabu}).  Finally, the remaining terms were recorded previously in (\ref{eqn:orbidimrangsp}).  In all, we now have the following result for the expected dimension
\begin{equation}\label{eqn:virDIMexpll}
\text{vdim}\big( \overline{\mathcal{M}}_{g,n}(X, \beta)\big) = \int_{\beta} c_{1}(X) + (\text{dim}X-3)(1-g) + n.
\end{equation}
Though not at all obvious from its definition (\ref{eqn:virdim}), the expected dimension does not depend on a particular stable map, and is therefore an invariant of the moduli space.

\begin{proppy}
The Zariski tangent space to $\overline{\mathcal{M}}_{g,n}(X, \beta)$ at $(C, p_{1}, \ldots, p_{n}, f)$ is given by the hyperext group $\mathbb{E}\text{xt}_{C}^{1}\big(\big\{ f^{*}\Omega_{X} \to \Omega_{C}(D)\big\}, \mathcal{O}_{C}\big)$.  It is clear from (\ref{eqn:virdim}) that the expected dimension is a lower bound for the actual dimension at each point.  The moduli space is smooth if and only if the Zariski tangent space doesn't jump in dimension over $\overline{\mathcal{M}}_{g,n}(X, \beta)$.  
\end{proppy}

The obstruction theory described here is a rather special kind.  We say the moduli space of stable maps carries a \emph{perfect obstruction theory}, defined in \cite{behrend_intrinsic_1997}.  Roughly speaking, such a theory is given by a perfect two-term complex governing deformations and obstructions.  Given a perfect obstruction theory, there is always an expected dimension which is constant over connected components of the moduli space, just as we observed in (\ref{eqn:virDIMexpll}).  Such a structure also leads to what is called the virtual fundamental class, which we will now briefly introduce.

\subsubsection{The Virtual Fundamental Class}

Recall that one motivation for Gromov-Witten theory is to count curves of a fixed genus $g$ and class $\beta$ in $X$ with $n$ generic incidence conditions.  The schematic way one might try to achieve this is by integrating certain cohomology classes (related to the incidence conditions) over $\overline{\mathcal{M}}_{g,n}(X, \beta)$.  However, we have seen that $\overline{\mathcal{M}}_{g,n}(X, \beta)$ is generally singular with many irreducible components of arbitrary dimensions; it does not have a fundamental class in the classical sense.  Instead, one would hope to find in the homology or Chow ring of $\overline{\mathcal{M}}_{g,n}(X, \beta)$ a special class in the degree of the expected dimension playing the role of the fundamental class.  Following from work of Behrend and Fantechi \cite{behrend_intrinsic_1997} such a class exists, and is called the virtual fundamental class.  

\begin{thm}
Given a proper scheme, or Deligne-Mumford stack $\mathcal{M}$ carrying a perfect obstruction theory of expected dimension $d$, there exists a cycle in the Chow ring
\[[\mathcal{M}]^{\text{vir}} \in A_{d}(\mathcal{M}, \mathbb{Q}) \longrightarrow H_{2d}(\mathcal{M}, \mathbb{Q})\]
or a class in homology called the virtual fundamental class (sometimes called just the virtual class).  
\end{thm}

This theorem may read strangely since we have not defined or constructed the virtual class.  However, the heuristic picture is that if one could deform the moduli data defining $\mathcal{M}$ so that it was smooth of the expected dimension, then $[\mathcal{M}]^{\text{vir}}$ would be its fundamental class.  It is an honest class, but only a \emph{virtual} fundamental class.  

Applying the seminal work of Behrend-Fantechi on perfect obstruction theories and virtual classes, Behrend introduced the virtual class of the moduli space of stable maps, and proved some basic axioms \cite{behrend_gromov-witten_1997}.  Underlying Gromov-Witten theory, we have the virtual class
\[[\overline{\mathcal{M}}_{g,n}(X, \beta)]^{\text{vir}} \in A_{d}\big(\overline{\mathcal{M}}_{g,n}(X, \beta), \mathbb{Q}\big)   \longrightarrow  H_{2d}\big(\overline{\mathcal{M}}_{g,n}(X, \beta), \mathbb{Q}\big)\]
where $d = \text{vdim}\big(\overline{\mathcal{M}}_{g,n}(X, \beta)\big)$ is the expected dimension (\ref{eqn:virDIMexpll}).  There are a few easy examples where we can understand the virtual class explicitly.  

\begin{Ex}[\bfseries Smooth of the Expected Dimension]
The simplest case is when $\overline{\mathcal{M}}_{g,n}(X, \beta)$ is smooth of the expected dimension.  By smoothness, the dimension of the tangent space $\mathbb{E}\text{xt}_{C}^{1}\big(\big\{ f^{*}\Omega_{X} \to \Omega_{C}(D)\big\}, \mathcal{O}_{C}\big)$ at a point does not jump over the moduli space.  Because $\overline{\mathcal{M}}_{g,n}(X, \beta)$ has the expected dimension, by (\ref{eqn:virdim}) it is clear that the obstruction space $\mathbb{E}\text{xt}_{C}^{2}\big(\big\{ f^{*}\Omega_{X} \to \Omega_{C}(D)\big\}, \mathcal{O}_{C}\big)$ vanishes.  In this case, the virtual class coincides with the ordinary fundamental class
\begin{equation}
[\overline{\mathcal{M}}_{g,n}(X, \beta)]^{\text{vir}} = [\overline{\mathcal{M}}_{g,n}(X, \beta)].
\end{equation}
By (\ref{eqn:LESINCOHOMhyperext}) it is clear that the obstruction space vanishes if $\text{Ext}_{C}^{1}(f^{*}\Omega_{X}, \mathcal{O}_{C}) \cong H^{1}(C, f^{*}T_{X})$ vanishes.  

Although rare, this case does occur in practice.  For example, if $X = \mathbb{P}^{N}$, then $\overline{\mathcal{M}}_{0,n}(X, \beta)$ is smooth of the expected dimension.  For any $X$, the same is true for $\overline{\mathcal{M}}_{0,n}(X, 0)$, and if $X$ is a point, then $\overline{\mathcal{M}}_{g,n}(X, 0) \cong \overline{\mathcal{M}}_{g,n}$ is also smooth of the expected dimension.  As a final example, the moduli space $\overline{\mathcal{M}}_{1,0}(E, d)$ of degree $d>0$ covers of an elliptic curve $E$ is also smooth of the expected dimension.  
\end{Ex}

\begin{Ex}[\bfseries Smooth of the Wrong Dimension]
It is also possible that the moduli space is smooth but of a dimension larger than the expected dimension.  By smoothness, the dimension of the tangent space $\mathbb{E}\text{xt}_{C}^{1}\big(\big\{ f^{*}\Omega_{X} \to \Omega_{C}(D)\big\}, \mathcal{O}_{C}\big)$ does not jump, and because the expected dimension is an invariant of the moduli space, the dimension of $\mathbb{E}\text{xt}_{C}^{2}\big(\big\{ f^{*}\Omega_{X} \to \Omega_{C}(D)\big\}, \mathcal{O}_{C}\big)$ therefore also does not jump.  This fact gives rise to a canonical vector bundle called the obstruction bundle
\begin{equation}
\mathcal{O}b \longrightarrow \overline{\mathcal{M}}_{g,n}(X, \beta)
\end{equation}
whose fiber at $(C, p_{1}, \ldots, p_{n}, f)$ is $\mathbb{E}\text{xt}_{C}^{2}\big(\big\{ f^{*}\Omega_{X} \to \Omega_{C}(D)\big\}, \mathcal{O}_{C}\big)$.  The rank of $\mathcal{O}b$ measures the difference between the actual and expected dimensions of $\overline{\mathcal{M}}_{g,n}(X, \beta)$.  Here, the virtual class is given by
\begin{equation}
[\overline{\mathcal{M}}_{g,n}(X, \beta)]^{\text{vir}} = \text{PD}\big( e(\mathcal{O}b) \big)
\end{equation}
where $e(\mathcal{O}b)$ is the Euler class and defines an element of $H^{\text{rk}(\mathcal{O}b)}\big(\overline{\mathcal{M}}_{g,n}(X, \beta), \mathbb{Q}\big)$.  The primary example of this case which we will consider is $\overline{\mathcal{M}}_{g,n}(X, 0) \cong \overline{\mathcal{M}}_{g,n} \times X$ for $2g-2+n>0$.  
\end{Ex}

\subsection{The Gromov-Witten Invariants}  

Throughout this chapter, the primary motivation for studying Gromov-Witten theory was to count curves in $X$ with fixed discrete invariants, and intersecting fixed cycles $Z_{1}, \ldots, Z_{n}$.  We now have enough technology to rigorously assign such invariants to $X$.  For each $i = 1, \ldots, n$ we have an evaluation map 
\begin{equation}
\text{ev}_{i} : \overline{\mathcal{M}}_{g,n}(X, \beta) \longrightarrow X
\end{equation}
defined by sending a stable map $(C, p_{1}, \ldots, p_{n}, f)$ to $f(p_{i})$.  We can therefore define the composition
\begin{equation}\label{eqn:compositionGWI}
H^{*}(X)^{\otimes n} \overset{\text{ev}_{1}^{*} \smile \cdots \smile \text{ev}_{n}^{*}}{\xrightarrow{\hspace*{1.7cm}}}H^{*}\big(\overline{\mathcal{M}}_{g,n}(X, \beta) \big) \overset{\int_{[\overline{\mathcal{M}}_{g,n}(X, \beta)]^{\text{vir}}}}{\xrightarrow{\hspace*{1.7cm}}}\mathbb{Q}
\end{equation}
where the first map is given by pullback $\text{ev}_{1}^{*} \smile \cdots \smile \text{ev}_{n}^{*}$ and the second map by integration against the virtual fundamental class.  Let $\alpha_{i} \in H^{*}(X)$ represent the Poincar\'{e} dual in $X$ of the homology class of the cycle $Z_{i}$.  To get a sensible invariant we want the cycles to be in generic position which is why we are using the data of their homology classes as opposed to the cycles themselves.  The classes $\alpha_{i}$ are typically referred to as \emph{insertions}.  

We define the \emph{Gromov-Witten invariants} $\langle \alpha_{1} \cdots \alpha_{n}\rangle^{X}_{g, \beta}$ as the image of $\alpha_{1} \smile \cdots \smile \alpha_{n}$ under the composition (\ref{eqn:compositionGWI}).  This is concretely expressed as
\begin{equation}\label{eqn:GWintegral}
\langle \alpha_{1} \cdots \alpha_{n}\rangle^{X}_{g, \beta} = \int_{[\overline{\mathcal{M}}_{g, n}(X, \beta)]^{\text{vir}}} \text{ev}_{1}^{*} (\alpha_{1}) \smile \cdots \smile \text{ev}_{n}^{*} (\alpha_{n}) \in \mathbb{Q}. 
\end{equation}   
The Gromov-Witten invariants are deformation invariants of $X$, and one should interpret $\langle \alpha_{1} \cdots \alpha_{n}\rangle^{X}_{g, \beta}$ to be a virtual count of stable maps from curves of arithmetic genus $g$, landing in class $\beta$, and intersecting each of the classes $[Z_{i}] = \text{PD}(\alpha_{i})$.  The invariants are rational precisely because the moduli space of stable maps is a Deligne-Mumford stack; the non-integrality is traced back to stable maps with finite, but non-vanishing automorphism group.    

By definition, the integral (\ref{eqn:GWintegral}) vanishes unless the sum of the degrees of the insertions $\alpha_{i}$ coincide with the (real) expected dimension.  This implies the constraint
\begin{equation}\label{eqn:GWconstr}
\int_{\beta} c_{1}(X) + (\text{dim}X-3)(1-g) + n = \frac{1}{2} \sum_{i=1}^{n} \text{deg}(\alpha_{i}) = \frac{1}{2}\sum_{i=1}^{n} \text{codim}(Z_{i}).
\end{equation}
Notice Calabi-Yau threefolds appear to be special in that (\ref{eqn:GWconstr}) is satisfied for all $g$ and $\beta$ without insertions.  For arbitrary smooth projective varieties, the virtual dimension may vanish for special $g$ and $\beta$, but one generally needs insertions.  

The following remark describes another trivial constraint on the Gromov-Witten invariants.  

\begin{rmk}\label{rmk:effrmk}
We say $\beta \in H_{2}(X, \mathbb{Z})$ is not an effective class if it is non-zero and if it cannot be represented by an algebraic curve in $X$.  In such a case, $\langle \alpha_{1} \cdots \alpha_{n}\rangle^{X}_{g, \beta}=0$ for all $g,n$ because the moduli space $\overline{\mathcal{M}}_{g,n}( X, \beta)$ is empty.  
\end{rmk}

We can define a tautological line bundle $\mathbb{L}_{i}$ on $\overline{\mathcal{M}}_{g,n}(X, \beta)$ whose fiber at a point $(C, p_{1}, \ldots, p_{n}, f)$ is the one-dimensional vector space $T_{C}^{\smvee}|_{p_{i}}$.  The classes $\psi_{i} = c_{1}(\mathbb{L}_{i})$ give cohomology classes on $\overline{\mathcal{M}}_{g,n}(X, \beta)$.  We can therefore define the \emph{descendent Gromov-Witten invariants} or \emph{gravitational descendants} to be
\begin{equation}\label{eqn:gravdesccccinv}
\big\langle \tau_{a_{1}}(\alpha_{1}) \cdots \tau_{a_{n}}(\alpha_{n})\big\rangle^{X}_{g, \beta} = \int_{[\overline{\mathcal{M}}_{g, n}(X, \beta)]^{\text{vir}}} \text{ev}_{1}^{*} (\alpha_{1}) \smile \psi_{1}^{a_{1}} \cdots \smile \text{ev}_{n}^{*} (\alpha_{n}) \smile \psi_{n}^{a_{n}} 
\end{equation}
where $\alpha_{i} \in H^{*}(X)$.  These invariants vanish unless $\frac{1}{2}\sum_{i=1}^{n} (\text{deg}(\alpha_{i}) + a_{i})$ coincides with the expected dimension of $\overline{\mathcal{M}}_{g,n}(X, \beta)$.  An enumerative interpretation of the gravitational descendants is not so clear in general, though they ultimately contribute to the correlation functions of the A-model topological string theory on $X$.

In the final sections, we will introduce the generating functions of Gromov-Witten invariants on Calabi-Yau threefolds as well as connections to string theory.  There are a number of deep topics within Gromov-Witten theory which we are forced to omit.  Most notably, we refer the reader to \cite{cox_mirror_1999} for a full discussion of quantum cohomology.  There are also certain axioms defining Gromov-Witten theory as well as connections to cohomological field theory \cite{kontsevich_gromov-witten_1994, hori_mirror_2003} which we must omit.  Finally, we have gone without many examples, so we refer the reader to \cite{hori_mirror_2003}.

\subsection{The Partition Function and Free Energy on a Calabi-Yau Threefold}

One departure of modern enumerative geometry from that of past centuries is that much less interest lies with the individual enumerative invariants.  Rather, the fundamental objects appear to be various generating functions of the invariants.  One the one hand, these generating functions often exhibit automorphic properties, hinting at hidden symmetries of a moduli space.  But they also seem to coincide with partition functions or correlation functions in physics, perhaps after twisting the physical theory and localizing the path integral.  We now introduce some of these ideas in the context of Gromov-Witten theory.  

Let $(X, \omega)$ be a smooth projective Calabi-Yau threefold with $\omega= B + iJ$ a complexified K\"{a}hler form.  Here $J$ is an honest K\"{a}hler form, and $B \in H^{2}(X, \mathbb{R})$ is called the B-field.  Let $\{S_{1}, \ldots, S_{k}\}$ be a positive basis of $H_{2}(X, \mathbb{Z})$, always modulo torsion.  This means that for any effective curve class $\beta \in H_{2}(X, \mathbb{Z})$, there exists non-negative integers $\{n_{a}\}_{a=1}^{k}$ such that $\beta = \sum_{a=1}^{k} n_{a}S_{a}$.  We call $n_{a}$ the degree of $\beta$ along $S_{a}$.  Using these definitions, we introduce the formal parameter $Q^{\beta}$ to be 
\begin{equation}\label{eqn:GWdegreetrack}
Q^{\beta} = \prod_{a=1}^{k} Q_{a}^{n_{a}}, \,\,\,\, \text{where} \,\,\,\, Q_{a} = e^{2 \pi i \int_{S_{a}} \omega}.
\end{equation}
We will think of $Q$, or more precisely $\{Q_{1}, \ldots, Q_{k}\}$, as degree-tracking parameters in Gromov-Witten (and also Donaldson-Thomas) theory.  Notice that $\int_{S_{a}} \omega $ has the interpretation of a (complexified) volume of the cycle $S_{a}$.  The limit $\int_{S_{a}} J \to \infty$ sends the volume of all curves in the class $[S_{a}]$ to infinity and by (\ref{eqn:GWdegreetrack}), clearly corresponds to $Q_{a} \to 0$.  If the limit is taken for all $a$, we effectively have $Q \to 0$.  This is known as the degree zero limit or the classical limit, for reasons which will become clear. 

We notice from (\ref{eqn:virDIMexpll}) that if $X$ is a Calabi-Yau threefold, the expected dimension of $\overline{\mathcal{M}}_{g,0}(X, \beta)$ vanishes for all $g$ and $\beta$.  We therefore do not require insertions to get a non-vanishing invariant (\ref{eqn:GWintegral}).  Changing notation slightly for convenience, the Gromov-Witten invariants are then given by  
\begin{equation}
\text{GW}_{g, \beta}(X) \coloneqq \langle \,\, \rangle_{g, \beta}^{X} = \int_{[\overline{\mathcal{M}}_{g,0}(X, \beta)]^{\text{vir}}} 1
\end{equation}
which is simply the degree of the zero-dimensional virtual class.  This is simply a sum over the multiplicities in a Chow zero-cycle, and is interpreted as a virtual count of the number of stable maps from genus $g$ curves into $X$ lying in the class $\beta$.  

We want to begin introducing the standard generating functions of the Gromov-Witten invariants on a Calabi-Yau threefold.  For a fixed genus $g$, the \emph{Gromov-Witten potential} $F_{g}$ is defined to be the generating function of Gromov-Witten invariants $\text{GW}_{g, \beta}(X)$ summing over all homology classes
\begin{equation}
F_{g}(Q) = \sum_{\beta \in H_{2}(X, \mathbb{Z})} \text{GW}_{g, \beta}(X) Q^{\beta}.
\end{equation}
Reminding the reader of Remark \ref{rmk:effrmk}, it would suffice to sum only over effective classes since the invariants vanish otherwise.  The \emph{Gromov-Witten free energy} $F_{\text{GW}}(X)$ is defined as the generating function of Gromov-Witten potentials $F_{g}$, weighted by a factor of $\lambda^{2g-2}$
\begin{equation}
F_{\text{GW}}(X) = \sum_{g=0}^{\infty} \lambda^{2g-2} F_{g}(Q)
\end{equation}
where $\lambda$ is a single parameter known as the \emph{string coupling constant}.  One can also define the Gromov-Witten potentials or free energy by summing only over a sublattice $\Gamma \subset H_{2}(X, \mathbb{Z})$.  For all classes $\beta \in H_{2}(X, \mathbb{Z})$, we define the free energy with fixed $\beta$ as
\begin{equation}
F_{\text{GW}}(X)_{\beta} = \sum_{g=0}^{\infty} \lambda^{2g-2} \text{GW}_{g, \beta}(X).  
\end{equation}
One finally then defines the \emph{Gromov-Witten partition function} to be the exponential of the free energy
\begin{equation}
Z_{\text{GW}}(X) = \text{exp}\big(F_{\text{GW}}(X)\big),
\end{equation}  
where the coefficients of the $\lambda$ and $Q$ expansions are interpreted as \emph{disconnected} Gromov-Witten invariants.  Both the Gromov-Witten partition function and free energy are functions of $\lambda$ and $Q$, which we will suppress from the notation, unless needed in a particular context.  

The degree-zero Gromov-Witten invariants $\text{GW}_{g, 0}(X)$ correspond to stable maps collapsing an entire genus $g$ curve to a point in $X$.  These contributions to the generating functions must be handled somewhat carefully, which we will do in the next section.  For now, we mention that by dividing away degree-zero contributions, we produce the \emph{reduced} Gromov-Witten partition function
\begin{equation}
Z'_{\text{GW}}(X) = \frac{Z_{\text{GW}}(X)}{Z_{\text{GW}}(X)_{0}}
\end{equation}
where $Z_{\text{GW}}(X)_{0} = \text{exp}\big( F_{\text{GW}}(X)_{0}\big)$.  One similarly may speak of reduced Gromov-Witten potentials or reduced Gromov-Witten free energy in the obvious manner.  Though instead of dividing, one must instead subtract the degree zero contributions from the Gromov-Witten free energy or potentials.

We have commented that $Q_{a} \to 0$ corresponds to sending the volume of all curves in the class $S_{a}$ to infinity, and $Q \to 0$ is the degree zero limit.  The reason for the name is the following facts
\begin{equation}\label{eqn:degreezerrlm}
\lim_{Q \to 0} Z_{\text{GW}}(X)  = Z_{\text{GW}}(X)_{0}, \,\,\,\,\,\,\,\,\,\,\,\,\,\,\,  \lim_{Q \to 0} Z'_{\text{GW}}(X) = 1.
\end{equation}
One should think very heuristically that if the volumes of all curve classes are sent to infinity, it would take infinite ``energy" for a map from a curve to wrap those classes.  Therefore, the only contributions would be in degree zero.

\subsubsection{Degree Zero Contributions to the Gromov-Witten Potentials}

In this section we will give a completely explicit description of the degree zero contributions to the Gromov-Witten potentials in all genus.  We know that for a map to be stable in degree zero, the domain curve must be stable.  Stability presents no problems for $g \geq 2$, but it requires marked points for $g=0,1$.  This means we have to deal explicitly with insertions in the generating functions.  Following the notation of \cite{pandharipande_three_2003}, we can let $X$ be a non-singular projective threefold, not necessarily Calabi-Yau.  For indexing sets $A$ and $D_{2}$, let $\{\gamma_{a}\}_{a \in A}$ denote a basis of $H^{*}(X, \mathbb{Z})$ modulo torsion, and let $\{\gamma_{a}\}_{a \in D_{2}}$ be the classes of degree two.  We also define formal variables $\{t_{a}\}_{a \in A}$ corresponding to $\{\gamma_{a}\}_{a \in A}$.  

Let $F_{g}^{(0)}(t)$ denote the degree zero contributions to the genus $g$ Gromov-Witten potential, where we abbreviate the collection of $t_{a}$ by $t$.  The degree zero contributions for $g=0$ simply encode the classical intersection numbers of $X$ as
\begin{equation}
F^{(0)}_{0}(t) = \sum_{a_{1}, a_{2}, a_{3} \in A} \frac{1}{3!} \, t_{a_{1}} t_{a_{2}} t_{a_{3}} \int_{X} \gamma_{a_{1}} \smile  \gamma_{a_{2}} \smile  \gamma_{a_{3}}.
\end{equation}
The three insertions correspond to the three marked points stabalizing a degree zero map from a rational curve.  Using that $\overline{\mathcal{M}}_{1,1}(X, 0) \cong \overline{\mathcal{M}}_{1,1} \times X$ is smooth of the wrong dimension, a virtual class computation shows that the degree zero contribution in $g=1$ takes the form
\begin{equation}
F^{(0)}_{1}(t) = \sum_{a \in D_{2}} t_{a} \langle \gamma_{a} \rangle^{X}_{1,0} = - \frac{1}{24} \sum_{a \in D_{2}} t_{a} \int_{X} \gamma_{a} \smile c_{2}(X).
\end{equation}
For $g \geq 2$, we have no need for insertions, so the degree zero contribution to $F_{g}$ is just a constant, independent of the variables $t_{a}$.  We have
\begin{equation}\label{eqn:deg0gGEN}
F^{(0)}_{g} = \frac{(-1)^{g}}{2} \int_{X} \big(c_{3}(X) - c_{1}(X) \smile c_{2}(X) \big) \int_{\overline{\mathcal{M}}_{g,0}} \lambda_{g-1}^{3}.
\end{equation}
The final integral is known as a Hodge integral.  Here, $\lambda_{g-1} = c_{g-1}(\mathbb{E})$ where $\mathbb{E} \to \overline{\mathcal{M}}_{g,0}$ is the Hodge bundle with fiber $H^{0}(C, \omega_{C})$ at a point $C \in \overline{\mathcal{M}}_{g,0}$.  The Hodge integrals have been explicitly computed by Faber and Pandharipande \cite{faber_hodge_1998} 
\begin{equation}
\int_{\overline{\mathcal{M}}_{g,0}} \lambda_{g-1}^{3} = -\frac{B_{2g}}{2g} \frac{B_{2g-2}}{2g-2} \frac{1}{(2g-2)!} \,\,\,\,\,\,\,\,\,\,\,\,\,\,\,\,\,\,\,\,\,\,\,\, (g \geq 2)
\end{equation}
where $B_{n}$ are the Bernoulli numbers\footnote{There are important signs to keep track of with Bernoulli numbers.  It turns out that $B_{2g}$ and $B_{2g-2}$ differ by a sign, so $|B_{2g}| \, |B_{2g-2}| = - B_{2g} B_{2g-2}$, for all $g \geq 2$}.  If $X$ is additionally Calabi-Yau with topological Euler characteristic $\chi(X) = \int_{X} c_{3}(X)$, then (\ref{eqn:deg0gGEN}) clearly specializes to
\begin{equation}\label{eqn:Deg0gCY3}
F^{(0)}_{g} = \frac{1}{2}\frac{|B_{2g}|}{2g} \frac{B_{2g-2}}{2g-2} \frac{\chi(X)}{(2g-2)!} \,\,\,\,\,\,\,\,\,\,\,\,\,\,\,\,\,\,\,\, (g \geq 2)
\end{equation}
where $|B_{2g}| = (-1)^{g-1} B_{2g}$.  

In the analysis of degree zero contributions to Gromov-Witten theory, an important function is the MacMahon function defined by 
\begin{equation} \label{eqn:McMahon}
M(q) = \prod_{n=1}^{\infty} \big(1-q^{n}\big)^{-n}.
\end{equation}
This is the generating function of plane partitions -- colloquially speaking, the coefficient of $q^{d}$ is the number of ways of stacking $d$ boxes into a corner of three-dimensional space.  One can study the asymptotic expansion of $\log M(e^{i \lambda})$ around $\lambda =0$, and it turns out that it very nearly encodes the $g \geq 2$ contributions (\ref{eqn:Deg0gCY3}) in degree zero.  For a Calabi-Yau threefold $X$, multiplying the result of \cite[Eqn. 42]{koshkin_quantum_2009} by $\frac{1}{2} \chi(X)$, we find 
\begin{equation}\label{eqn:asympEXPMcMahon}
\tfrac{1}{2} \chi(X) \log M(e^{i \lambda}) \sim - \frac{\chi(X)}{2} \frac{\zeta(3)}{\lambda^{2}} + \frac{\chi(X)}{24} \log (-i \lambda) + \frac{\chi(X)}{2} \zeta'(1) + \sum_{g=2}^{\infty} \lambda^{2g-2} F_{g}^{(0)}.  
\end{equation}
So the MacMahon function clearly plays a role in degree zero Gromov-Witten theory.  The strongest statement we can make is that we have an asymptotic equivalence 
\begin{equation}\label{eqn:AsymptExprGWW}
Z_{\text{GW}}(X)_{0} \sim M(e^{i\lambda})^{\frac{\chi(X)}{2}}
\end{equation}
where the precise meaning of $\sim$ here is that the logarithm of both sides agree identically in powers of $\lambda^{2}$ and higher.  Of course, the logarithm of the lefthand side contains the terms $F_{0}^{(0)}(t)$ and $F_{1}^{(0)}(t)$ discussed above, which clearly will not be encoded by the MacMahon function.

\subsection{Gromov-Witten Theory and Topological String Theory}\label{subsec:GWA-modTSSP}

Gromov-Witten theory is the rigorous mathematical formulation of what is called the A-model topological string theory.  One begins with a (closed string) non-linear sigma model, which is a two-dimensional quantum field theory of continuous maps
\begin{equation}\label{eqn:nonlinsigmmmmod}
\phi : C \longrightarrow X
\end{equation}
from a smooth projective connected curve $C$ called the \emph{worldsheet} into a K\"{a}hler manifold $X$ called the \emph{target space}.  Because the curve has no boundaries, we interpret the worldsheet to represent closed strings propagating through $X$, and interacting by joining and splitting.    

There are two possible topological twists, resulting in the A and B-model topological field theories \cite{witten_mirror_1991}, which are conjecturally exchanged by mirror symmetry.  The A-model correlation functions depend only on the K\"{a}hler moduli of $X$, and localize onto a moduli space of holomorphic maps from a \emph{fixed} curve into $X$.  The A-model topological string theory is simply the A-model topological field theory coupled to \emph{worldsheet gravity}, which means we also integrate over the moduli space of complex structures on the curve, resulting in a theory of quantum gravity.  The contact with Gromov-Witten theory is clear, and the correlation functions ultimately coincide with the gravitational descendants defined in (\ref{eqn:gravdesccccinv}).

\subsubsection{Introduction to A-model Topological Field Theory}

Let us begin with a brief discussion of the A-model twist.  The bosonic scalar fields are the holomorphic coordinate functions $\phi^{i}$ of the map (\ref{eqn:nonlinsigmmmmod}).  We also have anti-holomorphic fields $\overline{\phi^{i}} = \phi^{\bar{i}}$.  Both $i$ and $\bar{i}$ range over the complex dimension of $X$.  Let us break convention with previous sections and write $T_{X} = T_{X}^{(1,0)} \oplus T_{X}^{(0,1)}$ for the complexified tangent bundle with its decomposition into holomorphic and anti-holomorphic parts.  Also define $K_{C}$ to be the canonical bundle of $C$, with $\overline{K}_{C}$ its complex conjugate.  The fermionic fields are sections of certain bundles on $C$ as follows
\begin{equation}
\begin{split}
& \psi_{+}^{i} \in \Gamma( \phi^{*} T_{X}^{(1,0)}), \,\,\,\,\,\,\,\,\,\,\,\,\,\,\,\, \psi_{+}^{\bar{i}} = \Gamma( K_{C} \otimes \phi^{*}T_{X}^{(0,1)}) \\
& \psi_{-}^{\bar{i}} \in \Gamma( \phi^{*} T_{X}^{(0,1)}), \,\,\,\,\,\,\,\,\,\,\,\,\,\,\,\, \psi_{-}^{i} = \Gamma( \overline{K}_{C} \otimes \phi^{*}T_{X}^{(1,0)}).
\end{split}
\end{equation}
It is conventional to package $\psi_{+}^{i}$ and $\psi_{-}^{\bar{i}}$ into sections of $\phi^{*}T_{X}$.  Let $\chi^{I}$, with index $I$ ranging over the \emph{real} dimension of $X$, denote such sections whose holomorphic and anti-holomorphic parts are given by $\chi^{i} = \psi_{+}^{i}$ and $\chi^{\bar{i}} = \psi_{-}^{\bar{i}}$.  

If we choose coordinates $(z, \bar{z})$ on $C$ with volume form $d^{2}z$, we can write the A-model action as
\begin{equation}
S = \int_{C} d^{2}z \bigg( \frac{1}{2} (g_{IJ} + iB_{IJ}) \partial_{z} \phi^{I} \partial_{\bar{z}}\phi^{J} + i g_{i \bar{j}} \psi_{+}^{\bar{j}} D_{\bar{z}}\chi^{i} + i g_{i \bar{j}}\psi_{-}^{i} D_{z}\chi^{\bar{j}} - R_{i \bar{i} j \bar{j}} \psi_{-}^{i} \psi_{+}^{\bar{i}} \chi^{j} \chi^{\bar{j}}\bigg)
\end{equation}
where $R$ is the Riemann tensor on $X$ associated to metric $g$, and $D_{z}, D_{\bar{z}}$ are covariant derivatives on the appropriate bundles.  The first term is simply the complexified volume $\int_{C} \phi^{*}(B + iJ)$, where $J$ is the K\"{a}hler form produced from $g$.  

There is a supersymmetry operator $Q$ along with an infinitesimal fermionic parameter $\alpha$ which associates to any field $\mathcal{O}$ its variation $\delta \mathcal{O} \coloneqq - i \alpha \{Q, \mathcal{O}\}$.  Recall that supersymmetries exchange bosonic and fermionic fields.  The variations of the fields in the A-model described above are given by
\begin{equation}\label{eqn:Amodvarrrrs}
\begin{split}
& \delta \phi^{i} = i \alpha \chi^{i} \,\,\,\,\,\,\,\,\,\,\,\, \delta \chi^{i} =0 \,\,\,\,\,\,\,\,\,\,\,\, \delta \psi^{i}_{-} = - \alpha \partial_{\bar{z}} \phi^{i}-i \alpha \chi^{k} \Gamma^{i}_{km} \psi^{m}_{-} \\
& \delta \phi^{\bar{i}} = i \alpha \chi^{\bar{i}} \,\,\,\,\,\,\,\,\,\,\,\, \delta \chi^{\bar{i}} =0 \,\,\,\,\,\,\,\,\,\,\,\, \delta \psi^{\bar{i}}_{+} = - \alpha \partial_{z} \phi^{\bar{i}}-i \alpha \chi^{\bar{k}} \Gamma^{\bar{i}}_{\bar{k} \bar{m}} \psi^{\bar{m}}_{+}.
\end{split}
\end{equation}
The first column indicates for example that $\phi^{i}$ and $\chi^{i}$ are superpartners; likewise for their complex conjugates.  One can show that the action $S$ is invariant under these variations, and that $\delta^{2} = 0$.  In such a case, $Q$ is called a BRST operator, and we conventionally write $Q^{2}=0$ instead of $\delta^{2}=0$.

\subsubsection{The Physical Operators}  

The fact that $Q^{2}=0$ for a BRST operator means we get a complex whose entries are spaces of fields or operators.  The cohomology of this complex is called the BRST cohomology, and each class is represented by a $Q$-closed (or $Q$-invariant) operator.  The $Q$-closed operators are what we regard as the physical operators in the A-model, and from the variation (\ref{eqn:Amodvarrrrs}) we see that they must be formed out of only $\phi^{i}$ and $\chi^{i}$.  

Given a de Rham cohomology class $A \in H^{p}(X, \mathbb{C})$ described locally as $A_{i_{1} \cdots i_{p}}(x)dx^{i_{1}} \cdots dx^{i_{p}}$, we construct a local operator defined by
\begin{equation}
\mathcal{O}_{A}(x) \coloneqq A_{i_{1} \cdots i_{p}}(x) \chi^{i_{1}} \cdots \chi^{i_{p}}.  
\end{equation}
We can define a map from de Rham cohomology to the BRST cohomology by $A \mapsto \mathcal{O}_{A}$, but we really want this to be an isomorphism of cohomologies -- in other words, we want it to respect the differentials.  The differential in BRST cohomology is $\delta$, and one can show that
\begin{equation}
\delta \mathcal{O}_{A} = - \mathcal{O}_{dA}.
\end{equation}  
This means the above map takes closed forms to $Q$-closed forms, and likewise for exact forms.  We therefore identify the BRST cohomology of the A-model with the de Rham cohomology of $X$.  In what follows, we will write the A-model physical operators as $\mathcal{O}_{A}$, where $A \in H^{p}(X, \mathbb{C})$ and we will say that $\mathcal{O}_{A}$ has \emph{ghost number} $p = \text{deg}(A)$.

\subsubsection{Correlation Functions and the Ghost Number Anomaly}

The correlation functions of the A-model topological field theory are given by path integrals of the form
\begin{equation}\label{eqn:Amodddcorrrrfuncc}
\langle \mathcal{O}_{A_{1}} \cdots \mathcal{O}_{A_{n}} \rangle \coloneqq \int \mathcal{D} \phi \mathcal{D} \psi e^{-S}  \mathcal{O}_{A_{1}} \cdots \mathcal{O}_{A_{n}}.
\end{equation}
The topological nature of the theory is manifest in the invariance of the correlation functions under deformations of the worldsheet metric.  It is natural to decompose a correlation function in terms of the class $\beta \in H_{2}(X, \mathbb{Z})$ of the image of $\phi$.  We have
\begin{equation}
\langle \mathcal{O}_{A_{1}} \cdots \mathcal{O}_{A_{n}} \rangle = \sum_{\beta \in H_{2}(X, \mathbb{Z})} \langle \mathcal{O}_{A_{1}} \cdots \mathcal{O}_{A_{n}} \rangle_{\beta},
\end{equation}
where $\langle \mathcal{O}_{A_{1}} \cdots \mathcal{O}_{A_{n}} \rangle_{\beta}$ is simply the path integral (\ref{eqn:Amodddcorrrrfuncc}) restricted to configurations satisfying $\phi_{*}[C] = \beta$.  

There is a powerful localization principle in supersymmetric physics: given a fermionic symmetry $Q$, the path integral defining any correlation function of $Q$-invariant operators localizes onto field configurations such that all variations of fermionic fields vanish.  It turns out that both terms in $\delta \psi_{\pm}$ (\ref{eqn:Amodvarrrrs}) must vanish individually.  So in particular, 
\begin{equation}
\partial_{\bar{z}} \phi^{i} = \partial_{z} \phi^{\bar{i}} =0.   
\end{equation} 
This is simply the local statement that $\phi: C \to X$ is holomorphic.  In this context, we refer to such a holomorphic map as a \emph{worldsheet instanton} and say that the A-model localizes onto worldsheet instantons.\footnote{The name worldsheet instanton arises because the entire worldsheet lies in the target space $X$, and from the point of view of four-dimensional spacetime, appears as a configuration localized in time and space -- an instanton.}  If $X$ is compact in addition to K\"{a}hler, then a worldsheet instanton with $\phi_{*}[C] = \beta$ wraps a minimal volume cycle in $X$ amongst all representatives of the homology class $\beta$.    

One typically writes $\#( \chi \, \text{zero modes})$ for the dimension of the solution space of the differential equation 
\[D_{\bar{z}}\chi^{i} = D_{z} \chi^{\bar{i}}=0.\]  
Similarly, one writes $\#(\psi \, \text{zero modes})$ for the dimension of the solution space of $D_{\bar{z}} \psi_{+}^{\bar{i}} = D_{z} \psi_{-}^{i}=0$.  Individually, these two integers are hard to compute, but their difference can be computed as the index of the differential operator $D$.  It turns out that for a worldsheet instanton $\phi: C \to X$ where $C$ is a smooth genus $g$ curve and $\phi_{*}[C] = \beta$, we have
\begin{equation}\label{eqn:chizeropsizeroo}
\begin{split}
\#( \chi \, \text{zero modes}) - \#(\psi \,  \text{zero modes}) & = \chi(C, \phi^{*}T_{X}^{(1,0)})\\
& = \int_{\beta} c_{1}(T_{X}^{(1,0)}) + \text{dim}X(1-g).  
\end{split}
\end{equation}
A phenomenon called the \emph{ghost number anomaly} imposes a powerful constraint relating the ghost number of operators in a correlation function with the index (\ref{eqn:chizeropsizeroo}).  The quantity $\langle \mathcal{O}_{A_{1}} \cdots \mathcal{O}_{A_{n}} \rangle_{\beta}$ has total ghost number $\sum_{i=1}^{n} \text{deg}(A_{i})$ and vanishes unless we have
\begin{equation}\label{eqn:ghostnummmconddd}
\frac{1}{2} \sum_{i=1}^{n} \text{deg}(A_{i}) = \int_{\beta} c_{1}(T_{X}^{(1,0)}) + \text{dim}X(1-g).  
\end{equation}

Let $M_{C}(X, \beta)$ denote the moduli space of holomorphic maps $\phi : C \to X$ with \emph{fixed} curve $C$ such that $\phi_{*}[C] = \beta$.  Having localized to worldsheet instantons, we will find the correlation functions $\langle \mathcal{O}_{A_{1}} \cdots \mathcal{O}_{A_{n}} \rangle_{\beta}$ satisfying (\ref{eqn:ghostnummmconddd}) will be given as integrals over $M_{C}(X, \beta)$, at least when the following technical assumption holds.  If there are no $\psi$ zero modes, then the moduli space $M_{C}(X, \beta)$ is smooth and
\begin{equation}
\text{dim} \, M_{C}(X, \beta) = \#( \chi \, \text{zero modes}) =  \int_{\beta} c_{1}(T_{X}^{(1,0)}) + \text{dim}X(1-g).  
\end{equation}
In such a case, assuming (\ref{eqn:ghostnummmconddd}) to hold, the correlation functions can be computed as
\begin{equation}
\langle \mathcal{O}_{A_{1}} \cdots \mathcal{O}_{A_{n}} \rangle_{\beta} = e^{i \omega \cdot \beta} \int_{M_{C}(X, \beta)} \text{ev}_{1}^{*}(A_{1}) \smile \cdots \smile \text{ev}_{n}^{*}(A_{n})
\end{equation}
where $\text{ev}_{i}$ are the evaluation maps, and $\omega = B + iJ$ is the complexified K\"{a}hler class on $X$.

\subsubsection{Coupling to Worldsheet Gravity}  

What we have surveyed above is an overview of the A-model topological field theory, which is formulated as a topological twist of a non-linear sigma model on a fixed smooth curve $C$.  It should be clear that we are approaching Gromov-Witten theory.  The correlation functions in the A-model are reminiscient of the Gromov-Witten invariants, and the ghost number anomaly imposes the mathematical constraint that the degrees of the insertions are compatible with the expected dimension of the moduli space.  But a theory of quantum gravity should necessarily involve dynamical fluctuations of the metric on the domain of the sigma model.  This is known as coupling a topological field theory to \emph{worldsheet gravity}, resulting in a topological string theory.  Physically, we are only interested in Riemannian metrics on $C$ up to conformal equivalence, which is simply the data of a complex structure on $C$, since the dimension is one.  It is this topological string theory which coincides with Gromov-Witten theory.  

Schematically speaking, the correlation functions in topological string theory are integrals over $\overline{M}_{g,n}$ of correlation functions in the topological \emph{field} theory valued in a top-form on $M_{g,n}$
\begin{equation}\label{eqn:AMODCORRFUNC}
\int_{\overline{M}_{g,n}} \big\langle \prod_{i=1}^{n} \tau_{a_{i}}(\mathcal{O}_{A_{i}}) \prod_{j=1}^{3g-3+n} | (G, \mu_{j})|^{2} \big\rangle.
\end{equation}
Let us briefly define some of these components.  The $\mu_{j}$ are called Beltrami differentials, and they are valued in $H^{1}(C, T_{C}(-D))$ which we identify as the tangent space to $\overline{M}_{g,n}$ at $(C, p_{1}, \ldots, p_{n})$.  Recall the notation $D = \sum_{i=1}^{n}p_{i}$ for the sum of marked points.  The operator $G$ is defined to be the $Q$-variation of the energy-momentum tensor $G(z) = \{ Q, T(z)\}$, and we define the operator-valued one-forms on $\overline{M}_{g,n}$ given by $(G, \mu_{j}) \coloneqq \int_{C} G(z) \mu_{j}(z, \bar{z})dz d\bar{z}$.  

The operators $\tau_{a_{i}}(\mathcal{O}_{A_{i}})$ are known as gravitational descendants and have ghost number $\text{deg}(A_{i}) + a_{i}$.  Since the brackets above denote a correlation function in the topological field theory, the total ghost number must satisfy the constraint (\ref{eqn:ghostnummmconddd}).  It turns out that $G$ has ghost number $-1$, so we have
\begin{equation}
\frac{1}{2} \sum_{i=1}^{n} \big( \text{deg}(A_{i}) + a_{i}\big) -3g + 3 -n = \int_{\beta} c_{1}(T_{X}^{(1,0)}) + \text{dim}X(1-g).
\end{equation}  
Moving the $-3g+3-n$ to the other side, we recover precisely the expected dimension (\ref{eqn:virDIMexpll}) of Gromov-Witten theory.  Moreover, the correlation function (\ref{eqn:AMODCORRFUNC}) is exactly the descendent Gromov-Witten invariant presented in (\ref{eqn:gravdesccccinv}).  In the same manner one can construct the topological string free energy and partition function without insertions $\tau_{a_{i}}(\mathcal{O}_{A_{i}})$ on a Calabi-Yau threefold, and they will agree with their counterparts in Gromov-Witten theory.

\vskip4ex
\subsubsection{Applications of Gromov-Witten Theory to Physical String Theory}

At the face of it, topological string theory is quite physically unrealistic.  To describe one oddity, the worldsheet (which should model closed strings propagating in time) is sitting in the Calabi-Yau fibers, while the time direction lies in $\mathbb{R}^{1,3}$.  Nonetheless, the topological toy theory is attractive for several reasons.  One advantage is its mathematical tractability, but another is that it \emph{computes} certain quantities in a full physical superstring theory.  In certain models, the Gromov-Witten potentials contribute to observable effects in four-dimensional physics.  

One can compactify the Type IIA superstring theory on a Calabi-Yau threefold $X$.  The ten-dimensional spacetime is then written as $\mathbb{R}^{1,3} \times X$.  If $X$ is a generic Calabi-Yau threefold, the Type IIA theory famously gives rise to an effective $\mathcal{N}=2$ theory on $\mathbb{R}^{1,3}$.  It turns out that there are so-called \emph{F-terms} in the effective four-dimensional action of the form
\begin{equation}
\int d^{4}x  F_{g}(Q) R_{+}^{2} F_{+}^{2g-2}
\end{equation}
where the integral is over $\mathbb{R}^{1,3}$, $R_{+}$ is the self-dual part of the Riemann tensor, and $F_{+}$ is the self-dual part of the field strength of what is called the graviphoton field \cite{antoniadis_topological_1993, bershadsky_holomorphic_1993, gopakumar_m-theory_1998}.  We interpret this to mean that the Gromov-Witten potential $F_{g}(Q)$ provides worldsheet instanton corrections to the effective theory in four dimensions.

\section{Donaldson-Thomas Theory} \label{eqn:DTThhhh}

Recall that Gromov-Witten theory grew out of one possible compactification of the stack $\mathcal{C}(X, \beta)$ of smooth embedded curves in $X$ lying in the class $\beta$.  Instead of parameterizations, one can study curves as pure one-dimensional subschemes $(C, \mathcal{O}_{C})$ of $X$.  If $C$ is smooth, reduced, and connected, we have
\[\chi(X, \mathcal{O}_{C})=1-g.\]
One might therefore expect the holomorphic Euler characteristic of $\mathcal{O}_{C}$ to replace the genus in Gromov-Witten theory.  Taking the support of $\mathcal{O}_{C}$, we also get a homology class $[C] \in H_{2}(X, \mathbb{Z})$.  

The goal is to find a compactification of smooth curves in $X$ with fixed homology class and holomorphic Euler characteristic.  In order to construct a compact moduli space, one must account for arbitrary degenerations of the objects.  We must allow the curves to be singular and non-reduced.  Moreover, it turns out that under certain degenerations a pure one-dimensional subscheme may acquire zero-dimensional components.  The compact moduli space is the Hilbert scheme $\text{Hilb}_{\beta, n}(X)$ of one-dimensional subschemes with fixed discrete invariants, defined below.  

If $X$ is a threefold, the Hilbert scheme $\text{Hilb}_{\beta, n}(X)$ is isomorphic to the moduli space $M_{\beta, n}(X)$ of coherent sheaves $\mathscr{I}$ with trivial determinant and Chern character $\text{ch}(\mathscr{I}) = (1,0, -\beta, -n-\tfrac{1}{2} K_{X} \cdot \beta)$.  Despite the isomorphism as classical schemes, $\text{Hilb}_{\beta, n}(X)$ and $M_{\beta, n}(X)$ admit very different deformation-obstruction theories.  We will see that for Calabi-Yau threefolds, $M_{\beta, n}(X)$ carries a \emph{symmetric obstruction theory} and a zero-dimensional virtual class.  Using this structure, one can compute Donaldson-Thomas invariants as Behrend-weighted Euler characteristics and package them into partition functions.  The GW/DT correspondence says that under a non-trivial change of variables, the reduced Gromov-Witten and Donaldson-Thomas partition functions are equal.  

In physics, Donaldson-Thomas invariants are an example of a \emph{supersymmetric index}.  In Type IIA string theory on a Calabi-Yau threefold $X$ in the large volume limit, the invariants compute a virtual number of BPS states of particles arising from bound states of D6-D2-D0 branes in $X$.  Using known relations between IIA and IIB, we will sketch an argument interpreting the GW/DT correspondence as a consequence of the S-duality of Type IIB string theory.  We will also realize the Donaldson-Thomas invariants as quantities in the B-model topological string theory on $X$.

\subsection{From the Hilbert Scheme to Ideal Sheaves}

Let $X$ be a non-singular projective variety.  We begin by introducing the \emph{Hilbert scheme} $\text{Hilb}_{\beta, n}(X)$ of points and curves in $X$.  For all $\beta \in H_{2}(X, \mathbb{Z})$ and $n \in \mathbb{Z}$ we have
\begin{equation}
\text{Hilb}_{\beta, n}(X) = \bigg\{ Z \subset X \,\, \text{one-dimensional subscheme} \,\, \bigg| \,\,  [Z] = \beta, \,\, \chi(X, \mathcal{O}_{Z})=n \bigg\}.
\end{equation}
This Hilbert scheme is a classically known space parameterizing subschemes of $X$ consisting of unions of possibly singular and non-reduced curves, as well as a finite number of points.  Its nice features are that it is a projective scheme, and in fact a fine moduli space parameterizing flat families of one-dimensional subschemes with the fixed discrete invariants.  The Hilbert scheme is badly behaved in that it can be non-reduced and it may contain many irreducible components of arbitrary dimension with bad singularities.  We can however, give the following two examples where the Hilbert scheme is quite nice and easy to describe.  

\begin{Ex}\label{Ex:HilbSchPointsEX}
The most immediate example of the Hilbert scheme comes by setting $\beta =0$.  This gives the Hilbert scheme of points which we denote by $\text{Hilb}^{n}(X)$, changing the notation slightly.  Because $\beta=0$, there are no curve components allowed, which means $\text{Hilb}^{n}(X)$ parameterizes zero-dimensional subschemes of $X$ of length $n$.  In other words,
\begin{equation}
\text{Hilb}^{n}(X) = \bigg\{ Z \subset X \,\, \text{zero-dimensional subscheme} \,\, \bigg| \,\,  \chi(X, \mathcal{O}_{Z})=n \bigg\}.
\end{equation}
Generically, a subscheme $Z \in \text{Hilb}^{n}(X)$ consists of $n$ distinct, reduced points in $X$.  When points come together, data is not lost but rather encoded as a scheme-theoretic thickening, which preserves the length of the subscheme.  

If $X$ is a smooth projective curve or surface, then $\text{Hilb}^{n}(X)$ is a smooth projective variety of dimension $n$ when $X$ is a curve, and $2n$ when $X$ is a surface.  When $X$ is a surface, $\text{Hilb}^{n}(X)$ is a crepant resolution of singularities of $\text{Sym}^{n}(X)$.  In higher dimensions, the Hilbert scheme of points will typically be singular.  
\end{Ex}

\begin{Ex}
Let $X$ be a non-singular projective surface.  In this case, the one-dimensional components of a subscheme in the Hilbert scheme $\text{Hilb}_{\beta, n}(X)$ are simply divisors.  We therefore get a factorization 
\begin{equation}
\text{Hilb}_{\beta, n}(X) \cong \text{Div}_{\beta}(X) \times \text{Hilb}^{n - n_{\beta}}(X)
\end{equation}
where $\text{Div}_{\beta}(X) = \text{Hilb}_{\beta, n_{\beta}}(X)$ is the smooth projective moduli space of divisors in $X$ with class $\beta$, and $\text{Hilb}^{n - n_{\beta}}(X)$ is the Hilbert scheme of points introduced above.  Here, $n_{\beta}$ is the intersection number
\begin{equation}
n_{\beta} = - \frac{1}{2} \beta \cdot (K_{X} + \beta)
\end{equation}
which gives the contribution of a divisor to the holomorphic Euler characteristic.  If we require the total Euler characteristic to be $n$, then the difference $n - n_{\beta}$ must come from zero-dimensional subschemes in $X$.  In this example, $\text{Hilb}_{\beta, n}(X)$ is a smooth projective variety.
\end{Ex}

For multiple reasons to be illustrated below, it is desirable to consider, instead of subschemes, a certain class of ideal sheaves.  Recall from Definition \ref{defn:dettofcohshref} how to define the determinant of a coherent sheaf.  Let $\mathscr{I}$ be a torsion-free rank one coherent sheaf on a non-singular projective variety $X$.  In such a case, we get the following embedding
\begin{equation}
\mathscr{I} \lhook\joinrel\longrightarrow \mathscr{I}^{\smvee \smvee} \cong \text{det}(\mathscr{I}) \in \text{Pic}(X)
\end{equation}
where the assumption of torsion-free guarantees an embedding $\mathscr{I} \hookrightarrow \mathscr{I}^{\smvee \smvee}$, and the assumption of rank one further implies that $\text{det}(\mathscr{I}) \cong \mathscr{I}^{\smvee \smvee}$.  Therefore, \emph{torsion-free, rank one, and trivial determinant} are sufficient conditions to guarantee an embedding
\begin{equation}
\mathscr{I} \lhook\joinrel\longrightarrow \mathcal{O}_{X}.
\end{equation}
This establishes that such an $\mathscr{I}$ is an ideal sheaf in the sense familiar to algebraic geometers, and reviewed in Section \ref{subsec:idsheee}.  Note that not every ideal sheaf has trivial determinant.  For example, if $D \subset X$ is a divisor, $L = \mathcal{O}_{X}(-D)$ is an ideal sheaf and satisfies $\text{det}(L) = L$.  Effectively, the hypothesis of trivial determinant produces ideal sheaves $\mathscr{I}_{Z}$ whose corresponding subscheme $Z$ has only support in codimension two and above.

If we specialize to $\text{dim}(X)=3$, and let $\mathscr{I}$ be as above, the condition $\text{det}(\mathscr{I}) \cong \mathcal{O}_{X}$ implies that the Chern character of $\mathscr{I}$ can be expressed as 
\[ \text{ch}(\mathscr{I}) = \big(1,0,-\beta, -n-\tfrac{1}{2}K_{X} \cdot \beta \big)\]
for some effective class $\beta \in H_{2}(X, \mathbb{Z})$ and $n \in \mathbb{Z}$.  Let us denote by $M_{\beta, n}(X)$ the moduli space of sheaves with trivial determinant and Chern character $\big(1,0,-\beta, -n-\tfrac{1}{2}K_{X} \cdot \beta \big)$.  We have not mentioned stability or chosen an ample class because a rank one torsion-free sheaf is automatically (slope, therefore Gieseker) stable, independent of a polarization.  There are also no strictly semistable sheaves which implies that $M_{\beta, n}(X)$ is a projective scheme.  

For a threefold $X$, it is tempting to identify $M_{\beta, n}(X)$ with the scheme we began the discussion with: the Hilbert scheme $\text{Hilb}_{\beta, n}(X)$.  At the level of sets it is easy to establish a bijection.  Recalling the ideal sheaf exact sequence (\ref{eqn:idshexseq}) given an ideal sheaf $\mathscr{I}$ with trivial determinant, we get a uniquely determined subscheme $Z \subset X$ supported on points and curves.  By the additivity of the Chern character on short exact sequences, we know
\[\text{ch}(\mathscr{I}) + \text{ch}(\mathcal{O}_{Z}) = \text{ch}(\mathcal{O}_{X}) = (1,0, 0, 0).\]
Therefore, if $\mathscr{I} \in M_{\beta, n}(X)$, then $\text{ch}(\mathcal{O}_{Z}) = \big(0,0, \beta, n + \tfrac{1}{2}K_{X} \cdot \beta \big)$ and we conclude $Z \in \text{Hilb}_{\beta, n}(X)$.  We see that it is elementary to give a bijection between $M_{\beta, n}(X)$ and $\text{Hilb}_{\beta, n}(X)$.  The fact that the two are isomorphic as schemes is a very deep and non-trivial result.  

\begin{thm}[\bfseries {\cite[Thm 2.7]{pandharipande_curve_2009}}]
Let $X$ be a non-singular projective threefold.  We have the following isomorphism\footnote{If we had defined $M_{\beta, n}(X)$ to parameterize sheaves with Chern character $\big(1,0,-\beta, -n-\tfrac{1}{2}K_{X} \cdot \beta \big)$ without requiring trivial determinant, then we could have contributions from degree zero line bundles as well.  To get an isomorphism with the Hilbert scheme one can either impose trivial determinant as we did, or require $H^{1}(X, \mathcal{O}_{X})=0$, which is equivalent to $\text{Pic}^{0}(X)=0$.} of projective schemes
\begin{equation}\label{eqn:HilbIdShIso}
\text{Hilb}_{\beta, n}(X) \cong M_{\beta, n}(X)
\end{equation} 
between the Hilbert scheme $\text{Hilb}_{\beta, n}(X)$ and the moduli space $M_{\beta, n}(X)$ of coherent sheaves on $X$ with trivial determinant and Chern character $\big(1,0, -\beta, -n -\tfrac{1}{2}K_{X} \cdot \beta \big)$.  
\end{thm}

Despite the isomorphism (\ref{eqn:HilbIdShIso}), the Hilbert scheme and moduli space of sheaves carry very different deformation-obstruction theories.  We will comment more in the next section, but the deformation space is in general \emph{intrinsic} to a moduli scheme: it corresponds to the tangent space at a point.  Because $\text{Hilb}_{\beta, n}(X)$ and $M_{\beta, n}(X)$ are isomorphic as schemes, they have the same deformation spaces but different obstruction spaces.  One should think that $\text{Hilb}_{\beta, n}(X) \cong M_{\beta, n}(X)$ is an isomorphism of \emph{classical schemes}, but the two have very different derived structures.  We far prefer to work with sheaves partly because the deformation-obstruction theory is canonical and leads to a virtual class.  In particular, if $X$ is not only a threefold but also Calabi-Yau, we get a \emph{symmetric} obstruction theory.  Let us now turn to a discussion of this.

\subsection{Symmetric Obstruction Theories and the Virtual Fundamental Class}

Motivated by the discussion in the previous section, we would like to understand the canonical deformation-obstruction theory on $M_{\beta, n}(X)$ for a smooth projective threefold $X$.  But to lay some foundation, we begin in more generality.  We will let $X$ be a smooth projective scheme with $M$ a moduli space of stable or semistable sheaves on $X$ with fixed discrete invariants.  If $X$ is a smooth variety, we will also be interested in the subscheme $M(\mathcal{Q})$ of $M$ parameterizing sheaves with fixed determinant $\mathcal{Q} \in \text{Pic}(X)$.  

As we saw with $M_{\beta, n}(X)$, the moduli space $M$ or $M(\mathcal{Q})$ may not have a well-defined dimension: there may be many components of arbitrary dimensions.  But by the following theorem, the dimension of a moduli space of sheaves makes sense locally in terms of the Zariski tangent space, which can be identified with a familiar quantity.  A proof may be found in \cite[Thm. 2.7]{hartshorne2009deformation}.

\begin{thm}\label{thm:genmodZar}
Let $X$ be a projective scheme, and $M$ a moduli space of stable or semistable sheaves on $X$.  If $\mathscr{I} \in M$ is a stable moduli point, then the Zariski tangent space at $\mathscr{I}$ is given by
\begin{equation}
T_{\mathscr{I}}M \cong \text{Ext}^{1}(\mathscr{I}, \mathscr{I})
\end{equation}
and the canonical obstruction space for $M$ is $\text{Ext}^{2}(\mathscr{I}, \mathscr{I})$.  If $\text{Ext}^{2}(\mathscr{I}, \mathscr{I}) =0$, then $M$ is smooth at $\mathscr{I}$.  
\end{thm}
We think of elements of $\text{Ext}^{1}(\mathscr{I}, \mathscr{I})$ as \emph{infinitesimal deformations} of the sheaf $\mathscr{I}$.  There are \emph{obstructions} to lifting an infinitesimal deformation to infinite order, which live in the obstruction space $\text{Ext}^{2}(\mathscr{I}, \mathscr{I})$.  By a general result, there is a formal function called the \emph{Kuranishi map}
\begin{equation}
\kappa : \text{Ext}^{1}(\mathscr{I}, \mathscr{I}) \longrightarrow \text{Ext}^{2}(\mathscr{I}, \mathscr{I})
\end{equation}
such that the moduli space $M$ is locally the zero locus (not critical locus!) of $\kappa$.  The general principle is that a moduli space $M$ is smooth if and only if all infinitesimal deformations are unobstructed.  Therefore, as we saw in Theorem \ref{thm:genmodZar} if $\text{Ext}^{2}(\mathscr{I}, \mathscr{I})$ vanishes, then $M$ is smooth, but the converse is not true.  We can have $\text{Ext}^{2}(\mathscr{I}, \mathscr{I}) \neq 0$, but if $\kappa =0$ then $M$ will be smooth.

The above theorem is a generalization of a well-known fact about deformations of vector bundles.  Given a vector bundle $E$ on $X$, one can explicitly show that a first order deformation in the transition functions corresponds to an element of $H^{1}(X, \text{End}E)$.  When $E$ is a vector bundle, it is indeed true that
\[H^{1}(X, \text{End}E) \cong H^{1}(X, E^{\smvee} \otimes E) \cong \text{Ext}^{1}(E, E).\]
One can also verify that obstructions to lifting an infinitesimal deformation to infinite order live in the obstruction space $\text{Ext}^{2}(E, E) = H^{2}(X, \text{End}E)$.  
 
Recall that we are ultimately interested in $M_{\beta, n}(X)$ which is a moduli space of sheaves with fixed determinant.  In such a case, the deformation and obstruction spaces will be slightly different.  On a smooth projective scheme $X$, the Picard group $\text{Pic}(X)$ parameterizing isomorphism classes of invertible sheaves on $X$, is smooth.  If $X$ is in fact a smooth variety, given any coherent sheaf $\mathscr{I}$, we get a finite locally-free resolution $\dotr{E} \to \mathscr{I}$.  Therefore, given any moduli space $M$ of sheaves on $X$, we define a determinant map
\begin{equation}\label{eqn:detmappp}
\text{det}: M \longrightarrow \text{Pic}(X)
\end{equation}
by sending $\mathscr{I} \in M$ to $\text{det}(\mathscr{I}) \coloneqq \bigotimes_{i} \text{det}(E_{i})^{(-1)^{i}}$.  Given $M$, we can study the subscheme $M(\mathcal{Q})$ parameterizing sheaves with fixed determinant $\mathcal{Q} \in \text{Pic}(X)$, by way of the map (\ref{eqn:detmappp}).  

Given a sheaf $\mathscr{I}$ one can define trace maps \cite[Section 10.1]{huybrechts_geometry_2010}
\begin{equation}
\text{tr}^{i} : \text{Ext}^{i}(\mathscr{I}, \mathscr{I}) \longrightarrow H^{i}(X, \mathcal{O}_{X}).
\end{equation}  
This is once again a generalization of something which is straightforward in the case of a vector bundle $E$.  In such a case, $\text{Ext}^{i}(E, E) = H^{i}(X, \text{End}E)$ and the trace maps $\text{tr}^{i}$ are the morphisms on cohomology induced from the honest trace map $\text{End}E \to \mathcal{O}_{X}$.  

\begin{defn}
The traceless Ext group denoted $\text{Ext}^{i}(\mathscr{I}, \mathscr{I})_{0}$ is the kernel of the map $\text{tr}^{i}$.  
\end{defn}

Recall from Theorem \ref{thm:genmodZar}, given a stable moduli point $\mathscr{I} \in M$, the Zariski tangent space at $\mathscr{I}$ is given by $T_{\mathscr{I}}M \cong \text{Ext}^{1}(\mathscr{I}, \mathscr{I})$.  It follows that the tangent space to $\text{Pic}(X)$ at any point is $H^{1}(X, \mathcal{O}_{X})$.  But let us give an additional proof of this.  Because $X$ is compact and K\"{a}hler, the exponential sequence implies
\begin{equation}
0 \longrightarrow H^{1}(X, \mathcal{O}_{X})/H^{1}(X, \mathbb{Z}) \longrightarrow \text{Pic}(X)  \xrightarrow{\hspace*{0.1cm} c_{1} \hspace*{0.1cm}} \text{NS}(X) \longrightarrow 0
\end{equation}
where $\text{NS}(X) \coloneqq H^{2}(X, \mathbb{Z}) \cap H^{1,1}(X)$ is the Neron-Severi group and the first Chern class $c_{1}$ classifies line bundles on $X$ topologically.  It follows that for any fixed topological type, $H^{1}(X, \mathcal{O}_{X})/H^{1}(X, \mathbb{Z})$ is the moduli space of algebraic structures.  We think of $\text{NS}(X)$ as the discrete part of the Picard group and $H^{1}(X, \mathcal{O}_{X})/H^{1}(X, \mathbb{Z})$ as the continuous part.  Because the tangent space is a local object, it is insensitive to the discrete components of $\text{Pic}(X)$ as well as the lattice $H^{1}(X, \mathbb{Z})$.  Therefore, for all $L \in \text{Pic}(X)$
\begin{equation}
T_{L}\text{Pic}(X) \cong H^{1}(X, \mathcal{O}_{X}).
\end{equation}
With this in mind, it is tempting to expect that the map $\text{tr}^{1}$ is simply the differential, or pushforward, of the determinant map (\ref{eqn:detmappp}).  For all stable points $\mathscr{I} \in M$, this is indeed the case:
\begin{equation}\label{eqn:tr1differential}
\text{tr}^{1} : \text{Ext}^{1}(\mathscr{I}, \mathscr{I}) \cong T_{\mathscr{I}} M \longrightarrow T_{\text{det}(\mathscr{I})}\text{Pic}(X) \cong H^{1}(X, \mathcal{O}_{X}).
\end{equation}

\begin{thm}
Let $M$ be a moduli space of stable or semistable torsion-free sheaves on a smooth projective variety $X$, and let $M(\mathcal{Q})$ be the subscheme of such sheaves with fixed determinant $\mathcal{Q} \in \text{Pic}(X)$.  The Zariski tangent space to $M(\mathcal{Q})$ at a stable point $\mathscr{I} \in M(\mathcal{Q})$ is the kernel of the map (\ref{eqn:tr1differential}).  That is
\begin{equation}
T_{\mathscr{I}} M(\mathcal{Q}) = \text{Ext}^{1}(\mathscr{I}, \mathscr{I})_{0}.
\end{equation}
Moreover, the canonical obstruction space for $M(\mathcal{Q})$ is $\text{Ext}^{2}(\mathscr{I}, \mathscr{I})_{0}$, and $M(\mathcal{Q})$ is therefore smooth when $\text{Ext}^{2}(\mathscr{I}, \mathscr{I})_{0} =0$.  
\end{thm}

\subsubsection{The Canonical Perfect Obstruction Theory}

What we have outlined above are some generalities on the canonical deformation-obstruction theory carried by a moduli space of sheaves with and without fixed determinant.  It was shown by Richard Thomas \cite{p._thomas_holomorphic_2000} that in many cases, this canonical deformation-obstruction theory is in fact a \emph{perfect obstruction theory}, similar to that in Gromov-Witten theory.  The following theorem is paraphrased from \cite{p._thomas_holomorphic_2000}.   

\begin{thm}\label{thm:ThomasTHM}
Let $X$ be a smooth projective polarized variety with $M$ a moduli space of stable sheaves with fixed discrete invariants and $M(\mathcal{Q})$ the subscheme with fixed determinant $\mathcal{Q} \in \text{Pic}(X)$.  If the integers
\[ \text{dim} \, \text{Ext}^{i}(\mathscr{I}, \mathscr{I}), \,\,\,\,\,\,\,\,\,\,\,\, (i \geq 3),\]
are independent of $\mathscr{I} \in M$, then there is a canonical perfect obstruction theory on $M$ governed by $\text{Ext}^{1}(\mathscr{I}, \mathscr{I})$ and $\text{Ext}^{2}(\mathscr{I}, \mathscr{I})$.  For non-zero rank, if the same condition holds then $M(\mathcal{Q})$ carries a canonical perfect obstruction theory governed by $\text{Ext}^{1}(\mathscr{I}, \mathscr{I})_{0}$ and $\text{Ext}^{2}(\mathscr{I}, \mathscr{I})_{0}$.
\end{thm}

Let us now specialize to the case of $X$ a threefold.  Clearly, $\text{dim} \, \text{Ext}^{i}(\mathscr{I}, \mathscr{I})=0$ for all $i >3$, so to apply Theorem \ref{thm:ThomasTHM} to a given moduli space, the only technical condition to check is that $\text{dim} \, \text{Ext}^{3}(\mathscr{I}, \mathscr{I})$ is independent of $\mathscr{I}$.  In fact, we will be primarily interested in a moduli space with fixed determinant, which means we will use the traceless Ext groups.  We will need the following lemma.  

\begin{lemmy} \label{lemmy:threefoldExtVan}
If $X$ is a non-singular projective threefold, with $\mathscr{I}$ a rank one torsion-free sheaf, then $\text{dim} \, \text{Ext}^{3}(\mathscr{I}, \mathscr{I})$ is independent of $\mathscr{I}$, and $\text{dim} \, \text{Ext}^{0}(\mathscr{I}, \mathscr{I})_{0}, \text{dim} \, \text{Ext}^{3}(\mathscr{I}, \mathscr{I})_{0}$ both vanish.
\end{lemmy}

\begin{proof}
To prove the first claim, we can make use of Serre duality to write,
\begin{equation}
\text{Ext}^{3}(\mathscr{I}, \mathscr{I}) \cong \text{Ext}^{0}(\mathscr{I}, \mathscr{I} \otimes K_{X})^{\smvee} \cong \text{Hom}(\mathscr{I}, \mathscr{I} \otimes K_{X})^{\smvee}.
\end{equation}
Because $\mathscr{I}$ is rank one and torsion-free, we have $\text{Hom}(\mathscr{I}, \mathscr{I} \otimes K_{X})^{\smvee} \cong \text{Hom}(\mathcal{O}_{X}, K_{X})^{\smvee}$.  Therefore, we see that $\text{dim} \, \text{Ext}^{3}(\mathscr{I}, \mathscr{I})$ is independent of $\mathscr{I}$.  To prove the second claim, we apply Serre duality once more to see $\text{Hom}( \mathcal{O}_{X}, K_{X})^{\smvee} \cong H^{3}(X, \mathcal{O}_{X})$, which implies that the trace map $\text{tr}^{3}$ is an isomorphism, and $\text{dim} \, \text{Ext}^{3}(\mathscr{I}, \mathscr{I})_{0}=0$.  By the simplicity of $\mathscr{I}$, we know $\text{Ext}^{0}(\mathscr{I}, \mathscr{I}) = \text{Hom}(\mathscr{I}, \mathscr{I}) = \mathbb{C}$.  Hence, the trace map $\text{tr}^{0}$ from $\mathbb{C}$ to $H^{0}(X, \mathcal{O}_{X}) = \mathbb{C}$ is simply the identity.  The kernel is therefore trivial, and $\text{dim} \, \text{Ext}^{0}(\mathscr{I}, \mathscr{I})_{0}=0$.  
\end{proof}

\noindent Making use of this lemma, we regard the following result as a corollary of Theorem \ref{thm:ThomasTHM} above.

\begin{cory}
Let $X$ be a non-singular projective threefold and let $M_{\beta, n}(X)$ be the moduli space of sheaves on $X$ with trivial determinant and Chern character $\big(1,0, -\beta, -n -\tfrac{1}{2}K_{X} \cdot \beta \big)$.  There is a canonical perfect obstruction theory on $M_{\beta, n}(X)$ governed by $\text{Ext}^{1}(\mathscr{I}, \mathscr{I})_{0}$ and $\text{Ext}^{2}(\mathscr{I}, \mathscr{I})_{0}$ with expected dimension
\begin{equation}\label{eqn:virtdimDT}
\text{vdim}\big(M_{\beta, n}(X) \big) = \text{dim} \, \text{Ext}^{1}(\mathscr{I}, \mathscr{I})_{0} - \text{dim} \, \text{Ext}^{2}(\mathscr{I}, \mathscr{I})_{0} = \int_{\beta} c_{1}(X).
\end{equation}
Moreover, there exists a virtual class $[M_{\beta, n}(X)]^{\text{vir}}$ living in the Chow ring $A_{*}\big(M_{\beta, n}(X)\big)$, with degree equal to the virtual dimension.  
\end{cory}

\noindent Notice that individually, $\text{dim} \, \text{Ext}^{1}(\mathscr{I}, \mathscr{I})_{0}$ and $\text{dim} \, \text{Ext}^{1}(\mathscr{I}, \mathscr{I})_{0}$ depend on the sheaf $\mathscr{I}$, but their difference (\ref{eqn:virtdimDT}) does not.  The virtual dimension is an invariant of the moduli space $M_{\beta, n}(X)$.

\subsubsection{The Symmetric Obstruction Theory for $\bm{X}$ a Calabi-Yau Threefold}

When $X$ is a smooth projective Calabi-Yau threefold, the story enriches nicely.  Built into the Calabi-Yau assumption, one often assumes $H^{1}(X, \mathcal{O}_{X})=0$ which implies that $H^{2}(X, \mathcal{O}_{X})=0$, by Serre duality.  It follows that the groups $\text{Ext}^{1}(\mathscr{I}, \mathscr{I})$ and $\text{Ext}^{2}(\mathscr{I}, \mathscr{I})$ coincide with their traceless counterparts.  By Lemma \ref{lemmy:threefoldExtVan}, it remains true that $\text{Ext}^{0}(\mathscr{I}, \mathscr{I})_{0}$ and $\text{Ext}^{3}(\mathscr{I}, \mathscr{I})_{0}$ both vanish for a simple sheaf $\mathscr{I}$.  Most importantly, by Serre duality we have
\begin{equation}
\text{Ext}^{1}(\mathscr{I}, \mathscr{I}) \cong \text{Ext}^{2}(\mathscr{I}, \mathscr{I})^{\smvee}
\end{equation}
noting that $K_{X} \cong \mathcal{O}_{X}$.  Therefore, on a Calabi-Yau threefold, the deformations are dual to the obstructions.  This is what Behrend called a \emph{symmetric obstruction theory} \cite{behrend_donaldson-thomas_2009}.  By (\ref{eqn:virtdimDT}), it is clear that the expected dimension vanishes, in this case.    

Although there need not be a canonical linear map from $\text{Ext}^{1}(\mathscr{I},\mathscr{I})$ to $\text{Ext}^{2}(\mathscr{I},\mathscr{I})$, by a general result of Kuranishi, there exists a non-linear map called the \emph{Kuranishi map}
\begin{equation}
\kappa: \text{Ext}^{1}(\mathscr{I},\mathscr{I}) \longrightarrow \text{Ext}^{2}(\mathscr{I},\mathscr{I})
\end{equation}
which associates to an infinitesimal deformation of $\mathscr{I}$ the obstruction to lifting the deformation.  The moduli space is given locally by the zeros of the Kuranishi map, which corresponds to directions with unobstructed deformations.  If $X$ is a Calabi-Yau threefold, we have
\begin{equation}
\kappa: \text{Ext}^{1}(\mathscr{I},\mathscr{I}) \longrightarrow \text{Ext}^{1}(\mathscr{I},\mathscr{I})^{\smvee}
\end{equation}
and we can interpret $\kappa$ as a global section of the cotangent bundle of $\text{Ext}^{1}(\mathscr{I},\mathscr{I})$.  We know that a one-form on a vector space is exact, so there exists a function
\[f : \text{Ext}^{1}(\mathscr{I},\mathscr{I}) \longrightarrow \mathbb{C}\]
such that $\kappa = df$ and the moduli space is locally the critical locus of $f$.  In the case of Donaldson-Thomas theory, the function $f$ is called the \emph{holomorphic Chern-Simons functional} or the \emph{Chern-Simons superpotential}.  We can summarize these results in the following theorem.

\begin{thm}
For $X$ a smooth projective Calabi-Yau threefold, there is a canonical symmetric obstruction theory on $M_{\beta, n}(X)$ with virtual dimension zero, and a virtual fundamental class
\begin{equation}
[M_{\beta, n}(X)]^{\text{vir}} \in A_{0}\big(M_{\beta, n}(X)\big)
\end{equation}
giving a zero-cycle in the Chow ring.  In addition, $M_{\beta, n}(X)$ is locally the critical locus of the holomorphic Chern-Simons functional.  
\end{thm}

\subsection{The Donaldson-Thomas Invariants and Partition Function}

Given a proper scheme $Y$ carrying a perfect obstruction theory with expected dimension zero, we can apply the machinery of Behrend-Fantechi \cite{behrend_intrinsic_1997} to extract invariants.  We define a Donaldson-Thomas type invariant by the degree of the virtual class
\begin{equation}\label{eqn:DTinvschemeY}
\text{deg}[Y]^{\text{vir}} \coloneqq \int_{[Y]^{\text{vir}}} 1 \in \mathbb{Z}.
\end{equation}

\begin{defn}
In the case of the moduli space $M_{\beta, n}(X)$, the Donaldson-Thomas invariants denoted
\begin{equation}
\text{DT}_{\beta, n}(X) = \text{deg} [M_{\beta, n}(X)]^{\text{vir}} \in \mathbb{Z}
\end{equation}
are virtual counts of ideal sheaves with Chern character $(1, 0, -\beta, -n)$ on the smooth projective Calabi-Yau threefold $X$.
\end{defn}

Just as in Gromov-Witten theory, the Donaldson-Thomas invariants are deformation invariants of $X$.  They are however \emph{integral} invariants in contrast to Gromov-Witten theory.  This is because $M_{\beta, n}(X)$ is a scheme whereas the moduli space of stable maps is a Deligne-Mumford stack, and the non-integrality of $\text{GW}_{g, \beta}(X)$ arises from stable maps with finite, but non-vanishing automorphisms.  

The idea pioneered by Behrend is that if an obstruction theory on a proper scheme $Y$ is \emph{symmetric}, then the invariants should be computed as a weighted Euler characteristic \cite{behrend_donaldson-thomas_2009}.  In the simplest case, if $Y$ is smooth, the Donaldson-Thomas type invariants are given by integrating the Euler class of the obstruction bundle over $Y$.  By symmetry, the deformations are dual to the obstructions, so the obstruction bundle is simply the cotangent bundle $\Omega_{Y}$, and the invariants can be expressed by Chern-Gauss-Bonnet as
\begin{equation}
\text{deg}[Y]^{\text{vir}} = \int_{Y} e\big(\Omega_{Y}\big) = (-1)^{\text{dim}  Y} \chi(Y).  
\end{equation}
This formula was extended by Behrend to more general schemes.  Note that the symmetry of the obstruction theory is crucial, and no such formula exists in Gromov-Witten theory.

The fundamental quantity introduced in \cite{behrend_donaldson-thomas_2009} is a canonical constructible function $\nu_{Y}: Y \to \mathbb{Z}$ associated to any scheme $Y$ over $\mathbb{C}$, known today as the \emph{Behrend function}.  We will not provide an adequate discussion of the Behrend function, though we do mention that if $Y = Z(df)$ is the critical locus of a regular function $f$ on a smooth scheme $M$, then
\begin{equation}
\nu_{Y}(p) = (-1)^{\text{dim}  M} \big(1 - \chi(F_{p})\big)
\end{equation}
where $F_{p}$ is the Milnor fiber at $p \in Y$.  The importance of the Behrend function lies in the following result.  

\begin{thm}[\bfseries Behrend]
If $Y$ is a proper scheme over $\mathbb{C}$ carrying a symmetric obstruction theory, the Donaldson-Thomas type invariants (\ref{eqn:DTinvschemeY}) are expressed as
\begin{equation}
\text{deg}[Y]^{\text{vir}} = \chi(Y, \nu_{Y}) \coloneqq \sum_{k \in \mathbb{Z}} k \, \chi\big( \nu_{Y}^{-1}(k) \big)
\end{equation}
where $\chi(Y, \nu_{Y})$ is called the Behrend-weighted Euler characteristic of $Y$.  
\end{thm}

The Behrend function depends only on the scheme structure of $Y$, not on the obstruction theory.  It therefore follows from the above theorem that the Donaldson-Thomas type invariants do not depend on the particular symmetric obstruction theory $Y$ carries.  

If $Y$ is not proper, the invariant (\ref{eqn:DTinvschemeY}) is not well-defined.  One advantage of the Behrend-weighted Euler characteristic is that it makes sense for non-proper schemes as well, so one can \emph{define} a virtual count by $\chi(Y, \nu_{Y})$ in such a case.  

In the case of $M_{\beta, n}(X)$, we now know the Donaldson-Thomas invariants $\text{DT}_{\beta, n}(X)$ can be computed as a Behrend-weighted Euler characteristic.  We define the \emph{Donaldson-Thomas partition function} to be the generating function of these invariants
\begin{equation}
Z_{\text{DT}}(X) = \sum_{\beta \in H_{2}(X, \mathbb{Z})} Z_{\text{DT}}(X)_{\beta} Q^{\beta} = \sum_{\beta \in H_{2}(X, \mathbb{Z})} \sum_{n \in \mathbb{Z}} \text{DT}_{\beta, n}(X) Q^{\beta} q^{n}
\end{equation}
where $Z_{\text{DT}}(X)_{\beta}$ is the generating function of Donaldson-Thomas invariants with fixed class $\beta$, and the notation $Q^{\beta}$ was defined in (\ref{eqn:GWdegreetrack}).  Just as in Gromov-Witten theory, we can equivalently sum over only effective classes, because if $\beta$ is not effective, $\text{DT}_{\beta, n}(X)=0$.  

Unlike the asymptotic expression (\ref{eqn:AsymptExprGWW}) in Gromov-Witten theory, the degree zero Donaldson-Thomas partition function has been shown \cite{behrend_symmetric_2008, levine_algebraic_2006, li_zero_2006} to have the following exact form
\begin{equation}\label{eqn:DTdegggzeroo}
Z_{\text{DT}}(X)_{0} = M(-q)^{\chi(X)}
\end{equation}
where $M(q)$ is the MacMahon function defined in (\ref{eqn:McMahon}).  This was originally conjectured in \cite{maulik_gromovwitten_2006}.  The \emph{reduced} Donaldson-Thomas partition function is the full partition function divided by the degree zero contributions
\begin{equation}
Z'_{\text{DT}}(X) = \frac{Z_{\text{DT}}(X)}{Z_{\text{DT}}(X)_{0}}.
\end{equation}

\subsection{The GW/DT Correspondence}  

Recall that the relevant moduli spaces of Gromov-Witten and Donaldson-Thomas theory are different compactifications and stratifications of the moduli space $\mathcal{C}(X, \beta)$ of smooth curves in $X$ of class $\beta$.  Both invariants are therefore somehow related to counting curves in $X$, though the objects utilized are totally different in the two cases.  Gromov-Witten theory studies curves through parameterized maps, while Donaldson-Thomas theory uses structure sheaves of subschemes.  The degenerate contributions in Gromov-Witten theory come from multiple covers, while those in Donaldson-Thomas theory come from embedded points and non-reduced structure.  Nevertheless, a remarkable conjecture \cite{maulik_gromovwitten_2006} known as the \emph{GW/DT correspondence} is that the two reduced partition functions are closely related.  

\begin{conj}
On a smooth projective Calabi-Yau threefold $X$, the reduced Gromov-Witten partition function and reduced Donaldson-Thomas partition functions are equal
\begin{equation} \label{eqn:GW/DT-Corr}
Z'_{\text{DT}}(X; q, Q) = Z'_{\text{GW}}(X; \lambda, Q)
\end{equation}
under the change of variables $-q = e^{i \lambda}$.  Because both theories incorporate the homology class in the same way, the conjecture also stands if we restrict to a non-zero class $\beta$ or a lattice $\Gamma \subseteq H_{2}(X, \mathbb{Z})$.    
\end{conj}

One should be careful to note that this conjecture does \emph{not} provide a relationship between the full Gromov-Witten and Donaldson-Thomas partition functions, including degree zero contributions.  The conjecture has been proven by Pandharipande-Pixton \cite{pandharipande_gromov-witten/pairs_2012} for all compact Calabi-Yau threefolds which are complete intersections in products of projective spaces.  

In order for the GW/DT correspondence to make sense, we must have the following result which began as a conjecture of \cite{maulik_gromovwitten_2006}, and was later proven by Bridgeland \cite{bridgeland_hall_2011}.   

\begin{thm}
For all classes $\beta \neq 0$, $Z_{\text{DT}}(X)_{\beta}$ is the Laurent expansion of a rational function in $q$ invariant under $q \leftrightarrow 1/q$.  
\end{thm}

One notable consequence of the change of variables $-q = e^{i \lambda}$ is that the region of small $q$ corresponds to the region of large pure imaginary $\lambda$, and visa versa.  One can think of the reduced Gromov-Witten partition function as an asymptotic expansion around $\lambda =0$, while the reduced Donaldson-Thomas partition function is an asymptotic expansion around $q=0$.  Therefore, one aspect of the above conjecture is that the two partition functions are simply asymptotic expansions of \emph{the same function} about different points.

\subsection{The Physics of Donaldson-Thomas Theory}\label{subsecc:PHYSSDTtherry}

By compactifying Type IIA superstring theory on a smooth projective Calabi-Yau threefold $X$, one induces an $\mathcal{N}=2$ theory in the four-dimensional spacetime $\mathbb{R}^{1,3}$.  At this point, the reader is advised to review Section \ref{sec:Dbranesstabstrth} where we survey D-branes in string theory.  The Donaldson-Thomas invariants are deformation invariants of $X$ which compute the virtual number of BPS states of particles in $\mathbb{R}^{1,3}$ with a fixed charge vector.  Recall that D$p$-branes in Type IIA have a $(p+1)$-dimensional worldvolume where $p$ must be even.  Therefore, to engineer a particle in four dimensions, the branes should wrap a real even-dimensional cycle in $X$; the remaining dimension corresponds to the worldline in $\mathbb{R}^{1,3}$.  Because the mass will be proportional to the volume of the brane, BPS particles arise from D-branes wrapping \emph{algebraic} cycles in $X$, as these are volume-minimizing within their homology class.  We model part of the following discussion on \cite{cirafici_curve_2013}.  

In the large volume limit $\text{Vol}(X) \to \infty$, the BPS states are labelled by a charge vector $\gamma$ lying in the charge lattice $\Gamma$ defined by
\begin{equation}
\Gamma = H^{0}(X, \mathbb{Z}) \oplus H^{2}(X, \mathbb{Z}) \oplus H^{4}(X, \mathbb{Z}) \oplus H^{6}(X, \mathbb{Z})
\end{equation}  
where $H^{6-2k}(X, \mathbb{Z})$ corresponds to the D$(2k)$-brane charge through Poincar\'{e} duality.  The BPS states lie in a Hilbert space graded by charge vector 
\begin{equation}
\mathcal{H}_{\text{BPS}}^{X} = \bigoplus_{\gamma \in \Gamma} \mathcal{H}_{\gamma, \text{BPS}}^{X}.
\end{equation}
In a supersymmetric theory the Witten index counts BPS states by tracing a suitable operator over the Hilbert space of states.  For BPS states of fixed charge $\gamma$, the Witten index is given by
\begin{equation}
\Omega_{X}(\gamma) = \text{Tr}_{\mathcal{H}_{\gamma, \text{BPS}}^{X}} (-1)^{F}.
\end{equation}

In the large volume limit $\text{Vol}(X) \to \infty$, consider the charge vector $\gamma = (1,0, -\beta, -n)$.  Because we assume $H^{1}(X, \mathcal{O}_{X})=0$, there are no degree zero line bundles on $X$, and $\gamma$ corresponds to the Chern character of an ideal sheaf $\mathscr{I}_{Z}$ of a one-dimensional subscheme $Z$.  In such a case, the Witten index is precisely the Donaldson-Thomas invariant
\begin{equation}\label{eqn:DTBPSindTHM}
\sum_{k \in \mathbb{Z}} k \, \chi\big(\nu^{-1}(k)\big) = \Omega_{X}\big((1,0,-\beta, -n)\big) = \int_{[M_{\beta, n}(X)]^{\text{vir}}} 1.
\end{equation}
We think of $\mathcal{O}_{Z}$ as a bound state of D2-D0 branes, and $\mathcal{O}_{X}$ as a single D6-brane.  Recalling the ideal sheaf exact sequence (\ref{eqn:idshexseq}), one therefore concludes that Donaldson-Thomas invariants are virtual counts of BPS states of D6-D2-D0 branes in Type IIA string theory.  Packaging these integers into a generating function, we recover the familiar Donaldson-Thomas partition function 
\begin{equation}\label{eqn:BPSinstDTpartfuccc}
Z_{\text{DT}}(X) = \sum_{\beta \in H_{2}(X, \mathbb{Z})} \sum_{n \in \mathbb{Z}} \Omega_{X}\big((1,0,-\beta, -n)\big) Q^{\beta} q^{n}
\end{equation}
which is physically interpreted as the partition function of certain BPS black holes engineered from a single D6-brane, no D4-branes, and bound states of D2-D0 branes.

\subsubsection{Passing from Type IIA to Type IIB and the B-model Topological String}

Compactifying the time direction of $\mathbb{R}^{1,3}$ into a circle of radius $R$, in the limit $R \to \infty$, we recover the D-brane picture in Type IIA.  Performing T-duality on this time circle, we exchange Type IIA in the limit $R \to \infty$ with Type IIB in the limit $R \to 0$.  Therefore, a configuration of D6-D2-D0 branes in Type IIA is equivalent to a bound state of D5-D1-D(-1) branes in Type IIB, where the support in the Calabi-Yau is completely unchanged.  Recall that D$p$-branes in Type IIB are $(p+1)$-dimensional objects where $p$ must be odd.  But because in this particular case, the time direction has been contracted, a D5-branes wraps a 6-cycle in the Calabi-Yau, a D1-brane wraps a 2-cycle, and a D(-1)-brane is supported on a point.  Another consequence of time being contracted is that these brane configurations engineer BPS \emph{instantons} in three dimensions.  Here, BPS instanton refers to branes having pointlike (or localized) support in spacetime, whereas a BPS particle is supported along a worldline.  

One also hears about Donaldson-Thomas invariants as quantities in the B-model topological string theory with target space $X$.  This is because the B-branes in this theory wrap algebraic submanifolds.  \emph{We therefore have the following equivalent interpretations of Donaldson-Thomas invariants in string theory: they are virtual counts of B-branes in the B-model topological string theory on a smooth compact Calabi-Yau threefold.  In the large volume limit, they can also be thought of either as counts of states of BPS particles arising from D6-D2-D0 branes in Type IIA or as counts of states of BPS instantons arising from D5-D1-D(-1) branes in Type IIB, where the equivalence is induced by T-duality.}

\subsubsection{The GW/DT Correspondence as S-duality}

In Section \ref{subsec:GWA-modTSSP} we saw that Gromov-Witten theory corresponds to the A-model topological string, while we saw just above that Donaldson-Thomas theory corresponds to the B-model.  The GW/DT correspondence is a mathematical conjecture that the two reduced partition functions are equivalent.  One must note that this is \emph{not} a statement of mirror symmetry!  Indeed, the A and B-models are exchanged under mirror symmetry, but only after also exchanging the target space with its mirror partner.  To the contrary, the GW/DT correspondence is a relation between two theories on the same target.  This is understood roughly as follows.  

In physics, the GW/DT correspondence is a statement of S-duality \cite{nekrasov_s-duality_2004}, which is a highly non-trivial strong-weak coupling symmetry in the Type IIB superstring.  We have seen that the Donaldson-Thomas invariants are virtual counts of BPS instantons coming from bound states of D5-D1-D(-1) branes in Type IIB.  If the variable $q$ tracking the D(-1)-brane charge is small, the BPS instanton partition function (\ref{eqn:BPSinstDTpartfuccc}) is a perturbative expansion in $q$.  It turns out that $q$ is related to the string coupling constant by $q = e^{i \lambda}$.  As $q$ becomes large, we enter the region of small $\lambda$ and the correct physical description is to imagine the D1-branes as worldsheet instantons in Gromov-Witten theory.  

The principal of S-duality is that these strong and weak coupling descriptions are equivalent, and we get the conjecture
\begin{equation}
Z'_{\text{GW}}(X) = Z'_{\text{DT}}(X),
\end{equation}
under the change of variables\footnote{Whether one takes $q= e^{i \lambda}$ or $q = - e^{i \lambda}$ for the change of variables turns out to be not so important.} $q = e^{i \lambda}$.  To compute the contribution of a worldsheet instanton of fixed genus, we need to know the virtual counts of BPS instantons for \emph{all} values of the D(-1)-brane charge.  Similarly, one must know the contributions of worldsheet instantons of all genera to compute a single virtual count of BPS states with a fixed D(-1)-brane charge.

\section{The Gopakumar-Vafa (BPS) Invariants} \label{sec:GVinvvvv}

We have seen in previous sections that the Gromov-Witten and Donaldson-Thomas invariants are not truly enumerative, in that they do not genuinely count curves in a given class.  It has long been expected however, that both the Gromov-Witten and Donaldson-Thomas partition functions encode only finitely many integers for a fixed curve class, and that these integers are closer to honest curve-counting invariants.  Remarkably, such quantities emerged from physical considerations and are called the Gopakumar-Vafa (or BPS) invariants.

Motivated by what is called the M2-brane moduli space in M-theory, Gopakumar and Vafa showed that the reduced Gromov-Witten free energy on a Calabi-Yau threefold can be repackaged and written in terms of the Gopakumar-Vafa invariants.  Physically, these are virtual counts of BPS states of M2-branes in M-theory, or bound states of D2-D0 branes in string theory.  A proper, mathematical definition of the Gopakumar-Vafa invariants had been lacking until the recent proposal of Maulik-Toda \cite{maulik_gopakumar-vafa_2016}, which we will not be able to discuss here.

\subsection{Definition in Terms of GW/DT Invariants}

In the sequence of papers \cite{gopakumar_m-theory_1998, gopakumar_m-theory_1999} Gopakumar and Vafa (with motivations to be sketched shortly) show that a collection of numbers $n_{g, \beta}(X)$ are related to the reduced Gromov-Witten free energy of a Calabi-Yau threefold $X$ by the formula
\begin{equation} \label{eqn:BPSTopStr}
\setlength{\jot}{12pt}
\begin{split}
F'_{\text{GW}}(X) & =  \sum_{g \geq 0, \beta \neq 0} \text{GW}_{g, \beta}(X) \lambda^{2g-2} v^{\beta}  \\
& = \sum_{g \geq 0, \beta \neq 0} n_{g, \beta}(X) \lambda^{2g-2} \sum_{d >0} \bigg( \frac{2 \sin\big(\frac{d \lambda}{2}\big)}{\lambda}\bigg)^{2g-2} \frac{v^{d \beta}}{d}.
\end{split}
\end{equation}
The invariants $n_{g, \beta}(X)$ are known as the \emph{Gopakumar-Vafa invariants} or \emph{BPS invariants}, and for our purposes we take (\ref{eqn:BPSTopStr}) as their \emph{definition}.  Notice this is not a proper geometrical definition of $n_{g, \beta}(X)$.  Rather, one should think of it as simply a repackaging of the information in the Gromov-Witten free energy.  We will not go into the proposal of Maulik-Toda, but this is indeed expected to be a proper definition.  

Given that the Gromov-Witten invariants $\text{GW}_{g, \beta}(X)$ are rational, it appears that the $n_{g,\beta}(X)$ are as well.  However, the primary interest in the Gopakumar-Vafa invariants is generated from the following conjecture.  

\begin{conj}
For a Calabi-Yau threefold $X$, the Gopakumar-Vafa invariants $n_{g, \beta}(X)$ are integers for all genera $g$ and classes $\beta \in H_{2}(X, \mathbb{Z})$.  In addition, for a fixed class $\beta$, there are only finitely many $g$ such that $n_{g, \beta}(X)$ are non-vanishing.  
\end{conj}
\noindent Using the GW/DT correspondence, we can re-write (\ref{eqn:BPSTopStr}) in terms of the Donaldson-Thomas partition function and variables
\begin{equation}
Z'_{\text{DT}}(X) = \text{exp} \bigg(\sum_{g \geq 0, \beta \neq 0} \sum_{d >0} (-1)^{1-g} n_{g, \beta}(X) (-q)^{d(1-g)} \big(1-(-q)^{d}\big)^{2g-2}\frac{v^{d \beta}}{d}\bigg).
\end{equation}

We say a class $\beta \in H_{2}(X, \mathbb{Z})$ is \emph{irreducible} if it cannot be written as the sum of two effective classes.  An irreducible class is primitive and indivisible.  

\begin{rmk}
For $\beta \in H_{2}(X, \mathbb{Z})$ an irreducible non-zero class, the conjectural relationship between the Gopakumar-Vafa invariants and the Donaldson-Thomas partition function takes the simpler form
\begin{equation}
Z'_{\text{DT}}(X)_{\beta} = \sum_{g \geq 0} n_{g, \beta}(X) \big(q^{\frac{1}{2}} + q^{-\frac{1}{2}}\big)^{2g-2}.
\end{equation}
\end{rmk}

More generally, given the Donaldson-Thomas partition function $Z_{\text{DT}}(X)$, one can always write it as an infinite product
\begin{equation}
Z_{\text{DT}}(X) = \prod_{\beta \in H_{2}(X, \mathbb{Z})} \prod_{n \in \mathbb{Z}} \big(1-Q^{\beta} q^{n} \big)^{-c(\beta, n)}
\end{equation}
for \emph{some} collection of numbers $c(\beta, n)$ for all $\beta$ and $n$, or perhaps $\beta$ from a sublattice of $H_{2}(X, \mathbb{Z})$.  Generally this set of numbers might not have any nice properties, though we will see in Chapter \ref{ch:AutGWPotBanana} an example where they are the coefficients of an automorphic form.  These $c(\beta, n)$ are actually closely related to the Gopakumar-Vafa invariants of $X$.  One can show the following from \cite{katz_gromov-witten_2004}.  

\begin{proppy}
For all fixed non-zero classes $\beta \in H_{2}(X, \mathbb{Z})$, we have
\begin{equation}\label{eqn:cbetanGVrelll}
\sum_{g=0}^{\infty} n_{g, \beta}(X) \big(q^{\frac{1}{2}} + q^{-\frac{1}{2}}\big)^{2g-2} = \sum_{n \in \mathbb{Z}} c(\beta, n) (-q)^{n}.  
\end{equation}
\end{proppy}

\subsection{Simpson Stability and the D-brane Moduli Space}

In Section \ref{eqn:subsecSimppsStablllty} we introduced the notion of Simpson slope stability on pure one-dimensional sheaves on a Calabi-Yau threefold.  The following moduli space of Simpson semistable sheaves essentially appeared in the original M-theory papers on the Gopakumar-Vafa invariants, and also plays a leading role in the modern mathematical understanding of them.  

\begin{defn}
The D-brane moduli space or Simpson moduli space denoted $M_{\beta}(X)$ is the projective scheme parameterizing pure one-dimensional semistable sheaves $\mathscr{E}$ with $\chi(X, \mathscr{E})=1$ and support cycle $\beta$ on a polarized Calabi-Yau threefold $(X, H)$.  We know by (\ref{eqn:CherncharrrCY3onedish}) that this discrete data is equivalent to specifying the Chern character $\text{ch}(\mathscr{E}) = (0, 0, \beta, 1)$.  
\end{defn}

Recall that by Proposition \ref{proppy:onedshGieSimp}, for pure one-dimensional sheaves on a Calabi-Yau threefold, Gieseker-Simpson stability is equivalent to Simpson slope stability measured by the slope
\begin{equation} \label{eqn:slopstabb1}
\mu_{S}(\mathscr{\mathscr{F}}) = \frac{\chi(X, \mathscr{F})}{H \cdot \beta_{\mathscr{F}}}
\end{equation}
where $\beta_{\mathscr{F}}$ is the support cycle of $\mathscr{F}$.  The condition $\chi(X, \mathscr{E})=1$ is imposed in order to ensure that there are no strictly semistable sheaves, and that the moduli space is independent of the polarization.  To prove this, we first need the following lemma.  

\begin{lemmy}
Let $\mathscr{E}$ and $\mathscr{F}$ be two coherent sheaves on a smooth projective polarized variety with Hilbert polynomials $P(\mathscr{E})$ and $P(\mathscr{F})$.  If $P(\mathscr{E}) = P(\mathscr{F})$ and there exists an injective or surjective map $f: \mathscr{F} \to \mathscr{E}$, then $f$ is in fact an isomorphism.  
\end{lemmy}

\begin{proof}
Assume first that $f$ is injective.  There exists a coherent sheaf $\mathscr{G}$ fitting into a short exact sequence
\begin{equation}
0 \to \mathscr{F} \to \mathscr{E} \to \mathscr{G} \to 0.  
\end{equation}
Because tensoring by line bundles is exact, using the additivity of the holomorphic Euler characteristic on short exact sequences, we have $P(\mathscr{E}) = P(\mathscr{F}) + P(\mathscr{G})$.  However, if $P(\mathscr{E}) = P(\mathscr{F})$, then $P(\mathscr{G})=0$ which can happen if and only if $\mathscr{G}=0$.  Of course, if $f$ is surjective the argument is the same.   
\end{proof}

\begin{proppy}
There are no strictly semistable moduli points in the D-brane moduli space $M_{\beta}(X)$, and the moduli space is in fact independent of the polarization.  
\end{proppy}
\begin{proof}
Let $\mathscr{E}$ be a strictly semistable moduli point of $M_{\beta}(X)$.  There must therefore exist a proper subsheaf $\mathscr{F} \hookrightarrow \mathscr{E}$ such that $\mu_{S}(\mathscr{F}) = \mu_{S}(\mathscr{E})$.  Note that because $\mathscr{E}$ is pure, the subsheaf $\mathscr{F}$ cannot be zero-dimensional, which by the ampleness of $H$ means $H \cdot \beta_{\mathscr{F}} > 0$.  But $\mathscr{F}$ is a subsheaf of $\mathscr{E}$, which implies $0< H \cdot \beta_{\mathscr{F}} \leq H \cdot \beta$.  The equality of slopes, therefore gives the following inequality
\begin{equation}
0 < \chi(X, \mathscr{F}) = \frac{H \cdot \beta_{\mathscr{F}}}{H \cdot \beta} \leq 1.
\end{equation}
This implies that $\chi(X, \mathscr{F})=1$ and $H \cdot \beta_{\mathscr{F}} = H \cdot \beta$.  The Hilbert polynomials of $\mathscr{E}$ and $\mathscr{F}$ are therefore identical, and by the above lemma, noting the injective morphism $\mathscr{F} \hookrightarrow \mathscr{E}$, we conclude $\mathscr{E} \cong \mathscr{F}$, contradicting that $\mathscr{F}$ is a proper subsheaf.  As for the second claim, since $\chi(X, \mathscr{E})=1$ the condition that the moduli point $\mathscr{E}$ in $M_{\beta}(X)$ is stable is that for all proper subsheaves $\mathscr{F} \hookrightarrow \mathscr{E}$, we have $\chi(X, \mathscr{F}) \leq 0$.  Because this is independent of $H$, the moduli space is as well.  
\end{proof}

In \cite{katz_genus_2006}, it was shown by S. Katz that $M_{\beta}(X)$ carries a symmetric obstruction theory, which led to the correct mathematical definition of the Gopakumar-Vafa invariants in genus zero.

\begin{defn}
The genus zero Gopakumar-Vafa invariants $n_{0, \beta}(X)$ are given by the Behrend-weighted Euler characteristic 
\begin{equation}
n_{0, \beta}(X) =\int_{[M_{\beta}(X)]^{\text{vir}}} 1 = \chi\big(M_{\beta}(X), \nu\big) \in \mathbb{Z},
\end{equation}
of the D-brane moduli space.  If $M_{\beta}(X)$ is smooth, then $n_{0, \beta}(X) = (-1)^{\text{dim} M_{\beta}(X)} \chi\big(M_{\beta}(X)\big)$.  
\end{defn}

\noindent This definition is consistent with the more recent proposal of Maulik and Toda \cite{maulik_gopakumar-vafa_2016} which holds for all genus.  The Hilbert-Chow morphism
\begin{equation}
\pi : M_{\beta}(X) \longrightarrow \text{Chow}_{\beta}(X).
\end{equation}
plays an integral role in the general definition of the $n_{g, \beta}(X)$.  As we just saw, the Behrend-weighted Euler characteristic of $M_{\beta}(X)$ produces the genus zero invariants and in certain cases, the maximal genus invariants arise as weighted Euler characteristics of the Chow variety $\text{Chow}_{\beta}(X)$.

\subsection{The Physics of the Gopakumar-Vafa Invariants}\label{subsecc:PHYSGVinvv}

We will attempt to briefly outline the physics of the Gopakumar-Vafa invariants in a way which is hopefully approachable to mathematicians.  We will therefore leave out many physical details, and refer the interested reader to \cite{katz_m-theory_1999} as well as the original papers \cite{gopakumar_m-theory_1998, gopakumar_m-theory_1999}.  

Let us begin by considering M-theory compactified on a smooth compact Calabi-Yau threefold $X$.  Noting that M-theory is an 11-dimensional theory, the compactification induces a theory in five dimensions with $\mathcal{N}=2$ supersymmetry.  There exists a Hilbert space of states of particles in the five-dimensional theory.  It is well-known physically that there are $b_{2}(X)$-many $U(1)$ gauge fields, where $b_{2}(X)$ is the second Betti number of $X$ and is defined as the rank of $H_{2}(X, \mathbb{Z})$.  By coupling to these gauge fields, the five-dimensional particles therefore acquire a charge $\beta$ lying in the charge lattice $H_{2}(X, \mathbb{Z})$.  We can consider the subspace of the full Hilbert space corresponding to particles with fixed charge $\beta \in H_{2}(X, \mathbb{Z})$, and we can further focus on the subspace of BPS particles with charge $\beta$.  If $\beta \neq 0$ (which we will assume from here on) these BPS particles necessarily come from M2-branes wrapping a holomorphic two-cycle in $X$ of homology class $\beta$.  Recall that an M2-brane is a fundamental object in M-theory with a three-dimensional worldvolume, one direction of which is necessarily time.  The mass of the BPS particle is proportional to the integral of the K\"{a}hler form over $\beta$.  This same idea we have encountered a few times now -- the BPS condition means mass minimizing and therefore volume minimizing.  Holomorphic cycles minimize volume within a homology class on K\"{a}hler manifolds.  

One general principal in physics is that a \emph{particle} is understood as an irreducible representation of the spacetime symmetry group.  In five dimensions, massive particles transform under the rotation group $SO(4)$ which is isomorphic at the level of Lie algebras to $SU(2)_{L} \oplus SU(2)_{R}$.  The subscripts `L' and `R' are conventional in physics to distinguish the two factors.  We therefore need to review the representation theory of $SU(2)$.  Let $V(n)$ denote the unique irreducible representation of $SU(2)$ of dimension $n+1$ and highest weight $n$.  Of course, the representation is a morphism 
\begin{equation}\label{eqn:BPSreppsth}
\varphi_{n}: SU(2) \longrightarrow GL\big(V(n)\big)
\end{equation} 
but we understand $V(n)$ to be the $SU(2)$-module induced by $\varphi_{n}$.  All of the irreducible representations of $SU(2)_{L} \oplus SU(2)_{R}$ are therefore of the form $V(2j_{L}) \oplus V(2j_{R})$ for $j_{L}, j_{R} \in \tfrac{1}{2} \mathbb{Z}_{\geq 0}$.

Returning to the five-dimensional BPS particles, we see that in addition to charge $\beta \in H_{2}(X, \mathbb{Z})$, they are charged under $SU(2)_{L} \oplus SU(2)_{R}$.  Let $N^{\beta}_{j_{L}, j_{R}}$ be the number of five-dimensional BPS particles with charge $\beta \in H_{2}(X, \mathbb{Z})$ and transforming under the irreducible representation $V(2j_{L}) \oplus V(2 j_{R})$ of $SU(2)_{L} \oplus SU(2)_{R}$.  One should think of $j_{L}, j_{R}$ as the \emph{spin} of the particles.  It turns out that $N^{\beta}_{j_{L}, j_{R}}$ is \emph{not} a deformation invariant of the theory: it can jump in number upon deforming the complex structure of $X$.  For physical reasons, we are motivated to define the quantities 
\begin{equation}
N^{\beta}_{j_{L}} \coloneqq \sum_{j_{R} \in \tfrac{1}{2}\mathbb{Z}_{\geq 0}} (-1)^{2j_{R}}(2 j_{R} +1) N^{\beta}_{j_{L}, j_{R}}
\end{equation}
by summing over the spin $j_{R}$ in the particular way shown.  The point of this is that $N^{\beta}_{j_{L}}$ is indeed invariant under complex deformations of $X$ \cite{katz_m-theory_1999}.  

Let $\mathscr{R}$ be the representation ring of $SU(2)$.  The most natural collection of generators is given by
\[\bigg\{ V(n) \,\, \bigg| \,\, n =0, 1, 2, \ldots \bigg\}\]
but it will be convenient to choose a different basis of $\mathscr{R}$.  We can consider the representation $V(1) \oplus V(0) \oplus V(0) = V(1) \oplus 2 V(0)$ and define the collection 
\begin{equation}
\bigg\{ \big(V(1) \oplus 2V(0)\big)^{\otimes n} \,\, \bigg| \,\, n =0, 1, 2, \ldots \bigg\}.
\end{equation}
By noting that $V(n) \otimes V(1) = V(n+1) \oplus V(n-1)$, it is easy to see that this latter collection generates the representation ring $\mathscr{R}$.  Because it will play a key role, we define
\begin{equation}\label{eqn:Igdefnnn}
I_{g} \coloneqq \big(V(1) \oplus 2V(0)\big)^{\otimes g}.
\end{equation}
We can then consider the element $\sum_{j_{L}} N^{\beta}_{j_{L}} V(2 j_{L})$ of $\mathscr{R}$.  Because the $I_{g}$ are a basis of $\mathscr{R}$, there must be integers $n_{g, \beta}(X)$ such that 
\begin{equation}\label{eqn:eqninRepRing}
\sum_{j_{L} \in \tfrac{1}{2} \mathbb{Z}_{\geq 0}} N^{\beta}_{j_{L}} V(2 j_{L}) = \sum_{g \geq 0} n_{g, \beta}(X) I_{g}
\end{equation}
is an equality in $\mathscr{R}$, for fixed $\beta \in H_{2}(X, \mathbb{Z})$.  

Recall that a maximal toral subgroup of $SU(2)$ is one-dimensional and generated by diagonal elements of the form $\text{diag}(q, q^{-1})$ for $q \in U(1)$.  This toral subgroup is clearly isomorphic to $U(1)$.  Given an irreducible representation $V(n)$ of $SU(2)$ introduced above, we can consider the action of $\varphi_{n}(q) \coloneqq \varphi_{n}\big(\text{diag}(q, q^{-1})\big)$ on $V(n)$, where $\varphi_{n}$ was defined in (\ref{eqn:BPSreppsth}).  This induces a decomposition into eigenspaces
\begin{equation}\label{eqn:decompesyumod}
V(n) = V_{n} \oplus V_{n-2} \oplus \ldots \oplus V_{-(n-2)} \oplus V_{-n}
\end{equation}
where $\varphi_{n}(q)$ acts on $V_{k}$ as multiplication by $q^{k}$.  Note that the decomposition (\ref{eqn:decompesyumod}) is \emph{not} an isomorphism of representations, as non-toral elements of $SU(2)$ will not preserve the individual $V_{k}$.  Associated to the representation $V(n)$ we have a character $\chi_{n}: SU(2) \to \mathbb{C}$ defined by $\chi_{n}(g) \coloneqq \text{Tr}\varphi_{n}(g)$ for all $g \in SU(2)$.  By (\ref{eqn:decompesyumod}), evaluating $\chi_{n}$ on elements in the toral subgroup we have
\begin{equation}\label{eqn:characdiagg}
\chi_{n}(q) = q^{n} + q^{n-2} + \ldots + q^{-(n-2)} + q^{-n} \,\,\,\,\,\,\,\,\,\,\,\,\,\, q \in U(1).
\end{equation} 
Because characters are multiplicative on tensor products and additive under direct sums, we get a well-defined character $\chi$ on the full representation ring $\mathscr{R}$ of $SU(2)$.  Using the definition (\ref{eqn:Igdefnnn}) of $I_{g}$ in terms of $V(1)$ and $V(0)$, by (\ref{eqn:characdiagg}) it is clear that 
\begin{equation}
\chi(I_{g}) = (q + 2 + q^{-1})^{g}.
\end{equation}

Because (\ref{eqn:eqninRepRing}) is an equality in the representation ring, it must give rise to an equality after applying the character $\chi$ to both sides and evaluating on the toral element $q$.  Indeed, reindexing in terms of $g = 2 j_{L}$ we get,
\begin{equation}\label{eqn:relnspinbases}
\sum_{g \geq 0} N^{\beta}_{g/2} \big(q^{g} + q^{g-2} + \ldots + q^{-(g-2)} + q^{-g}\big) = \sum_{g \geq 0} n_{g, \beta}(X)(q + 2 + q^{-1})^{g}.
\end{equation}
The integers $n_{g, \beta}(X)$ are called the \emph{Gopakumar-Vafa invariants} and they are virtual counts of BPS states in five dimensions with non-zero charge $\beta \in H_{2}(X, \mathbb{Z})$ and spin $g \in \mathbb{Z}_{\geq 0}$.  They are deformation invariants of $X$.  Note that the spin $g$ is not the usual spin measured with respect to the basis $\{V(g)\}$ of the representation ring of $SU(2)$.  Rather, it is measured with respect to the basis $\{I_{g}\}$.  

As we have described, these BPS particles are engineered by M2-branes wrapping two-cycles in $X$.  The righthand side of (\ref{eqn:relnspinbases}) should be thought of as the partition function $Z_{\text{BPS}}(X)_{\beta}$ of BPS states with fixed charge $\beta$.  Summing over the charge lattice $H_{2}(X, \mathbb{Z})$, we get the full partition function of virtual BPS counts of M2-branes in $X$
\begin{equation}
Z_{\text{BPS}}(X) = \sum_{\substack{\beta \in H_{2}(X, \mathbb{Z}) \\ \beta \neq 0}} \sum_{g \geq 0} n_{g, \beta}(X)(q + 2 + q^{-1})^{g-1} Q^{\beta}
\end{equation}
where we have inserted an extra factor of $(q+2+q^{-1})^{-1}$, following convention.  Compactifying M-theory on a small circle, an M2-brane is interpreted as a bound state of D2-D0 branes.  Though unlike Donaldson-Thomas theory, the D0-branes lie in the D2-branes.  

\subsubsection{Gopakumar-Vafa Invariants and Black Hole Degeneracies}

The mass of the BPS particles in five dimensions is directly related to the class $\beta$, and is proportional to the volume of the M2 or D2-branes.  Because $X$ is a compact K\"{a}hler manifold, the volume of any closed complex subspace depends only on its homology class, and is given by integrating powers of the K\"{a}hler form against this class.  Therefore, if we consider a stack of $d$ M2-branes wrapping a fixed curve in X of class $\beta$, we get a resulting homology class $d \beta$.  The mass clearly also increases by a factor of $d$.  If we then take $d \gg 1$, we actually get a \emph{black hole} in five dimensions, and the Gopakumar-Vafa invariant $n_{g, d\beta}(X)$ computes the \emph{degeneracy} (the number of black hole states with fixed charges and spins).  This way of engineering black holes by wrapping large numbers of branes on fixed cycles was pioneered by Strominger and Vafa \cite{strominger_microscopic_1996-1}.

\subsubsection{Relation to the Topological String}  

In order to relate the Gopakumar-Vafa invariants to the topological string, it was the idea of Gopakumar and Vafa to use the relationship between 11-dimensional M-theory and 10-dimensional Type IIA string theory.  Working with the Euclideanized theory, consider M-theory on $\mathbb{R}^{4} \times X \times S^{1}$ where $X$ is a smooth compact Calabi-Yau threefold, and let $R$ be the radius of the M-theory circle $S^{1}$.  In the limit $R \to 0$, the theory reduces to Type IIA superstring theory on $\mathbb{R}^{4} \times X$. 

We recall that M2-branes are fundamental objects in M-theory, which have two spatial dimensions and a three-dimensional worldvolume.  We have seen that the Gopakumar-Vafa invariants are virtual counts of BPS states of M2-branes.  Writing the 11-dimensional spacetime as $\mathbb{R}^{4} \times X \times S^{1}$, there are two possibilities:

\begin{enumerate}

\item If one of the three dimensions of an M2-brane wraps the M-theory circle $S^{1}$, we get a \emph{fundamental Type IIA string} in $\mathbb{R}^{4} \times X$ upon taking the limit $R \to 0$.  This object, with a two-dimensional worldvolume in the limit, appears as a string worldsheet.  

\item If the M2-brane does \emph{not} wrap the M-theory circle, in the $R \to 0$ limit we get a D2-brane in Type IIA string theory on $\mathbb{R}^{4} \times X$, with a three-dimensional worldvolume.  

\end{enumerate}

\noindent  One brilliant idea of Gopakumar and Vafa was to take \emph{time} to be the M-theory circle.  Because an M2-brane necessarily has one time direction, the remaining two dimensions wrap a two-cycle in $X$.  From physical principles, we expect that a partition function of BPS states is independent of $R$.  So we can relate the computations at $R \to 0$ with $R \to \infty$.  By the above comments, in the $R \to 0$ limit we get a fundamental string in Type IIA, but it is wrapping a two-cycle in $X$ and it is sitting at a point in $\mathbb{R}^{4}$.  This is simply a worldsheet instanton, and its contributions are computed by the A-model topological string theory.  Conversely, in the $R \to \infty$ limit we get a BPS particle in five dimensions.  By relating these two limits, one is therefore able to write the topological string partition function in terms of the Gopakumar-Vafa invariants.

\chapter{Introduction to Topological Indices and Localization}

In this chapter we review equivariant cohomology and Atiyah-Bott localization with the goal of defining and explicitly computing equivariant versions of topological indices like the $\chi_{y}$-genus and the elliptic genus.  We prove that the elliptic genus can be interpreted as a (regularized) equivariant $\chi_{y}$-genus applied to the loop space of a compact complex manifold.  We give examples of equivariant indices applied to the simple non-compact toric variety $\mathbb{C}^{2}$.  More importantly, we give these same examples applied to moduli spaces of instantons on $\mathbb{C}^{2}$, which inherit a torus action.  This latter computation requires localization on a much more complicated space, but Nakajima-Yoshioka have shown that there are finitely many isolated torus fixed points whose equivariant tangent spaces can be given explicitly.  Summing over the topological charge of the instantons, we get a generating function known as a Nekrasov partition function of an $\mathcal{N}=2$ gauge theory.  

\section{Motivating Example: Topological Euler Characteristic}

The goal of this section is to preview the major recurring themes of this chapter by way of the simple example of the Euler characteristic.  The topological Euler characteristic is a topological invariant of a manifold given by the alternating sum of the Betti numbers.  If the manifold is compact and K\"{a}hler, it can also be expressed as an integral of the Euler class over the fundamental class.  These two perspectives are related by an index theorem which generalizes the Gauss-Bonnet theorem.  We will also state a powerful localization result often making the Euler characteristic easily computable when the manifold carries an action by tori.  Last but not least, we will see that in the context of supersymmetric sigma models, the Witten index is simply the Euler characteristic in physical guise.  Each of the above perspectives are important and will reappear in this chapter in less trivial examples.

The \emph{topological Euler characteristic} of a compact, orientable $m$-dimensional manifold $X$ is defined to be the alternating sum of the Betti numbers
\begin{equation}\label{eqn:eulchartopman}
\chi(X) = \sum_{k=0}^{m} (-1)^{k} b_{k}(X)
\end{equation}
The Betti numbers are topological invariants of $X$, and therefore the Euler characteristic is as well.  With the given hypotheses on $X$, Poincar\'{e} duality is manifest in the symmetry $b_{k}(X) = b_{m-k}(X)$ of the Betti numbers.  If $X$ is in addition, a complex manifold of (complex) dimension $n=m/2$, we can define Hodge numbers which are given by dimensions of sheaf cohomology groups $h^{p,q}(X) = \text{dim}H^{q}(X, \Omega^{p})$, for $p,q =1, \ldots n$.  Here, $\Omega^{p}$ is the sheaf of holomorphic $p$-forms on $X$.  Using the compactness of $X$, Serre duality is manifest in the symmetry $h^{n-p, n-q}(X) = h^{p,q}(X)$.  

Though the Hodge numbers are defined for any compact complex manifold, when $X$ is additionally K\"{a}hler they have the further symmetry $h^{p,q}(X) = h^{q,p}(X)$ and a relationship with the Betti numbers.  In particular, the Hodge decomposition induces the equality
\begin{equation}\label{eqn:hodge-betireln}
b_{k}(X) = \sum_{p+q=k}h^{p,q}(X).
\end{equation}
One should think of the Hodge numbers as a refinement of the Betti numbers for a compact K\"{a}hler manifold.  By (\ref{eqn:eulchartopman}) and (\ref{eqn:hodge-betireln}), we can express the Euler characteristic in terms of the Hodge numbers
\begin{equation}\label{eqn:hodge-eulcharelnn}
\chi(X) = \sum_{p,q=0}^{n} (-1)^{p+q}h^{p,q}(X).
\end{equation}
In general, the Hodge numbers are not topological invariants of $X$, though certain linear combinations of them may be.  For example, the Betti numbers (\ref{eqn:hodge-betireln}).   

Given a complex manifold $X$ of dimension $n$, the Euler class defined to be the top Chern class of $T_{X}$
\begin{equation}
e(T_{X}) \coloneqq c_{n}(X) = \prod_{i=1}^{n} x_{i}
\end{equation}
is an example of what we will define in Section \ref{sec:EquivCohABLoc} to be a multiplicative class associated to the polynomial $f(x)=x$.  Here $x_{1}, \ldots, x_{n}$ are the formal Chern roots of the tangent bundle $T_{X}$.  We will later consider topological indices to be integrals of multiplicative classes over $X$.  By the following proposition, if $X$ is compact and K\"{a}hler, the Euler characteristic is a topological index associated to the Euler class.  

\begin{proppy}[\bfseries Chern-Gauss-Bonnet Theorem]
On a compact K\"{a}hler manifold $X$ of dimension $n$, the topological Euler characteristic is given by
\begin{equation}\label{eqn:GENGBONNTHM}
\chi(X) = \int_{X} e(T_{X}).
\end{equation}
\end{proppy}

\begin{proof}
By Lemma \ref{lemmy:lambdaSChern} which is presented and proven in Section \ref{sec:Hirzzchiy}, it follows that
\begin{equation}\label{eqn:Borel-Serre}
\text{ch} \bigg( \bigoplus_{p=0}^{n} (-1)^{p} \Omega^{p}\bigg) \text{td}(X) = c_{n}(X).
\end{equation}
This is known as the Borel-Serre identity.  Because $X$ is K\"{a}hler, using the relationship (\ref{eqn:hodge-eulcharelnn}) between the Hodge numbers and Euler characteristic, we see that
\begin{equation}
\chi(X) =  \sum_{p,q=0}^{n} (-1)^{p+q} h^{p,q}(X) = \sum_{p=0}^{n}(-1)^{p} \sum_{q=0}^{n} (-1)^{q} \text{dim} H^{q}\big(X, \Omega^{p} \big).
\end{equation}
By the additivity of the Chern character on direct sums, we identify the righthand side with the holomorphic Euler characteristic $\chi(X, \oplus_{p} (-1)^{p} \Omega^{p})$ and the claim follows by applying (\ref{eqn:Borel-Serre}) within Hirzebruch-Riemann-Roch (\ref{eqn:HirzzzBBRoch}).  
\end{proof}

The Chern-Gauss-Bonnet theorem is a special case of the powerful Atiyah-Singer index theorem, relating the analytical index of a differential operator to a topological index.  In this thesis, we will not discuss analytical indices of operators.  For a general Riemannian manifold, the Euler characteristic can be computed as the integral of the Pfaffian applied to the curvature of the Levi-Civita connection.  The K\"{a}hler condition ensures compatibility between the complex and Riemannian theories, and identifies this Pfaffian (applied to the complexified tangent bundle) with the Euler class.  

When $X$ is some moduli space of interest, a topological index may be an interesting invariant of an enumerative problem or physical theory.  In the context of supersymmetric quantum field theories, a topological index may be referred to as a supersymmetric index.  A powerful computational tool for these invariants is what we will call \emph{localization}.  In the simple case of the Euler characteristic, we have the following useful localization result.  

\begin{proppy}
Let $X$ be a smooth compact manifold with an action by a real or algebraic torus $T$.  Then the Euler characteristic satisfies $\chi(X) = \chi(X^{T})$, where $X^{T}$ is the fixed point locus of the action.  
\end{proppy}

\noindent The moral of this result (which will reappear in less trivial examples) is that as far as the Euler characteristic of a manifold with torus action is concerned, one can localize to the fixed point locus without losing information.

\subsection{The Euler Characteristic as the Witten Index}\label{subsec:EULcharrasWITTIND}

The Euler characteristic has a purely physical manifestation as the \emph{Witten index} in supersymmetric quantum mechanics, and is an example of a supersymmetric index.  In general, a supersymmetric index may be simply a number or it may be given by a more exotic gadget than a numerical invariant, for example the elliptic genus.  We will be brief and omit many physical details, so the interested reader is referred to \cite[Section 10.4]{hori_mirror_2003}.

In general, \emph{supersymmetric quantum mechanics} consists of the following ingredients.  There exists a Hilbert space $\mathscr{H}$ of states of the system, admitting a $\mathbb{Z}/2$-grading
\begin{equation}
\mathscr{H} = \mathscr{H}_{B} \oplus \mathscr{H}_{F}
\end{equation}
where $\mathscr{H}_{B}$ is the Hilbert space of bosonic states and $\mathscr{H}_{F}$ is the Hilbert space of fermionic states.  There must also exist a supersymmetry operator $Q$ with partner $\overline{Q} = Q^{\dagger}$, which both act on $\mathscr{H}$ and exchange bosonic and fermionic states,
\begin{equation}
Q, \overline{Q} : \mathscr{H}_{B} \longleftrightarrow \mathscr{H}_{F}.
\end{equation}
We call $Q$ and $\overline{Q}$ \emph{supersymmetries}.  They satisfy $Q^{2}=0, \overline{Q}^{2} =0$ and their anticommutator is given by
\begin{equation}\label{eqn:anticommHAM}
\big\{ Q, \overline{Q} \big\} = 2H
\end{equation}
where $H$ is the Hamiltonian of the theory.  The above equation implies that $Q$ and $\overline{Q}$ commute with the Hamiltonian.  Thus, they are symmetries in the ordinary quantum mechanical sense.  

The eigenvalues of $H$ are the allowed energies, and states with zero energy are called \emph{ground states}.  The space of ground states is identified with the $Q$-cohomology $H^{*}(Q)$.  If in addition, there exists a fermion operator $F$ commuting with $H$, then we have a grading
\begin{equation}
H^{*}(Q) = \bigoplus_{k=0}^{\infty} H^{k}(Q)
\end{equation}
where the charge $k$ is the eigenvalue of $F$.  This charge $k$ is even for bosonic states and odd for fermionic states.  The \emph{Witten index} is a supersymmetric index defined by the trace $\text{Tr}_{\mathscr{H}}(-1)^{F}e^{-\beta H}$, and turns out to coincide with the difference between the number of bosonic and fermionic ground states
\begin{equation}\label{eqn:WittINDexpo}
\text{Tr}_{\mathscr{H}}(-1)^{F}e^{-\beta H} = n_{B}^{(0)} - n_{F}^{(0)}.
\end{equation}
The Witten index is independent of $\beta$, and the above difference between bosonic and fermionic ground states is insensitive to any changes in the parameters of the theory, though $n_{B}^{(0)}$ and $n_{F}^{(0)}$ may individually vary.

\subsubsection{Supersymmetric Quantum Mechanics on a Riemannian Manifold}  

All of the above structures emerge nicely in the geometric setting of supersymmetric particles moving in a Riemannian manifold $(M,g)$ of dimension $m$.  Such a theory is a (1+0)-dimensional non-linear sigma model, so we study continuous maps $\phi$ from $S^{1}$ or $\mathbb{R}$ into the target space $M$.  The bosonic fields are the local coordinates $\phi^{i}$ and the fermionic fields $\psi$ are local sections of the tangent bundle of $M$ (restricted to the image of $\phi$).  One can write down a classical action $S$, \cite{hori_mirror_2003} and one reason we need $M$ to be Riemannian is that the kinetic term must be proportional to $g_{ij} \tfrac{d\phi^{i}}{dt} \tfrac{d\phi^{j}}{dt}$.  

This system can be canonically quantized and after some details which we omit, one finds that the Hilbert space $\mathscr{H}$ is simply the exterior algebra of $M$
\begin{equation}
\mathscr{H} = \Lambda^{*}(M) \coloneqq \bigoplus_{k=0}^{m} \Omega^{k}(M)
\end{equation}
such that pure bosonic states live in $\Omega^{k}(M)$ with $k$ even, and fermionic states live in $\Omega^{k}(M)$ with $k$ odd.  The $L^{2}$-inner product on forms is the natural one described in Section \ref{sec:YMACfunnn}.  The supersymmetry partners are the operators 
\begin{equation}
Q = d \,\,\,\,\,\,\,\,\,\,\,\,\,\,\,\,\,\, \overline{Q} = d^{*}
\end{equation} 
where $d$ is the exterior derivative and $d^{*}$ is the adjoint with respect to the inner product (see Lemma \ref{lemmy:adjjjoplema}).   

Recalling (\ref{eqn:anticommHAM}), the Hamiltonian is related to the anticommutator $\{d, d^{*} \} = \Delta = 2H$, and we therefore identify the space of ground states of the theory with the zero-modes of the Laplacian $\Delta$, i.e. the harmonic forms $\mathcal{H}^{*}(M,g)$.  The physical fact that $\mathcal{H}^{*}(M,g)$ should be isomorphic to the $Q$-cohomology is consistent with the identification $Q=d$ in addition to the Hodge-de Rham isomorphism $\mathcal{H}^{*}(M,g) \cong H_{dR}^{*}(M)$.  

We can now compute the Witten index (\ref{eqn:WittINDexpo}) in this setting.  Because the ground states are given by the harmonic forms, and the Betti numbers $b_{k}(M)$ encode the number of harmonic $k$-forms, we have
\begin{equation}
\text{Tr}_{\mathscr{H}}(-1)^{F}e^{-\beta H} = \sum_{k-\text{even}} b_{k}(M) - \sum_{k-\text{odd}} b_{k}(M) = \sum_{k=0}^{m} (-1)^{k} b_{k}(M)
\end{equation}
which is none other than the Euler characteristic of $M$.  In addition to the above computation, a fundamental principle of quantum mechanics guarantees the same result by performing a path integral.  This principle gives rise to the equalities
\begin{equation}\label{eqn:pathintBettiredWI}
\sum_{k=0}^{m} (-1)^{k} b_{k}(M) = \text{Tr}_{\mathscr{H}} (-1)^{F} e^{-\beta H} = \int \mathcal{D} \phi \mathcal{D} \psi e^{-\mathcal{S}}
\end{equation}
and one can show that the path integral on the righthand side localizes to the expected integral representation of the Euler class involving the Riemann curvature.  If $M$ is compact and K\"{a}hler, then (\ref{eqn:pathintBettiredWI}) agrees with (\ref{eqn:GENGBONNTHM}).  The Witten index as an integral of the Euler class over $M$ is sometimes called a spacetime index, and one benefit of considering the more general field-theoretic formulation is that the index makes sense even if $M$ is non-compact.  

The takeaway from this brief survey should be that certain supersymmetric indices in physics give rise to interesting mathematical invariants.  Moreover, a physical duality may induce a mathematical conjecture, or an index theorem, such as (\ref{eqn:pathintBettiredWI}).  In fact, we have seen this earlier when interpreting the Donaldson-Thomas invariants as a supersymmetric index; note how analogous (\ref{eqn:pathintBettiredWI}) and (\ref{eqn:DTBPSindTHM}) are!  The Donaldson-Thomas invariants as integrals over the virtual class are analogous to the integral of the Euler class over $M$, while the Behrend-weighted Euler characteristic is analogous to $\text{Tr}_{\mathscr{H}}(-1)^{F}e^{-\beta H}$, and makes sense even when the moduli scheme is not proper.  Many of these principals reappear when studying more exotic indices.

\section{Equivariant Cohomology and Atiyah-Bott Localization} \label{sec:EquivCohABLoc}

Let $X$ be a compact complex manifold of dimension $d$ and let $E$ be a complex vector bundle on $X$ of rank $k$.  A \emph{multiplicative class} $A$ is a characteristic class applied to complex vector bundles on $X$, such that $A(E_{1} \oplus E_{2}) = A(E_{1})A(E_{2})$.  Associated to such a multiplicative class is a formal power series $f(x)$ such that $A(L) = f(c_{1}(L))$ for all line bundles $L$, which implies by the splitting principle $A(E) = \prod f(x_{i})$, where $x_{1}, \ldots, x_{k}$ are the formal Chern roots of $E$.  Motivated by the Euler characteristic example, we will study \emph{topological indices} $\Phi(X)$ defined by integrating a multiplicative class $A(E)$ over $X$
\begin{equation}\label{eqn:generalindex}
\Phi(X) = \int_{X} A(E) = \int_{X} \prod_{i=1}^{k} f(x_{i}).
\end{equation}
One very powerful method for computing such a $\Phi(X)$ in the case where $X$ carries a torus action is Atiyah-Bott localization.  If $X$ is not compact or infinite dimensional, then $\int_{X} A(E)$ does not make sense, but if the fixed locus is finite dimensional and compact, one can still define the topological index $\Phi(X)$ by performing the integral equivariantly.  It will therefore be necessary for us to understand equivariant cohomology and localization in some detail.

\subsection{An Introduction to Equivariant Cohomology}

Let $M$ be a smooth manifold carrying a left action by a group.  Though one can be more general, we will only consider the cases where the group is an algebraic torus $T = (\mathbb{C}^{*})^{n}$ or a real torus $T=(S^{1})^{n}$.  The $T$-equivariant cohomology of $M$, which we denote $H_{T}^{*}(M)$ is an invariant encoding topological data of both $M$ and the manner in which $T$ acts on $M$.  Unless stated otherwise, we will take coefficients in $\mathbb{Q}$.  

The general philosophy of equivariant cohomology is that instead of forgetting the group action, one should set up a cohomology theory remembering this data.  The main application will be that one can localize a problem to the fixed point locus of the action without losing information.  We will also require it to be functorial with respect to equivariant maps.  As an example, if $T$ acts freely on $M$, then $M/T$ is a smooth manifold, and the $T$-equivariant cohomology should agree with the ordinary cohomology of $M/T$
\begin{equation}\label{eqn:equivcohomfreeact}
H^{*}_{T}(M) = H^{*}(M/T).
\end{equation}  
However, when the action is not free, the ordinary cohomology of the orbit space does not retain enough information, and may in fact be trivial.   

The idea of Borel is that one should find a space homotopic to $M$ which encodes the $T$-action on $M$, and on which $T$ acts freely.  Moreover, this space should be canonical up to homotopy.  Equivariant cohomology should then be defined as the ordinary cohomology of the orbit space.  Milnor had proven that for all topological groups $G$, there exists a contractible space $EG$, unique up to homotopy, on which $G$ acts freely from the right \cite{milnor_construction_1956, milnor_construction_1956-1}.  The orbit space $BG \coloneqq EG/G$ is called the \emph{classifying space} of $G$, and $EG \to BG$ is the universal principal $G$-bundle on the classifying space.  

Returning to the case at hand, because $T$ acts freely on $ET$, it also acts freely on $ET \times M$, with diagonal action defined by $t \cdot (x,m) = (xt , t^{-1}m)$, for all $t \in T$.  The quotient is denoted
\begin{equation}
M_{T} \coloneqq ET \times_{T} M = \big(ET \times M\big)/T
\end{equation}
and can also be interpreted as an associated fiber bundle to the universal principal $T$-bundle $ET \to BT$.

\begin{defn}
The (Borel) $T$-equivariant cohomology of $M$ with rational coefficients is defined by
\begin{equation}
H^{*}_{T}(M) \coloneqq H^{*}(M_{T})
\end{equation}
where $H^{*}(M_{T})$ is the ordinary cohomology of the orbit space $M_{T}$.  
\end{defn}

The $T$-equivariant cohomology is independent of the choice of $ET$.  Because $M_{T} \to BT$ is a fiber bundle with fiber $M$, we can consider the inclusion $\iota_{M}: M \hookrightarrow M_{T}$, and we get the following ring homomorphism induced by pullback
\begin{equation}\label{eqn:incposhformwa}
\iota_{M}^{*}: H_{T}^{*}(M) \longrightarrow H^{*}(M).
\end{equation}
Moreover, pulling back by the equivariant map $M \to \text{pt}$ turns the $T$-equivariant cohomology $H_{T}^{*}(M)$ into a $H_{T}^{*}(\text{pt})$-module.  It is clear that in the case of $M=\text{pt}$, we have $M_{T} = BT$.  Therefore, the $T$-equivariant cohomology of a point agrees with the standard group cohomology
\begin{equation}\label{eqn:groupcohompoints}
H^{*}_{T}(\text{pt}) = H^{*}(BT).
\end{equation}
One can also show that in the case where $T$ acts freely on $M$, $M_{T}$ is homotopy equivalent to $M/T$ which verifies that the $T$-equivariant cohomology of $M$ agrees with the ordinary cohomology of $M/T$ (\ref{eqn:equivcohomfreeact}).  

It is important to find an explicit set of generators of the $T$-equivariant cohomology of a point.  First consider $T=\mathbb{C}^{*}$ or $T=S^{1}$.  We clearly have the following two ways of writing $\mathbb{P}^{N}$ as a quotient
\begin{equation}
\big(\mathbb{C}^{N+1}-\{0\}\big) / \mathbb{C}^{*} = \mathbb{P}^{N} = S^{2N+1}/ S^{1}.
\end{equation}
If we formally take $N \to \infty$, then $\mathbb{C}^{\infty}-\{0\}$ is a contractible space on which $\mathbb{C}^{*}$ acts freely from the right and similarly, $S^{\infty}$ is a contractible space on which $S^{1}$ acts freely from the right.  We then have
\begin{equation}\label{eqn:ETBTcrassspace}
\big(\mathbb{C}^{\infty}-\{0\}\big) / \mathbb{C}^{*} = \mathbb{P}^{\infty} = S^{\infty}/ S^{1}.
\end{equation}
Therefore, if $T=\mathbb{C}^{*}$, we should choose $ET = \mathbb{C}^{\infty} - \{0\}$, and for $T = S^{1}$, we choose $ET = S^{\infty}$.  By (\ref{eqn:ETBTcrassspace}), we see that in either case the classifying space is $BT = \mathbb{P}^{\infty}$.  Given that $H^{*}(\mathbb{P}^{\infty})$ is generated by a single element $\epsilon \in H^{2}(\mathbb{P}^{\infty})$ without relations, by (\ref{eqn:groupcohompoints}) we have
\begin{equation}
H^{*}_{T}(\text{pt}) = H^{*}(\mathbb{P}^{\infty}) = \mathbb{Q}[\epsilon].
\end{equation}

It is not hard to see that for either $T = (\mathbb{C}^{*})^{n}$ or $T=(S^{1})^{n}$, the classifying space is $(\mathbb{P}^{\infty})^{n}$, and
\begin{equation}\label{eqn:genequivcohomopppt}
H^{*}_{T}(\text{pt}) = H^{*}((\mathbb{P}^{\infty})^{n}) = \mathbb{Q}[\epsilon_{1}, \ldots, \epsilon_{n}] 
\end{equation}
with generators $\epsilon_{i} \in H^{2}((\mathbb{P}^{\infty})^{n})$.  In this case of $M = \text{pt}$, the map $\iota_{\text{pt}}^{*}$ in (\ref{eqn:incposhformwa}) is given by setting $\epsilon_{1} = \ldots = \epsilon_{n}=0$.  If $T = (\mathbb{C}^{*})^{n}$, then the $\epsilon_{i}$ are complex indeterminates while if $T=(S^{1})^{n}$, they are real indeterminates.  In fact, in either case the $\epsilon_{i}$ may be realized as generators of the Lie algebra $\mathfrak{t}$ of $T$, or equivalently as first Chern classes of (pullbacks of) canonical line bundles on $\mathbb{P}^{\infty}$.  This can be understood as follows.

\begin{defn}\label{defn:equivvvcharrcresass}
For $T$ either a real or algebraic torus, let $\widetilde{V}$ be a $T$-equivariant complex vector bundle on $M$.  This means that the action of $T$ on $M$ lifts to an action on $\widetilde{V}$, linear on each fiber.  Then $V_{T} \coloneqq ET \times_{T} \widetilde{V}$ is a complex vector bundle on $M_{T}$.  If $c_{*}(-)$ is any characteristic class, we define the equivariant class $c_{*}(-)_{T}$ by
\begin{equation}\label{eqn:defnequivclasss}
c_{*}(\widetilde{V})_{T} \coloneqq c_{*}(V_{T}) \in H^{*}_{T}(M).  
\end{equation}
\end{defn}

\noindent Associated to $T$, we can also consider the lattice $\Gamma_{T}$ of characters $\chi: T \to \mathbb{C}^{*}$, or $\chi: T \to S^{1}$ if $T$ is real.  
\begin{defn}\label{defn:constLBB}
Given a character $\chi \in \Gamma_{T}$, denote also by $\chi$ the one-dimensional $T$-module on which $(t_{1}, \ldots, t_{n}) \in T$ acts as multiplication by $\chi \big((t_{1}, \ldots t_{n})\big)$.  We will equivalently consider $\chi$ to be an equivariant line bundle over a point.  Recalling that $ET \to (\mathbb{P}^{\infty})^{n}$ is a principal $T$-bundle, we can use $\chi$ to construct the associated line bundle $\mathscr{L}_{\chi} \coloneqq ET \times_{\chi} \chi$ on $(\mathbb{P}^{\infty})^{n}$.
\end{defn}

\noindent For all $i=1, \ldots, n$ we have canonical characters $(t_{1}, \ldots, t_{n}) \mapsto t_{i}$, and making a mild (but conventional) abuse of notation, we denote also by $t_{i}$ the one-dimensional $T$-module on which $(t_{1}, \ldots, t_{n}) \in T$ acts as multiplication by $t_{i}$.  Denote by $\mathscr{L}_{i}$ the line bundle on $(\mathbb{P}^{\infty})^{n}$ constructed as in Definition \ref{defn:constLBB}, and note that
\begin{equation}
\mathscr{L}_{i}  \cong p_{i}^{*} \mathcal{O}_{\mathbb{P}^{\infty}}(-1)
\end{equation}
where $p_{i} : (\mathbb{P}^{\infty})^{n} \to \mathbb{P}^{\infty}$ is the projection onto the $i$-th factor, and $\mathcal{O}_{\mathbb{P}^{\infty}}(-1)$ is the tautological line bundle on $\mathbb{P}^{\infty}$.  Because $t_{i}$ is an equivariant line bundle over a \emph{point}, it produces no invariants via characteristic classes in classical cohomology.  But noting that for $M =\text{pt}$, $M_{T} = (\mathbb{P}^{\infty})^{n}$, by the definition (\ref{eqn:defnequivclasss}) of equivariant characteristic classes, we have
\begin{equation}
c_{1}(t_{i})_{T} = c_{1}(\mathscr{L}_{i}).
\end{equation}

The connection with the Lie algebra $\mathfrak{t}$ of the torus comes in the form of the isomorphism $\mathfrak{t}^{\smvee} \cong \Gamma_{T} \otimes_{\mathbb{Z}} \mathbb{C}$ between the complexified character lattice and the dual of the Lie algebra.  If $T$ is a real torus, one takes instead the realification of $\Gamma_{T}$.  These results culminate in the following proposition, whose proof is straightforward and left to the reader.

\begin{proppy}
Given a real or algebraic torus $T$ with Lie algebra $\mathfrak{t}$, we have the following isomorphism of lattices
\begin{equation}
\Gamma_{T} \cong H^{2}_{T}(\text{pt}; \mathbb{Z})
\end{equation}
by mapping a character $\chi$ to the first Chern class $c_{1}(\mathscr{L}_{\chi})$.  Using the relationship noted just above between $\mathfrak{t}^{\smvee}$ and $\Gamma_{T}$, we have the degree-doubling isomorphism
\begin{equation}
\mathbb{Q}[\mathfrak{t}] \coloneqq \text{Sym}(\mathfrak{t}^{\smvee}) \xrightarrow{\hspace*{0.2cm} \sim \hspace*{0.2cm}} H^{*}_{T}(\text{pt}).
\end{equation}
\end{proppy}

\begin{rmk}\label{rmk:multintttwarning}
By the above proposition, the generators $\epsilon_{i}$ of the full $T$-equivariant cohomology of a point (\ref{eqn:genequivcohomopppt}) can equivalently be thought of as generators of the Lie algebra $\mathfrak{t}$ or as equivariant first Chern classes $\epsilon_{i} = c_{1}(t_{i})_{T} = c_{1}(\mathscr{L}_{i})$.  This induces an interpretation of the parameters
\begin{equation}
t_{i} = e^{-\epsilon_{i}} = \text{ch}(\mathscr{L}_{i}^{\smvee})
\end{equation}
as equivalently elements of $\mathbb{C}^{*}$, or as the Chern character of the line bundle $\mathscr{L}^{\smvee}_{i} \to (\mathbb{P}^{\infty})^{n}$, or finally as one-dimensional $T$-modules.  The reader should be mindful of the multiple interpretations.    
\end{rmk}

\subsection{An Introduction to Atiyah-Bott Localization}\label{subsec:IntrABBBLOCC}

Atiyah-Bott localization is a powerful tool for computing integrals over smooth compact complex manifolds carrying a torus action.  The general idea is that one can ``localize" the integral to a sum of integrals, one over each of the fixed components on which the torus acts trivially.  Therefore, let us begin with some generalities where $M$ is a smooth manifold on which $T$ acts trivially.  In such a case we have $M_{T} = M \times (\mathbb{P}^{\infty})^{n}$ from which it follows that
\begin{equation}
H_{T}^{*}(M) = H^{*}(M) \otimes_{\mathbb{Q}} \mathbb{Q}[\epsilon_{1}, \ldots, \epsilon_{n}].
\end{equation}
In other words, the $T$-equivariant cohomology of $M$ factors into a product of the ordinary cohomology of $M$ and the $T$-equivariant cohomology of a point.  We allow $T$ to be either a real or algebraic torus.  Since the $\epsilon_{l}$ can be interpreted as generators of the Lie algebra of $T$, they will be either real or complex parameters, depending on which case we have.  

\begin{defn}\label{defn:defnequivLBfpr}
Let $M$ be a smooth manifold on which $T$ acts trivially.  For all $l=1, \ldots, n$ and for $k \in \mathbb{Z}$ we define an equivariant line bundle $\mathbb{C}_{l}(M, k)$ on $M$ via the relation
\[(t_{1}, \ldots, t_{n}) \cdot (m, z) \sim (m, t_{l}^{k} z)\]
for all $m \in M, z \in \mathbb{C}$, and all elements $(t_{1}, \ldots, t_{n})$ of the torus $T$.  
\end{defn}

\noindent One should think of $\mathbb{C}_{l}(M, k)$ as the trivial line bundle in the ordinary sense, but carrying a fiberwise action by $T$ via multiplication by $t_{l}^{k}$.  We refer to the integer $k$ as the \emph{weight} of the equivariant bundle.  In the case of a point, $\mathbb{C}_{l}(\text{pt}, k)$ is simply a $T$-module, and coincides with the one-dimensional $T$-module denoted $t_{l}^{k}$ (we remind the reader of Remark \ref{rmk:multintttwarning} where we warn about common multiple interpretations of symbols).  

Using Definition \ref{defn:equivvvcharrcresass} we can construct from $\mathbb{C}_{l}(M, k)$ an associated line bundle on $M \times (\mathbb{P}^{\infty})^{n}$.  This line bundle is trivial over $M$ and isomorphic to $\mathscr{L}_{l}^{k}$ over $(\mathbb{P}^{\infty})^{n}$.  We therefore see that 
\begin{equation}\label{eqn:firchcrassequivlb}
c_{1}\big(\mathbb{C}_{l}(M, k) \big)_{T} = k \, \epsilon_{l} \in H_{T}^{*}(M).  
\end{equation}

\begin{defn}
Let $T$ act trivially on a smooth manifold $M$ and let $E \to M$ be an ordinary complex vector bundle.  For a choice of $k_{l} \in \mathbb{Z}$ for each $l =1, \ldots, n$ an equivariant lift of $E$ is an equivariant bundle  
\[ \widetilde{E} \coloneqq E \, \bigotimes_{l=1}^{n} \mathbb{C}_{l}(M, k_{l}).\]  
\end{defn}

\noindent Again using Definition \ref{defn:equivvvcharrcresass} we can construct from $\widetilde{E}$ the bundle $E_{T}$ on $M \times (\mathbb{P}^{\infty})^{n}$ and compute Chern classes.  The bundle $E_{T}$ is isomorphic to $E$ along $M$ and isomorphic to $\otimes_{l=1}^{n} \mathscr{L}_{l}^{k_{l}}$ along $(\mathbb{P}^{\infty})^{n}$.  Using (\ref{eqn:firchcrassequivlb}) along with the fact that the first Chern class is additive on tensor products, we see
\begin{equation}
c_{1}(\widetilde{E})_{T} = c_{1}(E) + \sum_{l=1}^{n} k_{l} \epsilon_{l}.
\end{equation}

For the duration of this section we will let $X$ be a compact complex manifold of dimension $d$ carrying a $T$-action.  Define $F_{1}, \ldots, F_{p}$ to be the connected components of the fixed locus on which $T$ acts trivially.  Each $F_{j}$ is itself a smooth compact complex manifold.  There is a canonical lift of the action on $X$ to the tangent bundle $T_{X}$.  By the splitting principle we can \emph{formally} decompose $T_{X} = L_{1} \oplus \cdots \oplus L_{d}$ as a direct sum of line bundles.  Let $x_{i} = c_{1}(L_{i})$ be the Chern roots.  Using the canonical lift of the action to $T_{X}$, we can restrict $L_{i}$ to a fixed component $F_{j}$ and lift the action to get the equivariant line bundle 
\begin{equation}
\widetilde{L_{i} |_{F_{j}}} = L_{i} |_{F_{j}} \, \bigotimes_{l=1}^{n} \mathbb{C}_{l}(F_{j}, k_{l, i}^{(j)})
\end{equation}
for some weights $k_{l, i}^{(j)}$.  We can do this for all $i,j$.  We can therefore give the following \emph{formal} expression for the full equivariant tangent bundle restricted to a fixed component
\begin{equation}
\widetilde{T_{X}|_{F_{j}}} = \bigoplus_{i=1}^{d} \widetilde{L_{i} |_{F_{j}}}.
\end{equation}
Using the additivity of the first Chern class on tensor products as well as (\ref{eqn:firchcrassequivlb}) we can compute the \emph{equivariant Chern roots} of $\widetilde{T_{X}|_{F_{j}}}$ to be
\begin{equation}\label{eqn:equivvChrootequivbundr}
c_{1}\big(\widetilde{L_{i} |_{F_{j}}}\big)_{T} = x_{i} + \sum_{l=1}^{n} k_{l,i}^{(j)} \epsilon_{l} \in H_{T}^{2}(F_{j}, \mathbb{Z}).  
\end{equation}

In order to state and apply the Atiyah-Bott localization theorem, we need to understand the normal bundles $\mathcal{N}_{j}$ to each $F_{j}$ as an equivariant bundle.  Because $T$ acts trivially on $F_{j}$, the directions (labeled by $i=1, \ldots, d$) tangent to $F_{j}$ have weights $k^{(j)}_{l, i}=0$ for all $l =1, \ldots, n$.  Therefore, the equivariant normal bundle $\widetilde{\mathcal{N}_{j}}$ to $F_{j}$ is spanned by directions such that $k^{(j)}_{l, i} \neq 0$ for \emph{some} $l =1, \ldots, n$.  We are in particular interested in the equivariant Euler class of $\widetilde{\mathcal{N}_{j}}$ which we know should be the product of all the equivariant Chern roots computed in (\ref{eqn:equivvChrootequivbundr}).  We therefore get
\begin{equation}
e(\widetilde{\mathcal{N}_{j}})_{T} = \prod_{I}\bigg( x_{i} + \sum_{l=1}^{n} k_{l, i}^{(j)} \epsilon_{l}\bigg)
\end{equation}
where the $I$ is meant to indicate that the product is taken over all $i$ such that $k_{l, i}^{(j)} \neq 0$ for some $l =1, \ldots, n$.

\begin{thm}[\bfseries Atiyah-Bott \cite{atiyah_moment_1984}]
If we denote by $\mathbb{Q}(\epsilon_{1}, \ldots, \epsilon_{n})$ the field of fractions of the $T$-equivariant cohomology ring of a point $H_{T}^{*}(\text{pt}) = \mathbb{Q}[\epsilon_{1}, \ldots, \epsilon_{n}]$, we have the following isomorphism
\begin{equation}\label{eqn:ABlocalllzz}
H_{T}^{*}(X) \otimes_{H_{T}^{*}(\text{pt})}\mathbb{Q}(\epsilon_{1}, \ldots, \epsilon_{n}) \xrightarrow{\hspace*{0.2cm} \simeq \hspace*{0.2cm}} \bigoplus_{j=1}^{p} H_{T}^{*}(F_{j}) \otimes_{H_{T}^{*}(\text{pt})}\mathbb{Q}(\epsilon_{1}, \ldots, \epsilon_{n})
\end{equation}
induced by the map $\alpha \mapsto \sum_{j=1}^{p} \iota_{j}^{*}(\alpha) e(\widetilde{\mathcal{N}_{j}})^{-1}_{T}$, where the equivariant map $\iota_{j} : F_{j} \hookrightarrow X$ is the inclusion.  
\end{thm}

\noindent At least for our purposes, the above result is the primary payoff of studying equivariant cohomology: one can localize to the fixed locus of a torus action without losing information.  

The main application will be in evaluating (and in some cases \emph{defining}) integrals over $X$.  As described in an appendix of \cite{gasparim_nekrasov_2010}, we want the integral $\int_{X} : H_{T}^{*}(X) \to H_{T}^{*}(\text{pt})$ to satisfy the following conditions:
\begin{enumerate}
\item If $\alpha \in H^{q}_{T}(X)$ with $q < 2d$, then $\int_{X} \alpha =0$.  
\item If $\alpha \in H^{q}_{T}(X)$ with $q \geq 2d$, then $\int_{X} \alpha \in H^{q-2d}_{T}(\text{pt})$.  
\end{enumerate}
\noindent We note that if $q$ is odd, then $H^{q-2d}_{T}(\text{pt})=0$, while if $q$ is even, then $H^{q-2d}_{T}(\text{pt})$ consists of homogeneous polynomials in $\epsilon_{1}, \ldots, \epsilon_{n}$ of degree $q/2 - d$.  In particular, if $\alpha \in H^{2d}_{T}(X)$, then $\int_{X} \alpha \in \mathbb{Q}$ is independent of the equivariant parameters.

The following corollary of (\ref{eqn:ABlocalllzz}), known as the \emph{Atiyah-Bott localization formula} gives the following expression for the integral over $X$ of a $T$-equivariant cohomology class $\alpha \in H^{*}_{T}(X)$,
\begin{equation}\label{eqn:FullABLOCFORM}
\int_{X} \alpha = \sum_{j=1}^{p} \int_{F_{j}} \frac{\iota_{j}^{*} \alpha}{e(\widetilde{\mathcal{N}_{j}})_{T}}.
\end{equation}
By the discussion above, if $\alpha \in H^{2d}_{T}(X)$, then $\int_{X} \alpha \in \mathbb{Q}$ and the equivariant parameters which appear explicitly on the righthand side of (\ref{eqn:FullABLOCFORM}) must drop out after a simplification.

To end the section where it began, let $\Phi(X)$ be a topological index (\ref{eqn:generalindex}) associated to a multiplicative class with formal power series $f(x)$.  Applying the Atiyah-Bott localization formula (\ref{eqn:FullABLOCFORM}), we can compute
\begin{equation}\label{eqn:Atiyah-BottLocalization}
\Phi(X) = \sum_{j=1}^{p} \int_{F_{j}} \frac{1}{e(\widetilde{\mathcal{N}_{j}})_{T}} \prod_{i=1}^{d} f\bigg(x_{i} - \sum_{l=1}^{n} k_{l, i}^{(j)}\epsilon_{l}\bigg).
\end{equation}
Notice that if $X$ is a compact complex manifold, then $\Phi(X)$ is well-defined and can in theory be computed without localization.  It follows that the equivariant parameters $\epsilon_{l}$ must remarkably drop out of the righthand side of (\ref{eqn:Atiyah-BottLocalization}).  We will see an example of this in the next section.  We will also be interested in topological indices $\Phi(X)$ in the case when $X$ is either non-compact or infinite-dimensional.  However, if $X$ carries a $T$-action such that $X^{T}$ is a finite dimensional, smooth compact complex submanifold of $X$ with components $F_{j}$, then we can \emph{define} $\Phi(X)$ by the righthand side of (\ref{eqn:Atiyah-BottLocalization}).  The price we pay is that $\Phi(X)$ will now be a rational function in the equivariant parameters.

\section{The Hirzebruch $\boldmath{\chi_{y}}$-genus} \label{sec:Hirzzchiy}

Let us begin with some formal preliminaries which will be of use going forward.  Let $X$ be a compact complex manifold with $E$ a holomorphic vector bundle on $X$ of rank $n$.  We define the following elements in the formal power series ring in the variable $t$ with coefficients in holomorphic vector bundles
\begin{equation}\label{eqn:denfformbund}
\Lambda_{t}E \coloneqq \bigoplus_{p=0}^{n} t^{p} \Lambda^{p} E \,\,\,\,\,\,\,\,\,\,\,\,\,\,\,\,\,\,\, S_{t}E \coloneqq \bigoplus_{p=0}^{\infty} t^{p} S^{p} E
\end{equation}
where $\Lambda^{p}E$ and $S^{p}E$ are the $p$-th exterior and symmetric powers of $E$, respectively.  

\begin{lemmy} \label{lemmy:lambdaSChern}
If $E$ is a holomorphic vector bundle on $X$ of rank $n$ and with Chern roots $x_{i}$, then
\begin{equation} \label{eqn:Chcharformbunf}
\text{ch}(\Lambda_{t}E) = \prod_{i=1}^{n} (1 + t e^{x_{i}}) \,\,\,\,\,\,\,\,\,\,\,\,\,\,\,\,\,\,\,\, \text{ch}(S_{t}E) = \prod_{i=1}^{n} \frac{1}{1-te^{x_{i}}}.
\end{equation}
\end{lemmy}

\begin{proof}
Some standard facts are $\Lambda_{t}(E_{1} \oplus E_{2}) \cong \Lambda_{t}E_{1} \otimes \Lambda_{t}E_{2}$ and $S_{t}(E_{1} \oplus E_{2}) \cong S_{t}E_{1} \otimes S_{t}E_{2}$.  Applying the splitting principle and formally writing $E = L_{1} \oplus \cdots \oplus L_{n}$, we therefore have
\begin{equation}\label{eqn:lambdaStensor}
\Lambda_{t}E = \bigotimes_{i=1}^{n} \Lambda_{t} L_{i} \,\,\,\,\,\,\,\,\,\,\,\,\,\,\,\,\,\,\,\,  S_{t}E = \bigotimes_{i=1}^{n} S_{t} L_{i}.
\end{equation}
The exterior powers of a line bundle terminate after only two terms, so we have $\Lambda_{t} L_{i} = \mathcal{O}_{X} \oplus t L_{i}$.  But in the case of the symmetric powers we have
\begin{equation}
S_{t} L_{i} = \mathcal{O}_{X} \oplus t L_{i} \oplus t^{2} L_{i}^{2} \oplus \cdots
\end{equation}
By the additivity of the Chern character across direct sums and multiplicativity across tensor products, we have
\begin{equation}
\setlength{\jot}{8pt}
\begin{split}
& \,\,\,\,\,\,\,\,\,\,\,\,\,\,\,\,\,\,\,\,\,\,\,\, \text{ch}(\Lambda_{t} L_{i}) = \text{ch}(\mathcal{O}_{X}) + t \text{ch}(L_{i}) = 1 + te^{x_{i}} \\
& \text{ch}(S_{t}L_{i}) = \text{ch}(\mathcal{O}_{X}) + t \text{ch}(L_{i}) + t^{2} \text{ch}(L_{i})^{2} + \cdots = \frac{1}{1-te^{x_{i}}}.
\end{split}
\end{equation}
Finally, the multiplicativity of the Chern character on tensor products applied to (\ref{eqn:lambdaStensor}) yields the desired conclusions.  
\end{proof}

We can now finally give an example of a topological index, which is a refinement of the Euler characteristic in the case of a compact K\"{a}hler manifold.  

\begin{defn}
Let $X$ be a compact complex manifold.  The $\chi_{y}$-genus of $X$, denoted $\chi_{y}(X)$, is defined to be the holomorphic Euler characteristic of the bundle-valued polynomial $\Lambda_{y}T^{\smvee}_{X}$
\begin{equation} \label{eqn:inddefnchiy}
\chi_{y}(X) = \chi(X, \Lambda_{y}T^{\smvee}_{X}) = \int_{X} \text{ch}(\Lambda_{y}T^{\smvee}_{X}) \text{td}(X)
\end{equation}
where $T^{\smvee}_{X}$ is the holomorphic cotangent bundle.  
\end{defn}

\noindent Note that a priori, $y$ is simply a formal variable, but it can be analytically continued to an honest complex variable.  The $\chi_{y}$-genus is not a topological invariant of $X$ -- it generally depends on an almost complex structure.  However, we will soon see that it encodes certain topological invariants of $X$.  

\begin{proppy}
The $\chi_{y}$-genus is a topological index associated to multiplicative class $\text{ch}(\Lambda_{y}T^{\smvee}_{X}) \text{td}(X)$ and can be expressed in terms of the formal Chern roots $x_{1}, \ldots, x_{n}$ of the tangent bundle $T_{X}$
\begin{equation}\label{eqn:chiygenmultgen}
\chi_{y}(X) = \int_{X} \prod_{i=1}^{n} x_{i} \frac{1+ye^{-x_{i}}}{1-e^{-x_{i}}}
\end{equation}
where $n$ is the dimension of $X$.  It can also be expressed in terms of the Hodge numbers of $X$ as
\begin{equation}\label{eqn:chiygenhodgenumer}
\chi_{y}(X) = \sum_{p,q=0}^{n} (-1)^{q} y^{p} h^{p,q}(X).
\end{equation}
\end{proppy}

\begin{proof}
The Todd class $\text{td}(X)$ is a multiplicative class which can be written in terms of the formal Chern roots, and by (\ref{eqn:Chcharformbunf}) we have an expression for the multiplicative class $\text{ch}(\Lambda_{y}T^{\smvee}_{X})$.  Let us record both
\begin{equation} \label{eqn:ToddCrass}
\text{td}(X) = \prod_{i=1}^{n} \frac{x_{i}}{1-e^{-x_{i}}} \,\,\,\,\,\,\,\,\,\,\,\,\,\,\, \text{ch}(\Lambda_{y}T^{\smvee}_{X}) = \prod_{i=1}^{n} (1+ye^{-x_{i}})
\end{equation}
noting that if $x_{i}$ are the Chern roots of $T_{X}$, then $-x_{i}$ are the Chern roots of $T^{\smvee}_{X}$.  Therefore, by (\ref{eqn:inddefnchiy}), the expression of $\chi_{y}(X)$ in terms of the Chern roots, is as claimed.  To prove the second claim, by the linearity of the Chern character across direct sums we have
\begin{equation}
\chi_{y}(X) = \int_{X} \text{ch}\bigg(\bigoplus_{p=0}^{n} y^{p} \Lambda^{p}T^{\smvee}_{X} \bigg) \text{td}(X) = \sum_{p=0}^{n} y^{p} \chi(X, \Omega^{p})
\end{equation}
where $\Omega^{p} = \Lambda^{p}T^{\smvee}_{X}$ is the sheaf of holomorphic $p$-forms.  Moreover, the holomorphic Euler characteristic $\chi(X, \Omega^{p})$ can be expanded as the alternating sum of Hodge numbers.  This gives
\begin{equation}
\chi_{y}(X) = \sum_{p=0}^{n} y^{p} \chi(X, \Omega^{p}) = \sum_{p,q=0}^{n} (-1)^{q} y^{p} h^{p,q}(X).
\end{equation}
\end{proof}

\begin{cory}\label{cory:speccofchiygenn}
For any compact complex manifold $X$, the $\chi_{y}$-genus specializes to the holomorphic Euler characteristic of $\mathcal{O}_{X}$
\begin{equation}
\chi_{0}(X) = \chi(X, \mathcal{O}_{X}),
\end{equation}
and if $X$ is additionally K\"{a}hler, it can also be specialized to the topological Euler characteristic
\begin{equation}
\chi_{-1}(X) = \chi(X).
\end{equation}
\end{cory}

There is a similar looking one-variable refinement of the topological Euler characteristic called the \emph{Poincar\'{e} polynomial}, which is the generating function of the Betti numbers
\begin{equation}
P(y) \coloneqq \sum_{k=0}^{2n} b_{k}(X) y^{k} = \sum_{p,q=0}^{n} (-1)^{p+q} h^{p,q}(X) y^{p+q}.
\end{equation}
If $X$ is compact, recall that Poincar\'{e} duality implies $b_{k}(X) = b_{2n-k}(X)$.  It is therefore clear that $y^{-n}P(y)$ is invariant under $y \mapsto y^{-1}$.  If $X$ is additionally K\"{a}hler, the $\chi_{y}$-genus and the Poincar\'{e} polynomial are two independent one-variable refinements of $\chi(X)$.  They are essentially different ways of packaging the Hodge numbers into a polynomial.  The coefficients of $y$ in $\chi_{y}(X)$ are (alternating) sums over the diagonals of the Hodge diamond while the coefficients of $y$ in $P(y)$ are sums over the rows.

\begin{Ex}
Though not needed for this simple example, to illustrate its utility let us use Atiyah-Bott localization to compute $\chi_{y}(\mathbb{P}^{2})$.  The projective plane $\mathbb{P}^{2}$ has a canonical action by the torus
\[ T = \big\{ (u:v:w) \in \mathbb{P}^{2} \, \big| \, u,v,w \neq 0 \big\} \subset \mathbb{P}^{2},\]
which is isomorphic to $(\mathbb{C}^{*})^{2}$.  The action is defined by component-wise multiplication, and there are clearly three isolated fixed points
\[F_{1} = (1:0:0) \,\,\,\,\,\,\,\,\,\,\,\, F_{2} = (0:1:0) \,\,\,\,\,\,\,\,\,\,\,\, F_{3}=(0:0:1).\]
In the standard affine open charts centered at each $F_{j}$, it is clear that the torus action is given by coordinate-wise multiplication
\begin{equation}\label{eqn:affacttt}
\big(\zeta_{1}^{(j)}, \zeta_{2}^{(j)}\big) \cdot \big(x^{(j)}, y^{(j)}\big) = \big( \zeta_{1}^{(j)} x^{(j)} , \zeta_{2}^{(j)} y^{(j)}\big),
\end{equation}
for $j=1,2, 3$, where $\big(\zeta_{1}^{(j)}, \zeta_{2}^{(j)}\big) \in (\mathbb{C}^{*})^{2}$.  However, the crucial point to notice in this example is that by projecting the global $T$-action on $\mathbb{P}^{2}$ into the three affine charts, the $\big(\zeta_{1}^{(j)}, \zeta_{2}^{(j)}\big)$ are in fact, not independent.  If we define $(t_{1}, t_{2}) \coloneqq (\zeta_{1}^{(1)}, \zeta_{2}^{(1)})$, then we have
\begin{equation} \label{eqn:alphasss}
\big(\zeta_{1}^{(1)}, \zeta_{2}^{(1)}\big) = \big( t_{1}, t_{2}\big), \,\,\,\,\, \big(\zeta_{1}^{(2)}, \zeta_{2}^{(2)}\big) = \big(t_{1}^{-1}, t_{1}^{-1} t_{2}\big), \,\,\,\,\, \big(\zeta_{1}^{(3)}, \zeta_{2}^{(3)}\big) = \big( t_{2}^{-1}, t_{1} t_{2}^{-1}\big).
\end{equation}  
The next step is to lift the $T$-action to an action on $T_{\mathbb{P}^{2}}$.  Because the fixed points are isolated, we are simply turning the tangent space to $\mathbb{P}^{2}$ at the three fixed points into $T$-modules, precisely in the way prescribed by the action (\ref{eqn:affacttt}) in terms of the $t_{i}$ in (\ref{eqn:alphasss}).  Recall that we understand $t_{i}$, for $i=1,2$ to be one-dimensional $T$-modules where $(t_{1}, t_{2}) \in T$ acts as multiplication by $t_{i}$.  We therefore get equivariant bundles
\begin{equation}\label{eqn:equivCHROOT}
\begin{split}
& \widetilde{T_{\mathbb{P}^{2}}|_{F_{1}}} = t_{1} + t_{2} \\ 
&\widetilde{T_{\mathbb{P}^{2}}|_{F_{2}}} = t_{1}^{-1} + t_{1}^{-1} t_{2} \\ 
& \widetilde{T_{\mathbb{P}^{2}}|_{F_{3}}} = t_{2}^{-1} + t_{1}t^{-1}_{2}.
\end{split}
\end{equation}
The equivariant Chern roots can be easily determined from (\ref{eqn:equivCHROOT}) in terms of the equivariant parameters $\epsilon_{1} = c_{1}(t_{1})_{T}$ and $\epsilon_{2} = c_{1}(t_{2})_{T}$.  We can also record the Euler classes of the equivariant normal bundles $\widetilde{\mathcal{N}_{j}}$ in equivariant cohomology
\begin{equation}
e(\widetilde{\mathcal{N}_{1}})_{T} = \epsilon_{1} \epsilon_{2} \,\,\,\,\,\,\,\,\,\,\, e(\widetilde{\mathcal{N}_{2}})_{T} = \epsilon_{1}(\epsilon_{1}-\epsilon_{2}) \,\,\,\,\,\,\,\,\,\,\,\, e(\widetilde{\mathcal{N}_{3}})_{T} = \epsilon_{2}(\epsilon_{2} - \epsilon_{1}).  
\end{equation}
In terms of the parameters $t_{1} = e^{-\epsilon_{1}}$ and $t_{2} = e^{-\epsilon_{2}}$, we can finally apply Atiyah-Bott localization (\ref{eqn:Atiyah-BottLocalization}) to get
\begin{equation}
\chi_{y}(\mathbb{P}^{2}) = \frac{(1+yt_{1})}{(1-t_{1})} \frac{(1+y t_{2})}{(1-t_{2})} + \frac{(1+yt_{1}^{-1})}{(1-t_{1}^{-1})} \frac{(1+y t_{1}^{-1} t_{2})}{(1-t_{1}^{-1} t_{2})} + \frac{(1+yt_{2}^{-1})}{(1-t_{2}^{-1})} \frac{(1+y t_{1}t_{2}^{-1})}{(1-t_{1}t_{2}^{-1})}.
\end{equation}
Now, because $\mathbb{P}^{2}$ is compact and $\chi_{y}(\mathbb{P}^{2})$ is a polynomial in $y$, there is no dependence on the parameters $t_{1}$ and $t_{2}$, and they must drop out entirely after simplifying.  Indeed, after a trivial but tedious simplification we find
\begin{equation}
\chi_{y}(\mathbb{P}^{2}) = y^{2} -y + 1.
\end{equation}
As a consistency check, one verifies that $\chi_{-1}(\mathbb{P}^{2}) =3$, which matches the Euler characteristic of $\mathbb{P}^{2}$.  
\end{Ex}

\subsection{The \boldmath{$\chi_{y}$}-genus as a Supersymmetric Index}

The $\chi_{y}$-genus arises as an index in a one-dimensional supersymmetric sigma model whose target space is a $2n$-dimensional Riemannian manifold $(X, g)$ which also has a complex structure.  We will build on our discussion in Section \ref{subsec:EULcharrasWITTIND} where we study one-dimensional sigma models on Riemannian manifolds and identify the Witten index as the Euler characteristic.  The complex structure induces a splitting $d = \partial + \overline{\partial}$ of the exterior derivative, which refines the Hilbert space of the theory to 
\begin{equation}
\mathscr{H} = \bigoplus_{p,q=0}^{n} \Omega^{p,q}(X)
\end{equation}
where $\Omega^{p,q}(X)$ is the vector space of $(p,q)$-forms on $X$.  We also get a splitting of the fermion operator $F = F_{+} + F_{-}$ where eigenstates of $F_{+}$ with eigenvalue $p$ are forms in $\mathscr{H}$ whose holomorphic part is a $p$-form.  Similarly for $F_{-}$ and forms whose anti-holomorphic part is a $q$-form.  

We can define the following one-variable supersymmetric index
\begin{equation}\label{eqn:chiyGENSUSYYINDEX}
\text{Tr}_{\mathscr{H}} ( -1)^{F_{-}}y^{F_{+}} e^{- \beta H}
\end{equation}
which turns out to be independent of $\beta$.  In this theory we have supersymmetries which we identify as $\partial$ and $\overline{\partial}$ along with their adjoints.  Similarly to the Witten index, the index (\ref{eqn:chiyGENSUSYYINDEX}) localizes onto the ground states of the theory which we identify as the $\overline{\partial}$-cohomology.  But this is none other than the Dolbeault cohomology 
\begin{equation}\label{eqn:Dolbeaultthhm}
\bigoplus_{p,q=0}^{n} H^{p,q}(X) \cong \bigoplus_{p,q=0}^{n} H^{q}(X, \Omega^{p}).
\end{equation}
This allows us to evaluate the index (\ref{eqn:chiyGENSUSYYINDEX}), which we find to be exactly the $\chi_{y}$-genus of $X$
\begin{equation}
\text{Tr}_{\mathscr{H}}(-1)^{F_{-}}y^{F_{+}} e^{- \beta H} = \sum_{p,q=0}^{n} (-1)^{q}y^{p} h^{p,q}(X)= \chi_{y}(X).  
\end{equation}

If $X$ is compact and K\"{a}hler, then by the Hodge decomposition we can identify (\ref{eqn:Dolbeaultthhm}) with the ordinary de Rham cohomology $H^{*}(X, \mathbb{C})$.  It follows that by setting $y=-1$ we recover the Witten index, which is consistent with the mathematical results in Corollary \ref{cory:speccofchiygenn}.  The $\chi_{y}$-genus is therefore a supersymmetric index, refining the Witten index such that the variable $y$ tracks the charge $F_{+}$.  It is notable that both examples of topological indices presented thus far find a very natural home in the context of supersymmetric sigma models.

\section{The Ordinary Elliptic Genus}

We now want to introduce the (ordinary) elliptic genus which is an invariant of a compact complex manifold $X$.  It will be shown to be a refinement of both the Euler characteristic and the $\chi_{y}$-genus.  We will also interpret the elliptic genus to be an equivariant version of the $\chi_{y}$-genus applied to the infinite dimensional loop space of the manifold.  Just as we saw with the Euler characteristic and the $\chi_{y}$-genus, the elliptic genus can be expressed as a topological index as an integral over $X$ of a multiplicative class.  In physics, it is a supersymmetric index in a superconformal sigma model with target space $X$.  Unlike the previous indices we have seen, if certain geometric conditions on $X$ are satisfied, the elliptic genus will be an automorphic form known as a weak Jacobi form.

\begin{defn}
Let $X$ be a compact complex manifold of dimension $d$ with holomorphic tangent bundle $T_{X}$.  Using (\ref{eqn:denfformbund}) we define the following formal object
\begin{equation}\label{eqn:formbundleval}
\mathbb{E}_{q, y} = (-y)^{-d/2} \bigotimes_{n=1}^{\infty} \bigg( \Lambda_{y q^{n-1}}T^{\smvee}_{X} \otimes \Lambda_{y^{-1}q^{n}} T_{X} \otimes S_{q^{n}}T^{\smvee}_{X} \otimes S_{q^{n}}T_{X}\bigg)
\end{equation}
which aside from the factor $(-y)^{-d/2}$ is a power series in $q$ and $y^{\pm}$ whose coefficients are holomorphic vector bundles.  The elliptic genus of $X$ is simply the holomorphic Euler characteristic of $\mathbb{E}_{q,-y}$ 
\begin{equation}\label{eqn:indtheordefnELLGEN}
\text{Ell}_{q,y}(X) = \chi\big( X, \mathbb{E}_{q,-y} \big) = \int_{X} \text{ch}(\mathbb{E}_{q, -y}) \text{td}(X).
\end{equation}  
\end{defn}

The elliptic genus is a topological index associated to multiplicative class $\text{ch}(\mathbb{E}_{q, -y}) \text{td}(X)$.  Moreover, because the coefficient of a fixed power of $q$ and $y$ in $\mathbb{E}_{q,-y}$ is a holomorphic vector bundle, by Hirzebruch-Riemann-Roch the coefficients of the expansion of $\text{Ell}_{q,y}(X)$ in $q$ and $y$ are integers
\[ \text{Ell}_{q,y}(X) \in y^{-d/2} \, \mathbb{Z} \llbracket q, y^{\pm} \rrbracket. \]
At this stage, the definition of $\mathbb{E}_{q,y}$ may seem completely unmotivated.  We now want to show that the elliptic genus as defined is in fact an equivariant (and regularized) version of the $\chi_{y}$-genus of the loop space of $X$.

\subsection{The Elliptic Genus as an Equivariant Index on Loop Space}

Given any smooth manifold $M$, the \emph{loop space} $\mathcal{L}M$ is defined to be the space of continuous maps
\[ \mathcal{L}M \coloneqq \text{Map}(S^{1}, M). \]
The loop space is an infinite dimensional manifold, but it carries a natural $S^{1}$-action, defined simply by rotation of the domain circle.  The infinite dimensional nature of the loop space makes it difficult to work with, though one redeeming feature is that we can always find an embedded copy of $M$ within $\mathcal{L}M$ as the fixed point locus of this action.  The fixed point locus $\mathcal{L}M^{S^{1}}$ corresponds to constant maps, which is clearly identified with $M$ itself.  We therefore have
\[\mathcal{L}M^{S^{1}} \cong M \lhook\joinrel\longrightarrow \mathcal{L}M.\]

The case we are interested in is when $X$ is a compact complex manifold of dimension $d$.  Let $\mathcal{L}X$ be the loop space, and understand $X \hookrightarrow \mathcal{L}X$ to be the fixed point locus of the natural $S^{1}$-action.  We would like to compute the $\chi_{y}$-genus of $\mathcal{L}X$ making use of Atiyah-Bott localization (\ref{eqn:Atiyah-BottLocalization}) applied to the $S^{1}$-action.  However, because $\mathcal{L}X$ is infinite-dimensional it does not have a well-defined $\chi_{y}$-genus.  Nonetheless, notice that the righthand side of (\ref{eqn:Atiyah-BottLocalization}) \emph{is} well-defined since the fixed-point locus $X$ is smooth and compact.  We can therefore still compute the righthand side of (\ref{eqn:Atiyah-BottLocalization}) at the expense of the equivariant parameters not dropping out of the expression, in general.  This quantity will then be defined to be the equivariant $\chi_{y}$-genus of $\mathcal{L}X$.  We will show that this is precisely the elliptic genus with the variable $q$ arising as the equivariant parameter corresponding to the single generator of $S^{1}$.

\begin{proppy}
The restriction of the tangent bundle $T_{\mathcal{L} X}$ to the fixed locus $X$ can be lifted to an equivariant bundle
\begin{equation}\label{eqn:S1acttttangbundreee}
\widetilde{T_{\mathcal{L} X}|_{X}} = \bigoplus_{n \in \mathbb{Z}} T_{X} \otimes \mathbb{C}(X, n).  
\end{equation}
Here, $\mathbb{C}(X, n)$ is the equivariant line bundle on $X$ where $q = e^{2 \pi i \theta}$ with $\theta \in S^{1}$ acts via multiplication by $q^{n}$.  This is consistent with Definition \ref{defn:defnequivLBfpr} except we have dropped the subscript since $\text{dim}(S^{1})=1$.  
\end{proppy}

\begin{proof}
Given any loop $f : S^{1} \to X$, the elements of $(T_{\mathcal{L}X})_{f}$ are interpreted as infinitesimal deformations of the loop.  Such deformations are given by sections of $f^{*}T_{X}$.  In the case where $f$ is a constant map with image $x \in X$, the deformations are elements of $f^{*}(T_{X})_{x}$ which are simply maps $\delta f : S^{1} \to (T_{X})_{x}$.  Given an infinitesimal deformation $\delta f : S^{1} \to (T_{X})_{x}$ of a constant map $f$ with image $x \in X$, we can take the Fourier expansion
\begin{equation}\label{eqn:FourexxxpS1actt}
\delta f( \theta) = \sum_{n \in \mathbb{Z}} a(n) q^{n}, \,\,\,\,\,\,\,\,\,\,\,\,\,\,\,\, q = e^{2 \pi i \theta}
\end{equation}
with $a(n) \in (T_{X})_{x}$ for all $n \in \mathbb{Z}$.  By associating an infinitesimal deformation of a constant map to its spectrum of Fourier coefficients, we get a map $(T_{\mathcal{L} X})_{x} \longrightarrow \bigoplus_{n \in \mathbb{Z}} (T_{X})_{x}$ which is an isomorphism.  This then gives the isomorphism of infinite dimensional vector spaces
\begin{equation}\label{eqn:decomptanspls}
T_{\mathcal{L} X}\big|_{X} \cong \bigoplus_{n \in \mathbb{Z}} T_{X}.
\end{equation}   
To give $T_{\mathcal{L}X}|_{X}$ the structure of an equivariant bundle, we want to lift the $S^{1}$-action on $\mathcal{L}X$.  But by (\ref{eqn:FourexxxpS1actt}), the $n$-th summand $T_{X}$ on the righthand side of (\ref{eqn:decomptanspls}) is acted on by $q=e^{2 \pi i \theta}$ as multiplication by $q^{n}$.  This gives the $n$-th summand of (\ref{eqn:decomptanspls}) as the equivariant bundle $T_{X} \otimes \mathbb{C}(X, n)$, proving the proposition.  
\end{proof}

By the splitting principle we can formally write $T_{X} = L_{1} \oplus \cdots \oplus L_{d}$, and define $x_{i} = c_{1}(L_{i})$ to be the Chern roots.  By (\ref{eqn:S1acttttangbundreee}) we therefore have the following decomposition
\begin{equation}
\widetilde{T_{\mathcal{L}X}|_{X}} = \bigoplus_{n \in \mathbb{Z}}\bigg( \bigoplus_{i=1}^{d}L_{i} \otimes \mathbb{C}(X,n)\bigg).
\end{equation}
This allows us to compute the equivariant Chern roots of $\widetilde{T_{\mathcal{L}X}|_{X}}$ to be
\begin{equation}\label{eqn:equivChrootloopsp}
c_{1} \big( L_{i} \otimes \mathbb{C}(X, n) \big)_{S^{1}} = x_{i} + nu
\end{equation}
where we have used that the first Chern class is additive under tensor products, and we have defined $u$ by $H_{S^{1}}^{*}(\text{pt}) = \mathbb{Q}[u]$.  In other words, $u = c_{1}(\mathscr{L})$ where $\mathscr{L} \cong \mathcal{O}_{\mathbb{P}^{\infty}}(-1)$.  Recalling Remark \ref{rmk:multintttwarning} where we warn about multiple interpretations of symbols, $u$ is interpreted as a generator of the Lie algebra $\mathbb{R}$ of $S^{1}$.  Defining $q = e^{-u}$, we moreover have $q = \text{ch}(\mathscr{L}^{\smvee})$. 

Let $\mathcal{N}$ be the normal bundle to $X$ in $\mathcal{L}X$.  We want to lift the $S^{1}$-action on $\mathcal{L}X$ to give the equivariant normal bundle $\widetilde{\mathcal{N}}$.  By (\ref{eqn:S1acttttangbundreee}), the directions with weight $n=0$ correspond to $T_{X}$ itself, while those summands with $n \neq 0$ span the equivariant normal bundle
\begin{equation}
\widetilde{\mathcal{N}} = \bigoplus_{n \in \mathbb{Z} \setminus \{ 0 \}}  \bigg(T_{X} \otimes \mathbb{C}(X, n) \bigg).
\end{equation}
The equivariant Euler class is the product over all of the equivariant Chern roots (\ref{eqn:equivChrootloopsp}).  We therefore find
\begin{equation}\label{eqn:equivvEulclnblpsp}
e(\widetilde{\mathcal{N}})_{S^{1}} = \prod_{i=1}^{d} \prod_{n \in \mathbb{Z} \setminus \{0\}} \big(x_{i} + nu \big).
\end{equation}

\begin{defn}
Let $X$ be a compact complex manifold of dimension $d$.  Motivated by the righthand side of (\ref{eqn:Atiyah-BottLocalization}), and using the equivariant Chern roots (\ref{eqn:equivChrootloopsp}), as well as the equivariant Euler class of $\widetilde{\mathcal{N}}$ (\ref{eqn:equivvEulclnblpsp}), we define the equivariant $\chi_{y}$-genus of the loop space $\mathcal{L}X$ to be
\begin{equation}\label{eqn:EquivChiYGen}
\chi_{y}(\mathcal{L}X ; q)  = \int_{X} \prod_{j=1}^{d} \frac{\prod_{n \in \mathbb{Z}} \frac{x_{j} + nu}{1-e^{-x_{j} -nu}} \big(1 + y e^{-x_{j} -nu}\big)}{\prod_{n \in \mathbb{Z} \setminus \{0\}} (x_{j} + nu)} = \int_{X} \prod_{j=1}^{d} x_{j} \prod_{n \in \mathbb{Z}} \frac{\big(1+ yq^{n} e^{-x_{j}}\big)}{\big(1-q^{n}e^{-x_{j}}\big)}
\end{equation}
where the second equality is a trivial cancellation of terms as well as the definition of the parameter $q = e^{-u}$. 
\end{defn}

There is one subtlety arising from the infinite product in the above definition.  Technically, convergence is problematic in the $y$ variable, and we must impose zeta function regularization to get a well-defined index $\chi_{y}(\mathcal{L}X; q)^{\text{reg}}$.  The details of this regularization will emerge in the proof of the following proposition.  

\begin{proppy}
For any compact complex manifold $X$, the elliptic genus $\text{Ell}_{q,y}(X)$ coincides with the (regularized) equivariant $\chi_{y}$-genus $\chi_{-y}(\mathcal{L}X; q)^{\text{reg}}$ of the loop space
\begin{equation}
\text{Ell}_{q,y}(X) = \chi_{-y}(\mathcal{L}X; q)^{\text{reg}}
\end{equation}
after analytically continuing the equivariant $\chi_{y}$-genus in the variable $q$.  
\end{proppy}  

\begin{proof}
Let us begin by rewriting $\chi_{y}(\mathcal{L}X; q)$ from (\ref{eqn:EquivChiYGen}) in a more convenient fashion.  We can separate the infinite product over $n \in \mathbb{Z}$ into the $n=0$, $n>0$ and $n<0$ contributions.  It is easy to see we have
\begin{equation}
\chi_{y}(\mathcal{L}X; q) = \int_{X} \prod_{j=1}^{d} x_{j} \frac{1+ ye^{-x_{j}}}{1-e^{-x_{j}}} \prod_{n \geq 1} \bigg( \frac{1+q^{n}ye^{-x_{j}}}{1-q^{n}e^{-x_{j}}}\bigg)\bigg( \frac{1+q^{-n}y e^{-x_{j}}}{1-q^{-n}e^{-x_{j}}}\bigg)
\end{equation}
and with a little more work rewriting the second term in parentheses, we have
\begin{equation} \label{eqn:simplASS}
\chi_{y}(\mathcal{L}X; q) = \int_{X} \prod_{j=1}^{d} x_{j} \frac{1+ ye^{-x_{j}}}{1-e^{-x_{j}}} \prod_{n \geq 1} \bigg( \frac{1+q^{n}ye^{-x_{j}}}{1-q^{n}e^{-x_{j}}}\bigg)\bigg( \frac{1+q^{n}y^{-1} e^{x_{j}}}{1-q^{n}e^{x_{j}}}\bigg)(-y).
\end{equation}
It is now clear where we must employ regularization.  The first two terms in parentheses produce no convergence issues, yet thanks to the $(-y)$ factor (which is under the infinite product!) we have a divergent contribution of
\[\prod_{j=1}^{d} \prod_{n \geq 1} (-y) = \prod_{j=1}^{d} (-y)^{1 + 1 + \ldots} = (-y)^{d \zeta(0)} = (-y)^{-d/2}\] 
which we have tamed by a standard zeta function regularization, using $\zeta(0) = -1/2$.  With this, we may define the \emph{regularized} equivariant $\chi_{y}$-genus 
\begin{equation}\label{eqn:finalREGCHIY}
\chi_{y}(\mathcal{L}X; q)^{\text{reg}} = (-y)^{-d/2} \int_{X} \prod_{j=1}^{d} x_{j} \frac{1+ ye^{-x_{j}}}{1-e^{-x_{j}}} \prod_{n \geq 1} \frac{(1+q^{n}ye^{-x_{j}})}{(1-q^{n}e^{-x_{j}})}  \frac{(1+q^{n}y^{-1} e^{x_{j}})}{(1-q^{n}e^{x_{j}})}
\end{equation}
where the $(-y)^{d/2}$ factor arises from the above regularization, while all else is as in (\ref{eqn:simplASS}).  

Recalling the definition of $\mathbb{E}_{q,-y}$ in (\ref{eqn:formbundleval}), we now want to use properties of the Chern character and Todd class to simplify the expression of the elliptic genus $\text{Ell}_{q,y}(X)$.  We can compute
\begin{equation}
\setlength{\jot}{8pt}
\begin{split}
\text{ch}(\mathbb{E}_{q, -y}) & = y^{-d/2} \prod_{n=1}^{\infty} \text{ch}(\Lambda_{-yq^{n-1}}T^{\smvee}_{X}) \, \text{ch}(\Lambda_{-y^{-1}q^{n}}T_{X}) \, \text{ch}(S_{q^{n}}T^{\smvee}_{X}) \, \text{ch}(S_{q^{n}}T_{X})\\
& = y^{-d/2} \prod_{i=1}^{d} \prod_{n=1}^{\infty} \frac{(1-yq^{n-1} e^{-x_{i}})(1-y^{-1}q^{n}e^{x_{i}})}{(1-q^{n}e^{-x_{i}})(1-q^{n}e^{x_{i}})} \\
& = y^{-d/2} \prod_{i=1}^{d} (1-ye^{-x_{i}}) \prod_{n=1}^{\infty} \frac{(1-yq^{n} e^{-x_{i}})(1-y^{-1}q^{n}e^{x_{i}})}{(1-q^{n}e^{-x_{i}})(1-q^{n}e^{x_{i}})} 
\end{split}
\end{equation}
where the first equality follows from the behavior of the Chern character on tensor products, the second follows by applying (\ref{eqn:Chcharformbunf}), while the final equality is trivial reindexing.  Finally, recalling the expression (\ref{eqn:ToddCrass}) for $\text{td}(X)$ in terms of the Chern roots, we can apply the definition (\ref{eqn:indtheordefnELLGEN}) of the elliptic genus $\text{Ell}_{q,y}(X)$ to compute
\begin{equation} \label{eqn:finalELLGEN}
\text{Ell}_{q,y}(X) = y^{-d/2} \int_{X} \prod_{j=1}^{d} x_{j} \frac{1- ye^{-x_{j}}}{1-e^{-x_{j}}} \prod_{n \geq 1}  \frac{(1-q^{n}ye^{-x_{j}})}{(1-q^{n}e^{-x_{j}})}  \frac{(1-q^{n}y^{-1} e^{x_{j}})}{(1-q^{n}e^{x_{j}})}.
\end{equation}
Comparing (\ref{eqn:finalELLGEN}) and (\ref{eqn:finalREGCHIY}) we conclude that $\text{Ell}_{q,y}(X) = \chi_{-y}(\mathcal{L}X; q)^{\text{reg}}$.  
\end{proof}

It is conventional to freely apply the change of variables $q = e^{2 \pi i \tau}$ and $y = e^{2 \pi i z}$.  The elliptic genus $\text{Ell}_{q,y}(X)$ is a holomorphic function on $\mathfrak{H} \times \mathbb{C}$ where $\tau$ is a coordinate on the upper-half plane and $z$ is a coordinate on $\mathbb{C}$.  By the decomposition $\mathbb{E}_{q, -y} = \bigoplus E_{n, l} q^{n} y^{l}$, the elliptic genus admits a Fourier expansion of the form
\begin{equation}\label{eqn:FourDecompELLGEN}
\text{Ell}_{q,y}(X) = \sum_{n \geq 0, l \in \mathbb{Z}} c( n, l) q^{n} y^{l}.
\end{equation}
The integral coefficients $c(n, l)$ have the interpretation of topological indices of the vector bundles $E_{n, l}$
\begin{equation}\label{eqn:topindEllGenCoeff}
c(n, l) = \int_{X} \text{ch}(E_{n, l}) \text{td}(X).
\end{equation}

\begin{cory}\label{cory:corollaryEELLGEN}
For any compact complex manifold $X$, by setting $q=0$ the elliptic genus $\text{Ell}_{q,y}(X)$ specializes to the ordinary $\chi_{y}$-genus, up to a prefactor and a change in the sign of $y$
\begin{equation}
\text{Ell}_{0,y}(X) = y^{-d/2} \chi_{-y}(X).
\end{equation}
By setting $y=1$, it is clear from (\ref{eqn:finalELLGEN}) that all higher powers of $q$ vanish, and we recover the Euler characteristic.  In terms of the Fourier coefficients, we have 
\begin{equation}\label{eqn:EulCharEllGenCoeff}
\text{Ell}_{q,1}(X) = \chi(X) = \sum_{n \geq 0, l \in \mathbb{Z}} c(n, l) q^{n} = \sum_{l \in \mathbb{Z}} c(0,l).
\end{equation}
\end{cory}

Once we have introduced the relevant automorphic forms in the following chapter, we will note that the elliptic genus of a Calabi-Yau manifold is a weak Jacobi form.

\section{Equivariant Topological Indices on $\boldmath{\mathbb{C}^{2}}$} \label{sec:equivvvindiccsec}

With the foundations of equivariant cohomology and localization in place, as well as some examples of interesting topological indices $\Phi(X)$, we would now like to compute these indices for $X = \mathbb{C}^{2}$.  Because $\mathbb{C}^{2}$ is not compact, we cannot integrate ordinary cohomology classes, and the righthand side of (\ref{eqn:generalindex}) is not well-defined.  However, using the canonical algebraic torus action on $\mathbb{C}^{2}$, we can \emph{define} an equivariant topological index $\Phi(\mathbb{C}^{2})$ by applying Atiyah-Bott localization.  

Let $T_{t} = (\mathbb{C}^{*})^{2}$ be the algebraic torus acting on $\mathbb{C}^{2}$ in the natural way by coordinate-wise multiplication.  Clearly, the origin is the only fixed point of the action, and we should regard the tangent space to $\mathbb{C}^{2}$ at the origin as a $T_{t}$-module.  To briefly rehash the constructions of Section \ref{subsec:IntrABBBLOCC}, we denote by $t_{i}$ for $i=1,2$ the one-dimensional $T_{t}$-module on which $(t_{1}, t_{2}) \in T_{t}$ acts as multiplication by $t_{i}$.  We therefore regard $t_{1}$ and $t_{2}$ as $T_{t}$-equivariant line bundles over a point, with corresponding line bundles $\mathscr{L}_{i} \cong p_{i}^{*} \mathcal{O}_{\mathbb{P}^{\infty}}(-1)$ on $(\mathbb{P}^{\infty})^{2}$.  Recall that we write
\begin{equation}
\epsilon_{i} = c_{1}(t_{i})_{T_{t}} = c_{1}(\mathscr{L}_{i})
\end{equation}
as well as $t_{i} = e^{-\epsilon_{i}} = \text{ch}(\mathscr{L}_{i}^{\smvee})$.  We treat $\epsilon_{1}, \epsilon_{2}$ as the generators of the $T_{t}$-equivariant cohomology of a point, which we denote by $\mathbb{Q}[\epsilon_{1}, \epsilon_{2}]$.  The tangent space to $\mathbb{C}^{2}$ at the origin is clearly just the normal bundle $\mathcal{N}$ of the origin in $\mathbb{C}^{2}$.  As a $T_{t}$-equivariant bundle, it can be written as
\begin{equation}
\widetilde{\mathcal{N}} = t_{1} + t_{2}
\end{equation}
or equivalently, as $\mathscr{L}_{1} \oplus \mathscr{L}_{2}$ as a bundle on $(\mathbb{P}^{\infty})^{2}$.  Either way, it is straightforward to read off the equivariant Chern roots, as well as the equivariant Euler class of $\widetilde{\mathcal{N}}$
\begin{equation}
e(\widetilde{\mathcal{N}})_{T_{t}} = \epsilon_{1} \epsilon_{2}.  
\end{equation}

\begin{defn}
Let $\Phi(X)$ be a topological index (\ref{eqn:generalindex}) associated to a multiplicative class with formal power series $f(x)$.  Because $\mathbb{C}^{2}$ is not compact, we cannot simply apply (\ref{eqn:generalindex}), but we can define an equivariant index $\Phi(\mathbb{C}^{2})(t_{1}, t_{2})$ by the righthand side of (\ref{eqn:Atiyah-BottLocalization})
\begin{equation}\label{eqn:maindefneqiovindex}
\Phi(\mathbb{C}^{2})(t_{1}, t_{2}) \coloneqq \frac{1}{e(\widetilde{\mathcal{N}})_{T_{t}}} f(\epsilon_{1}) f(\epsilon_{2}) = \frac{1}{\epsilon_{1} \epsilon_{2}} f(\epsilon_{1}) f(\epsilon_{2}).  
\end{equation}
\end{defn}

\noindent In general, we expect $\Phi(\mathbb{C}^{2})(t_{1}, t_{2})$ to have poles when $t_{1} = 1$ or $t_{2} =1$ (equivalently, $\epsilon_{1} =0$ or $\epsilon_{2} =0$) due to the non-compactness of $\mathbb{C}^{2}$.  Note that $\Phi(\mathbb{C}^{2})(t_{1}, t_{2})$ may also depend on other parameters, as will be the case for the $\chi_{y}$-genus and elliptic genus.

\begin{Ex}[\bfseries The Euler Characteristic]
Recalling the standard index expression for the Euler characteristic of a compact complex manifold, we get an equivariant Euler characteristic by applying (\ref{eqn:maindefneqiovindex})
\begin{equation}
\chi(\mathbb{C}^{2})(t_{1}, t_{2}) = \frac{\epsilon_{1}\epsilon_{2}}{\epsilon_{1}\epsilon_{2}} = 1.
\end{equation}
The equivariant parameters drop out of the Euler characteristic entirely.  We will simply write $\chi(\mathbb{C}^{2}) =1$, which agrees with the topological computation using that $\mathbb{C}^{2}$ is contractible.  
\end{Ex}

\begin{Ex}[\bfseries The Equivariant $\boldmath{\chi_{y}}$-genus]
Recalling the expression for the $\chi_{y}$-genus of a compact complex manifold, we compute the equivariant version by applying (\ref{eqn:maindefneqiovindex})
\begin{equation}
\chi_{y}(\mathbb{C}^{2})(t_{1}, t_{2}) = \frac{1}{\epsilon_{1} \epsilon_{2}} \frac{\epsilon_{1} (1+ye^{-\epsilon_{1}})}{(1-e^{-\epsilon_{1}})}\frac{\epsilon_{2}(1+ye^{-\epsilon_{2}})}{(1-e^{-\epsilon_{2}})} = \frac{(1+yt_{1})(1+yt_{2})}{(1-t_{1})(1-t_{2})}.
\end{equation}
Note that under the variable specialization $y=-1$ we have $\chi_{-1}(\mathbb{C}^{2})(t_{1}, t_{2}) =1$, recovering the Euler characteristic.  As expected, we have singularities when $t_{1}=1$ or $t_{2}=1$, as a result of the non-compactness of $\mathbb{C}^{2}$.  
\end{Ex}

\begin{Ex}[\bfseries The Equivariant Elliptic Genus]
In exactly the same way as the above two examples, one can use the index expression of the ordinary elliptic genus on a compact two-dimensional complex manifold, and apply (\ref{eqn:maindefneqiovindex}) to get the equivariant version
\begin{equation}
\text{Ell}_{q,y}(\mathbb{C}^{2})(t_{1}, t_{2}) = y^{-1} \prod_{n \geq 1} \frac{(1-yq^{n-1}t_{1})(1-y^{-1}q^{n}t_{1}^{-1})(1-yq^{n-1}t_{2})(1-y^{-1}q^{n}t_{2}^{-1})}{(1-q^{n-1}t_{1})(1-q^{n}t_{1}^{-1})(1-q^{n-1}t_{2})(1-q^{n}t_{2}^{-1})}.
\end{equation}
Of course, by letting $q=0$, we specialize to $y^{-1}\chi_{-y}(\mathbb{C}^{2})$, and therefore also the Euler characteristic by further letting $y=1$.
\end{Ex}

It is often desirable to consider not the full $T_{t} = (\mathbb{C}^{*})^{2}$ action on $\mathbb{C}^{2}$, but rather the \emph{diagonal action} generated by $(t, t^{-1}) \in (\mathbb{C}^{*})^{2}$.  If $t_{1} = t_{2}^{-1}$, then $\epsilon_{1} + \epsilon_{2}=0$ which can be written
\[c_{1}( \mathscr{L}_{1} \oplus \mathscr{L}_{2}) = -(\epsilon_{1} + \epsilon_{2}) =0,\]
and thought of as an action preserving the Calabi-Yau or hyperk\"{a}hler structure of $\mathbb{C}^{2}$.  Of course, $T_{t} = (\mathbb{C}^{*})^{2} \subset GL_{2}(\mathbb{C})$ is a toral subgroup whereas the subgroup  generated by $(t, t^{-1})$ descends to a subgroup of $SL_{2}(\mathbb{C})$.  Instead of having independent parameters $t_{1}$ and $t_{2}$, we can specialize any of the above equivariant indices to the case of $t \coloneqq t_{1} = t_{2}^{-1}$.  The specialization of the equivariant elliptic genus will be especially important, so we record the result here
\begin{equation}\label{eqn:equivELLC2diagspecc}
\text{Ell}_{q,y}(\mathbb{C}^{2}; t) = y^{-1} \prod_{n \geq 1} \frac{(1-yq^{n-1}t)(1-y^{-1}q^{n}t^{-1})(1-yq^{n-1}t^{-1})(1-y^{-1}q^{n}t)}{(1-q^{n-1}t)(1-q^{n}t^{-1})(1-q^{n-1}t^{-1})(1-q^{n}t)}.
\end{equation}

\section{An Introduction to Nekrasov Partition Functions}\label{sec:FULLNEkkkSec}

To motivate the partition functions of Nekrasov, we would like to begin by studying instantons on $\mathbb{C}^{2} = \mathbb{R}^{4}$.  Recall that given a compact four-manifold $M$, a $U(r)$ or $SU(r)$ instanton is an ASD unitary connection on a Hermitian vector bundle over $M$.  Equivalently, it is a solution of the Yang-Mills equations and a global minimum of the Yang-Mills functional.  However, because $\mathbb{C}^{2}$ is not compact, in order for the theory to be well-defined, we must consider only field configurations which decay fast enough at infinity.  Well-defined instantons on $\mathbb{R}^{4}$ therefore come from genuine instantons on $S^{4}$ satisfying this property along with requiring trivial gauge transformations at the point at infinity.  These are called framed instantons and we denote by $M_{0}^{\text{reg}}(r,k)$ the moduli space of $SU(r)$ framed instantons on $S^{4}$ with second Chern class $k$.  This is a smooth non-compact hyper-K\"{a}hler manifold of real dimension $4rk$ and is isomorphic to the moduli space $\mathfrak{M}^{\text{loc}}(r,k)$ of locally-free sheaves of rank $r$ on $\mathbb{P}^{2}$ with second Chern class $k$ framed along the line at infinity (this framing will be defined shortly for general torsion-free sheaves on $\mathbb{P}^{2}$).  

The Uhlenbeck (partial) compactification of the genuine instantons $M_{0}^{\text{reg}}(r,k)$ is defined by
\begin{equation}\label{eqn:UHHHHlcomppppar}
M_{0}(r,k) = \cup_{k'=0}^{k} M_{0}^{\text{reg}}(r,k-k') \times \text{Sym}^{k'}(\mathbb{C}^{2})
\end{equation}
and acquires orbifold singularities.  It admits a resolution of singularities $\pi : \mathfrak{M}(r,k) \to M_{0}(r,k)$, where $\mathfrak{M}(r,k)$ is the moduli space of framed torsion-free sheaves on $\mathbb{P}^{2}$ with fixed discrete invariants, as we will explain shortly.  The moduli space $\mathfrak{M}(r,k)$ is called the Gieseker (partial) compactification and it is smooth, non-compact, and hyper-K\"{a}hler of real dimension $4rk$.  To summarize, we have the following diagram relating the moduli spaces introduced above,
\begin{equation}
\begin{tikzcd}
\mathfrak{M}^{\text{loc}}(r,k) \arrow[hookrightarrow]{rr}{} \arrow[swap]{d}{\simeq} & & \mathfrak{M}(r,k)\arrow{d}{\pi} \\
M_{0}^{\text{reg}}(r,k)\arrow[hookrightarrow]{rr}{}&  & M_{0}(r,k)
\end{tikzcd}
\end{equation}

We will begin by properly defining $\mathfrak{M}(r,k)$, and it will be our main focus throughout this section.  One should keep in mind that $\mathfrak{M}(r,k)$ is related to \emph{genuine} instantons on $\mathbb{C}^{2}$ in the fashion described above.  The Nekrasov partition functions are built from integrals of multiplicative classes over $\mathfrak{M}(r,k)$.  But since the moduli space is not compact, the integral must be defined equivariantly using Atiyah-Bott localization.  Physically, the Nekrasov partition functions are (the instanton parts of) partition functions of an $\mathcal{N}=2$ supersymmetric Yang-Mills theory on $\mathbb{C}^{2}$.  The precise physical theory will depend on the choice of multiplicative class to integrate over the moduli space.  

The original physics source for what follows is the paper of Nekrasov \cite{nekrasov_seiberg-witten_2003}, while the mathematical treatment was given in \cite{nakajima_instanton_2005}.  Gasparim and Liu \cite{gasparim_nekrasov_2010} generalized the results by replacing $\mathbb{P}^{2}$ with an arbitrary compact toric surface.

\subsection{Moduli Space of Framed Torsion-Free Sheaves on \boldmath{$\mathbb{P}^{2}$}}\label{subsec:ModFrInstP2}

Consider the projective plane $\mathbb{P}^{2}$ with line at infinity $\ell_{\infty} \cong \mathbb{P}^{1}$.  We will be interested in instanton counting on $\mathbb{C}^{2}$, thought of as the compliment $\mathbb{P}^{2} \setminus \ell_{\infty}$.  

\begin{defn}
A framed torsion-free sheaf on $\mathbb{P}^{2}$ is a pair $(E, \Phi)$ where $E$ is a torsion-free sheaf of rank $r$ on $\mathbb{P}^{2}$, locally-free in a neighborhood of $\ell_{\infty}$, and
\begin{equation}
\Phi : E \big|_{\ell_{\infty}} \xrightarrow{\hspace*{0.2cm} \sim \hspace*{0.2cm}} \mathcal{O}_{\ell_{\infty}}^{\oplus r}
\end{equation}
is a trivialization of $E$ restricted to the line at infinity $\ell_{\infty}$.  We denote by $\mathfrak{M}(r,k)$ the moduli space parameterizing such pairs $(E, \Phi)$ where $r$ is the rank of $E$ and $k=\int_{\mathbb{P}^{2}} c_{2}(E)$.  The moduli space $\mathfrak{M}(r,k)$ is called the Gieseker (partial) compactification of the genuine $SU(r)$ instantons on $\mathbb{C}^{2}$ with topological charge $k$.    
\end{defn}

\noindent If $(E, \Phi)$ is a framed torsion-free sheaf, then $c_{1}(E)=0$.  To see this, note that because $E$ is locally-free around $\ell_{\infty}$, and trivialized along $\ell_{\infty}$, we must have $\int_{\ell_{\infty}} c_{1}(E)=0$.  Because $\ell_{\infty}$ is a hyperplane in $\mathbb{P}^{2}$, this implies $c_{1}(E)=0$.  Therefore, just as with $SU(r)$ instantons on a surface, the only discrete invariants are the rank $r$ and the topological charge $k$.  

\begin{rmk}
What we are calling the topological charge here differs by a minus sign from the original definition in (\ref{eqn:topcharge4man}).  If $c_{1}(E)=0$, then $c_{2}(E) = -\text{ch}_{2}(E)$ so the discrepancy in sign is due to whether one uses $c_{2}(E)$ or $\text{ch}_{2}(E)$ to define $k$.  Because our sheaves originate from ASD bundles, we will use $c_{2}(E)$ which results in $k \geq 0$.  
\end{rmk}

The local structure of the moduli space $\mathfrak{M}(r,k)$ can be studied by way of the deformation-obstruction theory of sheaves, with a mild modification to account for the framing \cite{huybrechts_stable_1992}.  The infinitesimal automorphisms of a pair $(E, \Phi)$ are given by $\text{Ext}^{0}(E, E(- \ell_{\infty}))$, with the Zariski tangent space to $\mathfrak{M}(r,k)$ at $(E, \Phi)$ corresponding to $\text{Ext}^{1}(E, E(- \ell_{\infty}))$.  The obstruction space is of course $\text{Ext}^{2}(E, E(- \ell_{\infty}))$.  The following is a technical result showing that the deformations of framed torsion-free sheaves on $\mathbb{P}^{2}$ are unobstructed with no infinitesimal automorphisms.  A proof can be found in \cite{gasparim_nekrasov_2010, nakajima_instanton_2005}.

\begin{proppy}
Let $(E, \Phi) \in \mathfrak{M}(r,k)$ be a framed torsion-free sheaf on $\mathbb{P}^{2}$.  We have 
\begin{equation}
\text{Ext}^{0}(E, E(-\ell_{\infty})) = \text{Ext}^{2}(E, E(-\ell_{\infty}))=0.
\end{equation}  
\end{proppy}  

Recall that the dimension (or expected dimension) of a moduli space is frequently computed by a Riemann-Roch calculation.  Given any smooth variety $X$ of dimension $d$, for all coherent sheaves $\mathscr{E}$ and $\mathscr{F}$ on $X$ we can define the Euler pairing
\begin{equation}
\chi(\mathscr{E}, \mathscr{F}) \coloneqq \sum_{i=0}^{d} (-1)^{i} \text{dim}\, \text{Ext}^{i}(\mathscr{E}, \mathscr{F}) = \int_{X} \text{ch}^{\smvee}(\mathscr{E}) \text{ch}(\mathscr{F}) \text{td}(X)
\end{equation}
where the second equality requires that $X$ additionally be projective.  When $X$ is a surface, the alternating sum has only three terms, and in the case of $X = \mathbb{P}^{2}$ by the above proposition we have the following
\begin{equation}\label{eqn:dimframedmodspp}
\text{dim} \, \mathfrak{M}(r,k) =  \text{dim} \, \text{Ext}^{1}(E, E(- \ell_{\infty})) = - \chi(E, E(- \ell_{\infty})).
\end{equation}

\begin{cory}
The moduli space $\mathfrak{M}(r, k)$ is a smooth quasi-projective variety of dimension $2rk$.  
\end{cory}

\begin{proof}
The smoothness of the moduli space follows from the vanishing of the obstruction space, established in the above proposition.  By (\ref{eqn:dimframedmodspp}), to find the dimension we must compute 
\begin{equation}
\text{dim} \, \mathfrak{M}(r,k) = -\chi(E, E(- \ell_{\infty})) = - \int_{X} \text{ch}(E^{\smvee}) \text{ch}(E) \text{ch}(\mathcal{O}_{\mathbb{P}^{2}}(- \ell_{\infty})) \text{td}(X).
\end{equation}
If $\nu \in H^{4}(\mathbb{P}^{2}, \mathbb{Z})$ is the Poincar\'{e} dual of a point, i.e is normalized by $\int_{\mathbb{P}^{2}} \nu =1$, then we have
\begin{equation}
\text{ch}(E) = \text{ch}(E^{\smvee}) = (r, 0, - k \nu ).
\end{equation}
Likewise, if $\ell \in H^{2}(X, \mathbb{Z})$ is the Poincar\'{e} dual of the class of a line in $\mathbb{P}^{2}$, then it follows that 
\begin{equation}
\text{ch}\big(\mathcal{O}_{\mathbb{P}^{2}}(- \ell_{\infty})\big) = \big(1, -\ell, \tfrac{1}{2} \ell^{2}\big) \,\,\,\,\,\,\,\,\,\,\,\,\,\,\, \text{td}(\mathbb{P}^{2}) = \big(1, \tfrac{3}{2}\ell, \tfrac{1}{12}(9\ell^{2} + 3 \nu)\big)
\end{equation}
noting that $c_{1}(\mathbb{P}^{2}) = 3\ell$.  The claim then follows by a trivial computation.  
\end{proof}

\begin{Ex}
Perhaps the most important example in this thesis will come by considering the case where $E$ is a rank one torsion-free sheaf.  Because we argued above that $E$ cannot vanish on a divisor in $\mathbb{P}^{2}$, we know $E = \mathscr{I}_{Z}$ is an ideal sheaf with $Z$ a zero-dimensional subscheme.  Moreover, because $\mathscr{I}_{Z}$ is locally-free in a neighborhood of $\ell_{\infty}$, the subscheme must be supported in $\mathbb{C}^{2} = \mathbb{P}^{2} \setminus \ell_{\infty}$.  We therefore have
\begin{equation}\label{eqn:Gieeepartcommmp}
\mathfrak{M}(1,k) \cong \text{Hilb}^{k}(\mathbb{C}^{2}).
\end{equation}
The Hilbert scheme\footnote{Strictly speaking there are no abelian instantons, as discussed in Section 6.2 of \cite{hollowood_matrix_2008}.  The typical resolution is to consider instead $U(1)$ instantons on a non-commutative deformation of Euclidean space.  The Hilbert scheme $\text{Hilb}^{k}(\mathbb{C}^{2})$ is the moduli space of abelian instantons on a non-commutative $\mathbb{C}^{2}$.} is smooth, non-compact for $k>0$, and indeed has dimension $2k$.  

Recall the discussion surrounding (\ref{eqn:UHHHHlcomppppar}).  Since rank one vector bundles have vanishing second Chern class, $M_{0}^{\text{reg}}(1, k-k')$ is empty unless $k-k'=0$, in which case it is a point.  Therefore, the Uhlenbeck partial compactification is
\begin{equation}
M_{0}(1,k) \cong \text{Sym}^{k}(\mathbb{C}^{2})
\end{equation}
and the Gieseker partial compactification (\ref{eqn:Gieeepartcommmp}) is indeed a resolution of singularities.  
\end{Ex}

\subsection{The Torus Action on \boldmath{$\mathfrak{M}(r,k)$} and the Fixed Points}

Consider the torus $\widetilde{T} = T_{t} \times T_{e} = (\mathbb{C}^{*})^{2} \times (\mathbb{C}^{*})^{r}$, where $T_{t}$ is the torus acting naturally on $\mathbb{P}^{2}$, and $T_{e}$ is the maximal torus of $GL_{r}(\mathbb{C})$ consisting of diagonal matrices.  In physics, one should think of $\mathbb{C}^{2}$ as flat, four-dimensional spacetime, with $T_{t}$ the maximal torus of the Lorentz group of spacetime symmetries, and $T_{e}$ is the maximal torus of the gauge group.  Assume that $\ell_{\infty}$ is a $T_{t}$-invariant line in $\mathbb{P}^{2}$.  The moduli space $\mathfrak{M}(r,k)$ carries a natural action by the full torus $\widetilde{T}$ and moreover, has finitely many isolated fixed points.  Following closely the notation of \cite{gasparim_nekrasov_2010, nakajima_instanton_2005}, we will now describe this action as well as the fixed points.

Given $(t_{1}, t_{2}) \in T_{t}$, we can define an automorphism $F_{t_{1}, t_{2}}$ of $\mathbb{P}^{2}$ by $F_{t_{1}, t_{2}}(x) = (t_{1}, t_{2}) \cdot x$.  Because it is $T_{t}$-invariant, note that $\ell_{\infty}$ is preserved by all $F_{t_{1}, t_{2}}$.  In addition, for $\vec{e} = \text{diag}(e_{1}, \ldots, e_{r}) \in T_{e}$ we define the isomorphism $G_{\vec{e}}$ of $\mathcal{O}_{\ell_{\infty}}^{\oplus r}$ by
\begin{equation}
G_{\vec{e}}(s_{1}, \ldots, s_{r}) = (e_{1} s_{1}, \ldots, e_{r} s_{r}).
\end{equation}
This allows us to define a natural action of $\widetilde{T}$ on $\mathfrak{M}(r,k)$.  Given a framed torsion-free sheaf $(E, \Phi) \in \mathfrak{M}(r,k)$, we define
\begin{equation}
(t_{1}, t_{2}, \vec{e}\, ) \cdot (E, \Phi) = \big( (F_{t_{1}, t_{2}}^{-1})^{*}E, \Phi'\big)
\end{equation}
where the new framing $\Phi'$ is defined as follows.  Given the initial framing $\Phi$, the action by $F_{t_{1}, t_{2}}^{-1}$ induces the following commuting diagram of vector bundles on the fixed divisor $\ell_{\infty}$
\begin{equation}
\begin{tikzcd}
E\big|_{\ell_{\infty}} \arrow{rr}{\Phi} & & \mathcal{O}_{\ell_{\infty}}^{\oplus r} \\
(F_{t_{1}, t_{2}}^{-1})^{*} E\big|_{\ell_{\infty}}  \arrow{urr}{\sigma_{t_{1}, t_{2}}} \arrow[swap]{rr}{(F_{t_{1}, t_{2}}^{-1})^{*} \Phi} \arrow{u}{} & & (F_{t_{1}, t_{2}}^{-1})^{*} \mathcal{O}_{\ell_{\infty}}^{\oplus r} \arrow[swap]{u}{}
\end{tikzcd}
\end{equation}
In terms of the diagonal morphism $\sigma_{t_{1}, t_{2}}$ in the above diagram, we define the new framing as the composition 
\begin{equation}
\Phi' = G_{\vec{e}} \circ \sigma_{t_{1}, t_{2}}.  
\end{equation}
To describe the torus action intuitively, the sheaf $E$ is pulled back in a straightforward way involving only $(t_{1}, t_{2}) \in T_{t}$.  The framing however, is transformed by first pulling back along $F_{t_{1}, t_{2}}^{-1}$, and then multiplying by the diagonal matrix $\vec{e} = \text{diag}(e_{1}, \ldots, e_{r}) \in T_{e}$.  The action therefore intertwines $T_{t}$ and $T_{e}$ in a precise manner.  

Let $\mathcal{P}_{r,k}$ denote the set of $r$-tuples of one-dimensional partitions $\vec{Y} = (Y_{1}, \ldots, Y_{r})$ such that $\sum_{\alpha=1}^{r} k_{\alpha} =k$, where $k_{\alpha}$ is the number of boxes in $Y_{\alpha}$.  Let $s \in Y$ be a box in the partition $Y$, interpreted as a Young diagram.  The \emph{arm length} $a_{Y}(s)$ of $s$ is the number of boxes strictly to the right of $s$ while the \emph{leg length} $l_{Y}(s)$ is the number of boxes strictly below $s$.  If $s = (i,j)$, then the \emph{hook length} is defined by 
\begin{equation}\label{eqn:hoookleneqn}
h_{ij} = a_{Y}(s) + l_{Y}(s) +1.
\end{equation}

We would now like to understand the fixed point locus of the $\widetilde{T}$-action on $\mathfrak{M}(r,k)$ described above.  It was shown by Nakajima and Yoshioka \cite{nakajima_instanton_2005} that there are finitely many isolated fixed points parameterized by $r$-tuples of partitions $(Y_{1}, \ldots, Y_{r}) \in \mathcal{P}_{r,k}$.  It is first shown in \cite{nakajima_instanton_2005} that $(E, \Phi) \in \mathfrak{M}(r,k)$ is a fixed point if and only if there exists a splitting $E = \mathscr{I}_{1} \oplus \cdots \oplus \mathscr{I}_{r}$, such that for all $\alpha =1, \ldots, r$: 
\begin{enumerate}
\item $\mathscr{I}_{\alpha}$ is an ideal sheaf on $\mathbb{P}^{2}$ corresponding to a zero-dimensional subscheme $Z_{\alpha}$ supported outside $\ell_{\infty}$ and invariant under $T_{t}$.  
\item The action of $\Phi$ takes $\mathscr{I}_{\alpha} \big|_{\ell_{\infty}}$ to the $\alpha$-th factor of $\mathcal{O}_{\ell_{\infty}}^{\oplus r}$.  
\end{enumerate}
The subschemes $Z_{\alpha}$ being supported outside $\ell_{\infty}$ and invariant under $T_{t}$ clearly force the support to lie entirely at the orgin of $\mathbb{C}^{2} = \mathbb{P}^{2} \setminus \ell_{\infty}$.  This corresponds to a pile of $k_{\alpha} = c_{2}(\mathscr{I}_{\alpha}) = -\text{ch}_{2}(\mathscr{I}_{\alpha})$ boxes at the origin of $\mathbb{C}^{2}$.  Accounting for all $\alpha =1, \ldots, r$ we get an $r$-tuple of partitions or Young diagrams $(Y_{1}, \ldots, Y_{r})$ and by the additivity of the Chern character on direct sums, we have $\sum_{\alpha=1}^{r} k_{\alpha} = k$.  

As we have seen many times in this chapter, associated to the torus $\widetilde{T} = T_{t} \times T_{e}$ we can construct the equivariant line bundles $t_{i}$ and $e_{\alpha}$ on a point, or equivalently line bundles $\mathscr{L}_{i}$ and $\mathscr{L}_{\alpha}$ on $(\mathbb{P}^{\infty})^{r+2}$ for all $i =1,2$ and all $\alpha =1, \ldots, r$.  The equivariant first Chern classes of the bundles are defined to be
\begin{equation}\label{eqn:equivparammsNEkk}
\epsilon_{i} = c_{1}(t_{i})_{\widetilde{T}} = c_{1}(\mathscr{L}_{i}), \,\,\,\,\,\,\,\,\,\,\,\,  a_{\alpha} = c_{1}(e_{\alpha})_{\widetilde{T}} = c_{1}(\mathscr{L}_{\alpha}).
\end{equation}
Identifying $(\epsilon_{1}, \epsilon_{2}, a_{1}, \ldots, a_{r})$ as generators of the Lie algebra of $\widetilde{T}$, we have an equivalent interpretation of $(t_{1}, t_{2}, e_{1}, \ldots, e_{r})$ as,
\begin{equation}
t_{i} = e^{-\epsilon_{i}} = \text{ch}(\mathscr{L}_{i}^{\smvee}), \,\,\,\,\,\,\,\,\,\,\,\, e_{\alpha} = e^{-a_{\alpha}} = \text{ch}(\mathscr{L}_{\alpha}^{\smvee}).  
\end{equation}
Remaining mindful of these various interpretations, we now want to express the tangent space to $\mathfrak{M}(r,k)$ at $\vec{Y}$ as an equivariant vector bundle on a point, in terms of $t_{i}, e_{\alpha}$.  The following theorem is proven in \cite{nakajima_instanton_2005}. 

\begin{thm}
Let $(E, \Phi) \in \mathfrak{M}(r,k)$ be a $\widetilde{T}$-fixed point corresponding as above, to the $r$-tuple of partitions $\vec{Y} =(Y_{1}, \ldots, Y_{r}) \in \mathcal{P}_{r,k}$.  The equivariant decomposition of the tangent space to $\mathfrak{M}(r,k)$ at $\vec{Y}$ is
\begin{equation}\label{eqn:equivvdecompNek}
\begin{split}
& \,\,\,\,\,\,\,\,\,\,\,\,\,\,\,\,\,\,\,\,\,\,\,\,\,\,\,\,\,\,\,\,\,\,\,\,\,\,\,\,\,\,\,\,\,\,\,\,\,\,\,\,\,\, T_{\vec{Y}} \mathfrak{M}(r,k) = \sum_{\alpha, \beta =1}^{r} N_{\alpha, \beta}, \\
& N_{\alpha, \beta} = e_{\alpha} e_{\beta}^{-1} \times \bigg\{ \sum_{s \in Y_{\alpha}} t_{1}^{- l_{Y_{\beta}}(s)} t_{2}^{a_{Y_{\alpha}}(s) +1} + \sum_{t \in Y_{\beta}} t_{1}^{l_{Y_{\alpha}}(t) +1} t_{2}^{-a_{Y_{\beta}}(t)} \bigg\}.
\end{split}
\end{equation} 
It is straightforward to see that there are $2kr$ direct summands, consistent with the dimension of $\mathfrak{M}(r,k)$.  
\end{thm}

Because the equivariant first Chern class is additive under direct sums and tensor products, from (\ref{eqn:equivvdecompNek}) and (\ref{eqn:equivparammsNEkk}) we can extract the equivariant Chern roots.  In preparation to apply Atiyah-Bott localization, we can also record the equivariant Euler class of the normal bundle $T_{\vec{Y}}\mathfrak{M}(r,k)$
\begin{equation}\label{eqn:equivEulClassNekk}
\begin{split}
e\big(T_{\vec{Y}}\mathfrak{M}(r,k)\big)_{\widetilde{T}} = & \prod_{\alpha, \beta =1}^{r}  \prod_{s \in Y_{\alpha}} \bigg( (a_{\alpha}-a_{\beta}) - l_{Y_{\beta}}(s) \epsilon_{1} + \big(a_{Y_{\alpha}}(s) +1\big)\epsilon_{2}\bigg) \\
 & \times \prod_{t \in Y_{\beta}} \bigg((a_{\alpha}-a_{\beta}) + \big(l_{Y_{\alpha}}(t) +1\big)\epsilon_{1} - a_{Y_{\beta}}(t) \epsilon_{2}\bigg).
\end{split}
\end{equation}

\subsection{The Nekrasov Partition Functions}

Let $A$ be a multiplicative class with corresponding formal power series $f(x)$.  We can apply $A$ to the tangent bundle of $\mathfrak{M}(r,k)$, which we abbreviate to $T_{\mathfrak{M}}$.  By formally applying Atiyah-Bott localization with respect to the $\widetilde{T}$-action on $\mathfrak{M}(r,k)$, we can consider
\begin{equation}\label{eqn:genintNekk}
\int_{\mathfrak{M}(r,k)} A\big(T_{\mathfrak{M}}\big)_{\widetilde{T}} \in \mathbb{Q} \llbracket \epsilon_{1}, \epsilon_{2}, \vec{a} \rrbracket_{\mathfrak{m}} \subset \mathbb{Q}( \! (\epsilon_{1}, \epsilon_{2}, \vec{a}) \! )
\end{equation}  
where $A\big(T_{\mathfrak{M}}\big)_{\widetilde{T}} = \prod f(x_{i})$ is a multiplicative equivariant cohomology class and $x_{i}$ are the equivariant Chern roots of $T_{\mathfrak{M}}$.  Also, $\mathbb{Q} \llbracket \epsilon_{1}, \epsilon_{2}, \vec{a} \rrbracket_{\mathfrak{m}}$ is the localization of the formal power series ring $\mathbb{Q} \llbracket \epsilon_{1}, \epsilon_{2}, \vec{a} \rrbracket$ at the maximal ideal $\mathfrak{m}$ generated by $\epsilon_{1}, \epsilon_{2}, a_{1}, \ldots, a_{r}$.  If $f(x)$ happens to be a polynomial, then the integral (\ref{eqn:genintNekk}) lies in $\mathbb{Q} [\epsilon_{1}, \epsilon_{2}, \vec{a}]_{\mathfrak{m}}$.  Applying the Atiyah-Bott localization formula, we see
\begin{equation}
\int_{\mathfrak{M}(r,k)} A\big(T_{\mathfrak{M}}\big)_{\widetilde{T}} = \sum_{\vec{Y} \in \mathcal{P}_{r,k}} \frac{A\big(T_{\mathfrak{M}}\big)_{\widetilde{T}}}{e\big(T_{\vec{Y}} \mathfrak{M}(r,k)\big)_{\widetilde{T}}} = \sum_{\vec{Y} \in \mathcal{P}_{r,k}} \prod \frac{f(x_{i})}{x_{i}}.
\end{equation}

\begin{defn}
Let $r >0$ be a fixed integer and let $A$ be a multiplicative class associated to formal power series $f(x)$.  A Nekrasov partition function on $\mathbb{C}^{2}$ (also called an instanton partition function on $\mathbb{C}^{2}$) is a generating function of the form
\begin{equation}
\begin{split}
\mathcal{Z}^{\mathbb{C}^{2}, r}_{\text{inst}}\big(\epsilon_{1}, \epsilon_{2}, \vec{a}, Q)_{A} & = \sum_{k=0}^{\infty} Q^{k} \int_{\mathfrak{M}(r,k)} A\big(T_{\mathfrak{M}}\big)_{\widetilde{T}} \\
& = \sum_{k=0}^{\infty} Q^{k} \sum_{\vec{Y} \in \mathcal{P}_{r,k}} \prod \frac{f(x_{i})}{x_{i}} \in \mathbb{Q}( \! (\epsilon_{1}, \epsilon_{2}, \vec{a}) \! )\llbracket Q \rrbracket.
\end{split}
\end{equation}
If $f$ is a polynomial, then $\mathcal{Z}^{\mathbb{C}^{2},r}_{\text{inst}}\big(\epsilon_{1}, \epsilon_{2}, \vec{a}, Q)_{A} \in \mathbb{Q}( \epsilon_{1}, \epsilon_{2}, \vec{a})\llbracket Q \rrbracket$.  The class $A$ may possibly depend on other parameters as well (for example, with the $\chi_{y}$-genus or elliptic genus) in which case the partition function will depend on these as well.  Mildly abusing notation, in the examples below we will find a more creative way of decorating the partition function such that the multiplicative class $A$ is clear.    
\end{defn}

\begin{Ex}[\bfseries The Equivariant Volume]
Consider first the case of $A=1$, which means that the Nekrasov partition function is the generating function of equivariant volumes of $\mathfrak{M}(r,k)$
\begin{equation}\label{eqn:partfuncPURENeq2}
\mathcal{Z}^{\mathbb{C}^{2},r}_{\text{inst}}\big(\epsilon_{1}, \epsilon_{2}, \vec{a}, Q\big)_{\text{vol}} = \sum_{k=0}^{\infty} Q^{k} \int_{\mathfrak{M}(r,k)} 1 = \sum_{k=0}^{\infty} Q^{k} \sum_{\vec{Y} \in \mathcal{P}_{r,k}} \frac{1}{e\big(T_{\vec{Y}} \mathfrak{M}(r,k)\big)_{\widetilde{T}}}
\end{equation}
where the equivariant Euler class of the normal bundle is given explicitly in (\ref{eqn:equivEulClassNekk}).  Physically, this choice of $A$ gives the instanton partition function of pure $\mathcal{N}=2$ supersymmetric Yang-Mills theory on $\mathbb{C}^{2}$.  

Specifically in the case of rank $r=1$, the sum over Young diagrams simplifies in a nice way, as was shown in \cite{nakajima_instanton_2005}.  This gives the generating function of equivariant volumes of the Hilbert schemes of points $\text{Hilb}^{k}(\mathbb{C}^{2})$
\begin{equation}
\mathcal{Z}^{\mathbb{C}^{2},1}_{\text{inst}}\big(\epsilon_{1}, \epsilon_{2}, Q)_{\text{vol}} = \text{exp}\bigg(\frac{Q}{\epsilon_{1}\epsilon_{2}}\bigg).
\end{equation}
\end{Ex}

\begin{Ex}[\bfseries The Equivariant Euler Characteristic]
The landmark 1994 paper of Vafa and Witten \cite{vafa_strong_1994-1} proposed that the partition function of topologically twisted $\mathcal{N}=4$ supersymmetric Yang-Mills theory on a four-manifold $M$ is given by the generating function of Euler characteristics of an instanton moduli space on $M$.  Moreover, the partition function is expected to have modular properties inherited from S-duality.  We can choose $M = \mathbb{C}^{2}$ with the moduli space $\mathfrak{M}(r,k)$ of framed instantons and consider the Nekrasov partition function of equivariant Euler characteristics of $\mathfrak{M}(r,k)$.  This corresponds to choosing the multiplicative class $A$ to be the Euler class.  By the localization formula
\begin{equation}\label{eqn:NekkNeqfour}
\mathcal{Z}_{\text{inst}}^{\mathbb{C}^{2}, r}(Q)_{e} = \sum_{k=0}^{\infty} Q^{k} \int_{\mathfrak{M}(r,k)} e\big(T_{\mathfrak{M}}\big)_{\widetilde{T}} = \sum_{k=0}^{\infty} Q^{k} \sum_{\vec{Y} \in \mathcal{P}_{r,k}} 1
\end{equation}
the equivariant Euler classes cancel, leaving no dependence on the equivariant parameters.  In the case of rank one, the above partition function is simply the generating function of partitions of an integer.  Since the $\widetilde{T}$-fixed points of $\mathfrak{M}(r,k)$ are indexed by $r$-tuples of partitions, and the topological charge is the sum of the partition sizes, for $r>1$ we just get $r$ copies of the rank one result
\begin{equation}
\mathcal{Z}_{\text{inst}}^{\mathbb{C}^{2}, r}(Q)_{e} = \prod_{n=1}^{\infty} \frac{1}{(1-Q^{n})^{r}}.
\end{equation}
Notice that because $\chi(\mathbb{C}^{2}) =1$, when $r=1$ this is consistent with G\"{o}ttsche's result on generating functions of Euler characteristics of Hilbert schemes of points \cite{gottsche_betti_1990}.  
\end{Ex}

\begin{Ex}[\bfseries The Equivariant Elliptic Genera]
We can also consider the case where the multiplicative class gives rise to the equivariant elliptic genus, which we will denote $\text{Ell}_{q,y}(\mathfrak{M}(r,k))_{\widetilde{T}}$.  The corresponding Nekrasov partition function is the generating function of elliptic genera of $\mathfrak{M}(r,k)$, which of course only makes sense applying localization
\begin{equation}\label{eqn:genfuncelllgenNekk}
\begin{split}
& \mathcal{Z}^{\mathbb{C}^{2}, r}_{\text{inst}}\big(\epsilon_{1}, \epsilon_{2}, \vec{a}, Q,q,y\big)_{\text{Ell}} = \sum_{k=0}^{\infty} Q^{k} \text{Ell}_{q,y}\big(\mathfrak{M}(r,k)\big)_{\widetilde{T}} = \sum_{k=0}^{\infty} Q^{k} \sum_{\vec{Y} \in \mathcal{P}_{r,k}} \prod_{n=1}^{\infty} \prod_{\alpha, \beta =1}^{r}\\
& \times \prod_{s \in Y_{\alpha}} \frac{(1-y q^{n-1} e_{\alpha} e_{\beta}^{-1} t_{1}^{l_{Y_{\beta}}(s)} t_{2}^{-(a_{Y_{\alpha}}(s)+1)})(1-y^{-1} q^{n} e_{\alpha}^{-1} e_{\beta}t_{1}^{-l_{Y_{\beta}}(s)} t_{2}^{a_{Y_{\alpha}}(s)+1})}{(1-q^{n-1} e_{\alpha} e_{\beta}^{-1} t_{1}^{l_{Y_{\beta}}(s)} t_{2}^{-(a_{Y_{\alpha}}(s)+1)})(1-q^{n} e_{\alpha}^{-1} e_{\beta} t_{1}^{-l_{Y_{\beta}}(s)} t_{2}^{a_{Y_{\alpha}}(s)+1})} \\  
& \times \prod_{t \in Y_{\beta}} \frac{(1-y q^{n-1} e_{\alpha} e_{\beta}^{-1} t_{1}^{-(l_{Y_{\alpha}}(t)+1)} t_{2}^{a_{Y_{\beta}}(t)})(1-y^{-1} q^{n} e_{\alpha}^{-1} e_{\beta}t_{1}^{l_{Y_{\alpha}}(t)+1} t_{2}^{-a_{Y_{\beta}}(t)})}{(1-q^{n-1} e_{\alpha} e_{\beta}^{-1} t_{1}^{-(l_{Y_{\alpha}}(t)+1)} t_{2}^{a_{Y_{\beta}}(t)})(1-q^{n} e_{\alpha}^{-1} e_{\beta} t_{1}^{l_{Y_{\alpha}}(t)+1} t_{2}^{-a_{Y_{\beta}}(t)})}.
\end{split}
\end{equation}
From this general formula, one can specialize in a number of directions.  In particular, one can get the generating functions of the $\chi_{y}$-genera or $\chi_{0}$-genera of the instanton moduli spaces.  For our purposes in the final chapter, we will be particularly interested in the generating function of elliptic genera of the rank one instanton moduli spaces $\text{Hilb}^{k}(\mathbb{C}^{2})$ with the diagonal specialization $t \coloneqq t_{1} = t_{2}^{-1}$ corresponding to $\epsilon_{1} + \epsilon_{2} =0$.  Here we use the notation $\text{Ell}_{q,y}(\text{Hilb}^{k}(\mathbb{C}^{2});t)$.  Recalling the expression (\ref{eqn:hoookleneqn}) for the hook length $h_{ij}$ of the box $s =(i,j)$ in a Young diagram $Y$, it is straightforward to specialize (\ref{eqn:genfuncelllgenNekk}) to
\begin{equation}\label{eqn:sumpartPEllGenC2}
\begin{split}
& \,\,\,\,\,\,\,\,\,\,\,\,\,\,\,\,\,\,\,\,\,\,\,\,\,\,\,\,\,\,\,\,\,\,\,\,\,\,\,\,\,\,\,\,\,\,\,\,\,\,\,\,\,\,\,\, \mathcal{Z}^{\mathbb{C}^{2}, 1}_{\text{inst}}\big(Q,q,y ; t \big)_{\text{Ell}} = \sum_{k=0}^{\infty} Q^{k} \text{Ell}_{q,y}\big( \text{Hilb}^{k}(\mathbb{C}^{2}) ; t \big) \\
& =\sum_{Y \in \mathcal{P}} Q^{|Y|} \prod_{n=1}^{\infty} \prod_{(i,j) \in Y} \frac{(1-yq^{n-1}t^{h_{ij}})(1-y^{-1}q^{n}t^{-h_{ij}})(1-yq^{n-1}t^{h_{ij}})(1-y^{-1}q^{n}t^{h_{ij}})}{(1-q^{n-1}t^{h_{ij}})(1-q^{n}t^{-h_{ij}})(1-q^{n-1}t^{-h_{ij}})(1-q^{n}t^{h_{ij}})}
\end{split}
\end{equation}
where $\mathcal{P}$ is the infinite set of all one-dimensional partitions or Young diagrams.  
\end{Ex}

\begin{Ex}[\bfseries Pure $\boldmath{\mathcal{N}=2}$ SYM with Massive Adjoint Hypermultiplet]
Consider the multiplicative class $E_{m}$ corresponding to the polynomial $f(x) = x+m$, for complex parameter $m \in \mathbb{C}$.  Given a complex vector bundle $V$ of rank $n$
\begin{equation}\label{eqn:multclassEmm}
E_{m}(V) = m^{n} + c_{1}(V)m^{n-1} + \ldots + c_{n-1}(V)m + c_{n}(V).
\end{equation}
Notice for $m=1$ this is the total Chern class.  We can consider the corresponding Nekrasov partition function
\begin{equation} \label{eqn:Nekkeneqtwostar}
\mathcal{Z}_{\text{inst}}^{\mathbb{C}^{2}, r}( \epsilon_{1}, \epsilon_{2} , \vec{a}, Q ; m)_{\mathcal{N}=2^{*}} = \sum_{k=0}^{\infty} Q^{k} \int_{\mathfrak{M}(r,k)} E_{m}(T_{\mathfrak{M}})_{\widetilde{T}}.
\end{equation}
As indicated by our notation, this is the partition function of the $\mathcal{N}=2^{*}$ supersymmetric Yang-Mills theory on $\mathbb{C}^{2}$, which is simply pure $\mathcal{N}=2$ supersymmetric Yang-Mills with a single massive adjoint hypermultiplet of mass $m$.  It is clear from (\ref{eqn:multclassEmm}) that $\lim_{m \to 0}E_{m}(T_{\mathfrak{M}})_{\widetilde{T}} = e(T_{\mathfrak{M}})_{\widetilde{T}}$, which implies that
\begin{equation}
\lim_{m \to 0} \mathcal{Z}_{\text{inst}}^{\mathbb{C}^{2}, r}( \epsilon_{1}, \epsilon_{2} , \vec{a}, Q ; m)_{\mathcal{N}=2^{*}} = \mathcal{Z}_{\text{inst}}^{\mathbb{C}^{2}, r}(Q)_{e}
\end{equation} 
where the righthand side above is given in (\ref{eqn:NekkNeqfour}).  We can trivially rewrite (\ref{eqn:Nekkeneqtwostar}) in the following form  
\begin{equation}
\mathcal{Z}_{\text{inst}}^{\mathbb{C}^{2}, r}( \epsilon_{1}, \epsilon_{2} , \vec{a}, Q ; m)_{\mathcal{N}=2^{*}} = \sum_{k=0}^{\infty} (Qm^{2r})^{k} \int_{\mathfrak{M}(r,k)} \frac{1}{m^{2rk}}E_{m}(T_{\mathfrak{M}})_{\widetilde{T}}.
\end{equation}
Recalling that $\text{dim}(\mathfrak{M}(r,k)) = 2rk$, by (\ref{eqn:multclassEmm}) we have that $\lim_{m \to \infty} m^{-2rk} E_{m}(T_{\mathfrak{M}}) =1$.  If we therefore take the limits $m \to \infty$ and $Q \to 0$ precisely such that the variable $\Lambda \coloneqq Qm^{2r}$ is left finite, then the $\mathcal{N}=2^{*}$ partition function specializes to (\ref{eqn:partfuncPURENeq2}) with parameter $\Lambda$ instead of $Q$.  

This is all consistent with well-known physical facts: given $\mathcal{N}=2^{*}$ theory with mass $m$, in the massless $m \to 0$ limit we recover the $\mathcal{N}=4$ theory, while in the limit $m \to \infty$ and $Q \to 0$, we get the pure $\mathcal{N}=2$ theory with finite variable $\Lambda = Qm^{2r}$.  
\end{Ex}

\section{The Orbifold Elliptic Genera of Symmetric Products}

Given a manifold $X$ carrying an action by a finite group $G$, one can define the \emph{orbifold Euler characteristic}
\begin{equation}\label{eqn:obriiEulCharrr}
\chi_{\text{orb}}(X,G) \coloneqq \frac{1}{|G|} \sum_{gh =hg} \chi(X^{g,h})
\end{equation}
where the sum is over all commuting elements of $G$, and we denote by $X^{g,h}$ the fixed locus of both $g$ and $h$.  It was shown in \cite{hirzebruch_euler_1990} that the generating function of the orbifold Euler characteristic of the symmetric products of $X$ satisfies the product formula
\begin{equation}\label{eqn:OrbEulCherProdForm}
\sum_{m=0}^{\infty} Q^{m} \chi_{\text{orb}}\big(\text{Sym}^{m}(X)\big) = \prod_{n=1}^{\infty} \big(1-Q^{n}\big)^{-\chi(X)}.
\end{equation}
Here, $\chi_{\text{orb}}\big(\text{Sym}^{m}(X)\big) = \chi_{\text{orb}}(X^{m}, \Sigma_{m})$ where $\Sigma_{m}$ is the permutation group.  According to this formula, the generating function of the orbifold Euler characteristics is determined simply by the Euler characteristic of $X$; it is a universal function (the Euler function) raised to the power $\chi(X)$.  It was also shown in \cite{hirzebruch_euler_1990} that if $V$ is a smooth algebraic variety, and a crepant resolution of $V/G$ exists, then the ordinary Euler characteristic of this resolution agrees with the orbifold Euler characteristic.  If $X$ is a smooth algebraic surface, the Hilbert scheme $\text{Hilb}^{m}(X)$ is a crepant resolution of the symmetric product $\text{Sym}^{m}(X)$, and (\ref{eqn:OrbEulCherProdForm}) specializes to a well-known formula of G\"{o}ttsche \cite{gottsche_betti_1990}.  

In 1996 a refinement of the product formula (\ref{eqn:OrbEulCherProdForm}) emerged from string theorists R. Dijkgraaf, G. Moore, E. Verlinde, and H. Verlinde \cite{dijkgraaf_elliptic_1997}.  For any compact K\"{a}hler manifold $X$, they gave a physical derevation of the following formula
\begin{equation}\label{eqn:DMVVPRODFORM}
\sum_{m=0}^{\infty} Q^{m} \text{Ell}_{q,y}^{\text{orb}}\big( \text{Sym}^{m}(X) \big) = \prod_{\substack{ m >0, n \geq 0 \\ l \in \mathbb{Z}}} \big(1-Q^{m}q^{n}y^{l}\big)^{-c(mn, l)}
\end{equation}
where $c(mn, l)$ is the coefficient of $q^{nm}y^{l}$ in the ordinary elliptic genus of $X$, and $\text{Ell}^{\text{orb}}_{q,y}(-)$ is the orbifold elliptic genus, defined in \cite{borisov_elliptic_2003}.  We will refer to the product formula (\ref{eqn:DMVVPRODFORM}) as the \emph{DMVV formula}.  Because $\text{Ell}^{\text{orb}}_{q,1}(-) = \chi_{\text{orb}}(-)$, by (\ref{eqn:EulCharEllGenCoeff}) the DMVV formula indeed specializes to (\ref{eqn:OrbEulCherProdForm}).  One interesting feature of the DMVV formula is that the only information required is the elliptic genus of $X$ itself.  It is therefore sometimes called the \emph{second quantized elliptic genus} of $X$.

The orbifold elliptic genus $\text{Ell}^{\text{orb}}_{q,y}(X, G)$ was defined in \cite{borisov_elliptic_2003} for a finite group $G$ acting on an algebraic variety $X$, and a mathematical proof of the DMVV formula was given.  In addition, the authors prove that if $Y \to X/G$ is a crepant resolution, then $\text{Ell}_{q,y}(Y) = \text{Ell}^{\text{orb}}_{q,y}(X,G)$.  In particular, for a smooth compact algebraic surface $X$, the DMVV formula can be given as
\begin{equation}\label{eqn:DMVVsmmsurfSS}
\sum_{m=0}^{\infty} Q^{m} \text{Ell}_{q,y}\big( \text{Hilb}^{m}(X) \big) = \prod_{\substack{ m >0, n \geq 0 \\ l \in \mathbb{Z}}} \big(1-Q^{m}q^{n}y^{l}\big)^{-c(mn, l)}
\end{equation}
noting that the Hilbert scheme is a crepant resolution of the symmetric product.  For a K3 surface, we will see in the next chapter (\ref{eqn:chi10prodDMVVform}) that the above product formula is related to the Siegel modular form $\chi_{10}(\Omega)$.  

With respect to the natural torus action on $\mathbb{C}^{2}$, the following equivariant version of the DMVV formula was proven by R. Waelder \cite{waelder_equivariant_2008}, which we present for the diagonal specialization $t = t_{1} = t_{2}^{-1}$
\begin{equation}\label{eqn:WaelderDMVVFORM}
\sum_{m=0}^{\infty} Q^{m} \text{Ell}_{q,y}\big( \text{Hilb}^{m}(\mathbb{C}^{2}) ;t \big) = \prod_{\substack{ m >0, n \geq 0 \\ l, k \in \mathbb{Z}}} \big(1-Q^{m}q^{n}y^{l}t^{k}\big)^{-c(mn, l,k)}
\end{equation}
where $c(mn, l, k)$ is the coefficient of $q^{mn}y^{l}t^{k}$ in the Fourier expansion of $\text{Ell}_{q,y}(\mathbb{C}^{2}; t)$, shown in (\ref{eqn:equivELLC2diagspecc}).  For all fixed $k$, these coefficients depend only on the combination $4nm - l^{2}$, so we will from now on write them as $c(4nm - l^{2}, k)$.  

One should recognize the lefthand side of (\ref{eqn:WaelderDMVVFORM}) as one of the examples of a Nekrasov partition function we presented in (\ref{eqn:sumpartPEllGenC2}).  Recall that by way of equivariant localization on the Hilbert scheme $\text{Hilb}^{m}(\mathbb{C}^{2})$, the generating function of equivariant elliptic genera was expressed as a non-trivial sum over partitions.  Combining this with the result of Waelder, we get a remarkable formula relating an infinite product, with a sum over partitions  
\begin{equation}\label{eqn:resultofWaelderr}
\begin{split}
& \,\,\,\,\,\,\,\,\,\,\,\,\,\,\,\,\,\,\,\,\,\,\,\,\,\,\,\,\,\,\,\,\,\,\,\,\,\,\, \sum_{m=0}^{\infty} Q^{m} \text{Ell}_{q,y}\big( \text{Hilb}^{m}(\mathbb{C}^{2}) ; t \big) = \prod_{\substack{ m >0, n \geq 0 \\ l, k \in \mathbb{Z}}} \big(1-Q^{m}q^{n}y^{l}t^{k}\big)^{-c(4nm- l^{2} ,k)} \\
& =\sum_{Y \in \mathcal{P}} Q^{|Y|} \prod_{n=1}^{\infty} \prod_{(i,j) \in Y} \frac{(1-yq^{n-1}t^{h_{ij}})(1-y^{-1}q^{n}t^{-h_{ij}})(1-yq^{n-1}t^{h_{ij}})(1-y^{-1}q^{n}t^{h_{ij}})}{(1-q^{n-1}t^{h_{ij}})(1-q^{n}t^{-h_{ij}})(1-q^{n-1}t^{-h_{ij}})(1-q^{n}t^{h_{ij}})}.
\end{split}
\end{equation}

\chapter{A Brief Survey of Some Automorphic Forms}

Automorphic forms constitute a large and beautiful subject touching many distinct areas in mathematics and physics.  In this chapter we content ourselves to briefly surveying just three related types: ordinary modular forms, Jacobi forms, and Siegel modular forms.  Each of these will arise in our original results presented in the final chapter.  To a modern enumerative geometer, one reason to care about automorphic forms is that generating functions of enumerative invariants may be automorphic.  Having some understanding and control over these objects, one may be able to generate conjectures about the geometry which were otherwise not at all obvious.  One component of what is to follow, which perhaps is not so widely known, is a detailed discussion of Hecke operators on weak Jacobi forms and their use in defining the Maass lift.

\section{Introduction to Modular Forms}

Let $\mathfrak{H}$ be the complex upper-half plane, and consider the natural transitive action by $SL_{2}(\mathbb{R})$ on $\mathfrak{H}$ via fractional linear transformations
\[ \tau \mapsto \frac{a \tau + b}{c \tau + d}.\]
The maximal discrete subgroup of $SL_{2}(\mathbb{R})$ is the modular group $SL_{2}(\mathbb{Z})$ of invertible $2 \times 2$ matrices with integer entries, and unit determinant.  For reasons which will become clear upon introducing Siegel modular forms, we will often use the notation $\Gamma_{1} = SL_{2}(\mathbb{Z})$.  

\begin{defn}
A modular form of weight $k \in \mathbb{Z}$ on $SL_{2}(\mathbb{Z})$ is a holomorphic function $f: \mathfrak{H} \to \mathbb{C}$ satisfying the covariance property 
\begin{equation} \label{eqn:modformtrans}
f \bigg(\frac{a \tau + b}{c \tau +d} \bigg) = (c \tau + d)^{k} f(\tau), \,\,\,\,\,\,\,\,\,\,
\begin{pmatrix}
a & b\\
c & d
\end{pmatrix}
\in SL_{2}(\mathbb{Z}).
\end{equation}
\end{defn}
\noindent We refer to $(c \tau + d)^{k}$ as the \emph{automorphy factor}.  This transformation law implies that modular forms are periodic: $f(\tau + 1) = f(\tau)$.  Therefore, $f(\tau)$ has a Fourier expansion with $q=e^{2 \pi i \tau}$
\begin{equation} \label{eqn:Fourexp}
f(\tau) = \sum_{n=-\infty}^{\infty} a(n) q^{n}, \,\,\,\,\,\,\,\,\,\, a(n) \in \mathbb{Q}.
\end{equation}
There are therefore two different, yet equally important perspectives on a modular form.  One can either think of them as a holomorphic function on the upper-half plane with symmetry group $SL_{2}(\mathbb{Z})$, or equivalently as a Fourier expansion (\ref{eqn:Fourexp}) in $q$ with coefficients $a(n)$.  

\begin{rmk}\label{rmk:CosetDescrrr}
The coset description of the upper-half plane is the biholomorphism 
\begin{equation}\label{eqn:cosetdecrG1}
\mathfrak{H} \cong SL_{2}(\mathbb{R}) \big/ SO(2)
\end{equation}
where $SO(2) \subset SL_{2}(\mathbb{R})$ is a maximal compact subgroup.  That (\ref{eqn:cosetdecrG1}) is a diffeomorphism follows from identifying $SO(2)$ as the stabalizer of $i \in \mathfrak{H}$, but it is not a priori obvious that $SL_{2}(\mathbb{R}) \big/ SO(2)$ even has complex structure, so the biholomorphism takes more work \cite{milne_introduction_2005}.  Ultimately, one would say that the modular forms we have defined are automorphic forms on the Shimura variety
\begin{equation}
SL_{2}(\mathbb{Z}) \big\backslash SL_{2}(\mathbb{R}) \big/ SO(2).
\end{equation}
More general automorphic forms share the two perspectives described above for ordinary modular forms.  
\end{rmk}

Alternatively, instead of being given a modular form, one may have a collection of numbers $\{a(n) \}$ depending on the discrete invariant $n \in \mathbb{Z}$.  These might form an interesting arithmetic function, or these might be invariants coming from a one-parameter counting problem in math or physics.  The natural instinct is to package the invariants into a generating function (\ref{eqn:Fourexp}) and \emph{then} study its analytic and modular properties.  Remarkably, answers to counting problems often arise as the coefficients of a modular form.  An interesting converse problem is, given a modular form with integer coefficients, what exactly are the coefficients counting?  The answer will very often lead one completely away from the original setting of modular forms, down the path of algebraic geometry, representation theory, conformal field theory, and string theory.

Under the change of variables $q=e^{2 \pi i \tau}$ the upper-half plane is taken to the interior of the unit disk, such that the point at infinity is mapped to $q=0$.  As is standard, we will use the $q$ and $\tau$ variables interchangeably.  An additional piece of data defining a modular form is a specification of  the growth of $f(\tau)$ at the point at infinity of $\mathfrak{H}$, which we call the \emph{cusp}.  The growth at the cusp is reflected in the Fourier coefficients which we summarize with the following definitions. 

\begin{enumerate}
\item We say $f(\tau)$ is a holomorphic modular form if $a(n)=0$ for all $n<0$, and we denote by $M_{k}(\Gamma_{1})$ the vector space of holomorphic modular forms of weight $k$.  The ring of modular forms defined by
\begin{equation}\label{eqn:ringmodforms}
M_{*}(\Gamma_{1}) = \bigoplus_{k \in \mathbb{Z}}M_{k}(\Gamma_{1})
\end{equation}   
is clearly a graded ring since $M_{k} M_{l} \subset M_{k + l}$.  

\item We say $f(\tau)$ is a cusp form if $a(n)=0$ for all $n \leq 0$, and we denote by $S_{k}(\Gamma_{1})$ the space of cusp forms of weight $k$.  The ring $S_{*}(\Gamma_{1}) \subset M_{*}(\Gamma_{1})$ defined in the obvious way, is an ideal since the product of a cusp form with an arbitrary modular form is again a cusp form.    

\item More generally, if $f(\tau) = \mathcal{O}(q^{-N})$ for some $N \geq 0$, then $a(n)=0$ for $n < -N$.  Such a modular form is called a weakly holomorphic modular form, and we denote by $M_{k}^{!}(\Gamma_{1})$ the space of weight $k$ weakly holomorphic modular forms.    
\end{enumerate}
For a fixed weight $k$, these various classes of modular forms constitute finite-dimensional vector spaces over $\mathbb{C}$.  It is easy to see the containment
\begin{equation}
S_{k}(\Gamma_{1}) \subset M_{k}(\Gamma_{1}) \subset M_{k}^{!}(\Gamma_{1}).
\end{equation}  
Note that holomorphic modular forms are bounded at the cusp with $a(0)$ being the value attained there.  Cusp forms are characterized by vanishing at the cusp, which explains the name.  Weakly holomorphic modular forms diverge at the cusp, since there are negative powers of $q$, but they do so in a controlled fashion.  

The following two elementary propositions rule out the existence of non-trivial modular forms of certain weights.

\begin{proppy}
There are no non-zero modular forms of odd weight.
\end{proppy}

\begin{proof}
By the modular transformation law (\ref{eqn:modformtrans}) with $-1 \in SL_{2}(\mathbb{Z})$, we get $f(\tau) = (-1)^{k} f(\tau)$.  Because this must hold for all $\tau$, if $k$ is odd, then $f$ is identically zero.    
\end{proof}

\begin{proppy}
The only holomorphic modular forms of weight zero are the constants.
\end{proppy}

\begin{proof}
If $f$ is a modular form of weight zero, then $f$ is holomorphic on the upper-half plane, bounded at infinity, and invariant under $SL_{2}(\mathbb{Z})$.  Therefore, it descends to a bounded, holomorphic function on $\mathfrak{H}/ SL_{2}(\mathbb{Z}) \cong \mathbb{C}$.  The only such $f$ are constants.  
\end{proof}

We have seen that a modular form is a holomorphic function on the upper-half plane, transforming covariantly under $SL_{2}(\mathbb{Z})$.  A modular \emph{function}, as opposed to a weight zero modular form, is merely meromorphic on the upper-half plane.  Hence, whereas weight zero modular forms are constant, there are non-trivial modular functions.  The canonical example is the j-invariant $j(\tau)$ which classifies elliptic curves up to isomorphism.  The j-invariant is meromorphic on $\mathfrak{H}$ in a rather tame way: it is holomorphic outside of a simple pole at the cusp.  Therefore, $j(\tau)$ is an example of a weakly holomorphic modular form of weight zero.    

\subsection*{Example: The Eisenstein Series}

The following are important examples of weight $2k$ holomorphic modular forms for $k \geq 2$
\begin{equation}
G_{2k}(\tau) = \sum_{(m,n) \in \mathbb{Z}^{2} \setminus (0,0)}= (m + n \tau)^{-2k}.
\end{equation}
It can be shown that in the Fourier expansion of $G_{2k}(\tau)$, the constant term is $2 \zeta(2k)$, where $\zeta(z)$ is the Riemann $\zeta$-function.  It is convenient to instead work with the following normalized modular forms with constant term equal to one
\begin{equation}\label{eqn:EiiisssSER}
E_{2k}(\tau) = \frac{G_{2k}(\tau)}{2 \zeta(2k)} = 1-\frac{4k}{B_{2k}}\sum_{n=1}^{\infty} \sigma_{2k-1}(n) q^{n}.
\end{equation} 
We refer to these normalized modular forms of weight $2k$ as \emph{Eisenstein Series}.  Here, $B_{2k}$ denotes the Bernoulli numbers and 
\begin{equation}
\sigma_{2k-1}(n) = \sum_{d|n} d^{2k-1}
\end{equation}
is called the divisor function.  The first few Eisenstein series are given explicitly as
\begin{equation}
\begin{split}
& E_{4}(\tau) = 1+ 240 \sum_{n=1}^{\infty} \sigma_{3}(n) q^{n} \\
& E_{6}(\tau) = 1 - 504 \sum_{n=1}^{\infty} \sigma_{5}(n) q^{n}.
\end{split}
\end{equation}
A crucial structural result in this subject, which we will recall shortly, is that $E_{4}(\tau)$ and $E_{6}(\tau)$ actually suffice to generate all modular forms with respect to $SL_{2}(\mathbb{Z})$.  The Eisenstein series $E_{2}(\tau)$ is defined by
\begin{equation}
E_{2}(\tau) = 1- 24 \sum_{n=1}^{\infty} \sigma_{1}(n) q^{n}
\end{equation}
but it is not a modular form -- it is called \emph{quasi-}modular.    

In parts of this thesis we will make use of a function called the \emph{polylogarithm}, defined by
\begin{equation}\label{eqn:polylogrthm}
\text{Li}_{a}(x) \coloneqq \sum_{r=1}^{\infty} r^{-a} x^{r}.  
\end{equation}
The polylogarithm generalizes the ordinary logarithm, which we can recover as $\text{Li}_{1}(x) = - \log(1-x)$.  It is obvious from the definitions that we have the following relationship between the polylogarithm and the divisor function, for all $k \geq 1$
\begin{equation}
\sum_{n=1}^{\infty} \text{Li}_{1-2k}(q^{n}) = \sum_{n=1}^{\infty} \sigma_{2k-1}(n)q^{n}.  
\end{equation}
We can therefore express the Eisenstein series (\ref{eqn:EiiisssSER}) in term of the polylogarithm as
\begin{equation}\label{eqn:EISserexpprPOLY}
E_{2k}(\tau) = 1- \frac{4k}{B_{2k}} \sum_{n=1}^{\infty} \text{Li}_{1-2k}(q^{n})
\end{equation}
a relationship which also holds for the quasi-modular form $E_{2}(\tau)$.  We will make use of this when discussing Hecke operators on Jacobi forms.

\subsection*{Example: The Modular Discriminant Cusp Form}

We will see shortly that in a certain sense, the only cusp form is the modular discriminant $\Delta(\tau)$ defined by
\begin{equation}\label{eqn:JacobiEtaa}
\Delta(\tau) = \eta(\tau)^{24}, \,\,\,\,\,\,\,\,\,\,\,\,\,\,\,\, \eta(\tau) = q^{1/24} \prod_{n=1}^{\infty}(1-q^{n})
\end{equation}
where $\eta(\tau)$ is the Dedekind eta function.  The modular discriminant is a cusp form of weight 12.  In addition to the simple zero at $q=0$, any other zeros must lie on $|q|=1$, which corresponds to the real axis under the change of variables.  Therefore, $\Delta(\tau)$ is non-vanishing on the upper-half plane with a simple zero at infinity.  We will soon see a direct geometrical interpretation of $\Delta$, as well as more non-trivial appearances of $1/\Delta$ in string theory, algebraic geometry, and combinatorics.  

We can use the modular discriminant to prove the following two results.  

\begin{lemmy}
There are no holomorphic modular forms of negative weight.  
\end{lemmy}

\begin{proof}
For $k>0$, let $f$ be a holomorphic modular form of weight $-k$.  By the additivity of the weight, $f^{12} \Delta^{k}$ is a holomorphic modular form of weight zero, and hence is a constant.  Since the product of any modular form and a cusp form is again a cusp form, the constant term of $f^{12} \Delta^{k}$ vanishes, which means $f=0$.  
\end{proof}

\begin{lemmy}
There is an isomorphism $M_{k-12}(\Gamma_{1}) \cong S_{k}(\Gamma_{1})$ of complex vector spaces, induced by multiplication by $\Delta$.  
\end{lemmy}

\begin{proof}
Define the map $M_{k-12}(\Gamma_{1}) \to S_{k}(\Gamma_{1})$ by $f \mapsto \Delta f$.  The map is well-defined since the product of any modular form and a cusp form is again a cusp form, and the weights are consistent.  The map is clearly injective.  To show surjectivity, let $g \in S_{k}(\Gamma_{1})$ be a cusp form of weight $k$.  The modular form $g/\Delta$ indeed has weight $k-12$.  Because $\Delta$ is non-vanishing except for a simple zero at $q=0$, and $g$ has a zero of some positive order at $q=0$, it follows that $g/\Delta$ is regular on the upper-half plane.  
\end{proof}

\begin{cory}
There are no non-trivial cusp forms of weight less than 12.  In weight 12, we have the isomorphism $S_{12}(\Gamma_{1}) \cong \mathbb{C}$ induced by $\Delta$.  Finally, any cusp form of weight $k \geq 12$ is a suitable modular form multiplying some power of $\Delta$.  It is in this sense that the modular discriminant $\Delta$ is effectively the only cusp form.
\end{cory}   

For now, we would like to record an explicit relation determining the dimension of the space of modular forms $M_{k}(\Gamma_{1})$ for all $k$.  In addition, we would hope to find a nice set of generators for the full ring of modular forms (\ref{eqn:ringmodforms}).  To briefly summarize what we have shown so far: $M_{k}(\Gamma_{1})=0$ for $k<0$ and $k$ odd, and $M_{0}(\Gamma_{1}) = \mathbb{C}$.  We have shown that $M_{k}(\Gamma_{1}) \cong S_{k+12}(\Gamma_{1})$ and so computing dimensions of $M_{k}(\Gamma_{1})$ will give dimensions of the spaces of cusp forms.  The only examples of non-cusp modular forms we have seen so far are $E_{4}$ and $E_{6}$.  In fact, the content of the following theorem is that these are all we need.  

\begin{thm}[\bfseries Structure Theorem]\label{eqn:ModformStrTHMMM}
The ring of modular forms is freely generated over $\mathbb{C}$ by the Eisenstein series $E_{4}$ and $E_{6}$
\begin{equation}
M_{*}(\Gamma_{1}) = \mathbb{C} [ E_{4}, E_{6} ].  
\end{equation}
In other words, any modular form over $SL_{2}(\mathbb{Z})$ is simply a polynomial in $E_{4}$ and $E_{6}$ with complex coefficients.     
\end{thm}

\begin{thm}
The dimension of the space of holomorphic modular forms of weight $k$ is given by
\begin{equation}
\text{dim}_{\mathbb{C}}M_{k}(\Gamma_{1}) = 
\begin{cases}
0 & k<0 \\
0 & k \,\,  \text{odd} \\
\floor{\frac{k}{12}}  & k \equiv 2 \Mod{12} \\
\floor{\frac{k}{12}}+1  & \text{otherwise} 
\end{cases}
\end{equation}
\end{thm}

These two theorems have some non-trivial consequences.  First of all, we see that $M_{2}(\Gamma_{1})=0$: there are no weight two modular forms.  Moreover, for $k=4, 6, 8, 10, 14$ we have $M_{k}(\Gamma_{1})= \mathbb{C}$, generated by the Eisenstein series $E_{k}$.  This leads to non-trivial number-theoretic identities between the Eisenstein series
\begin{equation}
E_{8}(\tau) = E_{4}(\tau)^{2}  \,\,\,\,\,\,\,\,\,\,\,  E_{10}(\tau) = E_{4}(\tau) E_{6}(\tau)  \,\,\,\,\,\,\,\,\,\,\, E_{14}(\tau) = E_{8}(\tau) E_{6}(\tau).
\end{equation}
More generally, given $f, g \in M_{k}(\Gamma_{1})$ we only have to check agreement up to at most the first $\text{dim}_{\mathbb{C}}M_{k}(\Gamma_{1})$ coefficients to conclude whether $f,g$ coincide or not.  This aspect is very powerful in practice.

In addition, we see that $M_{12}(\Gamma_{1})$ is a two-dimensional complex vector space generated by $E_{4}^{3}$ and $E_{6}^{2}$.  But we have already encountered one modular form of weight 12: the modular discriminant $\Delta$.  Therefore, $\Delta$ must be expressible in terms of the Eisenstein series.  It turns out that
\begin{equation}\label{eqn:discmodformEIS}
\Delta(\tau) = \frac{1}{1728}\big(E_{4}(\tau)^{3} - E_{6}(\tau)^{2}\big).
\end{equation}
This expression gives a nice geometrical interpretation to the modular discriminant $\Delta$, and explains the name.  An \emph{elliptic curve} is given by a homogeneous cubic equation in $\mathbb{P}^{2}$, but in an affine chart, coordinates can be chosen such that the equation takes the form $y^{2} = x^{3} + px + q$.  Interpreting $\tau$ as a coordinate on the moduli space of elliptic curves, an affine cubic curve can be parameterized in the following form \cite{husemoller_elliptic_2010}
\begin{equation}
y^{2} = x^{3} - \frac{E_{4}(\tau)}{48} x - \frac{E_{6}(\tau)}{864}.
\end{equation}
The curve is smooth if and only if a quantity called the discriminant is non-vanishing.  Up to scale, the discriminant is given simply by (\ref{eqn:discmodformEIS}), and coincides with the modular discriminant.  Recalling that $\Delta$ has a single zero (in the upper-half plane) at the cusp, this corresponds to a singular elliptic curve.  The smooth elliptic curves are parameterized up to isomorphism by $\mathcal{M}_{1,1} \cong \mathfrak{H}/ SL_{2}(\mathbb{Z})$, and the Deligne-Mumford compactification $\overline{\mathcal{M}}_{1,1}$ adds a single point at the cusp corresponding to a nodal elliptic curve with $\Delta =0$.  

The modular discriminant makes a slightly less direct appearance in a combinatorial problem.  If we let $p_{24}(n+1)$ denote the number of partitions of the non-negative integer $n+1$ into 24 distinct colors, we have
\begin{equation}\label{eqn:p24part24colorss}
\sum_{n=-1}^{\infty} p_{24}(n+1) q^{n} = \frac{1}{\Delta(\tau)}.
\end{equation}
In other words, the interesting combinatorial quantities $p_{24}(n+1)$ arise as the Fourier coefficients of the weakly holomorphic modular form $1/\Delta$ of weight -12.  We will also present a physical manifestation of the coefficients $p_{24}(n+1)$.

\subsection{Modular Forms in String Theory and Quantum Black Holes}

Modular forms (as well as the Jacobi forms and Siegel modular forms to come) often arise in physics as partition functions where the Fourier coefficients count degeneracies of certain quantum black hole configurations in superstring theory.  Black holes are a phenomena in spacetime inherently belonging to quantum gravity. They are extremely massive (so they gravitate classically in general relativity) yet they have collapsed to such tiny sizes that quantum mechanical effects cannot be neglected. One may think of a black hole as a macroscopic object in our universe acting in some sense as a microscope to the structure of spacetime at the tiniest scales.

Suppose there are observable quantum numbers $(\sigma_{1}, \ldots, \sigma_{n})$ called \emph{charges} which characterize features of a black hole (for example mass, charge, or spin).  We expect there to exist quantities $d(\sigma_{1}, \ldots, \sigma_{n})$ called \emph{degeneracies} which are integers counting the microscopic black hole states with fixed charge $(\sigma_{1}, \ldots, \sigma_{n})$.  The presence of symmetries of the physical system may imply that the degeneracies do not depend on the charges independently, but rather only on a function, or collection of functions of the charges.  Of course, $d(\sigma_{1}, \ldots, \sigma_{n})$ is defined for all possible values of the charges, but we only get a black hole in some large mass limit.  

As we have seen in Section \ref{sec:Dbranesstabstrth}, by compactifying Type IIA superstring theory on a Calabi-Yau threefold $X$, we can get BPS particles in four-dimensions by wrapping D-branes on holomorphic cycles in $X$ such that the observable quantum numbers of the particle are given by the D-brane charges.  We interpret $d(\sigma_{1}, \ldots, \sigma_{n})$ to be the number of BPS states or black hole states with fixed charge $(\sigma_{1}, \ldots, \sigma_{n})$.  For example, the Gopakumar-Vafa invariants $n_{g, \beta}(X)$ count BPS states with spin related to $g$ and charge $\beta \in H_{2}(X, \mathbb{Z})$.  

One might hope to engineer black holes in four dimensions using large mass D-brane configurations in $X$.  In certain cases, the degeneracies $d(\sigma_{1}, \ldots, \sigma_{n})$ arise as coefficients of an $n$-variable automorphic form.  This way of engineering black holes was pioneered by Strominger and Vafa\footnote{Strictly speaking, Strominger and Vafa studied black holes in five dimensions, but by \cite{gaiotto_new_2006} one can relate 4d and 5d black holes.} \cite{strominger_microscopic_1996-1}.  The fact that this often recovers a macroscopic prediction of Bekenstein-Hawking, made it one of the most tantalizing applications of string theory to potentially observable physics.

\subsubsection{Type II Compactification on $K3 \times E$}

Consider the Type II superstring compactified on the compact Calabi-Yau threefold $K3 \times E$, where $E$ is an elliptic curve.  This induces an $\mathcal{N}=4$ supersymmetric theory in four dimensions; that is a theory with 8 supercharges.  One can then study what are called \emph{half-BPS states} which are states in the Hilbert space annihilated by 4 of the 8 supercharges.  It turns out \cite{dabholkar_quantum_2012} that such states are engineered by D-brane configurations in $K3 \times E$ with only one independent D-brane charge $m$.  We can then construct the partition function $\hat{Z}(\sigma)$ with coefficients $d(m)$ counting the number of half-BPS states with fixed charge $m$.  In this case, as shown by Vafa and Witten \cite{vafa_strong_1994-1}
\begin{equation}\label{eqn:coefffdemmm}
d(m) = \chi_{\text{orb}}\big( \text{Sym}^{m+1}(K3) \big) = \chi \big( \text{Hilb}^{m+1}(K3)\big)
\end{equation}
where the orbifold Euler characteristic $\chi_{\text{orb}}$ was introduced in (\ref{eqn:obriiEulCharrr}), and the second equality in (\ref{eqn:coefffdemmm}) follows by noting the Hilbert scheme is a crepant resolution of the symmetric product.  The generating function of half-BPS states therefore takes the form
\begin{equation}
\hat{Z}(\sigma) = \sum_{m=-1}^{\infty} \chi\big(\text{Hilb}^{m+1}(K3)\big)Q^{m}
\end{equation}
where $Q = e^{2 \pi i \sigma}$.  By a result of G\"{o}ttsche \cite{gottsche_betti_1990}, this is known to be simply the inverse modular discriminant
\begin{equation} \label{eqn:TypeIIFRAME}
\hat{Z}(\sigma) = \frac{1}{\Delta(\sigma)}.  
\end{equation}

As is common in physics, this system has a dual description.  This means there are two completely different paradigms giving rise to the same physics, and no experiments can determine which paradigm one is in.  In this case, the physical dual is a chiral conformal field theory on a bosonic string \cite{dabholkar_quantum_2012}.  In superstring theory, with both bosons and fermions, the geometrical background is a ten-dimensional manifold.  In bosonic string theory however, we consider strings propagating in a 26-dimensional background.  A propagating string is modeled as an embedding of the worldsheet in the ambient 26 dimensions.  The string can only oscillate in 24 of the directions because it cannot oscillate along its own two-dimensional worldsheet.  For each dimension, there are also discrete modes labeled by an integer $n>0$ into which we can put energy.  On a closed string, for each mode $n>0$ we can have both a left and right moving field.  

In string theory we can interpret the above geometrical setting as a \emph{conformal field theory} on the worldsheet.  For our purposes, we can consider only the left-moving sector; this system is then called a \emph{chiral} conformal field theory, or sometimes a system of 24 free chiral bosons.    

When quantizing the system, we introduce raising and lowering operators $a_{\mu n}^{\dagger}$ and $a_{\mu n}$ respectively, where $\mu =1, \ldots, 24$ and $n>0$.  These operators define an algebra with commutation relations
\begin{equation} \label{eqn:commrel}
[a_{\mu n}, a_{\nu m}^{\dagger}] = \delta_{\mu \nu} \delta_{nm}.
\end{equation}
Let $|0 \rangle$ denote the unique vacuum state defined by its annihilation by all lowering operators $a_{\mu n} | 0 \rangle=0$.  The Hilbert space of states $\mathscr{H}$ corresponds to the following Fock space of representations of the algebra spanned by states of the form
\[(a_{\mu_{1} n_{1}}^{\dagger})^{m_{1}} \cdots (a_{\mu_{k} n_{k}}^{\dagger} )^{m_{k}} | 0 \rangle.\]
The Hamiltonian of the system is given by the following operator
\begin{equation}
H=\sum_{\mu =1}^{24} \sum_{n>0} n a_{\mu n}^{\dagger} a_{\mu n} -1= L_{0} -1.
\end{equation}
Using the commutation relation (\ref{eqn:commrel}) it is straightforward to see that the states above are eigenstates of the Hamiltonian with energy eigenvalue $m_{1} n_{1} + \cdots + m_{k} n_{k} -1$.  Therefore, the $-1$ appearing in the Hamiltonian corresponds to the energy of the vacuum $|0 \rangle$ in the quantum theory.  The partition function of the theory is given by the following trace\footnote{Note that $Z(\tau)$ looks similar to the Witten index (\ref{eqn:WittINDexpo}).  However, the $(-1)^{F}$ is not present as there are no fermions in the system at hand.} over the Hilbert space, where $q = e^{2 \pi i \tau}$
\begin{equation} \label{eqn:bosonpartfunc}
Z(\tau) = \text{Tr}_{\mathscr{H}}(q^{H})= q^{-1} \text{Tr}_{\mathscr{H}}(q^{L_{0}}), \,\,\,\,\,\,\,\, (q=e^{2 \pi i \tau}).
\end{equation}

We can now turn to evaluating the partition function (\ref{eqn:bosonpartfunc}).  The proper way of handling exponentiated operators is by passing to a basis of eigenstates of the operator.  In our case, the Fock space states above are eigenstates of $L_{0}$ with eigenvalue $m_{1} n_{1} + \cdots + m_{k}n_{k}$.  The commutation relations (\ref{eqn:commrel}) indicate that we can treat each mode and each of the 24 directions independently.  This allows us to write the partition function as
\begin{equation}
Z(\tau) = q^{-1} \prod_{\mu=1}^{24} \prod_{n=1}^{\infty} \text{Tr}_{\mathscr{H}^{(\mu)}_{n}}\big(q^{L_{0}}\big)
\end{equation}
where $\mathscr{H}_{n}^{(\mu)}$ is the Hilbert space of states spanned by
\[\bigg\{ \,\, (a_{\mu n_{1}}^{\dagger})^{m_{1}} \cdots (a_{\mu n_{k}}^{\dagger} )^{m_{k}} | 0 \rangle \,\,\, \bigg| \,\,\,m_{1}n_{1} + \cdots + m_{k}n_{k}=n \,\, \bigg\}. \]
These are the states with $L_{0}$ eigenvalue $n$, excited in one fixed direction.  Hence, for all $\mu$ and all $n \geq 1$
\begin{equation}
\text{Tr}_{\mathscr{H}^{(\mu)}_{n}}\big(q^{L_{0}}\big) = 1 + q^{n} + q^{2n} + \cdots = \frac{1}{1-q^{n}}.
\end{equation}
We finally see that the full partition function
\begin{equation}
Z(\tau) = q^{-1} \prod_{n=1}^{\infty} \bigg(\frac{1}{1-q^{n}}\bigg)^{24} = \frac{1}{\Delta(\tau)}
\end{equation}
is given by the inverse of the modular discriminant, in perfect agreement with (\ref{eqn:TypeIIFRAME}) upon exchanging $\tau$ and $\sigma$.  The dual conformal field theory picture gives a physical interpretation of the quantities $p_{24}(n+1)$ in (\ref{eqn:p24part24colorss}) as the number of ways a bosonic string in 26 dimensions can distribute $n+1$ units of energy.  This is an example of a duality in physics which generates a mathematical relationship or conjecture; in this case
\begin{equation}
p_{24}(n+1) = \chi\big(\text{Hilb}^{n+1}(K3)\big).  
\end{equation}

We note that there was good reason for using two different variables $q = e^{2 \pi i \tau}$ and $Q = e^{2 \pi i \sigma}$, which will become clear later in the chapter.  To give a hint, the Euler characteristic $\chi(-)$ is a specialization of the elliptic genus $\text{Ell}_{q,y}(-)$, and the two partition functions $\hat{Z}(\sigma)$ and $Z(\tau)$ essentially correspond to different specializations of a certain Siegel modular form with variables $(\tau, z, \sigma)$.

\section{Introduction to Jacobi Forms}\label{sec:IntroJACform}

Jacobi forms are automorphic forms which arise as two-variable generalizations of modular forms by replacing the modular group $\Gamma_{1} = SL_{2}(\mathbb{Z})$ by the Jacobi group $SL_{2}(\mathbb{Z}) \ltimes \mathbb{Z}^{2}$.  They are in some sense, a twisted combination of an elliptic function and a modular form in one variable.  The original, and canonical reference is Eichler and Zagier \cite{eichler_theory_2013} where Jacobi forms were elucidated for the first time.  A modern account, with applications to physics, can be found in \cite{dabholkar_quantum_2012}.  

Let $\varphi: \mathfrak{H} \times \mathbb{C} \to \mathbb{C}$ be a holomorphic function with $\tau$ a coordinate on the upper-half plane $\mathfrak{H}$ and $z$ a coordinate on $\mathbb{C}$.  In a specific sense, we want $\varphi(\tau, z)$ to be modular in $\tau$ and elliptic in $z$.  

\begin{defn}\label{defn:Jacobifrmmmm}
A Jacobi form of weight $k\in \mathbb{Z}$ and index $m \in \mathbb{Z}_{\geq0}$ is a holomorphic function $\varphi: \mathfrak{H} \times \mathbb{C} \to \mathbb{C}$ satisfying the following two conditions
\begin{equation} \label{eqn:JacFormModul}
\varphi\bigg(\frac{a \tau + b}{c \tau + d}, \frac{z}{c \tau +d}\bigg) = (c \tau + d)^{k}e^{\frac{2 \pi i m c z^{2}}{c \tau +d}}\varphi(\tau, z), \,\,\,\,\,\,\,\,\,\,
\begin{pmatrix}
a & b\\
c & d
\end{pmatrix}
\in SL(2, \mathbb{Z})
\end{equation}

\begin{equation}\label{eqn:JacFormElll}
\varphi(\tau, z + \lambda \tau + \mu) = e^{-2 \pi i m(\lambda^{2}\tau + 2 \lambda z)}\varphi(\tau, z), \,\,\,\,\,\,\,\, \lambda, \mu \in \mathbb{Z}.
\end{equation}
\end{defn}

\noindent By (\ref{eqn:JacFormElll}), if $m=0$ then $\varphi$ is independent of $z$.  Therefore by (\ref{eqn:JacFormModul}), a Jacobi form of index $m=0$ is simply an ordinary one-variable modular form.  In addition, clearly $f(\tau) \coloneqq \varphi(\tau, 0)$ is a modular form of weight $k$.  

The two defining conditions of a Jacobi form imply that $\varphi$ is periodic in both components: $\varphi(\tau + 1, z) = \varphi(\tau, z)$ and $\varphi(\tau, z+1) = \varphi(\tau, z)$.  Therefore, $\varphi$ has a Fourier expansion in terms of the variables $q=e^{2 \pi i \tau}$ and $y = e^{2 \pi i z}$,
\begin{equation} \label{eqn:JACfrmFEXP}
\varphi(\tau, z) = \sum_{n ,l \in \mathbb{Z}}c(n, l)q^{n}y^{l}.
\end{equation}
By imposing specific growth conditions on the Fourier coefficients, we will soon refine the general definition of a Jacobi form into certain classes of interest.  However, we first establish some important symmetries manifest in the Fourier coefficients of Jacobi forms with particular weight and index, independent of growth conditions.  

\begin{lemmy}\label{lemmy:evenweight}
A Jacobi form $\varphi$ with Fourier coefficients $c(n,l)$ has even weight if and only if $c(n,l) = c(n, -l)$ for all $n,l$.  
\end{lemmy}

\begin{proof}
This follows directly from the modularity property (\ref{eqn:JacFormModul}) of Jacobi forms.  Using $-1 \in SL_{2}(\mathbb{Z})$, the Jacobi form transforms as
\begin{equation}
\varphi(\tau, -z) = (-1)^{k} \varphi(\tau, z).
\end{equation}
For even weight $k$, this happens if and only if we have the symmetry $c(n,l) = c(n, -l)$ of the Fourier coefficients for fixed $n$.  
\end{proof}
\noindent In practice, this lemma implies that upon Fourier expanding a Jacobi form of even weight, the coefficient of any fixed power of $q$ will be palindromic in $y$.  The following theorem establishes further critical symmetries of the Fourier coefficients.

\begin{thm}\label{thm:wJacfrmm}
Let $\varphi$ be a Jacobi form of weight $k$ and index $m$ with Fourier coefficients $c(n,l)$.  The coefficients depend only on the quantity $\Delta = 4nm-l^{2}$, and $l \in \mathbb{Z}/2m\mathbb{Z}$.  That is to say,
\begin{equation}
c(n,l) = c( \Delta, l), \,\,\,\,\,\,\,\, l \in \mathbb{Z}/2m \mathbb{Z}.
\end{equation}
In addition, for index $m=1$ the coefficients depend only on $\Delta$, and the weight of $\varphi$ is even.  In such a case, we will write $c(n,l) = c(\Delta)$.  
\end{thm}
 
\begin{proof}
The first part of the theorem follows from the ellipticity property (\ref{eqn:JacFormElll}) of Jacobi forms.  Choosing $\mu=0$, we have
\begin{equation}
\setlength{\jot}{12pt}
\begin{split}
\varphi(\tau, z) = \sum_{n', l' \in \mathbb{Z}} c(n',l') q^{n'} y^{l'} & = e^{2 \pi i m(\lambda^{2} \tau + 2 \lambda z)} \varphi(\tau, z + \lambda \tau) \\
& = q^{m \lambda^{2}} y^{2 m \lambda} \sum_{n, l \in \mathbb{Z}} c(n, l) q^{n + l \lambda} y^{l}.
\end{split}
\end{equation}
By comparing terms, we see that $c(n', l') = c(n,l)$ if and only if
\begin{equation}
\setlength{\jot}{12pt}
\begin{split}
& n' = n+ l \lambda + m \lambda^{2} \\
& l' = l + 2m \lambda.
\end{split}
\end{equation}
The second condition requires $l' \equiv l \Mod{2m}$, since $\lambda \in \mathbb{Z}$ is an arbitrary integer.  By a trivial computation one can see that $4mn'-l'^{2} = 4mn-l^{2}$.  This proves the first assertion.  

Turning to the second claim, we note that for index $m=1$, by the first part of the theorem the coefficients depend only on $\Delta = 4n-l^{2}$ and $l \in \mathbb{Z}/2\mathbb{Z}$.  However, the parity of the quantity $4n-l^{2}$ itself encodes the parity of $l$.  Therefore, the coefficients depend only on $\Delta$.  Finally, the dependence on $\Delta$ implies that $c(n,l) = c(n,-l)$.  By Lemma \ref{lemmy:evenweight} we conclude that the weight must be even.  
\end{proof}

In the same spirit as ordinary modular forms, placing particular restrictions on the Fourier coefficients allow us to refine the general definition of a Jacobi form into special types, in terms of its growth at infinity of the upper-half plane.  

\begin{defn}
Let $\varphi(\tau,z)$ be a Jacobi form of arbitrary weight and index with Fourier coefficients $c(n,l)$.  We say that
\begin{enumerate}
\item $\varphi$ is a weakly holomorphic Jacobi form if $c(n,l)=0$ unless $n \geq n_{0}$, for a non-positive integer $n_{0}$.  Let $\mathbb{J}^{!}_{k,m}$ denote the vector space of weakly holomorphic Jacobi forms of weight $k$ and index $m$.  

\item $\varphi$ is a weak Jacobi form if $c(n,l)=0$ unless $n \geq 0$.  Let $\mathbb{J}^{\text{w}}_{k,m}$ denote the vector space of weak Jacobi forms of weight $k$ and index $m$.  

\item $\varphi$ is a holomorphic Jacobi form if $c(n,l)=0$ unless $4nm \geq l^{2}$.  Let $\mathbb{J}_{k,m}$ denote the vector space of holomorphic Jacobi forms of weight $k$ and index $m$.  

\item $\varphi$ is a Jacobi cusp form if $c(n,l)=0$ unless $4nm > l^{2}$.  Let $\mathbb{J}^{0}_{k,m}$ denote the vector space of Jacobi cusp forms of weight $k$ and index $m$.  
\end{enumerate}
\end{defn}

\noindent It is straightforward from the definitions to see that every Jacobi cusp form is holomorphic, every holomorphic Jacobi form is weak, and every weak Jacobi form is weakly holomorphic.  

These various Jacobi forms can be at least partly characterized by their behavior at infinity of the upper-half plane $\mathfrak{H}$.  Because $q = e^{2 \pi i \tau}$, all non-constant terms in $q$ of a weak Jacobi form vanish at the point at infinity of $\mathfrak{H}$.  By the constraint defining Jacobi cusp forms, we must have $c(0,0)=0$, which forces them to vanish identically at infinity.  Finally, a weakly holomorphic Jacobi form diverges in a controlled way at infinity.  Notice that the conditions defining holomorphic Jacobi forms and Jacobi cusp forms do not merely constrain the behavior at the point at infinity of $\mathfrak{H}$; they also constrain the allowed powers of $y$ in a fashion depending on the power of $q$ as well as the index of the Jacobi form.

\subsection{Weak Jacobi Forms of Index One}

We have seen that Jacobi forms of index 0 are simply ordinary modular forms, so the next interesting case to consider is index 1.  Given $\varphi_{k,1} \in \mathbb{J}^{\text{w}}_{k,1}$, we proved in Lemma \ref{thm:wJacfrmm} that the weight $k$ must be even.  We can give two examples of weak Jacobi forms of even weight and index 1, and it will turn out that these are all we need.  To construct both examples we will use the classical Jacobi theta functions defined by 
\begin{equation}\label{eqn:defnnJAcThetaFuncc}
\begin{split}
& \theta_{1}(\tau, z) = - \sum_{n \in \mathbb{Z}} q^{\frac{1}{2}(n + \frac{1}{2})^{2}}(-y)^{n+\frac{1}{2}} \\
& \theta_{2}(\tau, z) = \sum_{n \in \mathbb{Z}} q^{\frac{1}{2}(n + \frac{1}{2})^{2}} y^{n + \frac{1}{2}} \\
& \theta_{3}(\tau, z) = \sum_{n \in \mathbb{Z}} q^{n^{2}/2} y^{n} \\
& \theta_{4}(\tau, z) = \sum_{n \in \mathbb{Z}} q^{n^{2}/2} (-y)^{n}.  
\end{split}
\end{equation}
These are not Jacobi forms precisely in the sense of Definition \ref{defn:Jacobifrmmmm}.  We will primarily focus on $\theta_{1}(\tau, z)$ which with a suitable extension of the definition, is a Jacobi form of weight $\frac{1}{2}$ and index $\frac{1}{2}$.  The elliptic transformation law is given for all $\lambda, \mu \in \mathbb{Z}$ by 
\begin{equation}\label{eqn:jactheetaELL}
 \theta_{1}(\tau, z + \lambda \tau + \mu) = (-1)^{\lambda + \mu} e^{- i \pi( \lambda^{2} \tau + 2 \lambda z)} \theta_{1}(\tau, z)
\end{equation}
which in comparison to (\ref{eqn:JacFormElll}) is nearly how one would na\"{i}vely expect a Jacobi form of index $\frac{1}{2}$ to transform.  We record the modular transformation laws for the two generators of $SL_{2}(\mathbb{Z})$
\begin{equation}\label{eqn:JaccthetaMODDDD}
\theta_{1}(\tau + 1, z) = e^{\frac{i \pi}{4}} \theta_{1}(\tau, z), \,\,\,\,\,\,\,\,\,\,\,\,\,\,\,\, \theta_{1} \big( -\frac{1}{\tau}, \frac{z}{\tau}\big) = -i \sqrt{\frac{\tau}{i}} e^{i \pi z^{2}/\tau} \theta_{1}(\tau, z).
\end{equation} 
It turns out that $\theta_{1}(\tau, z)$ can be expressed as the following infinite product
\begin{equation}\label{eqn:ClJacFormTheta1}
\theta_{1}(\tau, z) = - i q^{\frac{1}{8}} (y^{\frac{1}{2}}-y^{-\frac{1}{2}})\prod_{n=1}^{\infty} (1-q^{n})(1-yq^{n})(1-y^{-1}q^{n}).
\end{equation}
To verify this beginning with (\ref{eqn:defnnJAcThetaFuncc}) is a straightforward computation using the \emph{Jacobi triple product}
\begin{equation}\label{eqn:JACTRIPPRODD}
\sum_{n \in \mathbb{Z}} q^{n^{2}/2}y^{n}  =  \prod_{m=1}^{\infty} (1-q^{m})(1+q^{m -\frac{1}{2}}y)(1+q^{m - \frac{1}{2}} y^{-1}).
\end{equation}

Clearly $\theta_{1}(\tau, z)$ is not itself a weak Jacobi form of index 1, but we can use it to construct one.  Recalling the Dedekind eta function $\eta(\tau)$ given in (\ref{eqn:JacobiEtaa}), we define
\begin{equation}
\Theta(\tau, z) \coloneqq i \frac{\theta_{1}(\tau, z)}{\eta(\tau)^{3}}.  
\end{equation} 
From the transformation laws of $\theta_{1}(\tau, z)$ and $\eta(\tau)$, one can show directly that the square $\Theta(\tau, z)^{2}$ is a weak Jacobi form of weight -2 and index 1.  As such, this form is often called $\varphi_{-2, 1}(\tau, z)$, but we will stick to the notation $\Theta(\tau, z)^{2}$.  By (\ref{eqn:JacobiEtaa}) and (\ref{eqn:ClJacFormTheta1}) we also have an infinite product formula for $\Theta(\tau, z)^{2}$
\begin{equation} \label{eqn:varThetaa}
\Theta(\tau, z)^{2} = y^{-1}(1-y)^{2} \prod_{n = 1}^{\infty} \frac{(1-yq^{n})^{2}(1-y^{-1}q^{n})^{2}}{(1-q^{n})^{4}}.
\end{equation}
Taking the Fourier expansion, the first few terms in powers of $q$ are
\begin{equation}
\Theta(\tau, z)^{2} = \frac{(y-1)^{2}}{y} - \frac{2(y-1)^{4}}{y^{2}} q + \frac{(y-1)^{4} (y^{2} -8y + 1)}{y^{3}}q^{2} + \cdots
\end{equation}
One can observe that the coefficient of a fixed power of $q$ is a palindromic polynomial in $y$.  In addition, one sees that the coefficients $c(n,l)= c(\Delta)$ indeed depend only on the value of $\Delta = 4n - l^{2}$ which in this case, satisfies $\Delta \geq -1$.  Finally, notice that $\Theta(\tau, z)^{2}$ is \emph{not} a holomorphic Jacobi form as there is a non-zero coefficient $c(-1)$ corresponding to $\Delta =-1$.  

It is also evident from (\ref{eqn:varThetaa}) that $\Theta(\tau, z)^{2}$ vanishes for $y=1$ or equivalently, $z=0$.  This is indeed evident from the following Taylor expansion in the variable $\lambda = 2 \pi z$
\begin{equation}\label{eqn:thetasqEXPFORM}
\Theta(\tau, z)^{2} = -\lambda^{2} \text{exp} \bigg( \sum_{g=1}^{\infty} (-1)^{g} \frac{B_{2g}}{g (2g)!} E_{2g}(\tau) \lambda^{2g}\bigg).
\end{equation}
It turns out that $\Theta(\tau, z)^{-2}$ appears in certain applications in enumerative geometry.  This is a meromorphic Jacobi form of weight 2 and index -1.  By taking the reciprocal of (\ref{eqn:thetasqEXPFORM}), we have an expansion of the form
\begin{equation}
\frac{1}{\Theta(\tau, z)^{2}} = \sum_{g=0}^{\infty} \lambda^{2g-2} \mathcal{P}_{g}(\tau) = -\frac{1}{\lambda^{2}} \text{exp} \bigg( \sum_{g=1}^{\infty} (-1)^{g+1} \frac{B_{2g}}{g (2g)!} E_{2g}(\tau) \lambda^{2g}\bigg),
\end{equation} 
where $\mathcal{P}_{g}(\tau)$ is evidently a quasi-modular form of weight $g$.  That is, it is a weighted-homogeneous polynomial of degree $g$ in $E_{2}, E_{4}, E_{6}$.  This function arises in the enumerative geometry of the trivial K3 fibration over an elliptic curve as well as Calabi-Yau threefolds fibered in K3 surfaces over $\mathbb{P}^{1}$.

In terms of the additional classical Jacobi theta functions (\ref{eqn:defnnJAcThetaFuncc}), we can also define the weak Jacobi form $\varphi_{0,1}$ of weight 0 and index 1 by
\begin{equation}\label{eqn:phi01defn}
\varphi_{0,1}(\tau, z) = 4 \bigg( \frac{\theta_{2}(\tau, z)^{2}}{\theta_{2}(\tau)^{2}} +  \frac{\theta_{3}(\tau, z)^{2}}{\theta_{3}(\tau)^{2}} + \frac{\theta_{4}(\tau, z)^{2}}{\theta_{4}(\tau)^{2}} \bigg)
\end{equation}
where $\theta_{i}(\tau) \coloneqq \theta_{i}(\tau, 0)$.  Taking the Fourier expansion, we can record some of the low-order terms in $q$
\begin{equation}\label{eqn:phi01FEXP}
\varphi_{0,1}(\tau, z) = \frac{y^{2} + 10y +1}{y} + \frac{2(y-1)^{2}(5y^{2}-22y+5)}{y^{2}}q + \cdots
\end{equation}
Just as in the case of $\Theta^{2}$, one observes that for a fixed power of $q$, the coefficients are palindromic polynomials in $y$, and that the Fourier coefficients $c(n,l) = c(\Delta)$ depend only on the value of $\Delta$.  

Up to multiplication by a scalar, $\Theta^{2}$ and $\varphi_{0,1}$ are the unique weak Jacobi forms of index 1 and weight $-2$ and $0$, respectively.  But in fact, the following theorem makes an even stronger statement.  A proof can be found in \cite{eichler_theory_2013}.  

\begin{thm}[\bfseries Structure Theorem]
The ring of weak Jacobi forms $\mathbb{J}^{\text{w}}_{k,m}$ of even weight $k$ and index $m$ is generated by $\Theta^{2}$ and $\varphi_{0,1}$ as a module over the ring $M_{*}(\Gamma_{1})$ of modular forms.  That is to say, we have
\begin{equation}
\mathbb{J}^{\text{w}}_{k,m} = \bigoplus_{j=0}^{m} M_{k+2j}(\Gamma_{1}) (\Theta^{2})^{j} \varphi_{0,1}^{m-j}.
\end{equation}
\end{thm}

The weak Jacobi forms $\Theta^{2}$ and $\varphi_{0,1}$ combine to produce a well-known function.  The Weierstrass $\wp$-function is a meromorphic Jacobi form of weight 2 and index 0, which is expressed in the $q$ and $y$ variables as
\begin{equation}\label{eqn:WPfunncFexppp}
\wp(\tau, z) = \frac{1}{12} + \frac{y}{(1-y)^{2}} + \sum_{k,r \geq 1} k(y^{k}-2+y^{-k})q^{rk}.
\end{equation}

\begin{rmk}\label{rmk:WPconvrmk}
Beware that the Weierstrass $\wp$-function as it is often defined is $(2 \pi i)^{2}$ times our definition.  Our convention ensures that $\wp(\tau, z)$ has rational Fourier coefficients in $q$ and $y$, but one must exercise caution when applying certain formulas.  
\end{rmk}

In terms of the polylogarithm (\ref{eqn:polylogrthm}) as well as the Eisenstein series $E_{2}(\tau)$, it is straightforward to show
\begin{equation}\label{eqn:WPFUNCEISPOLY}
\wp(\tau, z) = \frac{1}{12}E_{2}(\tau) + \text{Li}_{-1}(y) + \sum_{n=1}^{\infty} \text{Li}_{-1}(q^{n}y) + \text{Li}_{-1}(q^{n}y^{-1}).  
\end{equation}
A meromorphic Jacobi form is a ratio of Jacobi forms, and indeed we can express the Weierstrass $\wp$-function as
\begin{equation} \label{eqn:WPFuncK3}
\wp(\tau , z) = \frac{1}{12} \frac{\varphi_{0,1}(\tau, z)}{\Theta^{2}(\tau, z)}.
\end{equation}

\subsubsection{Jacobi Cusp Forms of Index One}

We saw above that the weak Jacobi forms $\Theta^{2}$ and $\varphi_{0,1}$ are indeed, not holomorphic.  Nevertheless, we can use the modular discriminant $\Delta(\tau) = \eta^{24}(\tau)$ defined in (\ref{eqn:JacobiEtaa}) to construct forms which are.  We define holomorphic Jacobi forms of index 1 and weights 10 and 12, respectively by
\begin{equation}\label{eqn:JacCuspForms}
\varphi_{10,1}(\tau, z) = \Delta(\tau) \cdot \Theta^{2}(\tau, z), \,\,\,\,\,\,\,\,\,\,\,\,\,\,\,\,\,\, \varphi_{12,1}(\tau, z) = \Delta(\tau) \cdot \varphi_{0,1}(\tau, z).
\end{equation}
In fact, $\Delta(\tau)$ being a cusp form will force $\varphi_{10,1}$ and $\varphi_{12,1}$ to be Jacobi cusp forms of index 1.

\subsection{The Elliptic Genus of Calabi-Yau Manifolds}  

Weak Jacobi forms of weight zero arise in practice as elliptic genera of Calabi-Yau manifolds.  Recall from the previous chapter that we can interpret the elliptic genus $\text{Ell}_{q,y}(X)$ as a holomorphic function on $\mathfrak{H} \times \mathbb{C}$, and we can now ask how it transforms under the Jacobi group $SL_{2}(\mathbb{Z}) \ltimes \mathbb{Z}^{2}$.  A proof of the following can be found in \cite{gritsenko_elliptic_1999}.  

\begin{thm}\label{thm:EllGenCYMAN}
For a smooth compact Calabi-Yau manifold $X$ of even complex dimension $d$, the elliptic genus $\text{Ell}_{q,y}(X)$ is a weak Jacobi form of weight 0 and index $d/2$.  
\end{thm}

The elliptic genus is therefore a topological index which produces an automorphic form for all compact Calabi-Yau manifolds.  Though it contains no more data than merely all Chern numbers of $X$, the elliptic genus packages the information in an attractive way.  Because all Jacobi forms enjoy a Fourier expansion (\ref{eqn:JACfrmFEXP}), when $X$ is a compact Calabi-Yau manifold of even complex dimension, the Fourier expansion of the elliptic genus (\ref{eqn:FourDecompELLGEN}) agrees with that of a weak Jacobi form introduced above.    

Recalling the topological index interpretation (\ref{eqn:topindEllGenCoeff}) of the coefficients $c(n, l)$, on a compact Calabi-Yau manifold $X$, we can interpret $c(n, l)$ as the index of a Dirac operator twisted by $E_{n, l}$
\begin{equation}
c(n, l) = \int_{X} \text{ch}(E_{n, l}) \hat{A}(X)
\end{equation}
by noting that if $X$ is Calabi-Yau, then $X$ is a spin manifold, and the $\hat{A}(X)$-genus agrees with the Todd class.  Therefore, the elliptic genus of a compact Calabi-Yau manifold is a weak Jacobi form whose Fourier coefficients have an interpretation as twisted Dirac indices.  

\begin{Ex}[\bfseries Abelian Varieties]
The most immediate example of a Calabi-Yau manifold in each dimension is an abelian variety.  If $A$ is an abelian variety, using Atiyah-Bott localization we see that the elliptic genus $\text{Ell}_{q,y}(A)$ vanishes.  The reason of course, is that an abelian variety acts freely on itself.  Since the elliptic genus is given as an integral over $A$ of a particular class in cohomology, this localizes to the fixed point locus of the action, which is empty.  It follows that,
\begin{equation}
\text{Ell}_{q,y}(A) =0.
\end{equation}
\end{Ex}

\begin{Ex}[\bfseries K3 Surface]

If $X$ is a compact Calabi-Yau surface, then $\text{Ell}_{q,y}(X)$ is a weak Jacobi form of weight 0 and index 1.  Having already dismissed abelian surfaces, the remaining possibility is a K3 surface.  We know there is a unique weak Jacobi form of weight 0 and index 1 up to scale, which implies that $\text{Ell}_{q,y}(K3)$ is a multiple of $\varphi_{0,1}(\tau, z)$.  Because the elliptic genus is certainly a diffeomorphism invariant and all K3 surfaces are diffeomorphic, we are justified in writing $\text{Ell}_{q,y}(K3)$.  By considering the Euler characteristic, we must have $\text{Ell}_{q, 1}(K3)=24$, and from the low-order terms (\ref{eqn:phi01FEXP}) in the expansion of $\varphi_{0,1}(\tau, z)$, we can see that
\begin{equation} \label{eqn:EllgenK3}
\text{Ell}_{q,y}(K3) = 2 \varphi_{0,1}(\tau, z).
\end{equation}
Notice that the constant term in $q$ of $\text{Ell}_{q,y}(K3)$ is $2y + 20 + 2y^{1}$, which evidently contains all non-trivial entries in the Hodge diamond of a K3 surface, and is precisely $y^{-1}\chi_{-y}(K3)$.  

Finally, we note that as with all Jacobi forms of index one, the Fourier coefficients $c(n, l)$ only depend on the value of $\Delta = 4n -l^{2}$.  We will write the Fourier expansion as
\begin{equation}\label{eqn:K3EllGEnCoefff}
\text{Ell}_{q,y}(K3) = \sum_{n \geq 0, l \in \mathbb{Z}} c(4n- l^{2}) q^{n} y^{l}.  
\end{equation}
\end{Ex}

\subsection{Hecke Operators on Jacobi Forms} \label{sec:HeckeJacForm}

Hecke operators play a very large and important role in the theory of automorphic forms, which we obviously cannot do justice to here.  We will content ourselves to discussing some of the properties needed later in the thesis.  Of particular interest will be the role played by the Hecke operators in the Maass lift of Jacobi forms to Siegel modular forms.  Hecke operators (on holomorphic Jacobi forms) were originally introduced in \cite{eichler_theory_2013} with Borcherds \cite{borcherds_automorphic_1995} and Aoki \cite{aoki_notitle_2018} providing an important contribution in the weak case.  For our purposes, the exposition of \cite{kawai_string_2000} will also be helpful.  
  
\begin{defn}
Let $\varphi_{k, m} \in J^{w}_{k,m}$ be a weak Jacobi form of weight $k$ and index $m$.  The action of the Hecke operator $V_{N}$ on $\varphi_{k,m}$ is defined for all $N>0$ by
\begin{equation} \label{eqn:Heckedefn}
(\varphi_{k,m} | V_{N}) = N^{k-1} \sum_{\substack{ad =N \\ a>0}} \sum_{b=0}^{d-1} d^{-k} \varphi_{k,m}\bigg(\frac{a \tau + b}{d}, az\bigg).
\end{equation}  
\end{defn}
\noindent By setting $z=0$, we recover the definition of $V_{N}$ acting on the modular form $\varphi_{k,m}(\tau, 0)$.  From the Jacobi form transformation equations (\ref{eqn:JacFormModul}) and (\ref{eqn:JacFormElll}), one can show that 
\begin{equation} \label{eqn:Heckeindexch}
(\varphi_{k,m} | V_{N}) \in \mathbb{J}^{w}_{k, Nm}.
\end{equation} 

The following important lemma shows that the generating function of Hecke operators on weak Jacobi forms has a nice expression in terms of the polylogarithm defined in (\ref{eqn:polylogrthm}).

\begin{lemmy}\label{lemmy:genfunccheckoops}
Let $\varphi$ be a weak Jacobi form of weight $k$ with Fourier coefficents $c(n, l)$.  We then have
\begin{equation}\label{eqn: genfunccHeckkkeOpp}
\sum_{N=1}^{\infty} Q^{N} \big( \varphi \big| V_{N} \big) = \sum_{\substack{ m>0, n \geq 0, \\ l \in \mathbb{Z}}} c(nm, l) \text{Li}_{1-k}(Q^{m}q^{n}y^{l}).  
\end{equation}
\end{lemmy}  

\begin{proof}
We of course begin by directly applying the definition (\ref{eqn:Heckedefn}) as well as the Fourier expansion of a weak Jacobi form
\begin{equation}
\begin{split}
\sum_{N=1}^{\infty} Q^{N} \big( \varphi \big| V_{N} \big) & = \sum_{N=1}^{\infty} Q^{N} N^{k-1} \sum_{\substack{ad = N \\ a>0}} \sum_{b=0}^{d-1} d^{-k} \varphi\bigg(\frac{a \tau + b}{d}, az\bigg) \\
& = \sum_{N=1}^{\infty} Q^{N} N^{k-1} \sum_{\substack{ad = N \\ a>0}} \sum_{b=0}^{d-1} d^{-k} \sum_{n \geq 0, l \in \mathbb{Z}} c(n, l) q^{na/d} e^{2 \pi i nb /d} y^{al}. 
\end{split}
\end{equation}
To proceed, we observe that the sum over $b$ is simply a finite geometric series with ratio $e^{2 \pi i n/d}$.  This sum vanishes unless $d|n$, in which case it equals $d$.  We therefore get
\begin{equation}
\begin{split}
\sum_{N=1}^{\infty} Q^{N} N^{k-1} & \sum_{\substack{ad = N \\ a>0}}d^{-k+1} \sum_{n \geq 0, l \in \mathbb{Z}} c(dn, l) q^{na} y^{al} \\
& = \sum_{d=1}^{\infty} \sum_{n \geq 0, l \in \mathbb{Z}} c(dn, l) \sum_{a=1}^{\infty} a^{k-1} \big(Q^{d} q^{n} y^{l} \big)^{a}
\end{split}
\end{equation}
where we have used in the final equality that $(\frac{N}{d})^{k-1} = a^{k-1}$.  We notice that up to relabeling indices, this is exactly the righthand side of (\ref{eqn: genfunccHeckkkeOpp}), completing the proof.  
\end{proof}

The Hecke operator $V_{0}$ is also important, though it is more subtle to define.  

\begin{defn}
Let $\varphi_{k,m} \in \mathbb{J}^{w}_{k,m}$ be a weak Jacobi form of even weight $k \in 2 \mathbb{Z}$ with Fourier coefficients $c(n, l)$.  Then
\begin{equation}\label{T0WEAKHOL}
(\varphi_{k,m}| V_{0}) = c(0,0) \epsilon(k) + \sum_{\substack{ n \geq 0, l \in \mathbb{Z} \\ l >0 \, \text{if} \, n=0}} c(0,l) \text{Li}_{1-k} (q^{n} y^{l})
\end{equation}
where
\begin{equation}\label{eqn:epsilonkdef}
\epsilon(k) = 
\begin{cases}
\begin{aligned}
& \frac{1}{2}\zeta(1-k), & k<0 \\[1ex]
& 0, & k=0 \\[1ex]
& \frac{1}{2}\zeta(1-k) = -\frac{B_{k}}{2k}, & k>0.
\end{aligned}
\end{cases}
\end{equation} 
\end{defn}

\noindent If $\varphi_{k,m}$ is a holomorphic Jacobi form, then $c(0,l)=0$ for all $l \neq 0$, and (\ref{T0WEAKHOL}) therefore specializes to
\begin{equation}\label{eqn:T0HOL}
(\varphi_{k,m}|V_{0}) = - c(0,0) \frac{B_{k}}{2k} E_{k}(\tau).
\end{equation}

The definition in the holomorphic case appeared in \cite{eichler_theory_2013} while Borcherds \cite{borcherds_automorphic_1995} gave the definition of $V_{0}$ more generally in the weakly holomorphic case for positive weight.  For $k<0$, this definition can be found in \cite{kawai_string_2000}, where they also include the divergent term $\frac{1}{2}\zeta(1)$ for $k=0$.  We choose to omit this term.    

From (\ref{eqn:T0HOL}) one might be worried that the quasi-modular form $E_{2}(\tau)$ makes an appearance in $(\varphi_{2,m}|V_{0})$ for $\varphi_{2,m}$ a holomorphic Jacobi form of weight two.  However, since $\varphi_{2,m}(\tau, 0)$ must be a weight two holomorphic modular form, we must have $\varphi_{2,m}(\tau, 0) = 0$, which implies that $c(0,0)=0$.   

For $\varphi_{k,m} \in J^{w}_{k,m}$ with \emph{positive} even weight $k$, it was the idea of Borcherds \cite{borcherds_automorphic_1995} to use derivatives of the Weierstrass $\wp$-function to express $(\varphi_{k,m}|V_{0})$ as a meromorphic Jacobi form.  We define 
\begin{equation}
\wp^{(r)}(\tau, z) \coloneqq \frac{1}{(2 \pi i)^{r}} \frac{ \partial^{r}}{\partial z^{r}} \wp(\tau, z)
\end{equation}
and we remind the reader of the warning in Remark \ref{rmk:WPconvrmk}.  It is straightforward to see from (\ref{eqn:WPFUNCEISPOLY}) that for all even integers $k \geq 2$
\begin{equation}\label{eqn:WPderrr}
\wp^{(k-2)}(\tau, z) = \frac{\delta_{2,k}}{12} E_{2}(\tau) + \bigg(\text{Li}_{1-k}(y) + \sum_{n=1}^{\infty} \text{Li}_{1-k}(q^{n}y) + \text{Li}_{1-k}(q^{n}y^{-1})\bigg).
\end{equation}
Consistent with the above expression, it is clear from (\ref{eqn:WPFUNCEISPOLY}) that the Eisenstein series $E_{2}(\tau)$ only appears for $k=2$.  The derivatives $\wp^{(k-2)}(\tau, z)$ are meromorphic Jacobi forms of weight $k$ and index zero.  

Notice that by (\ref{eqn:T0HOL}) if $\varphi_{k,m}$ is holomorphic, then $(\varphi_{k,m}|V_{0})$ is a holomorphic Jacobi form of index zero, i.e. a modular form.  This is consistent with the behavior in (\ref{eqn:Heckeindexch}).  If $\varphi_{k,m}$ is not holomorphic, one might still expect $(\varphi_{k,m}|V_{0})$ to have index zero.  By the following proposition, this is indeed the case.  
\begin{proppy}
If $\varphi_{k,m}$ is a weak Jacobi form of even weight $k>0$, then $(\varphi_{k,m}| V_{0})$ is a meromorphic Jacobi form of weight $k$ and index 0.  More specifically, we have
\begin{equation}\label{eqn:WPV0wholjacfrm}
(\varphi_{k,m}|V_{0}) = - c(0,0) \frac{B_{k}}{2k} E_{k}(\tau) + \sum_{l > 0} c(0, l) \bigg( \wp^{(k-2)}(\tau, l z) - \frac{\delta_{2,k}}{12} E_{2}(\tau)\bigg).  
\end{equation}
\end{proppy}

\noindent Before beginning the proof, we make an important remark.  For $k=2$, because $\sum_{l \in \mathbb{Z}} c(0,l) =0$ for a weight two Jacobi form, by Lemma \ref{lemmy:evenweight} we know
\[\sum_{l >0} c(0,l) = -\frac{1}{2}c(0,0).\]
It follows that the terms proportional to $E_{2}(\tau)$ in (\ref{eqn:WPV0wholjacfrm}) cancel.  The full expression for $k=2$ is simply
\begin{equation}\label{eqn:weight2V0expr}
(\varphi_{2,m}|V_{0}) = \sum_{ l >0} c(0, l) \wp(\tau, l z).  
\end{equation}

\begin{proof}
We begin by rewriting the sum (\ref{T0WEAKHOL}) defining $(\varphi_{k,m}|V_{0})$ as follows
\begin{equation}\label{eqm:intermedddeqn}
\setlength{\jot}{8pt}
\begin{split}
(\varphi_{k,m}|V_{0}) & = -c(0,0) \frac{B_{k}}{2k} \bigg( 1 - \frac{2k}{B_{k}} \sum_{n=1}^{\infty}\text{Li}_{1-k}(q^{n})\bigg) \\
& + \sum_{n>0, l \in \mathbb{Z} \setminus \{ 0 \}} c(0, l) \text{Li}_{1-k}(q^{n} y^{l}) + \sum_{l > 0} c(0, l) \text{Li}_{1-k}(y^{l}).  
\end{split}
\end{equation}
By (\ref{eqn:EISserexpprPOLY}), the quantity in parentheses in the first term above, is simply the Eisenstein series $E_{k}(\tau)$.  Because the weight of $\varphi_{k,m}$ is even, we have $c(0, l) = c(0, -l)$ for all $l \in \mathbb{Z}$.  Therefore, using (\ref{eqn:WPderrr}) to provide an expression for $\wp^{(k-2)}(\tau, l z)$, it is straightforward to see that (\ref{eqm:intermedddeqn}) takes the desired form of (\ref{eqn:WPV0wholjacfrm}).  
\end{proof}

So far, we have not really motivated our interest in Hecke operators.  They are of great importance in many parts of number theory and automorphic forms, but for the purposes of this thesis, the main interest lies in the definition of the following object.

\begin{defn}
The Maass lift $\text{ML}(\varphi)$ of a weak Jacobi form $\varphi$ is defined by
\begin{equation}\label{eqn:maassriftttt}
\text{ML}(\varphi) = \sum_{m=0}^{\infty} Q^{m} \big( \varphi \big| V_{m}\big).
\end{equation}
The Maass lift is sometimes called the additive lift, because $\text{ML}(\varphi_{1} + \varphi_{2}) = \text{ML}(\varphi_{1}) + \text{ML}(\varphi_{2})$.  
\end{defn}

\noindent We will tend to think of $\text{ML}(\varphi)$ as a function of the three variables $(Q,q,y)$, though one can see from (\ref{eqn:maassriftttt}) that it is really a power series in $Q$ whose coefficients are holomorphic functions on $\mathfrak{H} \times \mathbb{C}$.  

The following lemma expresses the Maass lift of a weak Jacobi form in terms of its Fourier coefficients and the polylogarithm.  The proof follows immediately from Lemma \ref{lemmy:genfunccheckoops} and (\ref{T0WEAKHOL}).  
\begin{lemmy}
Let $\varphi_{k, m} \in \mathbb{J}^{w}_{k,m}$ be a weak Jacobi form of weight $k \in 2 \mathbb{Z}$ with Fourier coefficients $c(n, l)$.  Then
\begin{equation}\label{eqn:fullHecke}
\text{ML} (\varphi_{k,m}) = c(0,0) \epsilon(k) + \sum_{(m, n, l) >0} c(nm, l) \text{Li}_{1-k}\big(Q^{m} q^{n} y^{l}\big) 
\end{equation}
where $\epsilon(k)$ is defined in (\ref{eqn:epsilonkdef}).  The notation $(m, n, l)>0$ means any of the following conditions hold
\[ (i) \,\,\, m>0, \,\,\,\,\,\,\,\,\,\,\, (ii) \,\,\, m=0, \,\, n>0, \,\,\,\,\,\,\,\,\,\,\, (iii) \,\,\, m=n=0, \,\, l>0.\]
\end{lemmy}

This result has a purely formal proof because we are not yet making any claims about the automorphy of $\text{ML}(\varphi_{k,m})$.  It is a rather deep collection of results that the Maass lift of certain Jacobi forms indeed has automorphic properties.  For holomorphic or weak Jacobi forms, the Maass lift is a possibly meromorphic Siegel modular form.  We now take the opportunity to introduce these objects.

\section{A Brief Foray into Siegel Modular Forms and Maass Lifting}  

Recall that a Jacobi form is a two-variable generalization of a modular form given by replacing $\Gamma_{1} = SL_{2}(\mathbb{Z})$ by the Jacobi group $SL_{2}(\mathbb{Z}) \ltimes \mathbb{Z}^{2}$.  \emph{Siegel modular forms} provide an additional generalization of ordinary modular forms by replacing the upper-half plane $\mathfrak{H}$ by the Siegel upper-half plane $\mathfrak{H}_{g}$, and replacing $\Gamma_{1} $ with the group $\Gamma_{g} = Sp_{2g}(\mathbb{Z})$.  We call the integer $g \geq 1$ the \emph{degree} or \emph{genus} of the form.  A Siegel modular form of genus $g=1$ is simply an ordinary modular form.  Therefore, genus two Siegel modular forms are the next interesting case to consider.  We will reserve ourselves to $g=2$ in what follows, and refer the reader to \cite{van_der_geer_siegel_2006} for a more systematic, and complete account for all genus.

\subsection{Siegel Modular Forms of Genus Two}\label{sec:SMFrrrrmSs}

We first generalize the upper-half plane $\mathfrak{H}$ to the Siegel upper-half plane $\mathfrak{H}_{2}$, which is defined to consist of complex symmetric $2 \times 2$ matrices with positive definite imaginary part.\footnote{We define the imaginary part of a complex matrix to be the matrix of imaginary components of all entries.}  This can be given explicitly as
\begin{equation}\label{eqn:SiegelUPPHarfprane}
\mathfrak{H}_{2} = \bigg\{ \Omega = 
\begin{pmatrix}
\tau & z\\
z & \sigma
\end{pmatrix}
\in \text{Mat}_{2}(\mathbb{C}) \, \bigg| \, \mathfrak{Im}(\tau), \mathfrak{Im}(\sigma) >0, \text{det}\big(\mathfrak{Im}(\Omega)\big)>0 \bigg\}.
\end{equation}
This is clearly a generalization of the ordinary upper-half plane $\mathfrak{H}$ in genus one.  We know $\mathfrak{H}$ carries an action by the modular group $SL_{2}(\mathbb{Z})$, so the next order of business is to generalize the ordinary modular group to higher genus.  We define the real symplectic group
\begin{equation}
Sp_{4}(\mathbb{R}) = \big\{ M \in \text{Mat}_{4}(\mathbb{R}) \, \big| \, M J_{4} M^{T} = J_{4} \big\}, \,\,\,\,\,\, J_{4} = 
\begin{pmatrix}
0 & -1_{2}\\
1_{2} & 0
\end{pmatrix},
\end{equation}
to be the group of real $4 \times 4$ matrices preserving the symplectic form $J_{4}$.  This group can be given more concretely in the following block form
\begin{equation}
Sp_{4}(\mathbb{R}) = 
\bigg\{ M = 
\begin{pmatrix}
A & B\\
C & D
\end{pmatrix}
\, \bigg| \, AB^{T}=BA^{T}, \, CD^{T} = DC^{T}, \, AD^{T}-BC^{T}=1_{2} \bigg\}
\end{equation}
where $A, B, C, D$ are $2 \times 2$ real matrices.  Considering elements in this block form, we have a transitive action of $Sp_{4}(\mathbb{R})$ on $\mathfrak{H}_{2}$ defined by
\begin{equation}\label{eqn:SMGAction}
M \cdot \Omega = (A \Omega + B)(C \Omega + D)^{-1}, \,\,\,\,\,\, \Omega \in \mathfrak{H}_{2}.
\end{equation}
One must then check this is well-defined; for example $C \Omega + D$ must be invertible \cite{van_der_geer_siegel_2006}.  

Generalizing Remark \ref{rmk:CosetDescrrr} for ordinary modular forms, the coset description of the Siegel upper-half plane is the biholomorphism 
\begin{equation}\label{eqn:cosdestrcSIEGL}
\mathfrak{H}_{2} \cong Sp_{4}(\mathbb{R}) \big/ U(2)
\end{equation}
where the unitary group $U(2) \subset Sp_{4}(\mathbb{R})$ is a maximal compact subgroup.  It is clear that (\ref{eqn:cosdestrcSIEGL}) is a diffeomorphism, as $U(2)$ is the stabilizer of $i \cdot 1_{2} \in \mathfrak{H}_{2}$, but we refer to \cite{milne_introduction_2005} for understanding the complex structure on $Sp_{4}(\mathbb{R}) \big/ U(2)$, which is not at all obvious.  

The \emph{Siegel modular group} $\Gamma_{2} = Sp_{4}(\mathbb{Z}) \subset Sp_{4}(\mathbb{R})$ is the subgroup such that the matrices have integer entries.  The Siegel upper-half plane $\mathfrak{H}_{2}$ carries an action of $Sp_{4}(\mathbb{Z})$ by (\ref{eqn:SMGAction}) which is evidently a generalization of the $SL_{2}(\mathbb{Z})$ action on $\mathfrak{H}$.  Recalling that we have reserved ourselves to genus two, note that one can more generally define $\mathfrak{H}_{g}$ and $Sp_{2g}(\mathbb{Z})$, which we will not do.

\begin{defn}
A Siegel modular form of weight $k$ and genus two is a holomorphic function $F: \mathfrak{H}_{2} \to \mathbb{C}$ such that 
\begin{equation}\label{eqn:defnnnSMFDS}
F(M \cdot \Omega) = \text{det}(C \Omega + D)^{k} F(\Omega)
\end{equation}
for all $\Omega = \bigl( \begin{smallmatrix}\tau & z\\ z & \sigma \end{smallmatrix}\bigr) \in \mathfrak{H}_{2}$ and $M = \bigl( \begin{smallmatrix}A & B\\ C & D\end{smallmatrix}\bigr) \in Sp_{4}(\mathbb{Z})$.  We will write either $F(\Omega)$ or $F(\tau, z, \sigma)$.  
\end{defn}
 
\noindent Such an object is sometimes called a \emph{holomorphic} Siegel modular form.  These are automorphic forms on 
\begin{equation}
Sp_{4}(\mathbb{Z}) \big\backslash Sp_{4}(\mathbb{R}) \big/ U(2).  
\end{equation}  
We denote the vector space of genus two Siegel modular forms of weight $k$ by $\mathfrak{M}_{k}(\Gamma_{2})$.  The full ring of Siegel modular forms $\mathfrak{M}_{*}(\Gamma_{2})$ is a graded ring in the obvious way.

\subsection{The Fourier-Jacobi Expansion}  \label{subsec:TheFJEXPPP}

As with Jacobi forms, we introduce the parameters $q = e^{2 \pi i \tau}$, $y = e^{2 \pi i z}$, and we now additionally define $Q = e^{2 \pi i \sigma}$.  A Siegel modular form $F$ has the following Fourier expansion
\begin{equation}
F(\Omega) = \sum_{\substack{m, n, l \in \mathbb{Z} \\ m, n, 4nm -l^{2} \geq 0}} A(m, n, l)Q^{m} q^{n} y^{l}.
\end{equation}
From the above expansion, it is tempting to hope that the coefficient of a fixed power of $Q$ is a Jacobi form in variables $q$ and $y$.  This is indeed the case, and provides a nice connection to the theory of Jacobi forms.  

\begin{thm}[\bfseries {\cite[Thm. 6.1]{eichler_theory_2013}}]
Given a Siegel modular form $F$ of weight $k$ and genus two, we have an expansion
\begin{equation}\label{eqn:FJExpp}
F(\Omega) = \sum_{m=0}^{\infty} Q^{m} \varphi_{k,m}(\tau, z)
\end{equation}
where $\varphi_{k,m}$ is a (holomorphic) Jacobi form of weight $k$ and index $m$.  
\end{thm}

\noindent We refer to (\ref{eqn:FJExpp}) as the \emph{Fourier-Jacobi expansion} of $F$.  This theorem is special to genus two Siegel modular forms.  The coefficient $\varphi_{k,0}$ of the Fourier-Jacobi expansion of $F$ is a Jacobi form of index 0, which we have previously seen to be independent of $z$, and is in fact simply a modular form of weight $k$.  

Given an ordinary modular form, there is a simple procedure to extract the constant term in the Fourier expansion.  The analogous procedure for Siegel modular forms is applying the Siegel operator.

\begin{defn}
The Siegel operator $\Phi$ is a lowering operator on the genus $g$ of a Siegel modular form.  In the case of $g=2$, the linear map
\begin{equation}\label{eqn:SiegelOPPP}
\Phi : \mathfrak{M}_{k}(\Gamma_{2}) \longrightarrow M_{k}(\Gamma_{1})
\end{equation}
takes a Siegel modular form $F(\Omega) = F(\tau, z, \sigma)$ of weight $k$ and produces a modular form $\Phi(F)$ of weight $k$ defined by
\begin{equation}
\Phi(F) \coloneqq \lim_{t \to \infty} F(\tau, 0, it).
\end{equation}
\end{defn}

\noindent It is clear that $\Phi$ maps $F$ to the coefficient $\varphi_{k,0}$ of the Fourier-Jacobi expansion of $F$.  By the above discussion, $\varphi_{k,0}$ is independent of $z$ and is a modular form of weight $k$.  

\begin{defn}
We define the ideal of Siegel cusp forms of genus two and weight $k$ by
\begin{equation}
\mathfrak{S}_{k}(\Gamma_{2}) \coloneqq \text{ker}(\Phi).
\end{equation}  
\end{defn}

\noindent Note that this is entirely analogous to the definition of cusp modular forms $S_{k}(\Gamma_{1})$.  It is clear that a genus two Siegel modular form is a cusp form if and only if the coefficient $\varphi_{k,0}$ of the Fourier-Jacobi expansion vanishes.

\subsection{Maass `Spezialschar' and Index One Jacobi Forms} \label{subsec: MaassliftIndoneJFOR}

From the discussion of Hecke operators, given a holomorphic Jacobi form $\varphi_{k,m}$, for all $N \geq 0$, we have seen that $\big(\varphi_{k,m}\big| V_{N}\big)$ is a holomorphic Jacobi form of weight $k$ and index $Nm$.  Considering the special case of index 1 Jacobi forms, we have $(\varphi_{k,1} | V_{m}) \in \mathbb{J}_{k,m}$ and the Maass lift (\ref{eqn:maassriftttt}) appears to be the Fourier-Jacobi expansion of a Siegel modular form!  This is indeed the case.  

\begin{thm}[\bfseries {\cite[Thm. 6.2]{eichler_theory_2013}}]
Let $\varphi_{k,1} \in \mathbb{J}_{k,1}$ be a holomorphic Jacobi form of weight $k$ and index 1.  The Maass lift
\begin{equation}
\text{ML}(\varphi_{k,1}) = \sum_{m=0}^{\infty}Q^{m}\big(\varphi_{k,1}\big|V_{m}\big)
\end{equation}
is the Fourier-Jacobi expansion of a Siegel modular form $\text{ML}(\varphi_{k,1})$ of weight $k$ and genus two.  Defining the Maass `Spezialschar' as the image of the Maass lift in $\mathfrak{M}_{k}(\Gamma_{2})$, we have the following isomorphism
\begin{equation}
\mathbb{J}_{k,1} \overset{\sim}{\longleftrightarrow} \big\{ \text{Maass Spezialschar}  \big\} \subset \mathfrak{M}_{k}(\Gamma_{2})   
\end{equation}
where given $\varphi_{k,1} \in \mathbb{J}_{k,1}$ one associates the Maass lift $\text{ML}(\varphi_{k,1})$ and conversely, given a weight $k$ Siegel modular form $F$ in the Maass Spezialschar, one recovers a weight $k$ index 1 holomorphic Jacobi form as the coefficient of $Q$ in the Fourier-Jacobi expansion.  
\end{thm}

The Maass lift which we defined purely formally in (\ref{eqn:maassriftttt}) is now seen to have nice automorphic properties when the Jacobi form is holomorphic and index 1.  By results of H. Aoki \cite{aoki_formal_2014,aoki_notitle_2018}, the same is true for weak Jacobi forms of index 1, a fact which will be crucial for our results in the final chapter.  We should remark that what we are calling the Maass lift $\text{ML}(\varphi_{k,1})$ is often referred to as the Saito-Kurokawa lift or the additive lift of $\varphi_{k,1}$.

One should think of the \emph{Maass `Spezialschar'} as consisting of those special Siegel modular forms which are the Maass lift of an index 1 Jacobi form.  This space may be equivalently characterized by the behavior of Fourier coefficients \cite{dabholkar_quantum_2012}.  By writing the action of a Hecke operator in terms of Fourier coefficients, one can show that $F \in \mathfrak{M}_{k}(\Gamma_{2})$ lies in the Maass `Spezialschar' if and only if the coefficients $A(m, n, l)$ satisfy
\begin{equation}\label{eqn:MaassSPEZFCOES}
A(m,n,l) = \sum_{\substack{r | (m,n,l) \\ r>0}} r^{k-1} c\bigg( \frac{4nm-l^{2}}{r^{2}}\bigg), \,\,\,\,\,\,\,\,\,\,\,\,\,\, (m,n,l) \neq (0,0,0)
\end{equation}
where $c(\cdot)$ are the Fourier coefficients of the weight $k$ index 1 Jacobi form arising as the coefficient of $Q$ in the Fourier-Jacobi expansion of $F$.  Recall that the Fourier coefficients of index 1 Jacobi forms depend only on the quantity $4n-l^{2}$.  However, the Fourier coefficients $A(m, n, l)$ of the Maass lift depend on not only $4nm-l^{2}$, but also the divisibility of the triple $(m, n, l)$.

\subsection{The Ring of Genus Two Siegel Modular Forms}\label{subsec:RingofSMF}

We should now present the four classical examples of genus two Siegel modular forms, as introduced by Igusa.  It happens that these are the examples of interest to us in the final chapter, and they turn out to generate the ring of genus two Siegel modular forms of even weight.  

Let us refer the reader to (\ref{eqn:JacCuspForms}) where we introduced the Jacobi cusp forms $\varphi_{10, 1}$ and $\varphi_{12,1}$ of index 1.  The two most famous Siegel modular forms are defined as the Maass lifts of $\varphi_{10,1}$ and $\varphi_{12,1}$ respectively
\begin{equation}
\chi_{10}(\Omega) = \text{ML}(\varphi_{10,1}) = \sum_{m=1}^{\infty} Q^{m} \big(\varphi_{10,1} \big| V_{m}\big)
\end{equation}
\begin{equation}
\chi_{12}(\Omega) = \text{ML}(\varphi_{12,1}) = \sum_{m=1}^{\infty} Q^{m} \big(\varphi_{12,1} \big| V_{m}\big).
\end{equation}
We call $\chi_{10}$ and $\chi_{12}$ \emph{Igusa cusp forms} of weight 10 and 12, respectively.  Recalling (\ref{eqn:T0HOL}), the action of $V_{0}$ on a Jacobi cusp form vanishes, which is why the Igusa cusp forms have no $Q^{0}$ term.  They are the unique genus two Siegel cusp forms of their weight up to scale -- they respectively generate $\mathfrak{S}_{10}(\Gamma_{2})$ and $\mathfrak{S}_{12}(\Gamma_{2})$.  

The second collection of examples generalize the ordinary Eisenstein series, and are defined as follows for all $k >2$
\begin{equation}\label{eqn:defnS-EEEseriess}
\mathcal{E}_{k}(\Omega)  =  \sum_{(C, D)} \text{det}\big( C \Omega + D\big)^{-k}.
\end{equation}
The sum is over pairs of coprime symmetric integral matrices, non-associated with respect to multiplication on the left by $GL_{2}(\mathbb{Z})$.  We call $\mathcal{E}_{k}$ the \emph{Siegel-Eisenstein series} of weight $k$.  It is a holomorphic Siegel modular form of weight $k$ and genus two.  The normalization is chosen such that \cite{van_der_geer_siegel_2006}
\begin{equation}
\Phi(\mathcal{E}_{k}) = E_{k}
\end{equation}
where $\Phi$ is the Siegel operator defined in (\ref{eqn:SiegelOPPP}).  This simply says that the $Q^{0}$ term in the Fourier-Jacobi expansion of $\mathcal{E}_{k}$ is the Eisenstein series $E_{k}$ of weight $k$, normalized such that the constant Fourier coefficient is 1.  The Siegel operator is clearly not only a linear map, but a ring homomorphism, which means that $\Phi(\mathcal{E}_{k_{1}} \mathcal{E}_{k_{2}}) = E_{k_{1}} E_{k_{2}}$.  For example, we have
\begin{equation}
\Phi(\mathcal{E}_{10} - \mathcal{E}_{4}\mathcal{E}_{6}) = E_{10} - E_{4}E_{6} =0
\end{equation}
since $M_{10}(\Gamma_{1})$ is one-dimensional, and the Eisenstein series are normalized to have constant Fourier coefficient 1.  Therefore, $\mathcal{E}_{10} - \mathcal{E}_{4}\mathcal{E}_{6}$ is a Siegel cusp form of weight 10, and the space of such forms is one-dimensional.  We must then have that $\mathcal{E}_{10}-\mathcal{E}_{4}\mathcal{E}_{6}$ is proportional to $\chi_{10}$.  Similarly, one can show that $\mathcal{E}_{12}-\mathcal{E}_{6}^{2}$ is a Siegel cusp form of weight 12, proportional to $\chi_{12}$.  

The following is the structure theorem for genus two Siegel modular forms of even weight, analogous to Theorem \ref{eqn:ModformStrTHMMM} in the case of $M_{*}(\Gamma_{1})$.  It was proven by Igusa \cite{igusa_siegel_1962, igusa_modular_1967} in the 1960s.  

\begin{thm}\label{IgusacrassTHMMM}
The graded ring $\mathfrak{M}_{2*}(\Gamma_{2})$ of Siegel modular forms of even weight and genus two is given as the polynomial ring
\begin{equation}
\mathfrak{M}_{2*}(\Gamma_{2}) = \mathbb{C} \big[ \chi_{10}, \chi_{12}, \mathcal{E}_{4}, \mathcal{E}_{6} \big].
\end{equation}
In other words, it is freely generated over $\mathbb{C}$ by the Igusa cusp forms $\chi_{10}$ and $\chi_{12}$ as well as the Siegel-Eisenstein series $\mathcal{E}_{4}$ and $\mathcal{E}_{6}$.  
\end{thm}

\noindent Igusa also introduced a cusp form $\chi_{35}$ of odd weight whose square is an explicit polynomial in $\chi_{10}, \chi_{12}, \mathcal{E}_{4}$, and $\mathcal{E}_{6}$.  In this thesis, we will only be interested in the ring of even weight Siegel modular forms so we will not discuss $\chi_{35}$ further.

\subsection{Infinite Products and Borcherds Lifts}  

By a result of Gritsenko and Nikulin \cite{gritsenko_siegel_1997} the Igusa cusp form of weight 10 has the following \emph{infinite product} representation
\begin{equation}\label{eqn:GNIKULinfprodd}
\chi_{10}(\Omega) = Qqy\prod_{(m,n, l)>0} \big(1-Q^{m} q^{n} y^{l}\big)^{c(4nm-l^{2})}
\end{equation}
where in this case, the notation $(m,n, l)>0$ means either $m>0$, or $m=0, n>0$, or $m=n=0, l<0$.  The exponents $c(4nm-l^{2})$ in the product representation are the Fourier coefficients of the elliptic genus $\text{Ell}_{q,y}(K3)$ of a K3 surface (\ref{eqn:K3EllGEnCoefff}).  Using the definition of $(m,n, l)>0$ as well as some of the low-order Fourier coefficients of $\text{Ell}_{q,y}(K3) = 2 \varphi_{0,1}$ (\ref{eqn:phi01FEXP}), it is a simple exercise to show
\begin{equation}\label{eqn:GNikchi10prod}
\chi_{10}(\Omega) = Q \, \varphi_{10,1}(\tau, z) \prod_{\substack{m>0, n \geq 0 \\ l \in \mathbb{Z}}} \big(1 - Q^{m}q^{n}y^{l}\big)^{c(4nm-l^{2})}.
\end{equation}
One should recognize the infinite product in (\ref{eqn:GNikchi10prod}) as (the inverse of) the second quantized elliptic genus of a $K3$ surface from the DMVV formula.  Indeed, by (\ref{eqn:DMVVsmmsurfSS}) we have the following formula for the inverse of the Igusa cusp form
\begin{equation}\label{eqn:chi10prodDMVVform}
\frac{1}{\chi_{10}(\Omega)} = \frac{1}{Q \, \varphi_{10,1}(\tau, z)} \sum_{m=0}^{\infty} Q^{m} \text{Ell}_{q,y}(\text{Hilb}^{m}(K3)).
\end{equation}
The inverse of the Igusa cusp form is a \emph{meromorphic} Siegel modular form of weight $-10$, and it plays a leading role in the enumerative geometry of $K3 \times E$ \cite{oberdieck_holomorphic_2018}.  

We can also write $\chi_{10}$ as a \emph{Borcherd's lift} or \emph{multiplicative lift} of the elliptic genus $\text{Ell}_{q,y}(K3)$
\begin{equation}
\chi_{10}(\Omega) = Q \varphi_{10,1} \text{exp}\bigg( - \sum_{m=1}^{\infty} Q^{m} \big(\text{Ell}_{q,y}(K3) \big| V_{m}\big)\bigg).
\end{equation}
One can roughly think of a Borcherds lift \cite{borcherds_automorphic_1995} as the exponentiation of a Maass lift, up to a prefactor.  To summarize, $\chi_{10}$ is defined as the Maass lift of $\varphi_{10,1}$, and has an infinite product representation.  It is also the Borcherds lift of $\text{Ell}_{q,y}(K3)$ and is clearly related to the second quantized elliptic genus of a $K3$ surface.

\chapter{The Automorphic Properties of the Banana Manifold Partition Functions} \label{ch:AutGWPotBanana}

In this final chapter, much of the background material presented previously will culminate in the original results of this thesis.  We study the enumerative geometry of the banana manifold and find surprising connections to Hecke operators, Borcherds and Maass lifts, as well as Siegel modular forms.  Let us briefly summarize the results of the chapter, necessarily allowing some details and definitions to follow later.  For the reader's convenience, we refer back to previous sections in the thesis containing the necessary background material.  

The \emph{banana manifold} $X_{\text{ban}}$ is a smooth projective Calabi-Yau threefold fibered over $\mathbb{P}^{1}$ with generic fiber a smooth abelian surface.  One can construct $X_{\text{ban}}$ as follows: let $r : S \to \mathbb{P}^{1}$ be a generic rational elliptic surface, and form the self-fibered product $S \times_{\mathbb{P}^{1}} S$.  Because there are 12 singular fibers of $r$, each of which is a nodal elliptic curve, there are 12 conifold singularities in the fibered product.  All singularities are contained in the diagonal $\Delta \subset S \times_{\mathbb{P}^{1}} S$, which is a Weil divisor.  We define the banana manifold to be the full conifold resolution of singularities 
\begin{equation}\label{eqn:origdefnBAN}
X_{\text{ban}} \coloneqq \text{Bl}_{\Delta}(S \times_{\mathbb{P}^{1}} S).
\end{equation}
Defining the natural projection $\pi : X_{\text{ban}} \to \mathbb{P}^{1}$, the smooth fibers are isomorphic to $E \times E$, where $E$ is a smooth elliptic curve.    
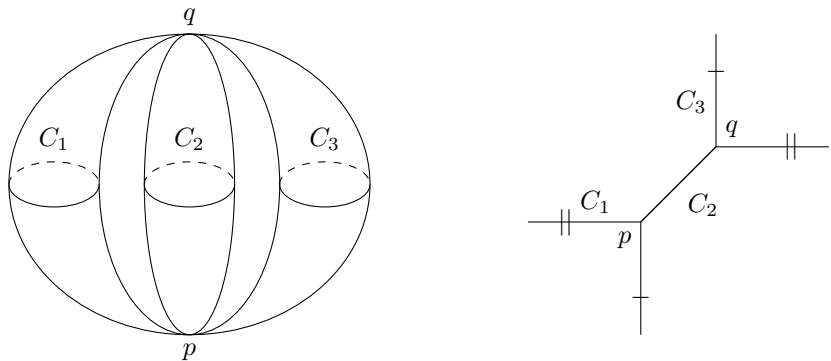
\begin{figure}[h] 
\centering
\begin{tikzpicture}[xshift=5cm,
		    scale = 1.0
		    ]

\begin{scope}  
\draw (0,0) ellipse (2.4 and 2);
\draw (0,0) ellipse (0.6cm and 2cm);
\draw (0,0) ellipse (1.2cm and 2cm);
\draw (-0.6,0) arc(180:360:0.6 and 0.3);
\draw[dashed](-0.6,0) arc(180:0:0.6cm and 0.3cm);
\draw (1.2,0) arc(180:360:0.6cm and 0.3cm);
\draw[dashed](1.2,0) arc(180:0:0.6cm and 0.3cm);
\draw (-2.4,0) arc(180:360:0.6cm and 0.3cm);
\draw[dashed](-2.4,0) arc(180:0:0.6cm and 0.3cm);

\draw (0,-2) node[below] {$p$};
\draw(0,2) node[above] {$q$};
\draw(0,0.6) node {$C_{2}$};
\draw(-1.8,0.6) node {$C_{1}$};
\draw(1.8,0.6)node {$C_{3}$};
\end{scope}


\begin{scope}[xshift=4.5cm,yshift=-2cm]
\draw (0,1.5)--(1.5,1.5)--(2.5,2.5)--(4,2.5);
\draw (1.5,0)--(1.5,1.5)--(2.5,2.5)--(2.5,4);
\draw (0.5,1.5)node{$||$};
\draw (3.5,2.5)node{$||$};
\draw (1.5,0.5)node{$-$};
\draw (2.5,3.5)node{$-$};
\draw (2.5,3.1)node[left]{$C_{3}$};
\draw (2,2)node[below right]{$C_{2}$};
\draw (0.9,1.5)node[above]{$C_{1}$};
\draw (1.5,1.5)node[below left]{$p$};
\draw (2.5,2.5)node[above right]{$q$};
\end{scope}
\end{tikzpicture}
\caption{A banana configuration of curves}
\label{fig: banana configuration}
\end{figure}

\noindent There are 12 singular fibers of $\pi$, each containing a \emph{banana configuration} of curves -- this consists of three rational curves $C_{1}, C_{2}, C_{3}$ all meeting in two distinct points $p,q \in X_{\text{ban}}$ (see Figure \ref{fig: banana configuration}).  The classes in homology of $C_{1}, C_{2}, C_{3}$ generate the three-dimensional lattice of fiber curve classes
\begin{equation}
\Gamma = \text{ker}(\pi_{*}) \subset H_{2}(X_{\text{ban}}, \mathbb{Z}).
\end{equation}
Let $Q_{1}, Q_{2}, Q_{3}$ be formal variables tracking degrees along the three banana curves.  It turns out there is a change of variables (\ref{eqn:changeofvarss}) into those defined by
\begin{equation}\label{eqn:innchvarsintrod}
Q=e^{2 \pi i \sigma}, \,\,\,\,\,\, q = e^{2 \pi i \tau}, \,\,\,\,\,\, y=e^{2 \pi i z} \,\,\,\,\,\,\,\,\,\,\,\,\,\,\,\,\,\, \Omega \coloneqq \begin{pmatrix} \tau & z \\ z & \sigma \end{pmatrix} \in \mathfrak{H}_{2}
\end{equation}
where $\Omega$ is an element of the Siegel upper-half plane $\mathfrak{H}_{2}$ which we introduced in Section \ref{sec:SMFrrrrmSs}.  For more details on the geometry of $X_{\text{ban}}$, see \cite{bryan_donaldson-thomas_2018,leigh_enumerative_2019,leigh_donaldson_2019}.

In this chapter, we will be studying the standard generating functions of Gromov-Witten, Donaldson-Thomas, and Gopakumar-Vafa invariants of the banana manifold.  These three curve-counting theories were introduced in Sections \ref{sec:GWThhhh}, \ref{eqn:DTThhhh}, and \ref{sec:GVinvvvv} respectively.  In terms of variables tracking curve classes, we will switch between $Q_{1}, Q_{2}, Q_{3}$ and those defined in (\ref{eqn:innchvarsintrod}) related to the Siegel upper-half plane.  The central object in the construction of the relevant generating functions is the equivariant elliptic genus of $\mathbb{C}^{2}$ (Section \ref{sec:equivvvindiccsec}) which we denote in this chapter for convenience as
\begin{equation}
\Phi_{0}(\tau, z, x) \coloneqq \text{Ell}_{q,y}(\mathbb{C}^{2}; t).
\end{equation}
We will prove that $\Phi_{0}$ is a weak Jacobi form of weight 0 and matrix index.  The Gopakumar-Vafa invariants of $X_{\text{ban}}$ are encoded non-trivially into the coefficients of $12 \Phi_{0}$, and have the interesting property that they depend for all genus \emph{only} on the value of the quadratic form $4nm-l^{2}$, where classes $(m,n,l)$ are those tracked by the variables $Q, q, y$.    

As a corollary to a theorem of J. Bryan \cite{bryan_donaldson-thomas_2018}, we prove that under a change of variables, the Donaldson-Thomas partition function of $X_{\text{ban}}$ restricted to the lattice of fiber classes $\Gamma$ can be identified as the formal Borcherds lift $\text{BL}(-)$ of $12 \Phi_{0}$ (see Proposition \ref{proppy:DTBryanchvarss}) 
\begin{equation}
Z_{\text{DT}}(X_{\text{ban}})_{\Gamma} = \text{BL}(12 \Phi_{0}) = \prod_{(m,n,l,k)>0} \big(1-Q^{m}q^{n}y^{l}t^{k}\big)^{-12 c(4nm-l^{2}, k)}.
\end{equation}
The formal Borcherds lift is related to the second quantized equivariant elliptic genus of $\mathbb{C}^{2}$.  We choose to emphasize the Borcherds lift perspective because it is closely analogous to other models.  For example, Kawai-Yoshioka \cite{kawai_string_2000} write the topological string partition function of certain $K3$ fibrations as infinite products by Borcherds lifting a weight 0 Weyl-invariant Jacobi form.  In addition, the Borcherds lift of $\text{Ell}_{q,y}(K3)$ -- the unique weak Jacobi form of weight 0 and index 1 -- gives the full Donaldson-Thomas partition function for $K3 \times E$ as an infinite product \cite{oberdieck_holomorphic_2018}.  Also, the formal Borcherds lift contains factors relevant to the Donaldson-Thomas partition function not encoded by the second quantized equivariant elliptic genus (\ref{eqn:rerrrnBLandSQEEG}).  

With this result in hand, we find very nice arithmetic and automorphic structure by passing to the Gromov-Witten theory of $X_{\text{ban}}$, assuming the GW/DT correspondence holds.  The equivariant elliptic genus admits the following Laurent expansion in the variable $\lambda = 2 \pi x$, where $t = e^{i \lambda}$
\begin{equation}
\Phi_{0}(\tau, z, x) = \sum_{g=0}^{\infty} \lambda^{2g-2} \psi_{2g-2}(\tau, z).
\end{equation} 
The functions $\psi_{2g-2}$ are explicit weak Jacobi forms of weight $2g-2$ and index 1 given below (\ref{eqn:defnofphys}).  Recalling the Maass lift $\text{ML}(-)$ of weak Jacobi forms introduced in Section \ref{sec:HeckeJacForm}, one of our main results is the following.  
\begin{thm}
For all genus $g \geq 2$, the Gromov-Witten potential $F_{g}$ of the banana manifold restricted to fiber classes is the Maass lift of $12 \psi_{2g-2}$
\begin{equation}
F_{g}(\Omega) = \text{ML}(12 \psi_{2g-2}).  
\end{equation}
\end{thm}

By results of H. Aoki, the Maass lift of a weak Jacobi form of weight $k$ and index 1 is a meromorphic Siegel modular form of weight $k$ and genus two.  The image of the Maass lift is known as the \emph{Maass `Spezialcshar'} inside the ring of meromorphic Siegel modular forms of genus two.  Following a suggestion of G. Oberdieck, we combine our results with those of H. Aoki to conclude the following.

\begin{thm}
For all genus $g \geq 2$, the Gromov-Witten potentials $F_{g}$ of $X_{\text{ban}}$ restricted to fiber classes are meromorphic Siegel modular forms of weight $2g-2$ and genus two lying in the Maass `Spezialschar.'  Moreover, $\chi_{10}^{g-1}F_{g}$ is a holomorphic Siegel modular form of genus two.  
\end{thm}

\noindent In Remark \ref{rmk:MIRRORSYYMM} we give a conjectural explanation in terms of mirror symmetry of why genus two Siegel modular forms are arising as the Gromov-Witten potentials in this particular geometry.  

For genus $2 \leq g < 6$, we are able to give a closed form expression for the potential $F_{g}$ using derivatives of the Weierstrass $\wp$-function.  The result is a polynomial in the ratio $\chi_{12}/ \chi_{10}$ as well as $\mathcal{E}_{4}$ and $\mathcal{E}_{6}$, where these Siegel modular forms were defined in Section \ref{subsec:RingofSMF}.  The first few potentials are:
\begin{equation}
\setlength{\jot}{10pt}
\begin{split}
& F_{2}(\Omega) = \frac{1}{240} \frac{\chi_{12}}{\chi_{10}} \\
& F_{3}(\Omega) = \frac{1}{60480}\bigg(5 \bigg(\frac{\chi_{12}}{\chi_{10}}\bigg)^{2}-6 \, \mathcal{E}_{4}\bigg) \\
& F_{4}(\Omega) = \frac{1}{7257600}\bigg( 35 \bigg( \frac{\chi_{12}}{\chi_{10}}\bigg)^{3} -63 \frac{\chi_{12}}{\chi_{10}} \mathcal{E}_{4} + 30 \mathcal{E}_{6}\bigg) \\
& F_{5}(\Omega) = \frac{1}{319334400} \bigg( 175 \bigg(\frac{\chi_{12}}{\chi_{10}}\bigg)^{4} - 420 \bigg(\frac{\chi_{12}}{\chi_{10}}\bigg)^{2} \mathcal{E}_{4} + 200 \frac{\chi_{12}}{\chi_{10}} \mathcal{E}_{6} + 42 \mathcal{E}_{8}\bigg). \\
\end{split}
\end{equation}

The diagram below nicely summarizes our results in this chapter: the bottom corner shows the equivariant elliptic genus of $\mathbb{C}^{2}$ which is related to the Gopakumar-Vafa invariants, as we will show.  The top left corner shows the Donaldson-Thomas partition function, and the top right is essentially the Gromov-Witten free energy.  We regard the banana manifold as an interesting example whereby a weight zero automorphic object encoding the Gopakumar-Vafa invariants (the elliptic genus) has standard arithmetic lifts producing the Donaldson-Thomas and Gromov-Witten theories in fiber classes.  

\begin{equation}
\begin{tikzcd}
Z_{\text{DT}}(X_{\text{ban}})_{\Gamma} \arrow[dashed, swap]{rrr}{\text{(Asymptotic) GW/DT Corresopndence}}     &     &      &      \sum_{g=0}^{\infty} \lambda^{2g-2} \text{ML}(12 \psi_{2g-2}) \\
                                                            &      &      &        \\
                                                            &      &      &        \\
                                                            &      &      &        \\
                                                            &      &  12 \Phi_{0}(\tau, z, x) = \sum_{g=0}^{\infty} \lambda^{2g-2} 12\psi_{2g-2}(\tau, z) \arrow{uuuull}{\text{Formal Borcherds Lift of}\, 12\Phi_{0}} \arrow[swap]{uuuur}{\text{Maass Lift of the}\, 12 \psi_{2g-2}}    &
\end{tikzcd}
\end{equation}

We will focus only on fiber classes, but O. Leigh has made progress on understanding the structure of invariants incorporating section classes of the banana manifold \cite{leigh_enumerative_2019,leigh_donaldson_2019}.

\section{Return to the Equivariant Elliptic Genus of \boldmath{$\mathbb{C}^{2}$}}

As mentioned in our introductory discussion above, the central object in the analysis of the partition functions of the banana manifold $X_{\text{ban}}$ will be the equivariant elliptic genus of $\mathbb{C}^{2}$.  This was computed in (\ref{eqn:equivELLC2diagspecc}) by equivariant localization.  Changing notations as indicated above, we recall the result here, along with a definition of the Fourier coefficients
\begin{equation}\label{eqn:copiedEQELLGENC2}
\setlength{\jot}{10pt}
\begin{split}
\Phi_{0}(\tau, z, x) & \coloneqq \text{Ell}_{q, y}(\mathbb{C}^{2}; t) =  \sum_{\substack{n \geq 0 \\ l,k \in \mathbb{Z}}} c(n, l, k) q^{n}y^{l}t^{k} \\
&= y^{-1} \prod_{n=1}^{\infty} \frac{(1-yq^{n-1}t)(1-y^{-1}q^{n}t^{-1})(1-yq^{n-1}t^{-1})(1-y^{-1}q^{n}t)}{(1-q^{n-1}t)(1-q^{n}t^{-1})(1-q^{n-1}t^{-1})(1-q^{n}t)}.
\end{split}
\end{equation}
Here $t$ is the equivariant parameter, and we will employ the change of variables $t = e^{2 \pi i x}$.  Notice that $\Phi_{0}$ has a pole at $t=1$ or equivalently, $x =0$.  This fact will prove to be meaningful in the enumerative geometry, and should be thought of as a manifestation of the non-compactness of $\mathbb{C}^{2}$.    

The equivariant elliptic genus $\Phi_{0}$ actually has automorphic properties.  In order to understand these, we must introduce the following notion of a weak Jacobi form of matrix index \cite{oberdieck_gromov-witten_2017}.  Let $L$ be a rational $r \times r$ symmetric matrix such that $2L$ is integral and has even diagonals.  Choose variables $\bm{w} = (w_{1}, \ldots, w_{r}) \in \mathbb{C}^{r}$, and define the symmetric bilinear form $\langle \bm{w}, \bm{w'} \rangle = \bm{w}^{t}L\bm{w'}$ on $\mathbb{C}^{r}$ with associated quadratic form $\mathcal{Q}(\bm{w}) = \langle \bm{w}, \bm{w} \rangle$.  Note that we allow $L$ to be non-degenerate, which will be important.

\begin{defn}
Given $L$ as above, a weak Jacobi form of weight $k \in \mathbb{Z}$ and matrix index $L$ is a meromorphic function $\Phi_{k} : \mathfrak{H} \times \mathbb{C}^{r} \to \mathbb{C}$ satisfying the following modularity and ellipticity conditions
\begin{equation}\label{eqn:MODtransWJFrm}
\Phi_{k}\bigg( \frac{a \tau + b}{c \tau + d}, \frac{\bm{w}}{c \tau +d}\bigg) = (c \tau + d)^{k} \text{exp}\bigg(\frac{2 \pi i c \, \mathcal{Q}(\bm{w})}{c \tau + d}\bigg) \Phi_{k}(\tau, \bm{w}) 
\end{equation}

\begin{equation}\label{eqn:ELLLtransWJFrm}
\Phi_{k}\big(\tau, \bm{w} + \bm{\alpha} \tau + \bm{\mu}) = \text{exp}\bigg(-2 \pi i \tau \mathcal{Q}(\bm{\alpha}) - 4 \pi i \langle \bm{w}, \bm{\alpha} \rangle\bigg) \Phi_{k}(\tau, \bm{w})
\end{equation}
for all $\begin{psmallmatrix}a & b\\ c & d\end{psmallmatrix} \in SL_{2}(\mathbb{Z})$, and $\bm{\alpha}, \bm{\mu} \in \mathbb{Z}^{r}$.  We also require the existence of the Fourier expansion
\begin{equation}\label{eqn:Fourexpmatind}
\Phi_{k}(\tau, \bm{w}) = \sum_{n=0}^{\infty} \sum_{\bm{\gamma} \in \mathbb{Z}^{r}} c(n, \bm{\gamma}) q^{n} \bm{\zeta^{\gamma}}
\end{equation}
where $q = e^{2 \pi i \tau}$ and we define $\bm{\zeta^{\gamma}} = \zeta_{1}^{\gamma_{1}} \cdots \zeta_{r}^{\gamma_{r}}$, where $\zeta_{i} = e^{2 \pi i w_{i}}$.  
\end{defn}

An ordinary weak Jacobi form (see Definition \ref{defn:Jacobifrmmmm}) of weight $k$ and index $m \in \mathbb{Z}_{\geq 0}$ is a weak Jacobi form of matrix index where $L = (m)$ is a $1 \times 1$ matrix and $\mathcal{Q}(z) = mz^{2}$ for all $z \in \mathbb{C}$.  When we refer to a weak Jacobi form without mentioning a matrix index, we will mean an ordinary weak Jacobi form.  

Recall from (\ref{eqn:ClJacFormTheta1}) that $\theta_{1}$ is a Jacobi form of weight $\frac{1}{2}$ and index $\frac{1}{2}$.  A straightforward calculation comparing (\ref{eqn:ClJacFormTheta1}) and (\ref{eqn:copiedEQELLGENC2}) verifies that the equivariant elliptic genus $\Phi_{0}$ can be written in the following way
\begin{equation}\label{eqn:ThetaEllGenct2o}
\Phi_{0}(\tau, z, x) = - \frac{\theta_{1}(\tau, z+x) \theta_{1}(\tau, z-x)}{\theta_{1}(\tau, x)^{2}}.
\end{equation}
As a consistency check, because $\theta_{1}(\tau, -x) = - \theta_{1}(\tau, x)$ we indeed have $\Phi_{0}(\tau, 0, x) =1$, which is the topological Euler characteristic of $\mathbb{C}^{2}$.  With this relationship, and using the known automorphic properties of $\theta_{1}$ provided in (\ref{eqn:jactheetaELL}) and (\ref{eqn:JaccthetaMODDDD}), we can now prove the following proposition.

\begin{proppy}\label{proppy:Phi0isWJFmatind}
The equivariant elliptic genus $\Phi_{0}(\tau, z, x)$ of $\mathbb{C}^{2}$ is a weak Jacobi form of weight 0 and matrix index
\begin{equation}\label{eqn:deggmatrPhi0zrp}
L = \begin{pmatrix}1 & 0\\ 0 & 0 \end{pmatrix}.
\end{equation}
\end{proppy}


Before proving this assertion, let us comment on why this is a satisfying way to view $\Phi_{0}$.  Recall from Theorem \ref{thm:EllGenCYMAN} that the elliptic genus of a compact Calabi-Yau surface is a weak Jacobi form of weight 0 and index 1.  Indeed, $\text{Ell}_{q,y}(K3)$ is the unique (up to scale) weak Jacobi form of this type.  The content of the above proposition is that the equivariant elliptic genus of the \emph{non-compact} Calabi-Yau surface $\mathbb{C}^{2}$ is a weight zero weak Jacobi form (consistent with the compact case) and it is nearly of index 1 by (\ref{eqn:deggmatrPhi0zrp}), though it has matrix index to account for the additional equivariant parameter.  The equivariant elliptic genus was called a ``generalized weak Jacobi form" in \cite{zhou_equivariant_2015}, though it was not regarded as having matrix index.   

Another reason which is not so transparent at the moment, is that there are a handful of models sharply analogous to our proposal later in this chapter.  The object playing the role of $\Phi_{0}$ in these models is a weight 0 automorphic object.

\begin{proof}
We will prove the proposition by directly verifying the modular (\ref{eqn:MODtransWJFrm}) and elliptic (\ref{eqn:ELLLtransWJFrm}) transformation laws of weak Jacobi forms using the modular (\ref{eqn:JaccthetaMODDDD}) and elliptic (\ref{eqn:jactheetaELL}) transformation laws of $\theta_{1}$.  We begin by noting that from (\ref{eqn:JaccthetaMODDDD}) it is clear that $\Phi_{0}(\tau + 1, z, x) = \Phi_{0}(\tau, z, x)$, which is the expected transformation of a weak Jacobi form.  We verify the transformation behavior under the other generator of $SL_{2}(\mathbb{Z})$ with the following computation
\begin{equation}
\setlength{\jot}{12pt}
\begin{split}
\Phi_{0} \big( - \frac{1}{\tau}, & \frac{z}{\tau}, \frac{x}{\tau} \big) = - \frac{\theta_{1}\big(-\frac{1}{\tau}, \frac{z+x}{\tau}\big) \theta_{1}\big(-\frac{1}{\tau}, \frac{z-x}{\tau} \big)}{\theta_{1}\big(-\frac{1}{\tau}, \frac{x}{\tau}\big)^{2}}\\
& = \text{exp} \bigg(i \pi \frac{(z+x)^{2}}{\tau} + i \pi \frac{(z-x)^{2}}{\tau}-2 \pi i \frac{x^{2}}{\tau}\bigg) \Phi_{0}(\tau, z, x) = \text{exp}\bigg( \frac{2 \pi i z^{2}}{\tau}\bigg) \Phi_{0}(\tau, z, x). 
\end{split}
\end{equation}
By (\ref{eqn:MODtransWJFrm}) this is indeed how a weight zero weak Jacobi form should transform where the quadratic form associated to the matrix (\ref{eqn:deggmatrPhi0zrp}) is $\mathcal{Q}(z, w) = z^{2}$.  We now verify the elliptic transformation law as follows
\begin{equation}
\setlength{\jot}{12pt}
\begin{split}
\Phi_{0}\big(\tau, z + & \alpha_{1} \tau + \mu_{1}, x + \alpha_{2}\tau + \mu_{2}\big) \\
& = - \frac{ \theta_{1}\big(\tau, z+x +\tau(\alpha_{1} + \alpha_{2}) + \mu_{1} + \mu_{2}\big) \theta_{1}\big(\tau, z-x +\tau(\alpha_{1} - \alpha_{2}) + \mu_{1} - \mu_{2}\big)}{\theta_{1}\big(\tau, x + \tau \alpha_{2} + \mu_{2}\big)^{2}}\\ 
&= \text{exp}\big(-2 \pi i\big( \tau \alpha_{1}^{2}  + 2 \alpha_{1} z\big)\big) \Phi_{0}(\tau, z, x).
\end{split}
\end{equation}
By (\ref{eqn:ELLLtransWJFrm}) this matches the expected transformation given the quadratic form associated to the matrix (\ref{eqn:deggmatrPhi0zrp}).

\end{proof}

We have already mentioned the change of variables $t = e^{2 \pi i x}$.  To avoid dealing with factors of $2 \pi$, let us define $\lambda = 2 \pi x$.  The Laurent expansion of $\Phi_{0}(\tau, z, x)$ in the variable $\lambda$ will play a crucial role on the Gromov-Witten side of our proposal.  By computations of Zhao \cite{zhou_regularized_2015}, this expansion takes the form
\begin{equation} \label{eqn:genexp}
\Phi_{0}(\tau, z, x) = \sum_{g = 0}^{\infty} \lambda^{2g-2} \psi_{2g-2}(\tau, z)
\end{equation}
where the $\psi_{2g-2}$ are weak Jacobi forms of weight $2g-2$ and index 1 for all $g \geq 0$.  Modifying results of \cite{zhou_regularized_2015} to match our conventions, they are given explicitly as
\begin{equation} \label{eqn:defnofphys}
\psi_{2g-2}(\tau, z) = \Theta^{2}(\tau, z) \cdot 
\begin{cases}
\begin{aligned}
& 1, & g=0 \\[1ex]
& \wp(\tau, z), & g=1 \\[1ex]
& \frac{|B_{2g}|}{2g (2g-2)!} \, E_{2g}(\tau), & g \geq 2
\end{aligned}
\end{cases}
\end{equation}
where $B_{2g}$ are the Bernoulli numbers, and where we recall that $\Theta^{2}$ presented in (\ref{eqn:varThetaa}) is the unique (up to scale) weak Jacobi form of weight -2 and index 1.  In addition, by (\ref{eqn:WPFuncK3}), $\Theta^{2} \wp$ is an equivalent way of writing the unique weak Jacobi form of weight 0 and index 1 up to scale.  It is apparent from (\ref{eqn:defnofphys}) that $\psi_{0}$ does not fit the pattern held by all other $\psi_{2g-2}$, a fact which arises in a sense due to the lack of modularity of $E_{2}$.  For all $g \neq 1$, we have
\[ \psi_{2g-2}(\tau, z) \in \Theta^{2}(\tau, z) \, M_{2g} (\Gamma_{1}), \,\,\,\,\,\,\,\,\,\,\,\,\,\,\, (g \neq 1).\]

Let us define the Fourier coefficients of $\psi_{2g-2}$ by
\begin{equation}
\psi_{2g-2}(\tau, z) = \sum_{n \geq 0, l \in \mathbb{Z}} c_{2g-2}(4n- l^{2}) q^{n} y^{l}.
\end{equation}
Because the $\psi_{2g-2}$ are weak Jacobi forms of index 1, by Theorem \ref{thm:wJacfrmm} the Fourier coefficients depend only on $4n-l^{2}$, as our notation suggests.  Comparing (\ref{eqn:copiedEQELLGENC2}) and (\ref{eqn:genexp}), we clearly have the following relationship between the Fourier coefficients $c(n, l, k)$ of $\Phi_{0}$ and the $c_{2g-2}(4n-l^{2})$ for all $n, l$
\begin{equation} \label{eqn:fourierref}
\sum_{g \geq 0} \lambda^{2g-2} c_{2g-2}(4n-l^{2}) = \sum_{k \in \mathbb{Z}} c(n, l, k) t^{k}.  
\end{equation}

\begin{proppy}\label{proppy:FcoeffEllgens23}
For all $n, l, k$, the coefficients $c(n, l, k)$ depend only on $4n - l^{2}$ and $k$.  Equivalently, if $4n - l^{2} = 4n' - l'^{2}$, then $c(n, l, k) = c(n', l', k)$ for all $k \in \mathbb{Z}$.  In addition, $c(n, l, k)$ vanishes unless $4n - l^{2} \geq -1$.    
\end{proppy}

\begin{proof}
For the first claim, because the lefthand side of (\ref{eqn:fourierref}) explicitly depends only on $4n-l^{2}$ we must have
\[ \sum_{k \in \mathbb{Z}} c(n, l, k) t^{k} = \sum_{k \in \mathbb{Z}} c(n', l', k) t^{k}.\]
But this implies that coefficients of corresponding powers of $t$ agree, which proves the first statement.  To show the second claim, we again use (\ref{eqn:fourierref}) along with the fact that the coefficients $c_{2g-2}(4n-l^{2})$ vanish unless $4n-l^{2} \geq -1$ since $\psi_{2g-2}$ are index 1 weak Jacobi forms of even weight.    
\end{proof}

We therefore change notation, and denote the Fourier coefficients of $\Phi_{0}$ as $c(4n-l^{2}, k)$.  We rewrite equation (\ref{eqn:fourierref}) in the new notation 
\begin{equation} \label{eqn:fourierrefTWOOO}
\sum_{g \geq 0} \lambda^{2g-2} c_{2g-2}(4n-l^{2}) = \sum_{k \in \mathbb{Z}} c(4n-l^{2}, k) t^{k}
\end{equation}
as this relation will be crucial when we pass from the Donaldson-Thomas to the Gromov-Witten theory of the banana manifold.  We will make use of the following closed form expressions for $c(4n-l^{2}, k)$ for a few of the smallest values of $4n-l^{2}$.  These can be computed directly from (\ref{eqn:copiedEQELLGENC2})
\begin{equation}\label{eqn:lowvalkcoeff}
c(-1,k)=
\begin{cases}
\begin{aligned}
& 0 , & k \leq 0 \\[1ex]
& -k , & k>0 \\[1ex]
\end{aligned}
\end{cases}    
\,\,\,\,\,\,\,\,\,\,\,\,\,\,\,\,\,\,\,\,\,\, c(0, k) =
\begin{cases}
\begin{aligned}
& 0 & k<0  \\[1ex]
& 1, & k=0 \\[1ex]
& 2k, & k >0
\end{aligned}
\end{cases}
\end{equation}
In addition, we will also make use of the following explicit values for the constant coefficients of the Jacobi forms $\psi_{2g-2}$ which can be computed directly from (\ref{eqn:defnofphys})
\begin{equation}\label{eqn:psissconstcoef}
c_{-2}(0) = -2, \,\,\,\,\,\,\,\,\,\,\,\,\,\,\,\,\,\,  c_{2g-2}(0) = -2 c_{2g-2}(-1) = - \frac{|B_{2g}|}{g (2g-2)!} \,\,\,\,\,\,\,\, (g \geq 2).
\end{equation}

\subsection{The Formal Borcherds Lift of the Equivariant Elliptic Genus}

In Section \ref{sec:HeckeJacForm} we gave an introduction to the theory of Hecke operators acting on ordinary weak Jacobi forms.  An important part of our work will be applying the results of that section to Hecke operators acting on the weak Jacobi forms $\psi_{2g-2}$.  This will come later but for now, we want to study their action on the equivariant elliptic genus $\Phi_{0}$.  Using Proposition \ref{proppy:Phi0isWJFmatind} where we established that $\Phi_{0}$ is a weak Jacobi form of weight 0 and matrix index, we need to slightly adjust the definitions and results on ordinary Jacobi forms in an obvious way.  

\begin{defn}
The action of the Hecke operator $V_{m}$ on $\Phi_{0}$ for $m>0$ is defined by
\begin{equation}
\big(\Phi_{0} \big| V_{m} \big) = \frac{1}{m} \sum_{\substack{ad = m \\ a > 0}} \sum_{b=0}^{d-1} \Phi_{0}\bigg( \frac{a \tau + b}{d}, az, ax\bigg).  
\end{equation}
\end{defn}

\noindent Notice this definition is the natural extension of (\ref{eqn:Heckedefn}) to matrix index Jacobi forms of weight 0.  For $m>0$, it is straightforward to show that the action of $V_{m}$ on a weak Jacobi form of weight $k$ and matrix index $L$ produces a weak Jacobi form of the same weight $k$ and matrix index $mL$.  

The following result is the natural generalization of Lemma \ref{lemmy:genfunccheckoops} to the case of $\Phi_{0}$.  We want give an expression for the generating function of the Hecke operators in terms of the Fourier coefficients and the polylogarithm defined in (\ref{eqn:polylogrthm}).  We omit the proof since it is completely formal and identical to that of Lemma \ref{lemmy:genfunccheckoops}.  

\begin{lemmy}
Let $c(4n-l^{2}, k)$ be the Fourier coefficients of the equivariant elliptic genus $\Phi_{0}$.  We have
\begin{equation}\label{eqn:genfunncHeccke}
\sum_{m=1}^{\infty} Q^{m} \big( \Phi_{0} \big| V_{m} \big) = \sum_{\substack{ m>0, n \geq 0, \\ l,k \in \mathbb{Z}}} c(4nm-l^{2}, k) \text{Li}_{1}(Q^{m} q^{n}y^{l}t^{k}).
\end{equation}
\end{lemmy}

We must also define the Hecke operator $V_{0}$ and its action on $\Phi_{0}$.  Recall that this is more subtle (\ref{T0WEAKHOL}), and in the case of weight 0 one must be careful to not allow ill-defined terms proportional to $\zeta(1)$.  We define
\begin{equation}\label{eqn:V0actionPhi0}
\big( \Phi_{0} \big| V_{0} \big) \coloneqq \sum_{(n, l,k)>0} c(-l^{2}, k) \text{Li}_{1}(q^{n}y^{l}t^{k})
\end{equation}
where $(n, l, k)>0$ means either $n>0$, or $n=0, l>0$, or $n=l=0, k>0$.  This is completely analogous to (\ref{T0WEAKHOL}) in the case of weight 0.  

\begin{defn}
We define the formal Maass lift of $\Phi_{0}$ to be
\begin{equation}
\text{ML}(\Phi_{0}) \coloneqq \sum_{m=0}^{\infty} Q^{m} \big( \Phi_{0} \big| V_{m} \big)
\end{equation}
as well as the formal Borcherds lift of $\Phi_{0}$ to be the exponential of the Maass lift
\begin{equation}
\text{BL}(\Phi_{0}) = \text{exp}\big( \text{ML}(\Phi_{0}) \big).
\end{equation}
We call these lifts ``formal" as we cannot make conclusive claims about their automorphy.  
\end{defn}

Combining (\ref{eqn:genfunncHeccke}) and (\ref{eqn:V0actionPhi0}), the formal Maass lift of $\Phi_{0}$ is easily seen to take the following form
\begin{equation}\label{eqn:fullformmMaassriftt}
\text{ML}(\Phi_{0}) = \sum_{(m,n,l,k)>0} c(4nm-l^{2}, k) \text{Li}_{1}(Q^{m} q^{n} y^{l} t^{k})
\end{equation}
where we define the notation $(m,n,l,k)>0$ to mean that any of the following hold
\begin{equation}\label{eqn:possscondds2}
(i) \,\, m>0  \,\,\,\,\,\,\, (ii) \,\, m=0, \, n>0  \,\,\,\,\,\,\, (iii) \,\, m=n=0, \, l>0 \,\,\,\,\,\,\,  (iv) \,\, m=n=l=0, \, k>0.
\end{equation}
An immediate observation one should make from (\ref{eqn:fullformmMaassriftt}) is that because $\text{Li}_{1}(x) = -\log(1-x)$, the exponential of the formal Maass lift of $\Phi_{0}$ can be written as an infinite product.  Because the formal Maass lift of an object of weight k contains $\text{Li}_{1-k}$, this phenomenon is special to the case of weight 0.  The upshot is that the formal Borcherds lift of $\Phi_{0}$ can be written as
\begin{equation}\label{eqn:formBORRRriftt}
\text{BL}(\Phi_{0}) = \prod_{(m,n,l,k)>0} \big(1-Q^{m}q^{n} y^{l} t^{k}\big)^{-c(4nm-l^{2}, k)}.
\end{equation}
It is this expression (raised to the power of 12) which we will soon prove to exactly coincide with the Donaldson-Thomas partition function of the banana manifold after a simple change of variables.

\subsection{Relation to Second Quantization and Analogy to $\boldmath{\chi_{10}}$}

We remark that the formal Borcherds lift (\ref{eqn:formBORRRriftt}) looks like an equivariant analog of the inverse of the Igusa cusp form $1/\chi_{10}$.  The infinite product form of Gritsenko-Nikulin (\ref{eqn:GNIKULinfprodd}) gives
\begin{equation}\label{eqn:GNIKULinfprodd222}
\frac{Qqy}{\chi_{10}(\Omega)} = \prod_{(m,n,l)>0} \big(1-Q^{m}q^{n}y^{l}\big)^{-c(4nm-l^{2})}
\end{equation}
where here, $c(4nm-l^{2})$ are the Fourier coefficients of $\text{Ell}_{q,y}(K3)$.  This looks nearly identical to (\ref{eqn:formBORRRriftt}) replacing the elliptic genus of $K3$ with $\Phi_{0}(\tau, z, x) = \text{Ell}_{q,y}(\mathbb{C}^{2}; t)$ and accounting for the equivariant parameter.  The analogy is even tighter noting that both $K3$ and $\mathbb{C}^{2}$ are Calabi-Yau surfaces, and both of their elliptic genera are weight 0 weak Jacobi forms, of matrix index in the case of $\mathbb{C}^{2}$.

Recall that starting with the Gritsenko-Nikulin form (\ref{eqn:GNIKULinfprodd222}), the inverse of the Igusa cusp form can be written as a prefactor multiplied by the generating function of the elliptic genera of the Hilbert schemes of points on a $K3$ surface.  This was recorded in (\ref{eqn:chi10prodDMVVform}) which we recall here for convenience
\begin{equation}
\frac{1}{\chi_{10}(\Omega)} = \frac{1}{Q \varphi_{10,1}(\tau, z)} \sum_{m=0}^{\infty} Q^{m} \text{Ell}_{q,y}\big( \text{Hilb}^{m}(K3) \big).
\end{equation}
The infinite sum above is known as the \emph{second quantization of the elliptic genus} $\text{Ell}_{q,y}(K3)$.  By the following proposition, $\text{BL}(\Phi_{0})$ has the same property.    

\begin{proppy}
The formal Borcherds lift of the equivariant elliptic genus of $\mathbb{C}^{2}$ can be written
\begin{equation}\label{eqn:rerrrnBLandSQEEG}
\begin{split}
& \,\,\,\,\,\,\,\,\,\,\,\,\,\,  \text{BL}(\Phi_{0}) = \mathcal{Z}_{0}(q,y,t) \sum_{n=0}^{\infty} Q^{n} \text{Ell}_{q, y} \big(\text{Hilb}^{n}(\mathbb{C}^{2}) ; t \big) \\
& \mathcal{Z}_{0}(q,y,t) = \frac{M(1,t)^{2}}{M(y,t)} \prod_{n=1}^{\infty} \frac{M(q^{n},t)^{2}}{(1-q^{n}) M(q^{n}y, t)M(q^{n}y^{-1},t)}
\end{split}
\end{equation}
where $M(x,t) = \prod_{n=1}^{\infty} (1-x t^{n})^{-n}$ is the weighted MacMahon function.  The infinite sum in the above expression is called the second quantization of the equivariant elliptic genus $\Phi_{0}$.    
\end{proppy}

\begin{proof}
By the result of Waelder (\ref{eqn:resultofWaelderr}) we know that the factors of the infinite product (\ref{eqn:formBORRRriftt}) with $m>0$ produce the generating function of the equivariant elliptic genera of $\text{Hilb}^{m}(\mathbb{C}^{2})$.  We therefore must verify that all the factors with $m=0$ takes the desired form.  Letting $\mathcal{Z}_{0}(q, y, t)$ denote the product of these terms, using (\ref{eqn:possscondds2}) we see
\begin{equation}
\mathcal{Z}_{0}(q,y,t) = \prod_{\substack{n \geq 1, \\ l,k \in \mathbb{Z}}} \big(1-q^{n}y^{l}t^{k}\big)^{-c(-l^{2}, k)} \prod_{k \in \mathbb{Z}}\big(1-yt^{k}\big)^{-c(-1, k)} \prod_{k \geq 1}\big(1-t^{k}\big)^{-c(0,k)}.
\end{equation}
Using the fact that $c(4nm-l^{2}, k)=0$ for $4nm-l^{2} < -1$, as well as some values of the coefficients (\ref{eqn:lowvalkcoeff}), it is straightforward to see that $\mathcal{Z}_{0}(q, y, t)$ takes the desired form.    
\end{proof}

Though the formal Borcherds lift is evidently very closely related to the second quantization of $\Phi_{0}$, the prefactors combining to give $\mathcal{Z}_{0}(q, y, t)$ are indeed important in the Donaldson-Thomas partition function of the banana manifold.  We therefore choose to emphasize the Borcherds lift in what follows.  Note that the second quantization of $\Phi_{0}$ is identified with the partition function of rank one (framed) instantons on $\mathbb{C}^{2}$.

\section{Donaldson-Thomas Partition Function of the Formal Banana Manifold}

As we defined in (\ref{eqn:origdefnBAN}), let $X_{\text{ban}}$ be the banana manifold, and let $C_{1}, C_{2}, C_{3}$ be the banana curves.  We will denote by $\underline{\bf{d}} = (d_{1}, d_{2}, d_{3})$ the classes $d_{1}C_{1} + d_{2}C_{2} + d_{3}C_{3}$ in the lattice of fiber classes $\Gamma \subset H_{2}(X_{\text{ban}}, \mathbb{Z})$, and we introduce formal variables $Q_{1}, Q_{2}, Q_{3}$ tracking degrees along the respective curves.  The following is a theorem of J. Bryan \cite{bryan_donaldson-thomas_2018}.  

\begin{thm} \label{thm:JimTHM}
The Donaldson-Thomas partition function of the banana manifold $X_{\text{ban}}$ restricted to fiber classes is given by the infinite product
\begin{equation} \label{eqn:DTBanpartfunc}
Z_{\text{DT}}(X_{\text{ban}})_{\Gamma}= \prod_{d_{1}, d_{2}, d_{3} \geq 0} \prod_{k \in \mathbb{Z}} \big(1-Q_{1}^{d_{1}} Q_{2}^{d_{2}} Q_{3}^{d_{3}}t^{k} \big)^{-12 c(||\underline{\bf{d}}||, k)}
\end{equation}
where we require $k \geq 1$ if $\underline{\bf{d}} = (0,0,0)$.  In this formula, $||\underline{\bf{d}}||$ is the quadratic form
\begin{equation}\label{eqn:quaddformdsss}
||\underline{\bf{d}}|| \coloneqq 2d_{1}d_{2} + 2d_{1}d_{3} + 2d_{2}d_{3} - d_{1}^{2}-d_{2}^{2}-d_{3}^{2}
\end{equation}
and $c(||\underline{\bf{d}}||, k)$ are the coefficients of the equivariant elliptic genus of $\mathbb{C}^{2}$, which we denote $\Phi_{0}$.       
\end{thm}

We want to identify this Donaldson-Thomas partition function with the Borcherds lift of $12 \Phi_{0}$ presented in (\ref{eqn:formBORRRriftt}).  This will of course require changing variables from the geometric variables $Q_{1}, Q_{2}, Q_{3}$ to $Q, q, y$.  It turns out the correct change of variables is 
\begin{equation}\label{eqn:changeofvarss}
Q = Q_{1}Q_{3}, \,\,\,\,\,\,\,\,\,\,\,\,\,\,\,\, q = Q_{2} Q_{3}, \,\,\,\,\,\,\,\,\,\,\,\,\,\,\,\, y=Q_{3}.
\end{equation}

\begin{proppy}\label{proppy:DTBryanchvarss}
Under the change of variables (\ref{eqn:changeofvarss}), the Donaldson-Thomas partition function (\ref{eqn:DTBanpartfunc}) can be written as the formal Borcherds lift of $12 \Phi_{0}$
\begin{equation}\label{eqn:DTpartfuncBorLiftTWOO}
Z_{\text{DT}}(X_{\text{ban}})_{\Gamma} = \text{BL}(12 \Phi_{0}) = \prod_{(m,n,l,k)>0} \big(1-Q^{m}q^{n}y^{l}t^{k}\big)^{-12 c(4nm-l^{2}, k)}
\end{equation}
where we recall that the notation $(m,n,l,k)>0$ was defined in (\ref{eqn:possscondds2}).  
\end{proppy}

\begin{proof}
A simple computation verifies that under the change of variables,
\begin{equation}
||\underline{\bf{d}}|| = 4nm - l^{2} \geq -1.  
\end{equation}
The second equality in (\ref{eqn:DTpartfuncBorLiftTWOO}) is simply the expression (\ref{eqn:formBORRRriftt}) of the formal Borcherds lift of $12 \Phi_{0}$.  Therefore, we have only to show that the infinite products are over the same ranges.  Performing the change of variables we clearly have 
\begin{equation}
Q^{m}q^{n} y^{l}t^{k} = Q_{1}^{m} Q_{2}^{n} Q_{3}^{l+m+n} t^{k}
\end{equation}
which motivates us to define $d_{1} = m$, $d_{2}=n$, and $d_{3}=l+m+n$.  The parameter $t$ tracking holomorphic Euler characteristic in Donaldson-Thomas theory is precisely the equivariant parameter of the elliptic genus.  Because $n,m \geq 0$ it follows that $d_{1}, d_{2} \geq 0$.  

We need to show that $d_{3} \geq 0$.  As we have defined it just above, this clearly holds if $l \geq 0$, but $l$ can be negative.  It turns out that the coefficients of $\Phi_{0}$ conspire to enforce the non-negativity of $d_{3}$.  To get a contradiction, let us assume $d_{3}<0$.  This means there exists an integer $\epsilon \geq 1$ such that 
\begin{equation}
d_{3} + \epsilon = l + m + n + \epsilon = 0.
\end{equation}
Solving for $l$, and using Proposition \ref{proppy:FcoeffEllgens23} we know that we must have
\begin{equation}
\begin{split}
-1 \leq 4nm - l^{2} & = 4nm - (m+n+ \epsilon)^{2} \\
& = -(n-m)^{2} - \epsilon^{2} - 2 \epsilon(n+m).
\end{split}
\end{equation}
Because $\epsilon \geq 1$, the only way for this inequality to be satisfied is if $\epsilon =1$ and $m=n=0$.  But by (\ref{eqn:possscondds2}), this means $l >0$ which contradicts our assumption.  We therefore have $d_{3} \geq 0$.  Finally, if all $d_{i}$ vanish, then $m=n=l=0$ which means $k>0$, also by (\ref{eqn:possscondds2}).  
\end{proof}

With this result in hand, we get the \emph{reduced} Donaldson-Thomas partition function by simply omitting option $(iv)$ in the definition (\ref{eqn:possscondds2}).  Indeed, using (\ref{eqn:lowvalkcoeff}) the degree zero terms combine to produce $M(t)^{24}$ which is consistent with (\ref{eqn:DTdegggzeroo}) noting that $\chi(X_{\text{ban}})=24$.  We therefore get
\begin{equation}\label{eqn:redDTparttsfung}
Z'_{\text{DT}}(X_{\text{ban}})_{\Gamma} = \prod_{(m,n,l)>0} \prod_{k \in \mathbb{Z}} \big(1- Q^{m} q^{n} y^{l} t^{k}\big)^{-12 c(4nm-l^{2}, k)}.  
\end{equation}

\section{The Automorphy of the Gromov-Witten Potentials}

Using the results detailed above, and assuming the GW/DT correspondence holds for the banana manifold $X_{\text{ban}}$, the goal of this section is to prove that the Gromov-Witten potentials are Siegel modular forms and to compute them explicitly for $g < 6$.  We will do this by identifying the genus $g$ potential as the Maass lift of $12 \psi_{2g-2}$ appearing in the expansion (\ref{eqn:genexp}) of $\Phi_{0}$.  Recall that the $\psi_{2g-2}$ are weak Jacobi forms of weight $2g-2$ and index 1, and are shown explicitly in (\ref{eqn:defnofphys}).  We will rely heavily on the results of Section \ref{sec:HeckeJacForm} where we studied Maass lifts and Hecke operators on ordinary weak Jacobi forms.  

Throughout this section we will assume the GW/DT correspondence holds for the banana manifold, and we will use the results above on the Donaldson-Thomas theory.  There is one major subtlety to contend with -- the GW/DT correspondence only relates the two \emph{reduced} partition functions, whereas the degree zero contributions will play a non-trivial role in our proposal.  We therefore must exercise caution when passing from Donaldson-Thomas to Gromov-Witten theory.  

We recall from (\ref{eqn:Deg0gCY3}) the form of the degree zero contributions to the Gromov-Witten potentials of the banana manifold in genus $g \geq 2$
\begin{equation}\label{eqn:degzeroGWgrtwo}
F_{g}^{(0)} = \frac{1}{2} \frac{|B_{2g}|}{2g} \frac{B_{2g-2}}{2g-2} \frac{\chi(X)}{(2g-2)!} = \frac{12}{(2g-2)!}\frac{|B_{2g}|}{2g} \frac{B_{2g-2}}{2g-2}
\end{equation}
which we will need for the following proposition.  In the second equality we have used that $\chi(X_{\text{ban}})=24$.  Let us also record the results of applying (\ref{eqn:fullHecke}) to the Maass lift of $12\psi_{2g-2}$
\begin{equation}\label{eqn:MLgengnew}
\text{ML}(12 \psi_{2g-2}) = 
\begin{cases}
\begin{aligned}
& -12\zeta(3) + \sum_{(m,n,l)>0} 12c_{-2}(4nm-l^{2})\text{Li}_{3}(Q^{m}q^{n}y^{l}), & g=0 \\[1ex]
& \sum_{(m,n,l)>0} 12 c_{0}(4nm-l^{2})\text{Li}_{1}(Q^{m}q^{n}y^{l}), & g=1 \\[1ex]
& F_{g}^{(0)} + \sum_{(m,n,l)>0} 12 c_{2g-2}(4nm-l^{2})\text{Li}_{3-2g}(Q^{m}q^{n}y^{l}), & g \geq 2
\end{aligned}
\end{cases}
\end{equation}
where $c_{2g-2}(4nm-l^{2})$ are the Fourier coefficients.  The constant terms in the above formula were easily computed using (\ref{eqn:psissconstcoef}) and (\ref{eqn:epsilonkdef}).  

\begin{proppy}\label{proppy:genfunscML}
Applying the change of variables $t=e^{i \lambda}$, we have the following equality 
\begin{equation}\label{eqn:genfuncMassssriftts}
\sum_{g=0}^{\infty} \lambda^{2g-2} \text{ML}(12 \psi_{2g-2}) = -\frac{12 \zeta(3)}{\lambda^{2}} + \sum_{g=2}^{\infty} \lambda^{2g-2} F_{g}^{(0)} + \text{log} \, Z'_{\text{DT}}(X_{\text{ban}})_{\Gamma}
\end{equation}
in the ring $\lambda^{-2} \mathbb{R} (y)^{y \leftrightarrow y^{-1}} \llbracket \lambda^{2}, Q, q \rrbracket$.  The coefficients are all in $\mathbb{Q}$ except for the term with the irrational number $\zeta(3)$.  Here, $Z'_{\text{DT}}(X_{\text{ban}})_{\Gamma}$ is the reduced Donaldson-Thomas partition function given in (\ref{eqn:redDTparttsfung}).   
\end{proppy}

\begin{proof}
We begin by taking the logarithm of the reduced partition function in (\ref{eqn:redDTparttsfung})
\begin{equation}
\begin{split}
\log Z'_{\text{DT}}(X_{\text{ban}})_{\Gamma} & = \sum_{(m,n,l)>0} \sum_{k \in \mathbb{Z}} -12c(4nm-l^{2}, k) \log\big(1-Q^{m}q^{n}y^{l}t^{k}\big) \\
& = \sum_{(m,n,l)>0} \sum_{k \in \mathbb{Z}} 12 c(4nm-l^{2}, k) \text{Li}_{1}\big(Q^{m}q^{n}y^{l}t^{k}\big) \\
& = \sum_{(m,n,l)>0} \sum_{k \in \mathbb{Z}} 12 c(4nm-l^{2}, k) \sum_{h=1}^{\infty} \frac{1}{h} Q^{mh} q^{nh} y^{lh}t^{kh}.
\end{split}
\end{equation}
We can perform the sum over $k \in \mathbb{Z}$ using (\ref{eqn:fourierrefTWOOO}), which results in 
\begin{equation}
\begin{split}
\log Z'_{\text{DT}}(X_{\text{ban}})_{\Gamma} & = \sum_{(m,n,l)>0} \sum_{h=1}^{\infty} \frac{1}{h} Q^{mh} q^{nh} y^{lh} \sum_{g=0}^{\infty} (h \lambda)^{2g-2} 12 c_{2g-2}(4nm-l^{2}) \\
& = \sum_{g=0}^{\infty} \lambda^{2g-2} \sum_{(m,n,l)>0} 12 c_{2g-2}(4nm-l^{2}) \text{Li}_{3-2g}(Q^{m}q^{n}y^{l}).
\end{split}
\end{equation}
By (\ref{eqn:MLgengnew}) we know the form of $\text{ML}(12 \psi_{2g-2})$ from which we can compute
\begin{equation}
\begin{split}
\sum_{g=0}^{\infty} \lambda^{2g-2} & \text{ML}(12 \psi_{2g-2}) = \\
& -\frac{12 \zeta(3)}{\lambda^{2}} + \sum_{g=2}^{\infty} \lambda^{2g-2} F_{g}^{(0)} + \sum_{g=0}^{\infty} \lambda^{2g-2} \sum_{(m,n,l)>0} 12 c_{2g-2}(4nm-l^{2}) \text{Li}_{3-2g}(Q^{m}q^{n}y^{l})\\
& = -\frac{12 \zeta(3)}{\lambda^{2}} + \sum_{g=2}^{\infty} \lambda^{2g-2} F_{g}^{(0)}  + \log Z'_{\text{DT}}(X_{\text{ban}})_{\Gamma} 
\end{split}
\end{equation}
which completes the proof.  
\end{proof}

There are a few important observations to make at this stage.  Notice that the righthand side of (\ref{eqn:genfuncMassssriftts}) is very close to the Gromov-Witten free energy $F_{\text{GW}}(X_{\text{ban}})_{\Gamma}$.  Indeed, the GW/DT correspondence says that
\begin{equation}
F'_{\text{GW}}(X_{\text{ban}})_{\Gamma} = \log Z'_{\text{DT}}(X_{\text{ban}})_{\Gamma}
\end{equation}
under the change of variables $t = e^{i \lambda}$.  For $g \geq 2$, the degree zero contributions are all accounted for in the righthand side of (\ref{eqn:genfuncMassssriftts}).  The degree zero contributions for $g=0$ and $g=1$ are not constants and are naturally not incorporated.  The GW/DT correspondence in the presence of degree zero contributions only gives an asymptotic equivalence.  Therefore, by Proposition \ref{proppy:genfunscML} we know the following.  

\begin{thm}\label{thm:asympGWDTcorrr}
Assuming the GW/DT correspondence holds for the formal banana manifold, we have the equivalences 
\begin{equation}\label{eqn:GWDTcorraympps}
\text{exp}\bigg(\sum_{g=0}^{\infty} \lambda^{2g-2} \text{ML}(12 \psi_{2g-2})\bigg) \sim Z_{\text{GW}}(X_{\text{ban}})_{\Gamma} \sim M(e^{i \lambda})^{12} Z'_{\text{DT}}(X_{\text{ban}})_{\Gamma}
\end{equation}
under the change of variables $t=e^{i \lambda}$.  The precise meaning of the equivalence $\sim$ is as follows: the logarithm of all three quantities has an expansion in $Q, q, y$ as well as a Laurent expansion in $\lambda$ with only a second order pole.  For terms non-constant in at least one of $Q, q, y$, $\sim$ is an equality of series in $\lambda$.  For terms constant in all of $Q, q, y$, $\sim$ is an equality of the $\lambda$-expansions in order $\lambda^{2}$ and above.  
\end{thm}

\begin{rmk}
Despite the fact that it is forgotten in the above asymptotic equivalences, the term proportional to $\zeta(3)$ in $\text{ML}(\psi_{-2})$ is well-known to arise in the asymptotic expansion (\ref{eqn:asympEXPMcMahon}) of the logarithm of the MacMahon function $M(e^{i \lambda})$.  Though it is mysterious from an enumerative point of view.  
\end{rmk}

\subsection{The Gromov-Witten Potentials as Siegel Modular Forms}

Using Theorem \ref{thm:asympGWDTcorrr}, we can relate the Maass lift of $12 \psi_{2g-2}$ with the genus $g$ Gromov-Witten potential of the banana manifold $X_{\text{ban}}$ in the variables $Q, q, y$.  This is already an interesting arithmetic property, despite not yet expressing the full automorphy.

\begin{cory}
Assuming the GW/DT correspondence for $X_{\text{ban}}$, we identify the Gromov-Witten potential $F_{g}$ for $g \geq 2$ as the Maass lift of $12 \psi_{2g-2}$
\begin{equation}\label{eqn:formmassliftggeq2}
F_{g}(Q, q, y) = \text{ML}(12 \psi_{2g-2})
\end{equation}
where $\text{ML}(12 \psi_{2g-2})$ is shown in (\ref{eqn:MLgengnew}).  For $g=1$, the Maass lift of $12 \psi_{0}$ is the reduced potential $F'_{1}$, but of course does not encode the degree zero terms.  Moreover, we have the formula 
\begin{equation}
F'_{1}(Q, q, y) = \text{ML}( 12 \psi_{0}) = \frac{1}{2} \log \bigg(\frac{Qqy}{\chi_{10}}\bigg).
\end{equation}
For $g=0$, the Maass lift $\text{ML}(12 \psi_{-2})$ encodes the full genus zero Gromov-Witten potential outside some degree zero terms -- it certainly encodes all non-degree zero terms, as well as the term proportional to $\zeta(3)$ which is well-known to arise in genus zero Gromov-Witten theory.  
\end{cory} 

\begin{proof}
The statements for $g=0$ and $g \geq 2$ follow directly from Theorem \ref{thm:asympGWDTcorrr} while for $g=1$, we need a short computation.  Also by Theorem \ref{thm:asympGWDTcorrr} we know
\begin{equation}
F'_{1}(Q, q, y) = \text{ML}(12 \psi_{0}) = \sum_{(m,n,l)>0} 12 c_{0}(4nm-l^{2}) \text{Li}_{1}(Q^{m} q^{n} y^{l})
\end{equation}
where the second equality follows from (\ref{eqn:MLgengnew}).  But we have
\begin{equation}
12 \psi_{0}(\tau, z) = 12 \Theta^{2}(\tau, z) \wp(\tau, z) = \frac{1}{2} \text{Ell}_{q,y}(K3).
\end{equation}
The claim then follows immediately from (\ref{eqn:GNIKULinfprodd222}).  
\end{proof}

We saw in Section \ref{subsec: MaassliftIndoneJFOR} that the Maass lift of a holomorphic Jacobi form of positive weight and index 1 is a holomorphic Siegel modular form of genus two.  Recall that $\mathfrak{M}_{*}(\Gamma_{2})$ is the graded ring of Siegel modular forms of genus two, and the \emph{Maass `Spezialschar'} is the image in $\mathfrak{M}_{*}(\Gamma_{2})$ of the Maass lift.  From (\ref{eqn:MLgengnew}) it is clear that we are interested in the Maass lift of \emph{weak} (not holomorphic) Jacobi forms of positive even weight and index 1.  By results of Hiroki Aoki \cite{aoki_formal_2014,aoki_notitle_2018} we still get Siegel modular forms of genus two, though they are meromorphic.  I am extremely grateful \cite{oberdieck_notitle_2018} to Georg Oberdieck for initially pointing out to me that the Maass lift of a weak Jacobi form should still be a Siegel modular form, as well as to Hiroki Aoki for graciously discussing the following result of his.    

\begin{thm}[\bfseries H. Aoki \cite{aoki_formal_2014,aoki_notitle_2018}]
Let $\varphi \in \mathbb{J}_{k,1}^{w}$ be a weak Jacobi form of even weight $k>0$ and index 1.  The Maass lift $\text{ML}(\varphi)$ is a meromorphic Siegel modular form of weight $k$ and genus two.  We get an isomorphism of vector spaces
\begin{equation}
\mathbb{J}_{k,1}^{w} \overset{\sim}{\longrightarrow} \big\{ \text{Maass `Spezialschar'} \big\} \subset \widetilde{\mathfrak{M}}_{k}(\Gamma_{2})
\end{equation}
where $\widetilde{\mathfrak{M}}_{k}(\Gamma_{2})$ is the vector space of meromorphic Siegel modular forms of weight $k$ and genus two, and the Maass `Spezialschar' is the image of the Maass lift.  
\end{thm}

Combining our results above with those of Aoki, we have the following remarkable conclusion on the automorphy of the Gromov-Witten potentials.  

\begin{cory}\label{cory:GWpotssSMFds}
For all $g \geq 2$ the Gromov-Witten potentials of the banana manifold $X_{\text{ban}}$
\begin{equation}
F_{g}(\Omega) = \text{ML}(12 \psi_{2g-2}) = \sum_{m=0}^{\infty} Q^{m} \big(12 \psi_{2g-2} \big| V_{m} \big) 
\end{equation}
are meromorphic Siegel modular forms of weight $2g-2$ and genus two, where $\Omega$ is an element of the Siegel upper-half plane related to $Q, q, y$ as in (\ref{eqn:innchvarsintrod}).  In other words, the genus $g$ Gromov-Witten potential is an element of the Maass `Spezialschar' inside of $\widetilde{\mathfrak{M}}_{2g-2}(\Gamma_{2})$.  
\end{cory}

\begin{rmk}\label{rmk:MIRRORSYYMM}
We can interpret Corollary \ref{cory:GWpotssSMFds} using mirror symmetry.  The Gromov-Witten potentials are functions of $Q_{1}, Q_{2}, Q_{3}$ (or $Q, q, y$) which are coordinates on the K\"{a}hler moduli space of $X_{\text{ban}}$.  Under the mirror correspondence $Q, q, y$ are interpreted as coordinates on the moduli space of complex structures on the mirror $\widetilde{X}_{\text{ban}}$.  But genus two Siegel modular forms are sections of line bundles over the moduli space of genus two curves.  As a consequence, and since we are only using fiberwise K\"{a}hler parameters, we conjecture that the moduli space of complex structures on $\widetilde{X}_{\text{ban}}$ contains a subspace isomorphic to the moduli space of genus two curves.  Indeed, it has been shown that the mirror of a local banana configuration is a genus two curve \cite{kanazawa_local_2016}.
\end{rmk}

We will spend much of the remainder of this section attempting to understand precise features of the $F_{g}$ as meromorphic Siegel modular forms.  The first step is to understand exactly what the denominator of the meromorphic function is.  To do this, we note that for $g \geq 2$ the weight $2g-2$ is positive so we can use (\ref{eqn:WPV0wholjacfrm}) to compute the action of the Hecke operator $V_{0}$ on $12 \psi_{2g-2}$.  In the case of $g=2$, by (\ref{eqn:weight2V0expr}) we have
\begin{equation}\label{eqn:psi2V0Maassrift}
\big( 12 \psi_{2}  \big| V_{0} \big) = 12 c_{2}(-1) \wp(\tau, z) = \frac{1}{20} \wp(\tau, z)
\end{equation}
where we use the coefficients recorded in (\ref{eqn:psissconstcoef}).  Since the constant term of $\wp$ is $\frac{1}{12}$, one should note that the constant term above agrees with the degree zero Gromov-Witten term $F_{2}^{(0)}$ in (\ref{eqn:degzeroGWgrtwo}).  For $g \geq 3$,
\begin{equation}\label{eqn:VzeroPSI2g-2}
\begin{split}
\big(12 \psi_{2g-2} \big| V_{0} \big) & = -12c_{2g-2}(0)\frac{B_{2g-2}}{2(2g-2)}E_{2g-2}(\tau) + 12 c_{2g-2}(-1) \wp^{(2g-4)}(\tau, z)\\
& = F_{g}^{(0)} E_{2g-2}(\tau) + \frac{6 |B_{2g}|}{g(2g-2)!}\wp^{(2g-4)}(\tau, z).
\end{split}
\end{equation}  

\begin{rmk}
In writing $F_{g}(\Omega) = \text{ML}(12 \psi_{2g-2})$, it is clear that the degree zero term $F_{g}^{(0)}$ must come entirely from $\big( 12 \psi_{2g-2} \big| V_{0} \big)$.  Indeed, one can check from (\ref{eqn:psi2V0Maassrift}) and (\ref{eqn:VzeroPSI2g-2}) that the constant terms agree with $F_{g}^{(0)}$ recorded in (\ref{eqn:degzeroGWgrtwo}).  
\end{rmk}

Because $\wp^{(2g-4)}$ is a meromorphic Jacobi form of weight $2g-2$ and index 0, we see from (\ref{eqn:psi2V0Maassrift}) and (\ref{eqn:VzeroPSI2g-2}) that for $g \geq 2$, $\big( 12 \psi_{2g-2} \big| V_{0} \big)$ is also a meromorphic Jacobi form of weight $2g-2$ and index 0.  It is this fact which allows one to determine the denominators of $F_{g}$ in the following lemma.   

\begin{lemmy}[\bfseries H. Aoki \cite{aoki_notitle_2018}, Sec. 3.4]
For all $g \geq 2$, the denominator of the Gromov-Witten potential $F_{g}$ is $\chi_{10}^{g-1}$.  In other words, 
\begin{equation}\label{eqn:defnnofSgeee}
\mathcal{S}_{g}(\Omega) \coloneqq \chi_{10}^{g-1} F_{g}(\Omega) \in \mathfrak{M}_{12g-12}(\Gamma_{2})
\end{equation}
is a holomorphic Siegel modular form of weight $12g-12$ and genus two.  We therefore have for $g \geq 2$
\begin{equation}
F_{g}(\Omega) \in \frac{1}{\chi_{10}^{g-1}} \mathfrak{M}_{12g-12}(\Gamma_{2}) \subset \widetilde{\mathfrak{M}}_{2g-2}(\Gamma_{2}).
\end{equation}
\end{lemmy}


Recall from Theorem \ref{IgusacrassTHMMM} that Igusa proved the ring $\mathfrak{M}_{2*}(\Gamma_{2})$ of Siegel modular forms of even weight and genus two is generated freely over $\mathbb{C}$ by $\chi_{12}, \chi_{10}, \mathcal{E}_{4}$, and $\mathcal{E}_{6}$.  Therefore, the power of the above lemma is evidently that only a finite number of computations are needed to determine $\mathcal{S}_{g}$ explicitly in the ring $\mathfrak{M}_{2*}(\Gamma_{2})$.  However, these computations are not so easy to carry out for large $g$.  

By (\ref{eqn:defnnofSgeee}), the lowest power of $Q$ appearing in the Fourier-Jacobi expansion of $\mathcal{S}_{g}$ is $Q^{g-1}$ from which we can at least conclude
\begin{equation}\label{eqn:exricitSgeefr2222}
\mathcal{S}_{g}(\Omega) = \sum_{i+j \geq g-1} f_{i,j}(\mathcal{E}_{4}, \mathcal{E}_{6}) \chi_{10}^{i} \chi_{12}^{j}
\end{equation}
where $f_{i,j}(\mathcal{E}_{4}, \mathcal{E}_{6})$ is a holomorphic Siegel modular form of weight $12(g-j-1)-10i$ which is a polynomial in $\mathcal{E}_{4}$ and $\mathcal{E}_{6}$.  In other words, the Fourier-Jacobi expansion of $\mathcal{S}_{g}$ indicates that it must consist of a sum of monomials having at least $g-1$ factors of the cusp forms $\chi_{10}$ and $\chi_{12}$.  The following lemma can be easily checked directly.  

\begin{lemmy}
For $2 \leq g <6$, the sum in (\ref{eqn:exricitSgeefr2222}) is over $i + j = g-1$.  In other words, we have
\begin{equation}
\mathcal{S}_{g}(\Omega) = \sum_{k=0}^{g-1}f_{k}(\mathcal{E}_{4}, \mathcal{E}_{6}) \chi_{10}^{k} \chi_{12}^{g-1-k}
\end{equation}
where $f_{k}$ is a Siegel modular form of weight $2k$ which is a polynomial in $\mathcal{E}_{4}$ and $\mathcal{E}_{6}$.  It follows that the Gromov-Witten potential in such genera is a polynomial in $\chi_{12}/\chi_{10}, \mathcal{E}_{4},$ and $\mathcal{E}_{6}$
\begin{equation}\label{eqn:GWpotg2345}
F_{g}(\Omega) = \sum_{k=0}^{g-1} f_{k}(\mathcal{E}_{4}, \mathcal{E}_{6}) \bigg(\frac{\chi_{12}}{\chi_{10}}\bigg)^{g-1-k}.  
\end{equation}
\end{lemmy}

It turns out that for genus $2 \leq g < 6$, the forms (\ref{eqn:VzeroPSI2g-2}) and (\ref{eqn:psi2V0Maassrift}) of $\big(12 \psi_{2g-2} \big| V_{0} \big)$ allow us to determine the polynomial explicitly.  Using the well-known formula 
\begin{equation}
\wp^{(1)}(\tau, z)^{2} = 4 \wp(\tau, z)^{3} - \frac{1}{12} E_{4}(\tau) \wp(\tau, z) + \frac{1}{216} E_{6}(\tau)
\end{equation}
it is straightforward to show the following lemma.  

\begin{lemmy}
For all $g \geq 2$ there exists a weighted homogeneous polynomial $\mathcal{P}_{g}(X, Y, Z)$ of degree $2g-2$ with weight 2 in $X$, weight 4 in $Y$, and weight 6 in $Z$ such that 
\begin{equation}\label{eqn:defnnnPgg}
\wp^{(2g-4)}(\tau, z) = \mathcal{P}_{g}\big( 12 \wp, E_{4}, E_{6} \big). 
\end{equation}
In other words, all even derivatives of the Weierstrass $\wp$-functions can be expressed as a polynomial in $\wp$, $E_{4}$, and $E_{6}$.  The first few examples are given by 
\begin{equation}
\setlength{\jot}{10pt}
\begin{split}
& \mathcal{P}_{2}(X, Y, Z) = \frac{1}{12}X \\
& \mathcal{P}_{3}(X, Y, Z) = \frac{1}{24}\big(X^{2} - Y\big) \\
& \mathcal{P}_{4}(X, Y, Z) = \frac{1}{72} \big(5X^{3}-9XY+4Z\big) \\
& \mathcal{P}_{5}(X, Y, Z) = \frac{1}{144} \big( 35X^{4} -84X^{2} Y + 40 XZ + 9 Y^{2}\big).
\end{split}
\end{equation}
\end{lemmy}

\begin{thm}
First for genus $g=2$, the Gromov-Witten potential of the banana manifold is the following meromorphic Siegel modular form of genus 2 and weight 2 
\begin{equation}
F_{2}(\Omega) = \frac{1}{240} \frac{\chi_{12}}{\chi_{10}}.
\end{equation}
For $g = 3, 4, 5$, the Gromov-Witten potentials are meromorphic Siegel modular forms of weight $2g-2$ given by the formula
\begin{equation}\label{eqn:explformGWpotts}
F_{g}(\Omega) = F_{g}^{(0)} \mathcal{E}_{2g-2} +  \frac{6 |B_{2g}|}{g (2g-2)!} \mathcal{P}_{g} \bigg( \frac{\chi_{12}}{\chi_{10}}, \mathcal{E}_{4}, \mathcal{E}_{6} \bigg)
\end{equation}
where the $F_{g}^{(0)}$ are the degree zero contributions to the Gromov-Witten potentials (\ref{eqn:degzeroGWgrtwo}).  Explicitly, we have
\begin{equation}
\setlength{\jot}{10pt}
\begin{split}
& F_{3}(\Omega) = \frac{1}{60480}\bigg(5 \bigg(\frac{\chi_{12}}{\chi_{10}}\bigg)^{2}-6 \, \mathcal{E}_{4}\bigg) \\
& F_{4}(\Omega) = \frac{1}{7257600}\bigg( 35 \bigg( \frac{\chi_{12}}{\chi_{10}}\bigg)^{3} -63 \frac{\chi_{12}}{\chi_{10}} \mathcal{E}_{4} + 30 \mathcal{E}_{6}\bigg) \\
& F_{5}(\Omega) = \frac{1}{319334400} \bigg( 175 \bigg(\frac{\chi_{12}}{\chi_{10}}\bigg)^{4} - 420 \bigg(\frac{\chi_{12}}{\chi_{10}}\bigg)^{2} \mathcal{E}_{4} + 200 \frac{\chi_{12}}{\chi_{10}} \mathcal{E}_{6} + 42 \mathcal{E}_{8}\bigg). \\
\end{split}
\end{equation}
\end{thm}

\begin{proof}
The method of the proof is to use the known form (\ref{eqn:GWpotg2345}) of the $F_{g}$, apply the Siegel operator $\Phi$ defined in (\ref{eqn:SiegelOPPP}) to produce the $Q^{0}$ coefficient in the Fourier-Jacobi expansion, and require this to agree with the known formulas (\ref{eqn:psi2V0Maassrift}) and (\ref{eqn:VzeroPSI2g-2}) for $\big( 12 \psi_{2g-2} \big| V_{0} \big)$.  We will use the relationship between the Weierstrass $\wp$-function and the Jacobi cusp forms
\begin{equation}
\wp(\tau, z) = \frac{1}{12} \frac{\varphi_{12,1}(\tau, z)}{\varphi_{10,1}(\tau, z)}.  
\end{equation}
In the case of $g=2$, we know from (\ref{eqn:GWpotg2345}) that the Gromov-Witten potential takes the form
\begin{equation}
F_{2}(\Omega) = \sum_{k=0}^{1} f_{k}(\mathcal{E}_{4}, \mathcal{E}_{6}) \bigg( \frac{\chi_{12}}{\chi_{10}} \bigg)^{1-k} = c \, \frac{\chi_{12}}{\chi_{10}} 
\end{equation}
for a constant $c$.  We cannot have a term $f_{1}(\mathcal{E}_{4}, \mathcal{E}_{6})$ because it would have to have weight 2.  Applying the Siegel operator, we find
\begin{equation}
\Phi(F_{2}) = c \, \frac{\varphi_{12,1}(\tau, z)}{\varphi_{10,1}(\tau, z)} = 12 \, c \, \wp(\tau, z).
\end{equation}
But since we must have $\Phi(F_{2}) = \big( 12 \psi_{2} \big| V_{0} \big)$, we know from (\ref{eqn:psi2V0Maassrift}) that $c=1/240$.  This completes the proof in the case of $g=2$.  

For $g=3,4,5$ we know the Gromov-Witten potentials take the form presented in (\ref{eqn:GWpotg2345}), and we apply the Siegel operator to get  
\begin{equation}
\Phi(F_{g}) = \sum_{k=0}^{g-1} f_{k}(E_{4}, E_{6})(12 \wp)^{g-1-k}.
\end{equation}
This is a weighted-homogeneous polynomial of degree $2g-2$ in $12 \wp, E_{4}$, and $E_{6}$ with weights $2,4,6$ respectively.  By (\ref{eqn:VzeroPSI2g-2}), since $g >2$ we must have
\begin{equation}
\Phi(F_{g}) = \big( 12 \psi_{2g-2} \big| V_{0} \big) = F_{g}^{(0)} E_{2g-2}(\tau) + \frac{6 |B_{2g}|}{g (2g-2)!} \wp^{(2g-4)}(\tau, z).
\end{equation}
Using the definition $\wp^{(2g-4)}(\tau, z) = \mathcal{P}_{g}(12 \wp, E_{4}, E_{6})$ of $\mathcal{P}_{g}$ in (\ref{eqn:defnnnPgg}), this uniquely fixes the $f_{k}$ and verifies that $F_{g}$ takes the form (\ref{eqn:explformGWpotts}) as claimed.   
\end{proof}

\subsection{Symmetries of the Gromov-Witten Potentials}\label{subsec:SyymmmGWpotss}

Because they are more convenient for studying the Hecke operators and Maass lifting, we have been working with the $Q, q, y$ variables for some time.  However, recall that the geometric variables on the banana curves are $Q_{1}, Q_{2}, Q_{3}$ with the relationship given by (\ref{eqn:changeofvarss}).  Applying the change of variables, from (\ref{eqn:MLgengnew}) and the definition of the polylogarithm, it is easy to see that the Gromov-Witten potentials in the geometric variables for genus $g \geq 2$ are given by 
\begin{equation}
F_{g}(Q_{1}, Q_{2}, Q_{3}) = F_{g}^{(0)} +  \sum_{\substack{ d_{1}, d_{2}, d_{3} \geq 0 \\ (d_{1}, d_{2}, d_{3}) \neq (0,0,0)}} 12 c_{2g-2}(||\underline{\bf{d}}||) \sum_{r=1}^{\infty} r^{2g-3} Q_{1}^{rd_{1}} Q_{2}^{rd_{2}}Q_{3}^{rd_{3}}
\end{equation}
where the quadratic form $||\underline{\bf{d}}||$ was defined in (\ref{eqn:quaddformdsss}).  From this formula above, we can extract the Gromov-Witten invariants for all classes $\underline{\bf{d}} \neq (0,0,0)$.  We find
\begin{equation}
\text{GW}_{g, \underline{\bf{d}}}(X) = \sum_{\substack{r | (d_{1}, d_{2}, d_{3}) \\ r >0}} r^{2g-3} 12 c_{2g-2} \bigg( \frac{||\underline{\bf{d}}||}{r^{2}} \bigg).
\end{equation}
(The invariants for $g=0,1$ of course look similar, but we will not discuss them here due to the subtleties in degree zero.) 

Notice that the Gromov-Witten invariants depend on not only the value of the quadratic form, but also the divisibility of the class.  We will soon find that the Gopakumar-Vafa invariants on the other hand, depend only on the quadratic form.  In either set of variables, the Gromov-Witten invariants exactly match the form (\ref{eqn:MaassSPEZFCOES}) for the Fourier coefficients of an object in the Maass `Spezialschar.'  This is of course to be expected from our results in the previous section.   

There are geometrical symmetries of the formal banana manifold which should induce symmetries of the Gromov-Witten potentials in the variables $Q_{1}, Q_{2}, Q_{3}$.  The banana curves are each identical in $X_{\text{ban}}$, so there is an obvious $S_{3}$-symmetry permuting the $Q_{i}$.  Less obvious, one can perform a \emph{flop} through the curve $C_{3}$ which should yield a symmetry of the potentials under the transformation
\begin{equation}\label{eqn:froptransformm}
\big(Q_{1}, Q_{2}, Q_{3} \big) \longmapsto \big( Q_{1}Q_{3}^{2}, Q_{2}Q_{3}^{2}, Q_{3}^{-1}\big).  
\end{equation}
By the permutation symmetry, one can of course equally well flop through the two other banana curves.  

It turns out that in the $Q, q, y$ variables, these symmetries come for free after identifying the $F_{g}$ as Siegel modular forms.  There is a natural embedding $GL_{2}(\mathbb{Z}) \hookrightarrow Sp_{4}(\mathbb{Z})$ defined by 
\begin{equation}
M \longmapsto
\begin{pmatrix}
M & 0 \\
0 & (M^{-1})^{T}
\end{pmatrix}.
\end{equation}
By (\ref{eqn:SMGAction}), the resulting action on the Siegel upper-half plane $\mathfrak{H}_{2}$ is by the adjoint action
\begin{equation}
\Omega \longmapsto M \Omega M^{T}, \,\,\,\,\,\,\,\,\,\,\,\,\,\,\,\,\,\,\, \Omega = \begin{pmatrix} \tau & z \\ z & \sigma \end{pmatrix} \in \mathfrak{H}_{2}
\end{equation}
where we recall the standard conventions $Q=e^{2 \pi i \sigma}, q = e^{2 \pi i \tau}$, and $y=e^{2 \pi i z}$.  Since $\text{det}(M)= \pm 1$, if $F$ is a genus two Siegel modular form of even weight, we have by (\ref{eqn:defnnnSMFDS})
\begin{equation}
F(M \Omega M^{T}) = F(\Omega).
\end{equation}
Therefore, all even weight genus two Siegel modular forms have a $GL_{2}(\mathbb{Z})$ invariance acting by the adjoint action on the Siegel upper-half plane.  In fact, because the action is by the adjoint, $-1$ clearly acts trivially, so we get a $PGL_{2}(\mathbb{Z})$ invariance.  The proof of the following proposition is an easy check.  

\begin{proppy}\label{proppy:chvarsssingfibrrl}
Let us write $GL_{2}(\mathbb{Z})$ with the following set of ordered generators
\begin{equation}
GL_{2}(\mathbb{Z}) = \langle \gamma_{1}, \gamma_{2}, \gamma_{3} \rangle \coloneqq \bigg\langle  
\begin{pmatrix}
1 & 0 \\ 0 & -1
\end{pmatrix},
\begin{pmatrix}
0 & 1 \\ 1 & 0
\end{pmatrix},
\begin{pmatrix}
1 & -1 \\ 0 & -1
\end{pmatrix}
\bigg\rangle.
\end{equation}
Under the change of variables (\ref{eqn:changeofvarss}), the adjoint action of $\gamma_{1}$ on the Siegel upper-half plane $\mathfrak{H}_{2}$ induces the flop transformation (\ref{eqn:froptransformm}).  The adjoint action of $\gamma_{2}$ on $\mathfrak{H}_{2}$ induces the permutation
\begin{equation}
\big(Q_{1}, Q_{2}, Q_{3} \big) \longmapsto \big( Q_{2}, Q_{1}, Q_{3} \big)
\end{equation}
and the adjoint action of $\gamma_{3}$ on $\mathfrak{H}_{2}$ induces another permutation 
\begin{equation}
\big(Q_{1}, Q_{2}, Q_{3} \big) \longmapsto \big( Q_{3}, Q_{2}, Q_{1} \big).
\end{equation}
Therefore, identifying the Gromov-Witten potential $F_{g}$ for $g \geq 2$, as an even weight genus two Siegel modular form, the $S_{3}$ and flop symmetries of $F_{g}$ correspond to the adjoint action of $PGL_{2}(\mathbb{Z})$ on $\mathfrak{H}_{2}$. 
\end{proppy}

Equivalently, we can define a rank three lattice $\big( V, ||\underline{\bf{d}}|| \big)$ generated by the banana curves $C_{1}, C_{2}, C_{3}$ and interpret the $\gamma_{i}$ as lattice automorphisms (automorphisms of $V$, preserving $||\underline{\bf{d}}||$).  From this perspective, the action of the $\gamma_{i}$ on the generators is
\begin{equation}\label{eqn:actonbanncurvves}
\begin{split}
& \gamma_{1} : \big(C_{1}, C_{2}, C_{3}\big) \longmapsto  \big(C_{1}+2C_{3}, C_{2}+2C_{3}, -C_{3}\big)  \\
&\gamma_{2} : \big(C_{1}, C_{2}, C_{3}\big) \longmapsto  \big(C_{2}, C_{1}, C_{3} \big)   \\
&\gamma_{3} : \big(C_{1}, C_{2}, C_{3}\big) \longmapsto \big( C_{3}, C_{2}, C_{1} \big).
\end{split}
\end{equation}

\section{The Gopakumar-Vafa (BPS) Invariants}

Recall from (\ref{eqn:cbetanGVrelll}) and the surrounding discussion that if the Donaldson-Thomas partition function is written as an infinite product, then there is a nice relationship between the Gopakumar-Vafa invariants and the exponents in the product.  We use the Donaldson-Thomas partition function $Z_{\text{DT}}(X_{\text{ban}})_{\Gamma}$ computed in (\ref{eqn:DTBanpartfunc}) and recall that a fiber class in the banana manifold is described as $\underline{\bf{d}} = (d_{1}, d_{2}, d_{3})$ with quadratic form
\begin{equation}\label{eqn:quaddfrmmmm}
||\underline{\bf{d}}|| = 2d_{1}d_{2} + 2d_{1}d_{3} + 2d_{2}d_{3} -d_{1}^{2}-d_{2}^{2}-d_{3}^{2}.
\end{equation}
We therefore have for all fiber classes $\underline{\bf{d}} \neq (0,0,0)$
\begin{equation}\label{eqn:GVcoeffEllGen1}
\sum_{g=0}^{\infty} n_{g, ||\underline{\bf{d}}||}(X_{\text{ban}})\big(t ^{\frac{1}{2}} + t^{-\frac{1}{2}}\big)^{2g-2} = \sum_{k \in \mathbb{Z}} 12 c(||\underline{\bf{d}}||, k) (-t)^{k}
\end{equation} 
where $n_{g, ||\underline{\bf{d}}||}(X_{\text{ban}})$ are the Gopakumar-Vafa invariants of the banana manifold $X_{\text{ban}}$ in fiber classes, and $c(||\underline{\bf{d}}||, k)$ are the coefficients of the equivariant elliptic genus of $\mathbb{C}^{2}$.  

\begin{cory}
The Gopakumar-Vafa invariants of $X_{\text{ban}}$ in fiber classes $\underline{\bf{d}}$ depend only on the quadratic form $||\underline{\bf{d}}||$.  In particular, they do not depend on the divisibility of the class.  
\end{cory}

\noindent This is an unusual property of $X_{\text{ban}}$, which also holds for a local $K3$ surface by the Katz-Klemm-Vafa conjecture from physics \cite{katz_m-theory_1999}, proved by Pandharipande-Thomas \cite{pandharipande_katz-klemm-vafa_2014}.  

Using the change of variables (\ref{eqn:changeofvarss}), we can equivalently give a fiber class $\underline{\bf{d}}$ in terms of K\"{a}hler parameters $(m,n,l)$ corresponding to variables $Q, q, y$.  The quadratic form transforms as
\begin{equation}\label{eqn:qfrmtranny}
||\underline{\bf{d}}|| = 4nm - l^{2}.  
\end{equation}

\begin{cory}
The Gopakumar-Vafa invariants of $X_{\text{ban}}$ are encoded into the equivariant elliptic genus of $\mathbb{C}^{2}$ as follows
\begin{equation}\label{eqn:relnGVtoELLGEN}
\sum_{a,g \geq 0,l \in \mathbb{Z}} \tfrac{1}{12} n_{g, 4a-l^{2}}(X_{\text{ban}}) q^{a}y^{l}\big(t^{\frac{1}{2}} + t^{-\frac{1}{2}}\big)^{2g-2} = \text{Ell}_{q,y}(\mathbb{C}^{2}; -t).  
\end{equation}
\end{cory}

\begin{proof}
First re-write (\ref{eqn:GVcoeffEllGen1}) replacing $||\underline{\bf{d}}||$ with $4nm-l^{2}$.  Noting that $4nm-l^{2}$ only depends on $n, m$ through $a=nm$, we let $q$ be a formal parameter tracking $a \geq 0$, and $y$ a formal parameter tracking $l \in \mathbb{Z}$.  Multiplying both sides of (\ref{eqn:GVcoeffEllGen1}) by $q^{a}y^{l}$ and summing over the parameter values, we have
\begin{equation}
\sum_{a,g \geq 0, l \in \mathbb{Z}} \tfrac{1}{12} n_{g, 4a-l^{2}}(X_{\text{ban}}) q^{a}y^{l}\big(t^{\frac{1}{2}} + t^{-\frac{1}{2}}\big)^{2g-2}  = \sum_{\substack{a \geq 0 \\ l,k \in \mathbb{Z}}}  c(4a-l^{2}, k)q^{a}y^{l}(-t)^{k}.
\end{equation}
The righthand side is precisely $\text{Ell}_{q, y}(\mathbb{C}^{2}; -t)$, which completes the proof.
\end{proof}

Let us now return to the diagram presented in the introduction, which we now are able to understand fully.  
\begin{equation}\label{eqn:DT/GW/GVdiaggs}
\begin{tikzcd}
Z_{\text{DT}}(X_{\text{ban}})_{\Gamma} \arrow[dashed, swap]{rrr}{\text{(Asymptotic) GW/DT Corresopndence}}     &     &      &      \sum_{g=0}^{\infty} \lambda^{2g-2} \text{ML}(12 \psi_{2g-2}) \\
                                                            &      &      &        \\
                                                            &      &      &        \\
                                                            &      &      &        \\
                                                            &      &  12\Phi_{0}(\tau, z, x) = \sum_{g=0}^{\infty} \lambda^{2g-2} 12\psi_{2g-2}(\tau, z) \arrow{uuuull}{\text{Formal Borcherds Lift of}\, 12 \Phi_{0}} \arrow[swap]{uuuur}{\text{Maass Lift of the}\, 12 \psi_{2g-2}}    &
\end{tikzcd}
\end{equation}
The top left corner shows the Donaldson-Thomas partition function, which by (\ref{eqn:DTpartfuncBorLiftTWOO}) we identify as the formal Borcherds lift of the equivariant elliptic genus $\Phi_{0}(\tau, z, x) = \text{Ell}_{q,y}(\mathbb{C}^{2}; t)$.  The top right corner shows the generating function of the Maass lifts of the $\psi_{2g-2}$ which by (\ref{eqn:GWDTcorraympps}) is essentially the Gromov-Witten free energy.  The two corners are related by the GW/DT correspondence, which is technically only an asymptotic statement since the degree zero contributions are important in this case.  Finally, the bottom corner shows $\Phi_{0}$ which by the above discussion, gives rise to the Gopakumar-Vafa invariants as in (\ref{eqn:relnGVtoELLGEN}).  

\begin{rmk}
It has long been expected that the Gopakumar-Vafa invariants underly both of the Donaldson-Thomas and Gromov-Witten theories on a Calabi-Yau threefold.  We regard diagram (\ref{eqn:DT/GW/GVdiaggs}) as an nice illustration of this whereby a modular object encoding the Gopakumar-Vafa invariants produces the Donaldson-Thomas and Gromov-Witten partition functions via nice arithmetic lifting procedures.
\end{rmk}

Performing the change of variables $t = e^{i \lambda}$ in (\ref{eqn:relnGVtoELLGEN}), it is a straightforward simplification to see that
\begin{equation}
\sum_{n, g \geq 0, l \in \mathbb{Z}} \tfrac{1}{12} n_{g,4n-l^{2}}(X_{\text{ban}})q^{n}y^{l} \lambda^{2g-2} \bigg( \frac{2 \sin(\lambda/2)}{\lambda}\bigg)^{2g-2} = \sum_{g=0}^{\infty} \lambda^{2g-2} \psi_{2g-2}(\tau, z)
\end{equation}
where for the righthand side, we have used that $\Phi_{0}(\tau, z, x) = \text{Ell}_{q,y}(\mathbb{C}^{2}; t)$ admits the $\lambda$-expansion (\ref{eqn:genexp}).  Recall that the weak Jacobi forms $\psi_{2g-2}$ for all $g \geq 0$ are shown in (\ref{eqn:defnofphys}).  By equating corresponding powers of $\lambda$ on each side, we can produce formulas for the fixed-genus generating functions
\[ \sum_{n \geq 0, l \in \mathbb{Z}} \tfrac{1}{12} n_{g, 4n-l^{2}}(X_{\text{ban}}) q^{n}y^{l}\]  
for all genus as a linear combination of the $\psi_{2g-2}$.  The results up to genus $g=5$ are as follows
\begin{equation} \label{eqn:lowgenGV}
\sum_{n \geq 0, l \in \mathbb{Z}} \tfrac{1}{12} n_{g,4n-l^{2}}(X_{\text{ban}}) q^{n} y^{l} =
\begin{cases}
\Theta^{2} & g=0 \\[1ex]
\Theta^{2}\big(\wp-\frac{1}{12}\big)  & g=1 \\[1ex]
\frac{\Theta^{2}}{240}\big(E_{4}-1\big)  & g=2 \\[1ex]
\frac{\Theta^{2}}{288}\big(\frac{E_{6}}{21}+\frac{E_{4}}{10}-\frac{31}{210}\big)   & g=3 \\[1ex]
\frac{\Theta^{2}}{864}\big(\frac{E_{8}}{200} + \frac{E_{6}}{42} + \frac{E_{4}}{25}-\frac{289}{4200}\big) & g=4 \\[1ex]
\frac{\Theta^{2}}{3840} \big(\frac{E_{10}}{1386} + \frac{E_{8}}{180} + \frac{E_{6}}{54} + \frac{E_{4}}{35} - \frac{317}{5940} \big) & g=5
\end{cases}
\end{equation}
These quantities are evidently weak Jacobi forms of index 1, but of non-homogeoeous weight.  We will return to these formulae shortly.  

Though it is an interesting structural result in the theory, the fact that the Gopakumar-Vafa invariants depend only on the value of the quadratic form (\ref{eqn:qfrmtranny}) means that formulae such as (\ref{eqn:relnGVtoELLGEN}) and (\ref{eqn:lowgenGV}) contain a lot of redundancy.  The quadratic form can only take values $-1,0 \Mod{4}$, so we can find a pair of one-index families of classes taking each value of the quadratic form once.  In the K\"{a}hler parameters $\underline{\bf{d}} = (d_{1}, d_{2}, d_{3})$, we choose for $n \geq 0$
\begin{equation}
\begin{split}
& (d_{1}, d_{2}, d_{3}) = (1, n, n), \,\,\,\,\,\,\,\,\,\,\,\,\,\,\,\,\,\,\,\,\,\,\,\, ||\underline{\bf{d}}|| = 4n-1 \equiv -1\Mod4 \\
&(d_{1}, d_{2}, d_{3}) = (1, n, n+1), \,\,\,\,\,\,\,\,\,\,\,\,\,\,   ||\underline{\bf{d}}|| = 4n \equiv  0\Mod4.
\end{split}
\end{equation}
Under the change of variables (\ref{eqn:changeofvarss}), these classes correspond to 
\begin{equation}
\begin{split}
& (m, n, l) = (1, n, 1), \,\,\,\,\,\,\,\,\,\,\,\,\,\,\,\,\, 4nm-l^{2} = 4n-1 \equiv -1\Mod4 \\
&(m, n, l) = (1, n, 0), \,\,\,\,\,\,\,\,\,\,\,\,\,\,\,\,\,  4nm-l^{2}  = 4n \equiv  0\Mod4
\end{split}
\end{equation}
meaning we can restrict attention to just coefficients of $y^{0}$ and $y^{1}$, with the power of $q$ parameterizing the value of the quadratic form.

Returning to formulae (\ref{eqn:lowgenGV}), let us consider the coefficient of $q^{n}$ or $q^{n}y$, depending on whether the quadratic form is even or odd.  One can check that in (\ref{eqn:lowgenGV}) for $g \geq 2$, as the genus increases the coefficients of the Eisenstein series conspire to eventually make the coefficient of $q^{n}$ or $q^{n}y$ vanish.  In other words, as the genus increases, the lowest power of $q$ in the quantities (\ref{eqn:lowgenGV}) does as well.  We are observing that for a fixed class, the Gopakumar-Vafa invariants vanish for large enough genus.  

In fact, there are some very nice features and patterns in the Gopakumar-Vafa invariants.  Let us now use (\ref{eqn:lowgenGV}) to present them numerically, differentiating quadratic form values $4n$ and $4n-1$.    

\begin{center}
\begin{tabular}{c|ccccccc} \toprule
    $\tfrac{1}{12} n_{g, 4n-1}(X)$ & {$g=0$} & {$g=1$} & {$g=2$} & {$g=3$} & {$g=4$} & {$g=5$} & {$g=6$} \\ \midrule
    $n=0$  & 1     &    0     &  0    &    0     &    0     &    0   & 0   \\
    $n=1$  &  8    &   -6     &  1    &    0     &    0     &    0   &  0 \\
    $n=2$  & 39 &   -46    &  17   &    -2     &    0     &     0    &   0 \\
    $n=3$  & 152      &    -242    &   139    &    -34     &     3    &     0     &    0 \\
    $n=4$  & 513  &    -1024  & 800  &      -304   &    56     &      -4    &  0 \\ 
    $n=5$  & 1560  & -3730 & 3683  &    -1912 &    548      &      -82    &  5 \\ 
\end{tabular}
\end{center}

\begin{center}
\begin{tabular}{c|ccccccc} \toprule
    $\tfrac{1}{12} n_{g,4n}(X)$ & {$g=0$} & {$g=1$} & {$g=2$} & {$g=3$} & {$g=4$} & {$g=5$} & {$g=6$} \\ \midrule
    $n=0$  & -2    &    1     &  0    &    0     &    0     &    0   &  0 \\
    $n=1$  &  -12    &    10     &  -2    &    0     &     0    &     0    &  0    \\ 
    $n=2$  &  -56  &   72  &  -30  &   4    &      0     &     0    &   0     \\
    $n=3$  & -208      &   352    &    -220  &      60     &     -6    &     0      &   0 \\ 
    $n=4$  & -684  & 1434 & -1194  &   492  &    -100      &      8    &  0 \\ 
    $n=5$  & -2032  & 5056 & -5252  &    2908 &    -902      &      148    &  -10 \\ \bottomrule
\end{tabular}
\end{center}
Let us record some of the notable features of these invariants.  As we will discuss in more detail below, the classes with quadratic form value $4n$ deform into the smooth fibers while those taking value $4n-1$ are fixed within the singular fibers.  

\begin{itemize}

\item Notice from the above tables, that for $n \geq 1$ both for classes which deform into the smooth fibers and those that do not, we have the following very simple expression for the maximal genus as a function of $n$
\begin{equation}
g_{\text{max}} = n+1.  
\end{equation}
For the classes which deform into the smooth fibers, this evidently holds for $n=0$ as well.  

\item Again for $n \geq 1$, in both cases the corresponding Gopakumar-Vafa invariant in the maximal genus exhibit the following behavior  
\begin{equation}
\tfrac{1}{12} n_{g_{\text{max}}, 4n-1}(X_{\text{ban}}) = (-1)^{n+1}n, \,\,\,\,\,\,\,\,\,\,\,\, \tfrac{1}{12} n_{g_{\text{max}}, 4n}(X_{\text{ban}}) = (-1)^{n}2n.
\end{equation}
\item For a fixed genus, the Gopakumar-Vafa invariants have the same sign for all $n$ whereas for a fixed class, the finitely many invariants alternate in sign as one increases the genus.  
\end{itemize}

\section{The Perspective from the Smooth Fibers}

The smooth fibers of the banana manifold can all be identified as the product $E \times E$ of some elliptic curve with itself.  We choose the following ordered set of generators for the Neron-Serveri group
\begin{equation}
\text{NS}(E \times E) = \big\langle E_{1}, E_{2}, E_{3} \coloneqq \Delta-E_{1}-E_{2} \big\rangle  \cong \mathbb{Z}^{3}
\end{equation}
where $E_{1}$ and $E_{2}$ are the classes corresponding to the two factors in the product while $\Delta$ is the class of the diagonal.  The Neron-Severi group is isomorphic to $\mathbb{Z}^{3}$ as an abelian group but the intersection form gives it the structure of a rank-three lattice.  The square of the class $\beta = mE_{1} + nE_{2} + \xi E_{3}$ is easily computed to be
\begin{equation}\label{eqn:intformsmmmmfibrr}
\beta^{2} = 2 nm -2 \xi^{2}
\end{equation}
since $E_{1}$ and $E_{2}$ are orthogonal to $E_{3}$, and $E_{3}^{2} = -2$.  We will see that it is not coincidental that $\beta^{2}$ looks similar to the quadratic form $4mn-l^{2}$ in the variables $Q, q, y$.  We additionally have the anti-diagonal class $\Delta^{op}$, which in terms of the generators is
\begin{equation}\label{eqn:ANTIdiagggg}
\Delta^{op} = 2 E_{1} + 2E_{2} - \Delta.  
\end{equation}

Recall that we denote by $C_{1}, C_{2}, C_{3}$ the banana curves in the banana manifold, and that they generate the cone of effective fiber classes of $X_{\text{ban}}$.  Certainly, the effective classes on the smooth fibers deform to some effective class on the singular fiber.  One can compute that 
\begin{equation}\label{eqn:relnsmsingfibbrr}
E_{1} = C_{1} + C_{3} \,\,\,\,\,\,\,\,\,\,\,\,\,\,\,\, E_{2} = C_{2} + C_{3} \,\,\,\,\,\,\,\,\,\,\,\,\,\,\,\, \Delta = C_{1} + C_{2} + 4 C_{3}.   
\end{equation}
It follows that $E_{3} = 2C_{3}$.  

We can introduce formal variables $v_{1}, v_{2}, v_{3}$ tracking degrees of classes along the curves $E_{1}, E_{2}, E_{3}$ respectively.  By (\ref{eqn:relnsmsingfibbrr}) we see the following relationship to the $Q_{1}, Q_{2}, Q_{3}$ variables
\begin{equation}
v_{1} = Q_{1} Q_{3}  \,\,\,\,\,\,\,\,\,\,\,\,\,\,\,\,   v_{2} = Q_{2} Q_{3} \,\,\,\,\,\,\,\,\,\,\,\,\,\,\,\, v_{3} = Q_{3}^{2}. 
\end{equation}
We notice that this is very close to (\ref{eqn:changeofvarss}).  We can therefore also transform into variables $Q, q, y$ by
\begin{equation}\label{eqn:relnsmsingfibrvarss}
v_{1} = Q \,\,\,\,\,\,\,\,\,\,\,\,\,\,\,\, v_{2} = q \,\,\,\,\,\,\,\,\,\,\,\,\,\,\,\, v_{3} = y^{2}.
\end{equation}

Recall that $(m,n, l)$ are the classes tracked by $Q, q, y$.  The quadratic form $4nm-l^{2}$ is even if and only if $l$ (the power of $y$) is even.  By (\ref{eqn:relnsmsingfibrvarss}), in order to get an integral power of $v_{3}$, we need the power of $y$ to be even.  In other words, only classes for which $4nm-l^{2} \equiv 0 \Mod4$ deform into the smooth fiber.  Note also that for $4nm-l^{2} \equiv 0 \Mod 4$, we have
\begin{equation}
||\underline{\bf{d}}|| = 4nm-l^{2} = 2 \beta^{2}
\end{equation}
where $\beta^{2}$ is the intersection form on the smooth fiber (\ref{eqn:intformsmmmmfibrr}), and $||\underline{\bf{d}}||$ is the quadratic form (\ref{eqn:quaddfrmmmm}) in the $Q_{1}, Q_{2}, Q_{3}$ variables.  This gives a nice interpretation of the quadratic form we have been studying: for classes which deform to the smooth fibers, it is simply (twice) the intersection form.

\subsection{The Automorphism Group of the Smooth Fibers}  

Recall that the main theme from Section \ref{subsec:SyymmmGWpotss} was the following: in the variables $(Q_{1}, Q_{2}, Q_{3})$ the Gromov-Witten potentials carry an $S_{3}$-symmetry by permuting the banana curves, as well as a flop symmetry (\ref{eqn:froptransformm}).  We proved that in the $(Q, q, y)$ variables these symmetries correspond to the $PGL_{2}(\mathbb{Z})$ invariances of even weight genus two Siegel modular forms.  It turns out that there is a $PGL_{2}(\mathbb{Z})$ arising naturally in the geometry of the smooth fibers $E \times E$ as well.  One should periodically refer back to Section \ref{subsec:SyymmmGWpotss}, as here we will be describing the symmetries on the singular fibers from the perspective of the smooth fibers.  

Let us denote points as $(p_{1}, p_{2}) \in E \times E$.  We have the following isomorphism 
\begin{equation}
GL_{2}(\mathbb{Z}) \overset{\sim}{\longrightarrow} \text{Aut}(E \times E)
\end{equation}
defined by producing automorphisms $\sigma$ of $E \times E$ as
\begin{equation}
\sigma(p_{1}, p_{2}) = (ap_{1} + bp_{2}, cp_{1} + dp_{2}) \,\,\,\,\,\,\,\,\,\,\,\,\,\,
\begin{pmatrix}
a & b\\
c & d
\end{pmatrix}
\in GL_{2}(\mathbb{Z}).
\end{equation}
The determinantal condition is necessary and sufficient for the above map to be invertible.  Strictly speaking, we get an isomorphism for generic elliptic curves $E$ without enhanced discrete automorphisms.  Just as in Section \ref{subsec:SyymmmGWpotss}, let us choose the following set of ordered generators of $GL_{2}(\mathbb{Z})$
\begin{equation}\label{eqn:GENGL2Z}
GL_{2}(\mathbb{Z}) = \langle \gamma_{1}, \gamma_{2}, \gamma_{3} \rangle \coloneqq \bigg\langle  
\begin{pmatrix}
1 & 0 \\ 0 & -1
\end{pmatrix},
\begin{pmatrix}
0 & 1 \\ 1 & 0
\end{pmatrix},
\begin{pmatrix}
1 & -1 \\ 0 & -1
\end{pmatrix}
\bigg\rangle
\end{equation}
and let us define the following automorphisms of $E \times E$ as the images of the generators under the above map
\begin{equation}
\sigma_{1}(p_{1},p_{2}) =(p_{1}, -p_{2}) , \,\,\,\,\,\, \sigma_{2}(p_{1}, p_{2}) = (p_{2}, p_{1}), \,\,\,\,\,\, \sigma_{3}(p_{1}, p_{2}) = (p_{1} - p_{2}, -p_{2}).
\end{equation}

Inside $E \times E$ we can use the locus $\{ (p,0) | p \in E \}$ as a representative of the class $E_{1}$, and similarly for $E_{2}$.  The diagonal class $\Delta$ can be represented by $\{(p,p) | p \in E \}$, and the anti-diagonal class $\Delta^{op}$ by the set $\{(p, -p) | p \in E \}$.  The isomorphism $GL_{2}(\mathbb{Z}) \cong \text{Aut}(E \times E)$ naturally induces a representation of $GL_{2}(\mathbb{Z})$ on $\text{NS}(E \times E)$ in the obvious way.  We interpret the $\sigma_{i}$ as lattice morphisms
\begin{equation}
\sigma_{i} : \text{NS}(E \times E) \longrightarrow \text{NS}(E \times E)
\end{equation}
which act in the following way on the generators of the Neron-Severi lattice 
\begin{equation}\label{eqn:sigmautEcE}
\begin{split}
& \sigma_{1} : \big(E_{1}, E_{2}, E_{3}\big) \longmapsto \big(E_{1}, E_{2}, -E_{3}\big) \\
& \sigma_{2} : \big(E_{1}, E_{2}, E_{3}\big) \longmapsto \big(E_{2}, E_{1}, E_{3} \big)\\
& \sigma_{3} : \big(E_{1}, E_{2}, E_{3}\big) \longmapsto \big( E_{1}, E_{1} + E_{2} + E_{3}, -2E_{1} -E_{3} \big).
\end{split}
\end{equation}

\begin{lemmy}
The action of $GL_{2}(\mathbb{Z})$ on $\text{NS}(E \times E)$ descends to a three-dimensional representation of $PGL_{2}(\mathbb{Z})$ on $\text{NS}(E \times E)$ by lattice automorphisms.  
\end{lemmy}

\begin{proof}
It is clear that $-1 \in GL_{2}(\mathbb{Z})$ acts on $E \times E$ by $(p_{1}, p_{2}) \mapsto (-p_{1}, -p_{2})$.  This is non-trivial as an automorphism, but is trivial at the level of homology.  The $PGL_{2}(\mathbb{Z})$ representation is defined by (\ref{eqn:sigmautEcE}) and it is easy to verify that these maps preserve the intersection form (\ref{eqn:intformsmmmmfibrr}).  
\end{proof}

Recall at the end of Section \ref{subsec:SyymmmGWpotss} we introduced the lattice $\big(V, ||\underline{\bf{d}}|| \big)$ generated by the banana curves $C_{1}, C_{2}, C_{3}$.  We define the lattice morphism
\begin{equation}
\alpha : \text{NS}(E \times E) \longrightarrow V
\end{equation}
via the relations (\ref{eqn:relnsmsingfibbrr}).  Equivalently, we associate a class in $E \times E$ to its image under a degeneration to a singular fiber.  It is easy to check that $\alpha$ preserves the quadratic forms, and is an injection.  Recalling that only half the classes in $V$ deform away from the singular fiber depending on the parity of $||\underline{\bf{d}}||$, we identify $\text{NS}(E \times E)$ as the index 2 sublattice of $V$ defined by those classes with $||\underline{\bf{d}}|| \equiv 0 \Mod 4$.  

The following theorem is a compatibility result between the $PGL_{2}(\mathbb{Z})$ representations on $\text{NS}(E \times E)$ and $\big(V, ||\underline{\bf{d}}|| \big)$ via the map $\alpha$.  

\begin{thm}
For $i=1,2$ the following diagram commutes    
\begin{equation}
\begin{tikzcd}
\text{NS}(E \times E) \arrow{r}{\alpha}\arrow[swap]{d}{\sigma_{i}}     &         V \arrow{d}{\gamma_{i}} \\
\text{NS}(E \times E) \arrow[swap]{r}{\alpha}                          &                   V
\end{tikzcd}
\end{equation}
For $i=3$, we must conjugate $\gamma_{3}$ by the flop symmetry $\gamma_{1}$.  Hence, if $\tilde{\gamma} \coloneqq \gamma_{1}^{-1} \circ \gamma_{3} \circ \gamma_{1}$, the following diagram commutes
\begin{equation}
\begin{tikzcd}
\text{NS}(E \times E) \arrow{r}{\alpha}\arrow[swap]{d}{\sigma_{3}}           &   V \arrow{d}{\tilde{\gamma}} \\
\text{NS}(E \times E) \arrow[swap]{r}{\alpha}                          &                   V
\end{tikzcd}
\end{equation}
\end{thm}

\begin{proof}
The proof follows easily from (\ref{eqn:actonbanncurvves}), (\ref{eqn:relnsmsingfibbrr}), and (\ref{eqn:sigmautEcE}).  First for $i=1$, we have
\[ \big(E_{1}, E_{2}, E_{3} \big) \xmapsto{\alpha} \big(C_{1} + C_{3}, C_{2} + C_{3}, 2C_{3} \big) \xmapsto{\gamma_{1}} \big(C_{1} + C_{3}, C_{2} + C_{3}, -2C_{3} \big)\]
and 
\[ \big(E_{1}, E_{2}, E_{3} \big) \xmapsto{\sigma_{1}} \big(E_{1}, E_{2}, -E_{3}\big) \xmapsto{\alpha} \big(C_{1} + C_{3}, C_{2} + C_{3}, -2C_{3} \big).\]
The argument is completely analogous for $i=2$.  For $i=3$, we follow the same direct method
\[\big(E_{1}, E_{2}, E_{3} \big) \xmapsto{\sigma_{3}} \big(E_{1}, E_{1}+E_{2}+E_{3}, -2E_{1}-E_{3}\big) \xmapsto{\alpha} \big(C_{1} + C_{3}, C_{1}+C_{2} + 4C_{3}, -2C_{1}-4C_{3} \big)\] 
and
\[
\begin{split}
\big(E_{1}, E_{2}, E_{3} \big) & \xmapsto{\alpha} \big( C_{1} + C_{3}, C_{2} + C_{3}, 2C_{3} \big) \xmapsto{\gamma_{1}} \big( C_{1} + C_{3}, C_{2} + C_{3}, -2C_{3} \big) \\
& \xmapsto{\gamma_{3}} \big(C_{1} + C_{3}, C_{1}+C_{2}, -2C_{1} \big) \xmapsto{\gamma_{1}^{-1}} \big(C_{1} + C_{3}, C_{1}+C_{2} + 4C_{3}, -2C_{1}-4C_{3} \big).
\end{split}
\]
\end{proof}

\section{The Case of the Schoen Calabi-Yau Threefold}

One takeaway from the previous sections is that the banana manifold fiberwise partition functions are given by certain standard arithmetic lifts of a modular object encoding the Gopakumar-Vafa invariants (the elliptic genus of $\mathbb{C}^{2}$).  We will show in this section that this phenomenon arises also for the Schoen manifold.  The Schoen manifold $X_{\text{Sch}}$ is defined as the fibered product
\begin{equation}
X_{\text{Sch}} = S \times_{\mathbb{P}^{1}} S'
\end{equation} 
of two distinct generic rational elliptic surfaces $S$ and $S'$.  There exists a natural projection $\pi: X_{\text{Sch}} \to \mathbb{P}^{1}$, and since every fiber of $\pi$ admits a free torus action, $\chi(X_{\text{Sch}})=0$.  There are 24 singular fibers of $\pi$, and the fiber classes $\Gamma = \text{ker}(\pi_{*}) \subset H_{2}(X_{\text{Sch}}, \mathbb{Z})$ form a rank 2 lattice generated by the classes $C_{1}, C_{2}$ of the two elliptic fibers.  

The Schoen manifold is related to the banana manifold through a conifold transition, and the Donaldson-Thomas partition function in fiber classes can be computed to be
\begin{equation}\label{eqn:SchhDT}
Z_{\text{DT}}(X_{\text{Sch}})_{\Gamma} = Z'_{\text{DT}}(X_{\text{Sch}})_{\Gamma} = \prod_{n=1}^{\infty}\big(1-Q^{n}\big)^{-12} \big( 1-q^{n} \big)^{-12}
\end{equation}
where $Q$ and $q$ are variables tracking degrees along the two fiber classes.  Because $\chi(X_{\text{Sch}})=0$, there is no factor of the MacMahon function.  Assuming the GW/DT correspondence, and taking the logarithm of (\ref{eqn:SchhDT}), we can extract the reduced Gromov-Witten potentials of $X_{\text{Sch}}$, all of which vanish except
\begin{equation}\label{eqn:Gen1GWpotSch}
F'_{1}(Q,q) = 12 \sum_{n=1}^{\infty} \bigg( \text{Li}_{-1}(1-Q^{n}) + \text{Li}_{-1}(1-q^{n})\bigg).
\end{equation} 
We denote by $(d_{1}, d_{2})$ the fiber class $d_{1} C_{1} + d_{2} C_{2}$.  All Gopakumar-Vafa invariants of $X_{\text{Sch}}$ vanish except in genus one, and in the class of the two elliptic fibers
\begin{equation}
n_{1, (1,0)}(X_{\text{Sch}}) = n_{1, (0,1)}(X_{\text{Sch}}) = 12.  
\end{equation}

Let us now make a seemingly vacuous observation.  We regard the Gopakumar-Vafa invariant 12 in a trivial sense, as the weight zero modular form 
\begin{equation}
12 = \sum_{n=0}^{\infty} c(n) q^{n}, \,\,\,\,\,\,\,\,\,\,\,\,\,\,\,\, c(n) = 12 \cdot \delta_{n,0}
\end{equation}
where $\delta_{n,0}$ is the Kronecker delta function.  Identifying the modular form 12 as a weight zero, index zero Jacobi form, we can compute the Maass lift $\text{ML}(12)$ via (\ref{eqn:fullHecke}) to be
\begin{equation}
\text{ML}(12) = \sum_{\substack{m,n \geq 0 \\ (m,n) \neq (0,0)}} c(mn) \text{Li}_{-1}\big(1-Q^{m}q^{n}\big) = 12 \sum_{n=1}^{\infty} \bigg( \text{Li}_{-1}(1-Q^{n}) + \text{Li}_{-1}(1-q^{n})\bigg).  
\end{equation}
Comparing with (\ref{eqn:Gen1GWpotSch}) and (\ref{eqn:SchhDT}), we conclude that 
\begin{equation}
F'_{1}(Q, q) = \text{ML}(12)
\end{equation}
and 
\begin{equation}
Z_{\text{DT}}(X_{\text{Sch}})_{\Gamma} =\text{BL}(12) = \text{exp} \big( \text{ML}(12) \big).  
\end{equation}
Therefore, the only non-vanishing Gromov-Witten potential of $X_{\text{Sch}}$ is the Maass lift of the constant 12, while the Donaldson-Thomas partition function is the formal Borcherds lift of 12.  

\emph{The main takeaway, and the key analogy with the banana manifold is the following:} the equivariant elliptic genus $12 \, \text{Ell}_{q,y}(\mathbb{C}^{2}; t)$ of $\mathbb{C}^{2}$ and 12 are both weight zero modular objects, and encode the fiberwise Gopakumar-Vafa invariants of $X_{\text{ban}}$ and $X_{\text{Sch}}$, respectively.  The formal Borcherds lifts give the respective Donaldson-Thomas partition functions.  Taking the $\lambda$-expansion of the elliptic genus (\ref{eqn:genexp}), the Maass lift of the component Jacobi forms give the Gromov-Witten potentials of $X_{\text{ban}}$, and the Maass lift of 12 gives the only non-vanishing Gromov-Witten potential of $X_{\text{Sch}}$.  


\bibliographystyle{plain}
\bibliography{MyLibrary}



\backmatter


\end{document}